\documentclass[english, 10pt]{amsart}
\usepackage{amsxtra,amscd,amsthm,amsmath,amssymb,mathtools}
\usepackage{longtable} \usepackage{bigstrut}
 \usepackage{enumerate}
\usepackage{verbatim}
\usepackage{graphicx}
\usepackage{xspace}
\usepackage{ifthen}
\usepackage{tabularx}
\usepackage[all,cmtip,line]{xy}

\usepackage{color}
\usepackage{mathrsfs}

\usepackage{amsthm}

\theoremstyle{plain}
\numberwithin{equation}{subsection}
\newtheorem{thm}[equation]{Theorem}

\newtheorem{prop}[equation]{Proposition}
\newtheorem{lemma}[equation]{Lemma}
\newtheorem*{lemma*}{Lemma}
\newtheorem{Definition}[equation]{Definition}

\newtheorem{Remark}[equation]{Remark}

\newtheorem{stass}[equation]{Standard assumptions}

\newtheorem{cor}[equation]{Corollary}

\theoremstyle{remark}
\newtheorem{para}[equation]{\bf}
\theoremstyle{plain}
\renewcommand{\subsubsection}{\addtocounter{equation}{1}{\vskip 6pt \bf\arabic{section}.\arabic{subsection}.\arabic{equation}}}
\theoremstyle{definition}

\newcommand{\quash}[1]{}  %%Anything in \quash is ignored
\newcommand{\nc}{\newcommand}
\nc{\on}{\operatorname}

\newcommand{\lps}{[\![}
\newcommand{\rps}{]\!]}
\newcommand{\llps}{(\!(}
\newcommand{\lrps}{)\!)}

\newcommand{\ur}{{\rm ur}}
\renewcommand{\phi}{\varphi}
\newcommand{\loc}{{\rm loc}}

\newcommand{\Isom}{{\rm Isom}}

\newcommand{\Lie}{{\rm Lie\,}}
\newcommand{\der}{{\rm der}}
\newcommand{\ab}{{\rm ab}}

\newcommand{\scrG}{{\mathscr{G}}}

\newcommand{\fraka}{{\mathfrak a}}

\newcommand{\frakm}{{\mathfrak m}}

\newcommand{\frakM}{{\mathfrak M}}

\newcommand{\Sh}{{\rm Sh}}
\newcommand{\SSh}{{\mathscr S}}
\newcommand{\frakS}{{\mathfrak S}}
\newcommand{\gS}{{\mathfrak S}}

\newcommand{\bbA}{{\mathbb A}}

\newcommand{\bbC}{{\mathbb C}}
\newcommand{\bbD}{{\mathbb D}}
\newcommand{\DD}{{\mathbb D}}

\newcommand{\bbG}{{\mathbb G}}
\newcommand{\bbH}{{\mathbb H}}

\newcommand{\bbM}{{\mathbb M}}
\newcommand{\wtM}{{\widetilde M}}

\newcommand{\bbQ}{{\mathbb Q}}
\newcommand{\bbR}{{\mathbb R}}
\newcommand{\bbS}{{\mathbb S}}

\newcommand{\bbZ}{{\mathbb Z}}
\newcommand{\calA}{{\mathcal A}}
\newcommand{\calB}{{\mathcal B}}

\newcommand{\calE}{{\mathcal E}}

\newcommand{\calG}{{\mathcal G}}

\newcommand{\calH}{{\mathcal H}}

\newcommand{\calJ}{{\mathcal J}}

\newcommand{\calL}{{\mathcal L}}

\newcommand{\calN}{{\mathcal N}}
\newcommand{\calO}{{\mathcal O}}

\newcommand{\calT}{{\mathcal T}}

\newcommand{\calV}{{\mathcal V}}
\newcommand{\calW}{{\mathcal W}}

\newcommand{\whW}{{\widehat W}}

\newcommand{\calZ}{{\mathcal Z}}

\newcommand{\ov}{\overline}

\newcommand{\mF}{\ensuremath{\mathbb{F}}\xspace}

\newcommand{\brQ}{\ensuremath{\breve{\mathbb{Q}}_p}}

\nc{\al}{{\alpha}} \nc{\be}{{\beta}}
\newcommand{\ga}{{\gamma}}
\nc{\ve}{{\varepsilon}} \nc{\Ga}{{\Gamma}}
\newcommand{\la}{{\lambda}}
\nc{\La}{{\Lambda}}

\def\0{\circ}

\newcommand{\cal}{\mathcal}

\newcommand{\fa}{\mathfrak a}

\newcommand{\G}{{\bf G}}

\newcommand{\br}{\breve}

\newcommand{\C}{{\mathbb C}}
\newcommand{\R}{{\mathbb R}}
\newcommand{\Q}{{\mathbb Q}}

\newcommand{\ep}{\epsilon}
\newcommand{\et}{{\text{\rm \'et}}}
\newcommand{\B}{{\mathcal B}}

\newcommand{\Hom}{{\rm Hom}}

\newcommand{\Gg}{{\mathcal G}}

\newcommand{\Hh}{{\mathcal H}}
\newcommand{\Gm}{{{\mathbb G}_{\rm m}}}
\newcommand{\Z}{{\mathbb Z}}
\newcommand{\F}{{\mathcal F}}
\newcommand{\ti}{\tilde}
\newcommand{\Spec}{{\rm Spec \, } }
\newcommand{\Spf}{{\rm Spf } }
\newcommand{\ad}{{\rm ad } }
 \renewcommand{\O}{{\mathcal O}}

\newcommand{\Gr}{{\rm Gr}}

\newcommand{\GL}{{\rm GL}}

\newcommand{\und}{\underline}

\renewcommand{\L}{{\mathcal L}}
\newcommand{\K}{{\mathcal K}}

\newcommand{\Res}{{\rm Res}}

\newcommand{\Tr}{{\rm Tr}}
\newcommand{\Ind}{{\rm Ind}}
\newcommand{\Rep}{{\rm Rep}}

\newcommand{\N}{{\mathcal N}}

\newcommand{\rM}{{\rm M}}

\def\thfill{\null\nobreak\hfill}

\def\endproof{\thfill\vbox{\hrule
  \hbox{\vrule\hbox to 5pt{\vbox to 5pt{\vfil}\hfil}\vrule}\hrule}}

\newcommand{\Gal}{{\rm Gal}}

\newcommand{\GSp}{{\rm GSp}}
\newcommand{\GO}{{\rm GO}}

\newcommand{\wh}{\widehat}

\newcommand{\Gam}{\Gamma}

\newcommand{\hook}{\hookrightarrow}

\newcommand{\Mloc}{{\rm M}^{\rm loc}}

\newcommand{\BMloc}{{\mathbb M}^{\rm loc}}

\newcommand{\un}{{\rm un}}

\newcommand{\fkT}{\ensuremath{\mathfrak{T}}}

\newcommand{\fkb}{\ensuremath{\mathfrak{b}}}

\newcommand{\fkg}{\ensuremath{\mathfrak{g}}}

\newcommand{\fkl}{\ensuremath{\mathfrak{l}}}
\newcommand{\fkm}{\ensuremath{\mathfrak{m}}}

\newcommand{\fkp}{\ensuremath{\mathfrak{p}}}

\newcommand{\fks}{\ensuremath{\mathfrak{s}}}
\newcommand{\fkt}{\ensuremath{\mathfrak{t}}}

\newcommand{\SL}{\ensuremath{\mathrm{SL}}}

\newcommand{\cur}{\ensuremath{\mathrm{cur}}}   
\newcommand{\FM}{\ensuremath{\mathrm{FM}}}
\newcommand{\lleq}{\ensuremath{\preccurlyeq}}
\newcommand{\ggeq}{\ensuremath{\succcurlyeq}}
\newcommand{\dom}{\ensuremath{\mathrm{dom}}}

\newcommand{\bx}{\ensuremath{\mathbf x}}

\newcommand{\bk}{\ensuremath{\mathbf k}}

\newcommand{\rmF}{{\mathrm{F}}}

\newcommand{\rmH}{{\mathrm{H}}}

\newcommand{\rmK}{{\mathrm{K}}}

\newcommand{\bfC}{{\mathbf{C}}}

\newcommand{\bfE}{{\mathbf{E}}}

\newcommand{\bfG}{{\mathbf{G}}}
\newcommand{\bfH}{{\mathbf{H}}}

\newcommand{\bfS}{{\mathbf{S}}}
\newcommand{\bfT}{{\mathbf{T}}}

\newcommand{\bfZ}{{\mathbf{Z}}}
\newcommand{\bfGSp}{{\mathbf{GSp}}}

\newenvironment{altenumerate}
   {\begin{list}
      {\textup{(\theenumi)} }
      {\usecounter{enumi}
       \setlength{\labelwidth}{15pt}
       \setlength{\labelsep}{5pt}
       \setlength{\leftmargin}{24pt}
       \setlength{\itemsep}{\the\smallskipamount}
       \renewcommand{\theenumi}{\arabic{enumi}}
      }}
   {\end{list}}
\newenvironment{altitemize}
   {\begin{list}
      {$\bullet$}
      {\setlength{\labelwidth}{15pt}
	   \setlength{\itemindent}{0pt}
       \setlength{\labelsep}{5pt}
       \setlength{\leftmargin}{24pt}
       \setlength{\itemsep}{\the\smallskipamount}
      }}
 {\end{list}}

\begin{document}

%\today

\title[Integral models of Shimura varieties]{Integral models of Shimura varieties with parahoric level structure, II} 

\author{Mark Kisin} 
\address{Department  of Mathematics, Harvard University, Cambridge, MA 02138, USA}
\email{kisin@math.harvard.edu}
 
\author{Georgios Pappas} 
\address{Department  of Mathematics, Michigan State University, E. Lansing, MI 48824, USA}
\email{pappasg@msu.edu}

\author{Rong Zhou}
\address{Department of Pure Mathematics and Mathematical Statistics, University of Cambridge, Cambridge, UK, CB3 0WA}
\email{rz240@dpmms.cam.ac.uk}

\begin{abstract}
We construct integral models of Shimura varieties of abelian type with parahoric level structure over odd primes.
These models are \'etale locally isomorphic to corresponding local models.
\end{abstract}

\date{\today}
\maketitle

\tableofcontents
 
\addtocontents{toc}{\protect\setcounter{tocdepth}{2}}

\section{Introduction}

 \subsection{Statement of the main result}\label{ss:Statement}

Let $(\G, X)$ be a Shimura datum in the  sense  of Deligne \cite{DelBour}, \cite{DeligneCorvallis}, so that $\G$ is a
reductive group over  $\Q$ and $X$ is a $\G_{\mathbb R}$-conjugacy class of homomorphisms
$h: {\mathbb S}={\rm Res}_{\C/\R}{\mathbb G}_m\to \G_{\R}$, satisfying the assumptions in \emph{loc. cit..} Let ${\mathbb A}_f$ denote the finite adeles of $\Q$ and suppose $\rmK\subset \G({\mathbb A}_f)$ is a open compact subgroup. The  Shimura variety $\Sh_{\rmK}(\G, X)$  is defined over  the reflex number field ${\mathbf E}\subset \C$ and has  complex points given by the double quotient
\[
\Sh_{\rmK}(\G, X)(\C)=\G(\Q)\backslash (X \times \G({\mathbb A}^f)/\rmK).
\]

The varieties $\Sh_{\rmK}(\G, X)$ are important for many applications in number theory, which often require a study of corresponding integral models. These are schemes which extend  $\Sh_{\rmK}(\G, X)$ over the 
ring of integers $\O_{\mathbf E}$ of $\mathbf E$, or over localizations or completions of $\O_{\mathbf E}$. In this paper, we consider the completions of $\O_{\mathbf E}$ at primes of $\mathbf E$ which lie over an odd rational prime $p$. We construct 
integral models over these completions when the Shimura datum $(\G, X)$ is of abelian type and the level subgroup $\rmK$ is parahoric or a stabilizer at $p$; we will explain these terms below. Our results extend the construction of \cite{KP} to all Shimura varieties of abelian type over odd primes. In particular, we dispense with the blanket restriction in \emph{op. cit.} that the group $\G$ splits over an extension of $\Q$ which is tamely ramified over $p$. In addition, we correct a serious gap in \cite{KP}
which also propagated to previous versions of \cite{KZhou}, see paragraphs \ref{par:Explain} and \ref{ss:ReltoKP}.

Recall that $(\G, X)$ is said to be of \emph{Hodge type} if there is an embedding
 $(\G, X)\hookrightarrow ({\mathbf {GSp}}_{2g}, S^{\pm })$ into the Shimura datum for a symplectic similitude group. This implies that the corresponding Shimura variety
 $\Sh_{\rmK}(\G, X)$ can be described as a moduli space for abelian varieties equipped with certain Hodge cycles. A Shimura datum
$(\G, X)$ is said to be of \emph{abelian type} if there is a datum of Hodge type $(\bfG_1,X_1)$
and a central isogeny between the derived groups $\G^\der_1\to \G^\der$ which induces an isomorphism 
$(\G^\ad_1, X^\ad_1)\xrightarrow{\sim} (\G^\ad, X^\ad)$. The class of Shimura data of abelian type is very general and includes 
almost all cases in which $\G$ is a classical group.

Now let us discuss the assumption on the level subgroup. We fix a prime $p>2$ and a prime $v$ of $\mathbf E$ which lies above $p$. Let $\Gg$ be a Bruhat-Tits \emph{stabilizer}  group scheme over $\Z_p$ with generic fiber the base change $G=\G_{\Q_p}$; this stabilizer is defined using the action of the group on its  affine building.  
The $\Z_p$-points of $\Gg$ give a level subgroup $\rmK_p=\Gg(\Z_p)\subset G(\Q_p)$ at $p$.
The corresponding \emph{parahoric} group scheme is the neutral connected component $\Gg^\circ$ of $\Gg$; we also consider the parahoric level subgroup $\rmK^\circ_p=\Gg^\circ(\Z_p)$ at $p$. Let ${\mathbb A}^p_f$ be the prime to $p$ finite adeles and let $\rmK^p\subset \G({\mathbb A}^p_f)$ be a sufficiently small compact open subgroup. We take the level subgroup to be $\rmK^\circ=\rmK^p\rmK^\circ_p\subset \G({\mathbb A}_f)$ or $\rmK=\rmK^p\rmK_p\subset \G({\mathbb A}_f)$
and consider $\Sh_{\rmK^\circ}(\G, X)$ or $\Sh_{\rmK}(\G, X)$.\footnote{Some of our results can be extended to the case that $\rmK_p$ is \emph{quasi-parahoric}, i.e. it lies between a stabilizer and its corresponding parahoric, see \cite{DvHKZ}.} %However,  we do not pursue this in this paper.
Note that our assumption that  the level subgroup at $p$ is either a parahoric or a stabilizer   is quite natural. It allows all cases with $\Gg$ reductive, when we have smooth reduction at $v$ (\cite{KisinJAMS}),
 but also includes many Shimura varieties with non-smooth reduction. In fact, for any reductive $\G$  over $\Q$ and prime $p$, the group $G(\Q_p)$ always contains parahoric subgroups. The case of stabilizer level  $\rmK_p$ naturally occurs when considering Shimura varieties described as moduli schemes and more generally for $(\G, X)$ of Hodge type and it plays a central role in the proofs.

Set $E={\mathbf E}_v$ for the completion at $v$. Our goal is to construct $\O_E$-integral models for $\Sh_{\rmK^\circ}(\G, X)$  and $\Sh_{\rmK}(\G, X)$
which satisfy two requirements roughly as follows; they are both important for applications.
First, the integral model is ``as proper as possible", i.e. 
it does not miss points in positive characteristic that should appear as reductions of points of the Shimura variety. Second, the \'etale local structure of the integral model is controlled  by a corresponding local model. We refer the reader to \cite{Picm} and \cite{PRS} for an account of past work on such integral models and local models.

Before stating the main result, we briefly recall some basic information about local models; these play a crucial role in the theory. Let $\{\mu\}$ be the geometric conjugacy class of the cocharacter $\mu=\mu_h$ of $G$ which corresponds to the hermitian symmetric domain $X$. 
The local model $\bbM^{\mathrm{loc}}_{\Gg, \mu}=\bbM^{\mathrm{loc}}_{\Gg^\circ, \mu}$ is associated to the triple $(G,\{\mu\}, \Gg^\circ)$, see \S \ref{ss:LM}.
It is a flat and normal proper scheme  over $\O_E$ and supports a $\Gg$-action with a finite number of orbits. Its generic fiber is the  homogeneous space over $E$ parametrizing parabolics in the conjugacy class of $P_{\mu^{-1}}$,  the parabolic subgroup corresponding to the minuscule cocharacter $\mu^{-1}$, and is an $E$-form of the compact dual of  $X$. Its special fiber is reduced and, in fact,  $\bbM^{\mathrm{loc}}_{\Gg, \mu}$ is uniquely determined by its corresponding $v$-sheaf on perfectoid spaces, which is given a priori by Scholze-Weinstein \cite{Schber}.

\begin{para}
The main theorem of this paper is the following:

\begin{thm}\label{introthm: LMD abelian type}
	Assume $p>2$. Let $(\bfG,X)$ be a Shimura datum  of abelian type and $\rmK_{p}^\circ=\calG^\circ(\bbZ_p)$ a parahoric subgroup. There exists a pro-system of $\calO_{E}$-schemes   $\mathscr{S}_{\rmK^\circ_{p}\rmK^p}(\bfG,X)$ with generic fibers $\Sh_{\rmK_p^\circ\rmK^p}(\bfG,X)$ and with finite \'etale transition maps, for varying sufficiently small $\rmK^p\subset \G({\mathbb A}^p_f)$, such that the $\calO_{E}$-scheme
 \[
 \mathscr{S}_{\rmK^\circ_{p}}(\bfG,X)=\varprojlim_{\rmK^p}\SSh_{\rmK^\circ_{p}\rmK^p}(\bfG,X)
 \]
 with $\bfG(\bbA^p_f)$-action extends  $\Sh_{\rmK^\circ_{p}}(\bfG,X)=\varprojlim_{\rmK^p}\Sh_{\rmK^\circ_{p}\rmK^p}(\bfG,X)$ and satisfies

	\begin{altenumerate}

		\item For  $R$  a discrete valuation ring of mixed characteristic  $(0,p)$, the map 
		$$\mathscr{S}_{\rmK^\circ_{p}}(\bfG,X)(R)\rightarrow\Sh_{\rmK^\circ_{p}}(\bfG,X)(R[1/p])$$ 
		is a bijection.  
		
		\item For $\rmK^p$ a sufficiently small compact open subgroup, 
$	\SSh_{\rmK^\circ_{p}\rmK^p}(\bfG,X)	$
		 is \'etale locally isomorphic to $\bbM^{\mathrm{loc}}_{\calG,\mu}$.
	
		\item There  exists a diagram 
\[\begin{aligned}
		\xymatrix{ &\widetilde{\mathscr{S}}^{\mathrm{ad}}_{\rmK^\circ_{p}}(\bfG,X)\ar[dr]^q\ar[dl]_\pi&\\
			\mathscr{S}_{\rmK^\circ_{p}}(\bfG,X) & &\bbM^{\mathrm{loc}}_{\Gg, \mu },}
\end{aligned}\]
where the morphism $\pi$ is a $\bfG(\bbA_f^p)$-equivariant ${\calG}^{\ad}$-torsor and the morphism $q$ is ${\calG}^{\ad}$-equivariant, smooth and 
		$\bfG(\bbA_f^p)$-equivariant, when $\bbM^{\mathrm{loc}}_{\calG, \mu}$ is equipped with the trivial $\bfG(\bbA_f^p)$-action.
				If in addition $(\bfG,X)$ is (NE), then  $\pi$ reduces to a ${\calG}^{\ad,\circ}$-torsor which still maps to $\bbM^{\mathrm{loc}}_{\calG, \mu}$ via the restriction of $q$.
	\end{altenumerate}
\end{thm}

Above, $\Gg^\ad$ is a smooth group scheme over $\Z_p$ with generic fiber the adjoint group $G^\ad$ of $G$.
It is not necessarily a stabilizer group scheme.
The neutral connected component $\calG^{\ad,\circ}$ is the parahoric group scheme of the adjoint group $G^{\ad}$ associated to $\calG$,
see Theorem \ref{thm: LMD abelian type} and \S \ref{par:7.1.12} in the text.   Using the smoothness of $\Gg$ one sees that (3) implies (2).   
The condition (NE) in the statement is explained in paragraph \ref{ss:strategyProof} below. In fact, by combining this result with work of Daniels--van Hoften--Kim--Zhang \cite{DvHKZ} which uses the theory of $p$-adic shtukas, we see that the condition (NE) can be removed, cf. Corollary \ref{cor: torsor for connected gp}. 

\begin{cor}\label{corintro:PlusShtukas}
The $\calG^{\ad}$-torsor $\pi$ in  Theorem \ref{introthm: LMD abelian type} (3) can be refined to a $\calG^{\ad,\circ}$-torsor and this fits in a $\calG^{\ad,\circ}$-equivariant local model diagram refining the diagram in Theorem \ref{introthm: LMD abelian type} (3).
\end{cor}

Here, by a \emph{local model diagram} we mean a diagram of morphisms with the smoothness and equivariance properties mentioned in (3) above. We will also give more precise results that refine the diagram (3) under certain additional conditions, and similar results for the Shimura (pro-)varieties $\Sh_{\rmK_p}(\G, X)$ with stabilizer level subgroup $\rmK_p$. The reader is referred to \S \ref{sec:SV} for these. 
 In addition, we refer the reader to paragraph 1.3 of this introduction for a discussion of other related results and, in particular, for a comparison 
with corresponding statements in \cite{KP} and previous versions of \cite{KZhou}.
 \end{para}

\begin{para}

The results of this paper have several applications. 

The integral models we construct are  used in \cite{KZhou} to show $\ell$-independence of Frobenius conjugacy classes for abelian varieties. The proof in \emph{loc. cit.} uses the existence of the local model diagram in Theorem \ref{introthm: LMD abelian type} to define the Kottwitz--Rapoport stratification on the models in order to apply an ``induction on strata'' argument. 
 
 The local model diagram is also used as a crucial input in determining the local zeta function at $p$ of the Shimura variety via the Langlands-Kottwitz method in \cite{HZZ}, as it allows us to understand the nearby cycles at points on the special fiber of integral models. 
 
 As explained below, the proof of Theorem \ref{introthm: LMD abelian type} uses the construction of the universal deformation space of a $p$-divisible groups equipped with crystalline  tensors. This
construction is applied in a different way to prove the representability of integral local Shimura varieties of abelian type in \cite{PRlsv}. 
\end{para}

 \subsection{Strategy of the proof}\label{ss:strategyProof}
We will now discuss the proof of Theorem \ref{introthm: LMD abelian type}. The overall strategy is the same as in \cite{KP} which covered only tamely ramified groups $G$. However, there is a complication:  An important condition which is necessary for the construction was erroneously omitted in \emph{loc. cit.}. As we will explain below, the condition is needed for the construction in \cite[\S 3]{KP} of the universal deformation of a $p$-divisible group equipped with crystalline tensors; the error was brought to the authors' attention by Manuel Hoff, see \cite[Rem. 2.29]{Hoff} and Remark \ref{remark:Hoff}.

In this paper, we correct the omission in \cite{KP} and also explain why this condition is satisfied in enough cases so that the proofs go through. 
In addition, we provide simplifications and generalizations of several other arguments of \emph{loc. cit.}. As a result, we can now also cover all groups $\G$ with $(\G,X)$ of abelian type.

\begin{para}\label{par:Explain}
Let us explain  this in some more detail: Suppose that the Shimura datum $(\G, X)$ is of Hodge type; this is the crucial case. The argument in \cite{KP} starts by finding a   Hodge embedding $\rho: (\G, X)\hook ({\mathbf {GSp}}(V,\psi), S^{\pm})$ and a $\Z_p$-lattice $\La$ in the $\Q_p$-vector space $V_{\Q_p}$ such that $G\hook \GL(V_{\Q_p})$ extends to a closed immersion of group schemes $\Gg\hook \GL(\La)$.   
Moreover, it is arranged so that the alternating form $\psi$ takes $\Z_p$-integral values on $\La$. Then $\rho$ induces an embedding of the Shimura variety $\Sh_{\rmK}(\G, X)$
in a Siegel moduli variety of polarized abelian schemes with appropriate level structure. This level structure is determined by a subgroup $\rmK'$ of the adelic symplectic similitude group whose choice depends on $\rmK$. This Siegel variety has  a $\Z_p$-integral model $\calA_{g, \rmK'}$ given by the natural extension of the moduli functor to schemes over $\Z_p$. Then, the normalization  of the Zariski closure of $\Sh_{\rmK}(\G, X)$ in $\calA_{g, \rmK'}\otimes_{\Z_p}\O_E$ gives an $\O_E$-integral model $\mathscr{S}_{\rmK}(\G,X)$ of $\Sh_{\rmK}(\G, X)$. Even if the notation does not indicate this, the scheme $\mathscr{S}_{\rmK}(\G,X)$ a priori depends on the above choices of the Hodge embedding and the lattice. 

The essential point now becomes to control the structure of $\mathscr{S}_{\rmK}(\G,X)$. In particular, the desired result is that  $\mathscr{S}_{\rmK}(\G,X)$ is  \'etale locally isomorphic to the local model  $\BMloc_{\Gg,\mu}$.
In fact, one aims for a more precise result: the existence of a local model diagram. 
This  amounts to a smooth morphism 
\[
 \mathscr{S}_{\rmK}(\G,X)\xrightarrow{ \ \ } [\Gg\backslash\,\BMloc_{\Gg,\mu}],
\]
with target the stack quotient of the $\Gg$-scheme $\BMloc_{\Gg,\mu}$. 

To achieve this control, we need to choose the Hodge embedding and the lattice $\La$ carefully. We first arrange so that the embedding $\Gg\hook\GL(\La)$ 
induces a closed immersion $\BMloc_{\Gg,\mu}\hook {\rm Gr}(d,\La)_{\O_E}$ of the local model in the base change of a Grassmannian scheme, where $d$ depends on $\mu$. When this closed immersion occurs, we say that we have an integral local Hodge embedding $(\Gg,\{\mu\})\hook (\GL(\La),\mu_d)$ which is ``good". In what follows, we assume that this has been arranged.

We now consider a finite collection of tensors $(s_a)$ in the tensor algebra $V_{\bbZ_{(p)}}^\otimes$ which ``cut out'' $\G_{\bbZ_{(p)}}$, cf. \S \ref{ss:VG definition}. Here $V_{\bbZ_{(p)}}$ is the unique 
$\Z_{(p)}$-lattice in $V$ whose $p$-adic completion $V_{\bbZ_p}$ is $\La$ and $\G_{\bbZ_{(p)}}$ the unique   affine $\Z_{(p)}$-model of $\G$ whose $p$-adic completion is $\Gg$. Then $(s_a)$ also cut out $\Gg$ in $\GL(\La)$.
The Betti-\'etale comparison isomorphism gives corresponding  tensors $s_{a,\et}\in \calV_p^\otimes$, where $\calV_p$ is the $\bbZ_p$-local system on $\Sh_{\rmK}(\bfG,X)$ corresponding to the  dual of the $p$-adic Tate-module of the pullback of the  universal abelian variety. 

Now consider $x\in  \mathscr{S}_{\rmK}(\G,X)(k)$, where $k$ is an algebraic closure 
of the residue field $k_E$ of $\O_E$ and set $\breve\bbQ_p=W(k)[1/p]$. We let $\scrG_x$ denote the $p$-divisible group of the abelian variety associated to $x$ and let $\bbD$ be the Dieudonn\'e module of $\scrG_x$. For a finite field extension $K/\breve\bbQ_p$  and  $\tilde{x}\in\SSh_{\rmK}(\bfG,X)(\calO_K)$ a point lifting $x$, the $p$-adic comparison isomorphism gives rise to tensors $s_{\alpha,0}\in \bbD[1/p]^\otimes$. These tensors lie in the submodule $\bbD^\otimes$ and are independent of the choice of lift $\tilde{x}$. Moreover, the scheme of tensor preserving isomorphisms $\underline{\Isom}_{s_a,s_{a,0}}(\La\otimes_{\Z_p}W(k),\bbD)$ is a trivial $\calG$-torsor and we can choose an identification $\bbD=\La\otimes_{\Z_p}W(k)$ matching $s_{a,0}$ with $s_a\otimes 1$. These facts follow by the argument in \cite[\S 3.3]{KP} using the general purity result of \cite{An} to cover the case of a general $\G.$   
We also see that the de Rham filtration on $\bbD\otimes_{W(k)}k$ corresponds to a point $y\in \Gr(d, \Lambda)(k)$ which lies in $\BMloc_{\calG,\mu}(k)$. 

Let $A $ denote the completion of the local ring of $\BMloc_{\calG,\mu}$  at $y$. (In the text this is usually denoted by $R_G$.) The crux of the matter is to show that the completion of the local ring of $\SSh_{\rmK}(\bfG,X)$ at $x$ is isomorphic to $A$. Roughly speaking, this follows if we construct a suitable   deformation of the $p$-divisible group $\scrG_x$  over $A$ which is equipped with tensors extending $s_{a,0}$. When $\Gg$ is reductive such a   deformation is given in \cite{KisinJAMS} following a construction of Faltings. For the general case, \cite{KP} use Zink's theory of displays. In the following discussion, we will use the usual notations of this theory, see \S \ref{1.1}, \cite[\S 3]{KP}.

Set $M=\La\otimes_{\Z_p}\whW(A)$ and denote by $\hat I_{A}M\subset M_1\subset M$ the unique $\whW(A)$-submodule corresponding to the   $A$-valued point of the Grassmannian given by $\BMloc_{\Gg,\mu}\hook {\rm Gr}(d,\La)_{\O_E}$.  To the ``Dieudonn\'e pair'' $(M, M_1)$, we  associate a finite free $\whW(A)$-module $\wtM_1$ with 
\[
p\phi^*M\subset \wtM_1 \subset \phi^*M.
\]

Now set $\fa =\frakm ^2+\pi_EA  \subset A $, where $\frakm $ is the maximal ideal of $A$ and $\pi_E$ a uniformizer of $E$. There is a canonical ``infinitesimal connection'' isomorphism 
\begin{equation*} 
c:  \widetilde{\bbD}_{1}\otimes_{W(k)}\whW(A /\fa )\xrightarrow{\sim} \wtM_1\otimes_{\whW(R )}\whW(A /\fa ),
\end{equation*}
see Lemma \ref{319}.
Here, $\widetilde{\bbD}_{1}$ is the $W(k)$-submodule of $\phi^*\bbD$ obtained
by the same construction but over $k$. 

The tensors $ \ti s_a:=s_a\otimes 1\in  \La^\otimes\otimes_{\Z_p}\whW(A)=(\phi^*M)^\otimes$ lie in $\wtM_1^{\otimes}$. Similarly, $s_{a,0}\in \bbD^{\otimes}$ lie in $\widetilde{\bbD}_{1}^{\otimes}$.   
We say ``the  tensor $\ti s_a$ is horizontal'' if
\[
c( s_{a,0}\otimes 1)=\ti s_a\otimes 1.
\]
If this holds for all $\ti s_a$, then the arguments in \cite{KP} construct the desired   deformation of the $p$-divisible group $\scrG_x$ over $A$ and the rest follows. 

However, it is not clear that the tensors $\ti s_a$ are horizontal in general. This is implicitly claimed to hold in \cite[3.2.12]{KP} but the argument depends on an erroneous construction  of the isomorphism $c$ in \cite[Lem. 3.1.9]{KP}, see the proof of Lemma \ref{319} for more details. 

When $\Gg $  is cut out 
by tensors $(s_a)\subset \La^\otimes$ such that all $\ti s_a$ are horizontal, we say that the integral Hodge embedding is ``very good'' (Definition \ref{def:vgood}). The constructions of \cite{KP} carry through under this additional condition, see Theorem \ref{thm:  main SV Hodge}.  Much of the work in the current paper is about showing that we can almost always choose an integral Hodge embedding which is very good. In fact, we conjecture that any good integral Hodge embedding is also very good, though we are not able to show this in general.
\end{para}

 \begin{para}
The main technique we use to produce sufficiently many very good embeddings relies on the following two properties. We let $s_a\in \Lambda^\otimes$ be fixed by $\Gg$ and $\ti s_a\in \wtM_1^\otimes$ the corresponding tensor.
\begin{altenumerate}

	\item If the tangent space $\BMloc_{\Gg,\mu}\otimes_{\O_E}k$ at $y$ is spanned as a $k$-vector space by the images of tangent spaces
	 of smooth formal curves, then $\ti s_a$ is horizontal; see Definition  \ref{def:span}, Proposition \ref{span}.
	
	\item If $s_a$ is an endomorphism (i.e. $s_a\in \Lambda\otimes_{\bbZ_p} \Lambda^\vee$), then $\ti s_a$ is horizontal;  see Lemma \ref{endo}.
\end{altenumerate}

 To produce very good embeddings, we first show  (Theorem \ref{corLMSpan}):

\begin{thm}\label{thm:spanIntro}
Let $(G,\{\mu\},\Gg)$ be a local model triple with $\Gg=\Res_{\O_F/\Z_p}\Hh$, the restriction of scalars of a   reductive group scheme $\Hh$ of classical type over $\O_F$. Suppose that the pair $(G^\ad, \{\mu^\ad\})$ is of abelian type and does not have a factor of type $D^{\mathbb H}$.
 Then the tangent spaces
 of $\BMloc_{\Gg,\mu}\otimes_{\O_E}k$ at all $k$-points are spanned by smooth formal curves.
 \end{thm}

 To prove this, we view $\BMloc_{\Gg,\mu}\otimes_{\O_E}k$ as a union of Schubert varieties in an affine Grassmannian for a certain equicharacteristic group over $k\lps t\rps$ which is of the same type as $\Hh$.  The smooth formal curves  are produced by using the curves coming from (conjugates of) the unipotent groups associated to affine roots.  The tangent directions spanned by these curves are then compared to an upper bound for the tangent space of $\BMloc_{\Gg,\mu}\otimes_{\O_E}k$ arising from a construction which is motivated by a conjectural modular description of Schubert varieties due to Finkelberg--Mirkovic \cite{FM}, see also  \cite[\S6]{HainesPlucker}. A detailed combinatorial analysis of these bounds carried out in \S\ref{s:RootCurves}, which may be of independent interest,  then proves the spanning property in the above cases,  see Theorem \ref{thm:curve span}. By property (1) above, this ensures that for any such group, a good embedding is also very good.

This remarkable property of  tangent spaces does not hold for the  local models of general stabilizer group schemes $\Gg$.  For example, it fails for $\Gg={\rm Res}_{\O_F/\Z_p}{\mathcal I}$, when $F/\Q_p$ is a ramified quadratic extension and $\mathcal I$ is an Iwahori group scheme for $\GL_{2}/F$, see Remark \ref{remSpanHilbert} (2).  However, we can still handle most of these cases as follows: We first present stabilizer group schemes  as the (tame) Galois fixed points of the Weil restriction of scalars of split reductive group schemes. This presentation is shown by applying a -more or less- standard argument with subdivision of apartments in the corresponding Bruhat-Tits buildings and crucially uses that $p$ is odd, see Proposition \ref{Fixed}. Tameness is important here so we can apply ``Edixhoven's lemma":  The fixed point locus of a tame finite group action on a smooth scheme is smooth. Now consider the fixed point group scheme $\Res_{\O_{K'}/\O_K}\GL(\La')^\Ga$ where $\La'$ is an $\O_{K'}$-lattice which is stable under the Galois group $\Ga=\Gal(K'/K)$. This fixed point scheme is cut out in $\Res_{\O_{K'}/\O_K}\GL(\La')$ by the endomorphisms of $\La'$ (considered as a $\O_K$-lattice by restriction of scalars) which are given by the Galois action. We then use this observation to show that there is a good Hodge embedding in which the group $\Gg$ is the stabilizer of the union of two sets of tensors: the first cuts out the Weil restriction of scalars of a split group and the second is given by endomorphisms, cf. Proposition \ref{prop: better lattice}. Since tensors given by endomorphisms  are always horizontal we can combine with the above to conclude that we have a very good embedding, cf. Theorem \ref{thm:main}.

The above argument cannot handle directly two types of ``exceptional cases": The first is when $(G^\ad, \mu^\ad)$ contains factors of type $D^{\mathbb H}_n$. The second is when  the adjoint group $G^\ad$ contains factors of the form ${\rm Res}_{F/\Q_p}{\rm PGL}_m(D)$, where $D$ is a central division algebra over $F$ with index divisible by $p$.  We call these cases ``exceptional type D'' and ``exceptional type A'' respectively. When  $(G^{\ad},
\mu^{\ad})$ does not contain factors of these forms, we say that $(\bfG,X)$ is ``non-exceptional" \emph{(NE)}, see \S\ref{sec: NE cases}. The reason for the first exception was already mentioned above. The second exception occurs because, in that case, the stabilizer group schemes cannot be written as the tame Galois fixed points of the Weil restrictions of split groups. Although there is a similar description for the stabilizer groups for a wild Galois action, taking wild fixed points does not always preserve smoothness. So there is no corresponding description for the group schemes.  
Fortunately, in both of these cases there are integral Hodge embeddings in which the group at $p$ is cut out in a symplectic group scheme by endomorphisms of the lattice (one could call these cases ``essentially of PEL type"). We show that these embeddings are very good by a modified version of the argument above, see \S \ref{ss:DnH}, \S \ref{ss:A_n}. However, in the exceptional cases, this somewhat restricts the Hodge embeddings that can be shown to be very good.  

This roughly explains the argument for most Shimura varieties of Hodge type. Extending the results to the rest and to Shimura varieties of abelian type is done by finding suitable Hodge type lifts in the sense of Deligne and closely follows \cite{KP}.   Here we need to make sure that we can find Hodge type lifts that support very good embeddings. There are some additional technical complications imposed by the aforementioned restriction on the Hodge embeddings we can use in the exceptional cases and, in the paper, we go in detail over the parts of the argument that are different.   We can then apply the argument in  \cite[\S4.4-6]{KP} in our setting to give Theorem \ref{introthm: LMD abelian type}. A crucial ingredient for this  is the notion of $R$-smoothness for tori developed in \cite{KZhou} which is used to extend the twisting construction of \cite{KP} beyond the tamely ramified case.
\end{para}

\begin{para}
We now return to briefly discuss the initial step of constructing good integral Hodge embeddings $(\Gg,\mu)\hook (\GL(\La), \mu_d)$ (which are later shown to be very good). 

The paper \cite{KP} uses results of Landvogt about functoriality of Bruhat--Tits buildings and arguments with Weyl modules to establish the existence of lattices $\La$ which give good integral Hodge embeddings $(\Gg,\mu)\hook (\GL(\La), \mu_d)$. Again using $R$-smoothness, it is possible to generalize this and to prove the result without the tameness hypothesis; this was the approach taken in earlier versions of 
\cite{KZhou}.
Here we give a different and simpler argument which does not use the results of \cite{Lan}, but,  instead, starts with finding linear closed embeddings for  reductive group schemes, see Proposition \ref{prop:findlattice}. Then,  appropriate closed embeddings for general stabilizers are constructed by taking Galois fixed points, after using Proposition \ref{Fixed}. We also take advantage of the improvement to the theory of local models  provided by Scholze-Weinstein in \cite{Schber} by the use of $v$-sheaves over perfectoid spaces. Indeed, \cite{Schber} gives  a characterization of local models via their associated $v$-sheaves and this implies that local models are functorial. 
This also makes it is easier to find lattices which give  good  integral Hodge embeddings. 
In \cite{KP}, these are  obtained  from suitable embeddings of buildings, constructed using \cite{Lan}. These produce compatible embeddings of group schemes over $\Z_p[u]$ and eventually give corresponding closed embeddings of local models. Now instead, if the local model satisfies the Scholze-Weinstein conjecture, then for each lattice giving an integral Hodge embedding, we know by functoriality that there is a uniquely determined morphism from the local model to a corresponding Grassmannian. We just have to make sure it is a closed embedding.
To apply this argument, we show that  the local models we use in this paper, which are given following the constructions of \cite{PZ}, \cite{Levin}, satisfy the characterization of \cite{Schber}, i.e. they satisfy the Scholze-Weinstein conjecture, see Theorem \ref{thm:twoLM}. The proof of this result follows a standard blueprint of reducing to the case of $\GL_n$ and is intertwined with the construction of good integral Hodge embeddings as above, see Theorem \ref{thm: LM embedding}, Theorem \ref{thm:twoLM}. It again uses the technique of writing
stabilizer group schemes as the tame Galois fixed points of the Weil restrictions of split groups. 
 
\end{para}

 \begin{para}
We emphasize that Theorem \ref{introthm: LMD abelian type} and most of the  main results of the current paper are shown completely independently of the theory of $p$-adic shtukas.  
In fact, techniques that use $p$-adic shtukas alone do not seem enough to construct integral models which are  \'etale locally isomorphic to the corresponding local model, not even in a single non-trivial example.
\end{para}

 \subsection{Corrections to \cite{KP} and relation to the current paper}\label{ss:ReltoKP}
 
 In this paragraph, we list the parts of \cite{KP} that need correction and explain how to replace them with corresponding parts of the current paper. We then compare Theorem \ref{introthm: LMD abelian type} and other results of this paper with statements that appear in  \cite{KP} (and previous versions of \cite{KZhou}).

\begin{para} We first discuss corrections to \cite{KP}.
 The construction of  the commutative diagram in Lemma 3.1.9  of \cite{KP} is not correct and has to be replaced by the construction in Lemma \ref{319} of this paper, see Remark \ref{remark:Hoff}. The construction of the isomorphism $\Psi_{R_G}$ in \cite[3.2.12]{KP} is not correct. Such an isomorphism only exists under the additional assumption that the Hodge embedding is \emph{very good}, with this term as defined in this paper. Lemma 3.2.14 and Propositions 3.2.17 and 3.3.13 of \cite{KP} require $\Psi_{R_G}$ and hence they are applicable only under the assumption that the Hodge embedding is very good, with the same proofs. The same is true
 for \cite[Theorem 4.2.7]{KP} and \cite[Cor. 4.2.12, Cor. 4.2.13]{KP} which now require the additional assumption that the Hodge embedding is very good. In fact, the rest of the arguments in \cite[Sec. 4]{KP} are unaffected. The conclusions in the statements hold when one adds the assumption that the Hodge embeddings that appear are very good, so that the above corrections to \cite[\S 3.1, \S 3.2, \S 3.3]{KP} can be applied. In particular, \cite[Theorem 4.2.7]{KP} should be replaced by Theorem \ref{thm: main SV Hodge} (which also relaxes the tameness assumption). This affects the final statement \cite[Theorem 4.6.23]{KP} about the existence of the local model diagram in abelian type cases, which should 
 be replaced by Theorem \ref{thm: LMD abelian type} (Theorem \ref{introthm: LMD abelian type}). Some additional minor corrections to proofs in \cite{KP} are listed in
\S \ref{ss:Errata}.
 \end{para}
 
 \begin{para}  
 Theorem 4.6.23 of \cite{KP} is a version  of Theorem \ref{introthm: LMD abelian type} in which the restriction in part (3) only rules out factors of type $D^{\bbH}$. (A corresponding result appeared in previous versions of \cite{KZhou}.)  Thus, Theorem \ref{introthm: LMD abelian type} gives a slightly weaker result in the case that $(G^{\ad},\mu^{\ad})$ contains factors of ``exceptional type A''. However, this extra restriction is removed by Corollary \ref{corintro:PlusShtukas}.   Also, in Hodge type cases, the argument of \cite{KP}, as corrected and extended in this paper under the assumption that the Hodge embedding is very good, gives  
 Theorem \ref{thm: main SV Hodge}. This  theorem implies that the normalization of the Zariski closure of the Shimura variety for stabilizer level under the Hodge embedding has the correct \'etale local structure. This last, slightly weaker, statement can be shown without the assumption that the Hodge embedding is very good, see Theorem \ref{thm:sh}. Both of these improvements require additional inputs from \cite{PRShtuka}, \cite{PRlsv}, \cite{DvHKZ}, which use Scholze's theory of $p$-adic shtukas.
\end{para}
 
 \subsection{Organization of the paper} We conclude the introduction by explaining the organization of the paper in some more detail.

 In \S \ref{s:ParahoricFixed}, under certain conditions, we show how to write stabilizer group schemes  as the tame Galois fixed points of the Weil restriction of scalars of split reductive group schemes. We give some applications to showing that certain representations of reductive groups extend to closed immersions between stabilizer schemes. These results are also applied later in showing existence of certain good  embeddings in sections \S \ref{s:LM} and very good embeddings in \S \ref{sec: very good embeddings}.

 In \S \ref{s:LM} we discuss  local models of Shimura varieties and prove the cases of the Scholze-Weinstein conjecture on local models that we need. This is intertwined with the construction of good embeddings mentioned above.

In \S \ref{s:RootCurves} we study  tangent spaces of local models of abelian type for restriction of scalars of reductive group schemes. We prove Theorem \ref{thm:spanIntro} which shows that they are spanned by smooth curves with very few exceptions. This involves quite heavy combinatorial computations.

In  \S \ref{s:Displays} we explain the connection isomorphism for displays, the omitted condition in \cite{KP} and give the key definition  of a very good  embedding. We show the main properties of very good embeddings that we will use in the sequel.

 The main constructions of very good embeddings are contained in \S \ref{sec: very good embeddings}; these are divided in the non-exceptional (NE) and exceptional cases. 
 
 Finally, in \S \ref{sec:SV} we give the application to integral models of Shimura varieties and state and prove the main results. We also give some errata for \cite{KP} and \cite{PCan}.
 
 \medskip
 
 \emph{Acknowledgements:} We would like to thank M. Hoff for pointing out the gap in \cite{KP} and 
M. Rapoport, Y. Luo and P. Wu for useful comments. We thank the referees for a careful reading of the manuscript and for many useful comments which greatly improved the paper. M.\,K. is supported by NSF grant \#DMS-2200449. G.\,P. is supported by NSF grant \#DMS-2100743. R.\,Z. is supported by  EPSRC grant ref. EP/Y030648/1.
 
\medskip
 
\textbf{Notations:} 
If $F/\Q_p$ is a non-archimedean local field, we let $\overline{F}$ be a fixed choice of algebraic closure of $F$. We let $\br F$ denote the completion of the maximal unramified extension of $F$ in $\overline{F}$. The rings of integers are denoted by $\calO_F$, resp. by $\br\calO_F$. We denote by $k_F$ the residue field of $F$. For most of the paper, $k$ is an algebraic closure of a finite field. 

If $X$ is an $A$-scheme and $B$ an $A$-algebra we write $X\otimes_AB$ or $X_B$ instead of $X\times_{\Spec(A)}\Spec(B)$.

For a connected reductive group $G$ over a field, we let $G^{\der}$ (resp. $G^{\ad}$) denote the derived group (resp. adjoint group) of $G$, and we let $G^{\mathrm{sc}}$ denote the simply-connected cover of $G^{\der}$. We denote by $\pi_1(G)$ Borovoi's algebraic fundamental group of $G$, i.e. $\pi_1(G)$ is the quotient of the cocharacter group by the coroot lattice over a separable closure of the ground field.

 %%%%%%%%%%%%

    \bigskip
    
    \addtocontents{toc}{\protect\setcounter{tocdepth}{2}}
  
  \section{Parahorics and embeddings of group schemes}\label{s:ParahoricFixed}

This section mainly contains preliminaries about parahoric and stabilizer group schemes that we will use later. This includes the notion of $R$-smoothness for tori which is recalled in \S2.1, and results from \cite{PRtame} on realizing parahorics and stabilizers as  fixed points of reductive group schemes in \S2.2.
  
  \subsection{Stabilizers, parahorics and buildings}

  \begin{para}
  Fix a prime $p>2$.  Let  $K$ be a  finite extension of $\bbQ_p$ or a finite extension of $\brQ$ and let $G$ be a (connected) reductive group over $K$. 
   We let $\calB(G, K)$ denote the extended building and $\bar\calB(G, K)=\calB(G^\ad, K)$ the reduced (``classical") building \cite{BT1}, \cite{BT2}. Recall that a quasi-parahoric  group scheme for $G$  is a smooth affine scheme $\Gg$ over the integers $\O=\O_K$ with $G=\Gg\otimes_{\O}K$,
  	whose neutral connected component  is a parahoric group scheme and
  	with $\br\O$-valued points satisfying
  	\[
  	\Gg_\bx^\circ(\br\O)\subset \Gg(\br\O)\subset \Gg_\bx(\br\O),
  	\]
  	for some point $\bx$ in the extended building $ \B(G, K)$ of $G$ over $K$, \cite{BT2}, \cite{KalethaPrasad}. Here $\Gg_{\bx}$ is the Bruhat--Tits  stabilizer group scheme 
  	associated to $\bx$ by Bruhat-Tits in \cite{BT2}. 
  	Then the neutral component $\Gg^\circ=\Gg^\circ_\bx$ is the associated parahoric
  	 and the inclusions above give quotients which are 
  	finite abelian groups, see \cite{HR}. Most of the time we will consider the case $\Gg=\Gg_\bx$, for some $\bx\in \B(G, K)$.

  \end{para}

\begin{para}
  	If $\ti K/K$ is a finite extension, then we have the building $\calB(G,\ti K)$ over $\ti K$, and for $\bx'\in \calB(G,\ti K)$, we let $\ti \calG_{\bx'}$ denote the stabilizer scheme over $\ti \calO=\calO_{\ti K}$ associated to $\bx'$. Let $H=\Res_{\ti K/K}G_{\ti K}$. Then by \cite[Prop. 4.6]{HaRi}, we have an identification $\calB(G,\ti K)\cong \calB(H,K)$ and there is  an isomorphism $\Res_{\ti \calO/\cal O}\ti\calG_{\bx'}\cong \calH_{\bx'}$, where $\calH_{\bx'}$ is the stabilizer scheme of $H$ for $\bx'$ considered as a point in $\calB(H,K)$.
  	
	Now assume $\ti K/K$ is a finite tame Galois extension with Galois group $\Ga=\Gal(\ti K/K)$ contained in an algebraic closure $\bar K$. 
By \cite{PYu}, the natural map $\B(G,K)\to \B(G,\ti K)$ gives  identifications
\begin{equation}\label{PYtame} 
\B(G, K)=\B(G, \ti K)^\Gamma,\quad \bar\B(G, K)=\bar\B(G, \ti K)^\Gamma
\end{equation}
with the fixed points by the natural action of $\Ga$.

  \end{para}
  
\begin{para}We now recall the notion of $R$-smoothness from \cite{KZhou} which will play an important role in what follows. 
	
	Let $T$ be a torus over $K$ and let $\ti K/K$ be a finite extension. We let $\calT$ (resp. $\ti \calT$) denote the lft N\'eron model for  $T$ (resp. the base change $T_{\ti K}$); see \cite[\S10]{BLR}. Then $\Res_{\ti \calO/\calO}\calT_{\ti K}$ is the lft N\'eron model for $\Res_{\ti K/K}T_{\ti K}$. 
	
Now fix a $\ti K/K$ such that $T$ splits over $\ti K$. Recall \cite[Def. 2.4.3]{KZhou} that the torus $T$ is said to be \emph{$R$-smooth} if the Zariski closure of $T$ inside $\Res_{\ti \calO/\calO}\ti\calT$ is smooth.\footnote{As explained in \cite[\S4.4.8]{BT2},  this definition is independent of the choice of splitting field $\ti K$.} If $G$ is a reductive group over $K$, we say that $G$ is $R$-smooth if the centralizer of a (equivalently any) maximal $\breve K$-split torus in $G$ is $R$-smooth. The following summarizes the main results on $R$-smoothness from \cite{KZhou} that we will need.

\begin{prop}\label{prop: R-smoothness properties}
	\begin{altenumerate}
		\item Let $T\cong\Res_{K_i/K}S_i$ where $K_i/K$ is finite and $S_i$ is a torus over $K_i$ which splits over a tamely ramified extension of $K_i$ (we call such a torus {\rm quasi-tame}, cf. Definition \ref{def: quasi-tame}). Then $T$ is $R$-smooth. 
		
		\item  If $T$ is the extension of an $R$-smooth torus by an $R$-smooth torus, then $T$ is $R$-smooth.
		
		\item Let  $\ti K/K$ be a finite extension and $G\rightarrow G'$ be a closed immersion of reductive groups which induces an isomorphism $G^{\der}\cong G'^{\der}$  and let $\bx\in \calB(G,K)$ with image $\bx'\in \calB(G',\ti K)$. Assume $p>2$ and that $G$ is $R$-smooth. Then the natural morphism $G\rightarrow \Res_{\ti K/K}G_{\ti K}'$ extends to a closed immersion of stabilizer schemes $$\calG_{\bx}\rightarrow\Res_{\ti \calO/\calO}\ti\calG'_{\bx'}.$$
	\end{altenumerate}
\end{prop}
\begin{proof}
Part (1) and (2) is \cite[Prop. 2.4.6]{KZhou}, and (3) follows from the argument of \cite[Prop. 2.4.10]{KZhou} using that $T\rightarrow T'$ extends to a closed immersion of finite type N\'eron models by \cite[Lem. 2.4.4]{KZhou}. Here $T$ is a centralizer of a maximal $\breve K$-split torus $S$ in $G$ whose apartment contains $\bx$, and $T'$ is the corresponding centralizer for some maximal $\breve K$-split torus of $G'$ which contains the image of $S$.
\end{proof}

\end{para}

  \subsection{Parahorics as Galois fixed points of reductive group schemes}\label{ss:Fixed}

   \begin{para}\label{ss:Fixed2} We now assume that $G$ is a classical reductive group over $K$ (i.e. there are no exceptional factors in $G^{\ad}$;  by convention,  this also excludes triality forms.) We will show that the identification \eqref{PYtame} allows us to realize stabilizer schemes as the Galois fixed points of hyperspecials over a tame extension.
   	
  	 Let $H_0$ be the Chevalley (split) form of $G$ over $\Z_p$. We   assume that $G$ is tamely ramified, i.e. there is a tame finite (Galois) extension $\ti K/K$ of ramification degree $e$ with $\Ga=\Gal(\ti K/K)$ such that $G\otimes_K\ti K\simeq H_0\otimes_{\Z_p}\ti K$. By adjoining an unramified extension, we can always assume that $\ti K$ contains a uniformizer $\ti\pi$ with $\ti\pi^e\in K^\un$, where $K^\un$ is the maximal unramified extension of $K$ contained in $\ti K$.

  	\begin{prop}\label{Fixed}
  		Assume that $G$ is as above. If $G^\ad$ contains a simple factor isomorphic to $\Res_{L/K}({\rm PGL}_m(D))$, where $D$ is a division central $L$-algebra
  		and $L/K$ is a tame extension, assume in addition that the index of $D$ is prime to $p$.

  		Suppose that $\Gg=\Gg_\bx$ is the Bruhat-Tits group scheme over $\O=\O_K$ with $\Gg\otimes_\O K=G$ which is the stabilizer of a point $\bx\in \calB(G, K)$  generic in its facet. \footnote{``generic'' here is meant in the sense that $\Gg_\bx=\Gg_{\bx'}$, for all $\bx'$ in some open nbd of $\bx$ in that facet.}
  		\begin{altenumerate}
  		
  		\item There is a point $\bx'\in \B(G, K)$ such that $\Gg_{\bx}=\Gg_{\bx'}$ and a finite Galois tame  extension $\ti K/K$ with Galois group $\Gamma=\Gal(\ti K/K)$ such that $G\otimes_K \ti K$ is split and $\bx'$ is hyperspecial in $\B(G, \ti K)$. 
  		
  		\item  The corresponding stabilizer group scheme $\ti\Gg_{\bx}$ over $\ti \O=\O_{\ti K}$ with generic fiber $G\otimes_K\ti K$ is reductive and supports a $\ti \O$-semilinear 
  		$\Gamma$-action which extends the $\Gamma$-action on $G\otimes_K\ti K\simeq H_0\otimes_{\Z_p} \ti K$. 
  		The isomorphism $G\simeq \Res_{\ti K/K}(G\otimes_K \ti K)^{\Gam}$
  		extends to  an isomorphism of group schemes
  		\[
  		\Gg\simeq (\Res_{\ti \O/\O}\ti \Gg_\bx)^{\Gam}.
  		\]
  		\end{altenumerate}
  	\end{prop}

  	Suppose $\bx\in \B(G, K)$ is such that $\Gg_\bx$ is connected. Then, if ${\bf y}\in \B(G, K)$ is generic in the smallest facet containing $\bx$, we have $\Gg_{\bf y}=\Gg_\bx$.
  	Hence, the result applies to all stabilizers group schemes that are parahoric, i.e. connected.

  	\begin{proof} The statement is a variation of \cite[Prop. 2.8]{PRtame}. We will explain how the proof in \emph{loc. cit.} can be extended to give this result. First we note that it is enough to show:
  		
  		($*$) There is a point $\bx'\in \B(G, K)$ such that $\Gg_{\bx}=\Gg_{\bx'}$ and a finite tame  extension $\ti K/K$ such that $G\otimes_K \ti K$ is split and $\bx'$ is hyperspecial in $\B(G, \ti K)$. 
  		
  		A hyperspecial point remains hyperspecial after every finite field extension. Hence, assuming ($*$) we can pass to the normal closure and  make sure that $\ti K/K$ is in addition Galois with group $\Gamma=\Gal(\ti K/K)$. Then the rest follows by the standard argument which uses the smoothness of fixed points of a smooth scheme for a tame finite group action (\cite[3.4: Prop.]{Edix}). 
  		
  		Statement ($*$) is shown in the course of the proof of \cite[Prop. 2.7, Prop. 2.8]{PRtame} when $G$ is absolutely simple and simply connected. We will show how this argument extends under our assumptions.
  		
  		First let us assume that $G$ is semi-simple. Write 
  		\[
  		G^{\rm sc}=\prod_i \Res_{L_i/K} G_i
  		\]
  		with $G_i$ over $L_i$, simply connected and absolutely simple. This gives
  		\[
  		\B(G, K)=\B(G^{\rm sc}, K)=\prod_i \B(G_i, L_i); \quad \bx\mapsto (\bx_i).
  		\]
  		Since $\bx$ is generic in its facet, each $\bx_i\in \B(G_i, L_i)$ is generic in its facet. By applying the argument in the proof of \cite[Prop. 2.8]{PRtame} which considers a ``tame subdivision'' of the apartment with its simplicial structure, we see that there exists a ``nearby'' $\bx'_i\in \B(G_i, L_i)$ which is hyperspecial in $\B(G_i, \ti L_i)$, where $\ti L_i$ is a finite tame extension of $K$. In fact, by enlarging $\ti L_i$, we can find $\bx'_i$ with these properties which, using the standard metric of the apartment, is as close to $\bx_i$ as we like. The assumption for groups of type A enters  in  the existence of the suitable tame subdivision, see the proof of  \cite[Prop. 2.8]{PRtame},  also \cite[Rem. 2.9]{PRtame}. Consider $\bx'=(\bx'_i)\in \B(G, K)$ which is then close to $\bx$ and defines the same stabilizer group scheme as $\bx$. By passing to the normal closure $\ti K$ of the join of the $\ti L_i$'s in an algebraic closure of $K$, we can assume that $\bx'_i\in \B(G_i, \ti K)$ is hyperspecial for all $i$. We now have
  		\[
  		\B(G, \ti K)=\prod_{(i, \alpha)} \B(G_i, \ti K)
  		\]  
  		the indexing set also including all $\alpha: L_i\hookrightarrow \ti K$ over $K$. For   $\alpha: L_i\hookrightarrow \ti K$, find $\tau\in \Gamma=\Gal(\ti K/K)$ such that $\alpha=\tau_{|L_i}: L_i\hookrightarrow \ti K$. Then the projection of the image of $\bx'$ to the factor indexed by $(i, \alpha)$,  is $\tau_*(\bx'_i)$. In this, $$ \tau_*: \B(G_i\otimes_{L_i}\ti K, \ti K)\to \B(G_i\otimes_{L_i, \tau}\ti K, \ti K)$$ is induced by functoriality of buildings by the Galois automorphism $ \tau: \ti K\to \ti K$. Hence, $ \tau_*(\bx'_i)$ is hyperspecial, and so $\bx'=( \tau_*(\bx'_i))_{i, \alpha}$ is hyperspecial in $\B(G,\ti K)$. This shows ($*$) when $G$ is semi-simple.

  		Now we discuss the general reductive case. Note that for a split group $H$, a point in $\bx\in \B(H, K)$ is hyperspecial if and only if its image $\bar \bx\in \bar\B(H,K)=\B(H^\der, \ti K)$ is hyperspecial. 
  		
  		We have $G(\br K)_{\bx}=G(\br K)_{\bar \bx}\cap G(\br K)^1$. Here, $G(\br K)_{\bar \bx}$ is the stabilizer of $\bar \bx\in \bar\B(G, K)\subset \bar\B(G, \br K)$ under the natural action of $G(\br K)$ on $\bar\B(G, \br K)$; the group $G(\br K)^1$ is  the kernel of 
  		\[
  		G(\br K)\xrightarrow{\kappa_G} \pi_1(G)_I\to \pi_1(G)_I/\{\rm torsion\}
  		\]
  		obtained from the Kottwitz homomorphism, see \cite[Rem. 11]{HR}, \cite[4.2.16]{BT2}. 
  		If $\bar \bx$ is generic in a facet and $\bar \bx'$ is nearby, $G(\br K)_{\bar \bx}=G(\br K)_{\bar \bx'}$; hence we also have $G(\br K)_{\bx}=G(\br K)_{\bx'}$ and so $\Gg_\bx=\Gg_{\bx'}$. 
  		
  		Recall that we know ($*$) for $G^\der$. Consider $\bx\in \B(G, K)$ generic in its facet with corresponding point $\bar \bx\in \B(G^\der,  K)$, also generic in its facet. By  ($*$) for $G^\der$, there is nearby $\bx'_\der\in \B(G^\der, K)$   
  		and a tame Galois extension $\ti K/K$ which splits $G^\der$ such that $\bx'_\der\in \B(G^\der, \ti K)$ 
  		is hyperspecial. By enlarging $\ti K/K$ we can assure that $G$ is also split over $\ti K$. 
  		Now lift $\bx'_\der$ to $\bx'\in \B(G, K)$, i.e. with $\bar \bx'=\bx'_\der$. The point $\bx'$ is hyperspecial in $\B(G, \ti K)$ and by the argument above $\Gg_\bx=\Gg_{\bx'}$. This shows ($*$) for $G$ and $\bx$.
  	\end{proof}
  \end{para}

  \subsection{Lattices and parahoric subgroups}\label{lattices}
 Let $V$ be a finite dimensional $K$-vector space. In this subsection, we give a more explicit  description of the construction in Proposition \ref{Fixed} in the case  $G=\GL(V)$.
 
 Fix, once and for all, a  volume form  on $V$, i.e. an isomorphism $\wedge^{\dim(V)} V\simeq K$. This allows us to identify the (extended) building $\B(\GL(V), K)$ with pairs $(\L, c)$ consisting of a periodic $\O$-lattice chain $ \L=\{ \La_\bullet\}$ in $V $ and a grading function $c: \L\to {\mathbb R}$ (see \cite{BTclass}, \cite[Cor. 5.1.28]{KalethaPrasad}).
  For each periodic lattice chain $\L$ we can choose a  ``determining segment",
  \[
  \La_s=\pi\La_{0}\subset \La_{s-1}\subset\cdots \subset \La_0
  \]
  in the obvious sense. If $\bx=(\L, c)$, then the corresponding parahoric subgroup of $\GL(V)$ is the common stabilizer of the lattices in the lattice chain, or in a determining segment of the lattice chain.
  The corresponding parahoric group scheme, which we write simply as $\GL(\L)$, is determined by its $\br \O$-points  (\cite[Prop. 1.7.6]{BT2}) which are
  \[
  \GL(\L)(\br \O)=\bigcap_i \GL(\La_i\otimes_\O\br \O)= \bigcap_{i=0}^{s-1} \GL(\La_i\otimes_\O\br \O).
  \]
 In this situation, we set
  \[
  {\rm tot}(\L):=\La_0\oplus\La_1\oplus\cdots \oplus \La_{s-1}\subset V^{\oplus s}
  \]
  for the direct sum of the lattices in the segment. We can consider the stabilizer group scheme $\GL({\rm tot}(\L))$.

  \begin{lemma}\label{BTGL}  (\cite[3.8]{BTclass}, cf. \cite[Prop. A.4]{R-Z})
  	There is a group scheme homomorphism
  	$
  	\GL(\L)\to \GL({\rm tot}(\L))
  	$
  	which extends the diagonal 
  	\[
  	\GL(V)\to \GL(V)^s\hook \GL(V^{\oplus s})
  	\]
  	and which is a closed immersion.\qed
  \end{lemma}

  \begin{para}\label{latticesI} Let $\ti K/K$ be a finite tame Galois extension with Galois group $\Ga$, inertia subgroup $I\subset \Ga$, and ramification index $e=|I|$.  Let $\ti\La\subset V\otimes_K\ti K$
  	be an $\ti\O$-lattice. We assume that $\ti\La$ is $\Ga$-stable. Let $\ti\L$ be the periodic lattice chain given by all scalar multiples $\ti\pi^i\ti\La$ of $\ti\La$
  	and consider the grading function $\ti c$ given by $\ti c(\ti\pi^i\ti\La)=i$. Then $(\ti\L, \ti c)$ is a periodic graded $\ti \O$-lattice chain in $V\otimes_K \ti K$ corresponding to a point $\ti \bx\in \B(\GL(V), \ti K)$ which is fixed by $\Ga$. The corresponding parahoric group scheme for  $\GL(V\otimes_K\ti K)$ over $\ti\O$ is   the group scheme  of $\ti\O$-linear automorphisms of $\ti\La$;
  	we denote this group scheme simply by $\GL(\ti\La)$.
  	
  	By tame descent on buildings (\ref{PYtame}), $\ti \bx$ is  identified with a point $\bx\in \B(\GL(V), K)$ which
  	corresponds to a periodic graded lattice chain $(\L, c)$ in $V$. 
  	We have
  	\[
  	\GL(\L)={\rm Res}_{\ti \O/\O}\GL(\ti\La)^\Ga
  	\]
  	for the parahoric $\GL(\L)$ of $\GL(V)$ given as the stabilizer of $\bx$.

  	\begin{lemma}\label{fixedptsGL}
  		The parahoric $\Res_{\ti \O/\O}\GL(\ti\La)^\Ga$ of $\GL(V)$ is equal to the stabilizer $\GL(\L)$ of the periodic lattice chain $\L$ given by $\{\Lambda_i\}_{i\in \Z}$ 
  		where
  		\[
  		\Lambda_i=(\ti\pi^i\ti\La)^\Ga\subset (V\otimes_K \ti K)^\Ga=V
  		\]
  		and $\Lambda_{i+1}\to \Lambda_i$ are the natural injective maps given by $\ti\pi^{i+1}\ti \La\subset\ti\pi^i\ti \La$.
  	\end{lemma}
  	
  	Note that in the above, we could have $\Lambda_{i+1}=\Lambda_i$ for some $i$. The periodic lattice chain $\L$ given by $\{\Lambda_i\}_{i\in \Z}$ is, by definition, the set of the lattices $\La_i$. Since $\ti\pi^e\ti \O=\pi\ti \O$ we have $\Lambda_{e}=(\pi\ti\La )^\Ga=\pi\Lambda_0$.
  	
  	\begin{proof} 
  	Both the group schemes  $\Res_{\ti \O/\O}\GL(\ti\La)^\Ga$ and $\GL(\L)$ are smooth affine with generic fiber $\GL(V)$ and, by \cite[Prop. 1.7.6]{BT2}, it is enough to show they have the same $\br \O$-points.
  		For this, we base change to $\br \O$ and assume that $K=\br K$. So, it is enough to show
  		\begin{equation}\label{intersect}
  		\GL(V)\cap \GL(\ti\La)=\bigcap_{i=0}^{e-1} \GL(\La_i)
  		\end{equation}
  		(the intersection taking place in $\GL(V\otimes_K\ti K)$.)
  		Let $\ti\pi\in \ti \O$ be a uniformizer with $\ti\pi^e\in \O$. Let $\chi: I\to k^*={\rm Aut}_k((\ti\pi)/(\ti\pi)^2)$ be the standard inertia character. Write
  		\[
  		\ti \La= \bigoplus_{i\in \Z/e\Z} \ti \La_{i}
  		\]
  		for the decomposition into eigenspaces for the action of the inertia.
  		Here 
  		\[
  		\ti\La_i=\ti \La_{i\,{\rm mod}\, e}=\{x\in \ti\La\ |\ \gamma(x)=[\chi(\gamma)]^i\cdot x\},
  		\]
  		with $[\ ]: k^*\to \O^*$ the Teichm\" uller map. The eigenspaces $\ti\La_i$ are $\O$-modules and
  		\[
  		\La_i=(\ti\pi^i\ti\La)^I=\ti\pi^i \ti\La_{-i\,{\rm mod}\, e}\xrightarrow{\sim} \ti\La_{-i\,{\rm mod}\, e},
  		\]
  		the last map given by multiplying by $\ti\pi^{-i}$. So, we have
  		\begin{equation}
  		\ti\La=\bigoplus_{i=0}^{e-1} \ti\pi^{-i}\La_i\subset V\otimes_K \ti K=\bigoplus_{i=0}^{e-1} \ti\pi^{-i}V.
  		\end{equation} 
  		Multiplication by $g\in \GL(V)\cap \GL(\ti\La)$ respects the eigenspace decomposition of $\ti\La$ and commutes with scaling by $\ti\pi$, 
  		so the LHS of (\ref{intersect}) is contained in the RHS. Suppose $g\in \GL(\La_i)$, for all $i$. Then, by the above, $g$ considered in $\GL(V\otimes_K \ti K)$ gives an automorphism of $\ti\pi^i \ti\La_{-i\,{\rm mod}\, e}$ and hence of
  		$\ti\La$. This shows that the RHS is contained in the LHS.
  	\end{proof}
  	
  \end{para}

  \begin{para}\label{TOT} In the  lemma above, $\{\La_i\}_{i\in \Z}$ is given by the $\pi^\Z$ multiples of its segment
  	\[
  	\pi\La_0\subset \La_{e-1}\subset \cdots\subset \La_1\subset \La_0.
  	\]
  	Assuming $\ti\pi^e\in \O$ and that $K=K^{\rm un}$, i.e. $\ti K/K$ is totally ramified, the proof of the lemma gives 
  	\begin{equation}\label{rescale}
  	\La_0\oplus\La_1\oplus\cdots \oplus \La_{e-1}\xrightarrow{\sim}\ti\La_{0}\oplus\ti\La_{-1}\oplus\cdots\oplus \ti\La_{-(e-1)}=\ti\La
  	\end{equation}
  	as $\O$-modules, with the map given by multiplication by $(1,\ti\pi^{-1},\ldots, \ti\pi^{-(e-1)})$. Hence,
  	\[
  	{\rm tot}(\L)\subset \La_0\oplus\La_1\oplus\cdots \oplus \La_{e-1}\simeq \ti\La
  	\]
  	and it is a direct summand. (The inclusion is proper when we have $\La_i=\La_{i+1}$, for some $i$.)  
  	It follows that multiplication by corresponding powers of $\ti\pi$ on the graded pieces gives an isomorphism
  	\begin{equation}\label{generalDirectSum}
  	L\oplus {\rm tot}(\L)\xrightarrow{\ \sim\ } \ti\La.
  	\end{equation}
  	where $L$ is a certain direct sum of $\La_i$.
  \end{para}

  \subsection{Embedding of parahorics}

 \begin{para}Let $p>2$ and $\rho:G\rightarrow \GL(V)$ a faithful representation of a reductive group over $K$. 
 	We have the following proposition which generalizes \cite[Lem. 2.3.1]{KisinJAMS} with a similar proof.
 		\begin{prop}\label{prop:findlattice}
 		
 		Let $\ti K/K$ be a finite Galois extension with $\Ga=\Gal(\ti K/K)$ and with the following property:
 		
 		There is a split reductive group scheme
 		$\ti\Gg$ over $\ti \O$ such that
 		\begin{altitemize}
 			\item[1)] $\ti\Gg\otimes_{\ti \O}\ti K=G\otimes_{K}\ti K$ (in particular, $G$ splits over $\ti K$),
 			
 			\item[2)] $\ti\Gg$ supports an $\ti\O$-semilinear $\Ga$-action which extends the standard $\ti K$-semilinear $\Gamma$-action on $G\otimes_K\ti K$.
 		\end{altitemize}
 		Then there is a $\ti \O$-lattice $\ti\La$ in $V\otimes_K\ti K$ which is $\Ga$-invariant and such that the base change $\rho\otimes_K\ti K: G\otimes_K\ti K\to \GL(V\otimes_K\ti K)$ extends to a closed group scheme immersion
 		\[
 		\ti\Gg\hook \GL(\ti \La).
 		\]
 	\end{prop}
 	\begin{proof}  
 		Let $M$ be the  maximal unramified extension of $\ti K$ in an algebraic closure of $\ti K$. Then $M/K$ is an (infinite) Galois extension. The natural homomorphism $\Gal(M/K)\to \Ga=\Gal(\ti K/K)$ is a surjection with kernel $\Gal(M/\ti K)$ which identifies inertia subgroups. We first show that there is a $\Gal(M/K)$-invariant $\O_M$-lattice
 		$\La_M$ in $V\otimes_K M$ which is also preserved by the action of  $\rho(\ti\Gg(\O_M))$. 
Observe that  $\ti\Gg(\O_M)$ is bounded and the Galois group $\Gal(M/K)$ is compact. We can consider the semi-direct product $\ti\Gg(\O_M)\rtimes \Gal(M/K)$ (obtained by 
 		the $\Gal(M/K)$-action on $\ti\Gg(\O_M)$ given by the semi-linear $\Ga$-action on $\ti\Gg$) and apply the same argument as in the proof of \cite[Lem. 2.3.1]{KisinJAMS} to obtain   the existence  of $\La_M$.
 		Then, by \cite[1.7.6]{BT2}, $\rho$ extends to a group scheme homomorphism
 		\[
 		\ti\Gg\otimes_{\ti \O}\O_M\to \GL(\La_M).
 		\]
 		Since $\ti\Gg$ is reductive and $p$ is odd, this is a closed immersion
 		by \cite[1.3]{PYuGS}. We can then take 
 		\[
 		\ti \La:=( \La_M)^{{\rm Gal}(M/\ti K)}.
 		\]
 		This is an $\ti\O$-lattice in $V\otimes_K\ti K$ by \'etale descent along $\O_M/\ti \O$ and the rest follows.
 	\end{proof}
 	  	\begin{Remark}\label{rem:findlattice}
 		{\rm 
 			a) After applying restriction of scalars and then $\Ga$-fixed points to $\ti\Gg\hook \GL(\ti\La)$, we obtain a closed immersion of group schemes
 			\[
 			(\Res_{\ti\O/\O}\ti\Gg)^\Ga\hook \Res_{\ti\O/\ti O}\GL(\ti \La)^\Ga
 			\]
 			which gives $\rho: G\to \GL(V)$ on generic fibers.
 			
 			b) Note that we do not need that $\ti K/K$ is tame in Proposition \ref{prop:findlattice}. 
 			However, under this additional assumption, we see, using Edixhoven's lemma \cite[3.4:Prop.]{Edix}, that both the target and the source of the closed immersion in (a) above are smooth affine schemes over $\O$.
 			By Lemma \ref{fixedptsGL} and \'etale descent, the target is a parahoric group scheme for $\GL(V)$. In fact, it is 
 			the parahoric group scheme given as the stabilizer of the chain of $\O$-lattices
 			$\{(\ti\pi^i\ti\La)^\Ga\}_{i\in\Z}$.
 			
 		}
 	\end{Remark}
 \end{para}
     \begin{para}\label{2immersions}
  We now assume that $G$ and $\ti K/K$ are as in \S\ref{ss:Fixed} and let $\rho: G\to \GL(V)$ be a faithful representation over $K$.
  Suppose $\bx\in \B(G, K)$ is generic in its facet and that after replacing $\bx$ by a nearby point with the same stabilizer group scheme, $\bx$ is hyperspecial in $\B(G, \ti K)$
  and hence the corresponding parahoric group scheme $\ti\Gg=\ti\Gg_{\bx}$ of $G\otimes_K\ti K$ is reductive. This is possible by Prop. \ref{Fixed}, under the assumptions stated there.

  	By Proposition \ref{prop:findlattice}, there is a $\Ga$-stable $\ti \O$-lattice $\ti\La$ in $V\otimes_K\ti K$ such that $\rho$ extends to a closed immersion of group schemes
  	\[
  	\rho: \ti\Gg_\bx\hook \GL(\ti\La).
  	\]
  	Taking restriction of scalars and then $\Ga$-fixed points gives a closed immersion
  	\begin{equation}\label{im1or}
  		\rho: \Gg_\bx=(\Res_{\ti\O/\O}\ti\Gg_\bx)^\Ga \hook (\Res_{\ti\O/\O}\GL(\ti\La)^\Ga\hook \GL(\ti\La)
  	\end{equation}
  	where in the target we consider $\ti\La$ as an $\O$-module by restriction of scalars.
  	
  \end{para}

\section{Local models and embeddings}\label{s:LM}

In this section, we discuss the formalism of local models, we exhibit local models as closed subschemes of suitable Grassmannians
and (re)prove the cases of the Scholze-Weinstein conjecture that we need. Our arguments are independent of \cite{AGLR}.
The main results are Theorems \ref{thm:twoLM} and Theorem \ref{thm: LM embedding}.
 We also give the related definition of a good integral Hodge embedding, see Definition \ref{def:Good}.
 
\subsection{Local model triples and local models}\label{ss:LM}

  \begin{para}\label{LMparNEW}
In this section, we let $F$ be a finite extension of $\bbQ_p$.
Let  $(G, \{\mu\}, \Gg )$ be a  {\sl local model triple} over $F$.
 By definition, in these triples 
 
 \begin{altitemize}
 
 \item{} $G$ is a (connected) reductive group over $F$,

\item{} $\{\mu\}$ is the $G(\ov F)$-conjugacy class of a minuscule cocharacter $\mu: \Gm_{\ov F}\to G_{\ov F}$,
where $\ov F$ is an algebraic closure of $F$,

 \item{} 
 $\Gg$ is a quasi-parahoric stabilizer group scheme over $\O_F$ for $G$.
 \end{altitemize}

A morphism of local model triples $(G,\{\mu\},\Gg) \rightarrow (G',\{\mu'\},\Gg')$ is a group scheme homomorphism $\Gg \rightarrow \Gg'$ 
taking $\{\mu\}$ to $\{\mu'\}.$ 

As usual, we denote by $E$ the reflex field of the pair $(G, \{\mu\})$. It is a subfield of $\bar F$ containing $F$. To simplify notation, we often write $(G, \mu)$ 
for   $(G, \{\mu\})$ and $(\Gg,\mu)$ instead of  $(G, \{\mu\}, \Gg )$.

\begin{Definition}
We say that the pair $(G, \mu)$ is of (local) \emph{Hodge type}, if there is an embedding $\rho: G\hook \GL(V)$ such that
\begin{altitemize}
\item $\rho$ is a minuscule representation,

\item  $\rho\circ \mu$ is conjugate to the standard minuscule cocharacter $\mu_d$ of $\GL(V_{\ov F})$; here $\mu_d(a)=\mathrm{diag}(a^{(d)},1^{(h-d)})$ where $h=\dim V$,

\item $\rho(G)$ contains the scalars.
\end{altitemize}
Such a $\rho$ will be said to give a Hodge embedding $\rho: (G,\mu)\hook (\GL(V),\mu_d)$.
\end{Definition}

By definition, an \emph{integral} Hodge embedding  for $(\Gg,\mu)$ is a closed immersion of group schemes $\Gg\hook\GL(\La)$ over $\O_F$, where $\La$ is an $\O_F$-lattice in $V$, such that
 the homomorphism of generic fibers $G\hook \GL(V)$ is a Hodge embedding in the sense above. 

\begin{Definition}
We say that the pair $(G, \mu)$ is of (local) \emph{abelian type}, if there is a pair $(G_1,\mu_1)$ of Hodge type and an isomorphism 
$(G^\ad_1, \mu^\ad_1)\simeq (G^\ad,\mu^\ad)$. 
\end{Definition}

\begin{Definition}\label{def: quasi-tame}
\begin{altenumerate}
\item We  say that a reductive group $G$ over $F$ is \emph{quasi-tame}, if $
G\simeq  \prod_{i=1}^s\mathrm{Res}_{K_i/F}H_i
$
where, for all $i$, $K_i/F$ is a finite extension and $H_i$ 
is a reductive group  over $K_i$ which splits over a tamely ramified extension of $K_i$.

\item We  say that a reductive group $G$ over $F$ is \emph{essentially tame}, if $G^\ad$ is quasi-tame, cf. \cite[App.]{PRShtuka}.
\end{altenumerate}
\end{Definition}

\begin{stass}\label{stdASSumptions}
 We assume $p>2$, the pair $(G, \mu)$ is of abelian type and, in addition, that $G$ is essentially tame and classical.
\end{stass}

 In this situation, $G$ is  classical when $G^\ad\simeq  \prod_{i=1}^s\mathrm{Res}_{K_i/F}H_i$, with each $H_i$ splitting over a tamely ramified extension of $K_i$, and of classical type. (By definition, ``classical type" excludes triality groups.)

\begin{Remark}\label{rem: standard assumption}
{\rm a) Suppose $p>2$ and $(G, \mu)$ is of abelian type. 
Write 
$$
(G^\ad,\mu^\ad)\simeq (\prod_{i=1}^s\Res_{K_i/F}H_i, \{\mu_i\}),
$$
 where, for all $i$, $K_i/F$ is a finite extension and $H_i$ is absolutely simple over $K_i$. As we will explain below, if  $\mu_i$ is non-trivial, then  $H_i$ is of classical type, and also splits over a tamely ramified extension of $K_i$. Hence, the additional condition ``essentially tame and classical" in the standard assumptions above, is only relevant when $\mu_i=1$, for some $i$.
 
Indeed, when $\mu_i\neq 1$,  $H_i$  is of type $A$, $B$, $C$, or $D$: This follows from Deligne's argument classifying  Hodge embeddings which also applies in this local case. 
The triality forms of type $D_4$ are excluded:  Indeed, the existence of a rational minuscule embedding implies that the Galois group cannot act transitively on the set of end vertices in the Dynkin diagram of a simple factor of type $D_4$. Now, 
if $p>3$ any reductive group $G$ over $F$ is essentially tame. If $p=3$, there are $G$ which are not essentially tame: However, they are all triality forms and these are excluded. For details, see \cite[Prop. 7.2.1 and its proof]{PRlsv}, cf. \cite[\S 2.3.8]{DeligneCorvallis} and Prop. \ref{prop:PRlsv} below. 

b) If $p>2$ and $(G, \mu_h)$ is obtained, by completion at $p$, from a (global) Shimura datum $(\G, X)$ of abelian type in the sense of \cite{DeligneCorvallis}, then the pair
$(G, \mu_h)$ satisfies the standard assumptions.}
\end{Remark}

 \end{para}

\begin{para}
In what follows, we write $\BMloc_{\Gg,\mu}$ for the local model associated to the local model triple $(G,\{\mu\},\Gg)$. By definition, $\BMloc_{\Gg,\mu}=\BMloc_{\Gg^\circ,\mu}$ and is the unique, up to unique isomorphism, proper flat $\O_E$-scheme with $\Gg$-action, with generic fiber $G/P_\mu$ and reduced special fiber, which represents the $v$-sheaf ${\rm M}^v_{\Gg,\mu}$ over ${\rm Spd}(\O_E)$ defined in \cite{Schber}. (This is denoted by ${\rm Gr}_{\Gg,{\rm Spd}(\O_E),\mu}$ in \cite[Lect. 21]{Schber}.)

The existence of $\BMloc_{\Gg,\mu}$ was conjectured by Scholze-Weinstein \cite[Conj. 21.4.1]{Schber} and is shown in \cite{AGLR} under mild assumptions (which are weaker than the standard assumptions above), and in general in \cite{GL}.

In fact, under the above standard assumptions, we will construct $\BMloc_{\Gg,\mu}$ following  the work in \cite{PZ}, \cite{Levin}, \cite{HPR}, independently of the arguments of \cite{AGLR}, \cite{GL}, see Theorem \ref{thm:twoLM}. Our specific construction of $\BMloc_{\Gg,\mu}$ is important for the rest of the argument,  and is intertwined with the construction of a suitable embedding of the local model in a Grassmannian, see Theorem \ref{thm: LM embedding}.

\end{para}
 
 \begin{para}\label{perfectionLM}
The perfection of the special fiber of the local model $\BMloc_{\Gg,\mu}$ is a closed subscheme of the (perfect)
Witt vector affine Grassmannian ${\rm Gr}_\Gg^W=L^W\Gg/L^{W+}\Gg$ (\cite{ZhuAfGr}, \cite{BS}), see \cite[Thm 2.1, Thm. 7.23]{AGLR}. If $\K$ is an algebraically closed field of characteristic $p$,
\[
{\rm Gr}_\Gg^W(\K)=\frac{G(W(\K)[1/p])}{\Gg(W(\K))},
\]
hence there is a natural equivariant embedding
\[
\BMloc_{\Gg,\mu}(\K)\subset {\rm Gr}_\Gg^W(\K)=\frac{G(W(\K)[1/p])}{\Gg(W(\K))}.
\]

 \end{para}
 
 \begin{para}\label{vsheafmorphisms}
 Now consider local model data $(\Gg, \mu)$ of Hodge type and integral Hodge embeddings $\Gg\hook\GL(\La)$ extending $\rho: (G,\mu)\hook (\GL(V),\mu_d)$. 
 By functoriality and by using the full-faithfulness result of \cite[Prop. 18.4.1]{Schber}, we see that there is a canonical 
 equivariant morphism 
 \[
 \rho_*: \BMloc_{\Gg,\mu}\to {\rm Gr}(d,\La)_{\O_E}=\BMloc_{\GL(\La),\mu_d}\otimes_{\O}\O_E
 \]
  attached to $(\Gg,\mu)\hook (\GL(\La),\mu_d)$, where $\Gr(d,\Lambda)$ is the smooth Grassmannian classifying $d$-dimensional subspaces of $\Lambda$. This morphism identifies $\BMloc_{\Gg,\mu}$ with the normalization of its scheme theoretic image.
  Note that by \cite[Cor. 21.6.10 and its proof]{Schber}, we have ${\rm Gr}(d,\La)^\Diamond={\rm M}^v_{\GL(\La),\mu_d}$, and so $\BMloc_{\GL(\La),\mu_d}={\rm Gr}(d,\La)$
(this proves the Scholze-Weinstein conjecture for $\GL_n$). 
  
Suppose  that $\K$ is an algebraically closed field extension of $k_E$. Then, by combining with \S \ref{perfectionLM}, we obtain
 a commutative diagram of inclusions
\begin{equation}\label{diagram:Witt}
 \begin{aligned}
\xymatrix
{\BMloc_{\Gg,\mu}(\K)  \ar@{^{(}->}[r]\ar@{^{(}->}[d] &  {\rm Gr}_\Gg^W(\K) \ar@{^{(}->}[d]   \\
{\rm Gr}(d, \La)(\K) \ar@{^{(}->}[r] &{\rm Gr}^W_{\GL(\La)}(\K),}
\end{aligned}
\end{equation}
with  the   vertical arrows induced by $\Gg\hook \GL(\La)$.

\subsection{Local models via Beilinson-Drinfeld affine Grassmannians.}\label{sec: affine grassmannian}

\begin{para}\label{par:1.1.1}
	Let    $G$ be a (connected) reductive group over a field $\kappa$.  
	We let $\Gr_G:=LG/L^+G$ denote the affine Grassmannian for $G$; thus $\Gr_{G}$ is the ind-scheme over $\Spec(\kappa)$ which represents the fpqc sheaf associated to the functor given by  $R\mapsto G(R(\!(t)\!))/G(R[\![t]\!])$ on $\kappa$-algebras $R$. The affine Grassmannian $\Gr_G$
	 also  represents the functor on $\kappa$-algebras
	which sends $R$ to the isomorphism classes of pairs $(\calE,\phi)$ where
	\begin{altitemize}\item  $\calE$ is a $G$-torsor over $\Spec R\lps t\rps $,
		\item  $\phi:\calE[1/t]\xrightarrow{\sim} \calE^0[1/t]$ is a trivialization of the restriction $\calE[1/t]$ of the $G$-torsor $\calE$ to $\Spec(R\llps t\lrps)$.
	\end{altitemize}
	Here, $\calE^0$ denotes the trivial $G$-torsor.
	\end{para}

\begin{para}\label{para:Levin1}

Let $K_0/F$ be a finite unramified extension. Let $P(u)\in\calO_{K_0}[u]$ be a monic polynomial and $\underline{\calG}$ a smooth affine group scheme over $\calO_{K_0}[u]$ with geometrically connected fibers. We consider the functor $\mathrm{Fl}^{P(u)}_{\underline{\calG},0}$ on $\calO_{K_0}$-algebras $R$ given by $$\mathrm{Fl}^{P(u)}_{\underline{\calG},0}(R)=\{\text{iso. classes of pairs }(\calE,\beta)\},$$
where $\calE$ is a $\underline{\calG}$-torsor over $R[u]$ and $\beta:\calE|_{R[u][1/P(u)]}\xrightarrow{\sim} \calE^0$ is an isomorphism of $\underline{\calG}$-torsors, where $\calE^0$ denotes the trivial  $\underline{\calG}$-torsor. We then define 
 the mixed characteristic affine Grassmannian
 $$
 \mathrm{Fl}^{P(u)}_{\underline{\calG}}:=\mathrm{Res}_{\calO_{K_0}/\calO_{F}}\mathrm{Fl}^{P(u)}_{\underline{\calG},0}.
 $$
 By embedding $\underline\Gg$ into a general linear group, one deduces as in \cite[Prop. 4.1.4]{Levin}, that 
 $\mathrm{Fl}^{P(u)}_{\underline{\calG}}$ is representable 
  by an ind-scheme over $\calO_F$. 

 \end{para}
 
 \begin{para}\label{para:Levin2}

Let $(G,\{\mu\},\Gg)$  be a local model triple with $G\cong \mathrm{Res}_{K/F}H$.  Assume that $\Gg$ is the stabilizer of a point $\bx\in\calB(G, F)$. Then by \cite[Prop. 4.7]{HaRi}, we have $\Gg\cong \mathrm{Res}_{\O_K/\O_F}\Hh$. 

Assume now that  $H$ splits over a tamely ramified extension of $K$.
Let $K_0$ denote the maximal unramified extension of $F$ contained in $K$ and write $\O_{K_0}$ (resp. $k_0$) for its ring of integers (resp. residue field). We let $\O_{K_0}[u^\pm]$ denote the ring $\calO_{K_0}[u,u^{-1}]$. We fix a uniformizer $\pi$ of $K$ and we write $E(u)\in \calO_{K_0}[u]$ for the Eisenstein polynomial which is the  minimal polynomial for $\pi$ over $K_0$. Fix also a rigidification $(H, A, S, P)$ of $H$ in the sense of \cite[Def. 2.7]{PZ}, cf. \cite[\S 3.1]{Levin}, in which $A$ is a maximal split torus of $H$ over $K$ such that $\bx\in \B(H, K)$ lies in the apartment corresponding to $A$. Denote by $\underline H$ the reductive group scheme over $\O_{K_0}[u^{\pm}]$ constructed by \cite[Prop. 3.1.2]{Levin}. This extends the group $H$ in the sense that the base change of $\underline H$ by $\O_{K_0}[u^{\pm}]\to \O_K$, $u\mapsto \pi$, is $H$. Then, \cite[Thm 3.3.3]{Levin}, cf. \cite[Thm. 4.1]{PZ}, gives a smooth affine group scheme $\underline\Hh^\circ$ over $\O_{K_0}[u]$,  with geometrically connected fibers, extending $\underline H$ which also specializes to $\calH^\circ$ under the map $\calO_{K_0}[u]\rightarrow \calO_{K},$ $u\mapsto \pi$.

 Applying the construction of \S\ref{sec: affine grassmannian} we obtain the ind-scheme $\mathrm{Fl}_{\underline{\calH}^\circ}^{E(u)}$ 
 over $\calO_F$ which is ind-projective by \cite[Thm. 4.2.11]{Levin}.

\begin{Remark}
{\rm In \cite[Thm. 3.3]{Levin}, \cite[Thm 4.1]{PZ}, it is assumed that the group scheme is parahoric, in particular connected.\footnote{In \cite{PZ}, $\Gg_{\bx}$ stands for the connected stabilizer; here this is denoted $\Gg_{\bx}^\circ$.} A similar argument as in loc. cit., can also be used to construct a smooth affine
$\underline\Hh$ over $\O_{K_0}[u]$  extending $\underline H$ which   specializes to the Bruhat-Tits stabilizer $\calH$ under the map $\calO_{K_0}[u]\rightarrow \calO_{K},$ $u\mapsto \pi$. 
Such a construction will appear in \S \ref{par: Levin local models again},  under some additional assumptions.  
}
\end{Remark}
 \end{para}

\begin{para}{}\label{sec: Levin local models}

For a $K_0$-algebra $R$, the completion $\widehat {R[u]}$ of $R[u]$ along the ideal $(E(u))$, contains the completion of $K_0[u]$ along $(E(u))$. The latter ring may be identified with $K\lps t\rps$, by a map sending $t$ to $E(u)$ and inducing the identity on residue fields.
Then $\widehat {R[u]}$ may be identified with $(K\otimes_{K_0} R)\lps t\rps$ by sending $t$ to $E(u)$. 
This induces an isomorphism from 
the generic fiber of $\mathrm{Fl}_{\underline{\calH}^\circ,0}^{E(u)}$ to the affine Grassmannian $\mathrm{Gr}_{\mathrm{Res}_{K/K_0}H}$ (cf. \cite[Cor. 3.5]{HaRi}), and hence an isomorphism from the generic fiber of $\mathrm{Fl}^{E(u)}_{\underline{\calH}^\circ}$ to $\mathrm{Gr}_{\mathrm{Res}_{K/F}H}\cong\mathrm{Gr}_{G}$.

A representative $\mu$ of $\{\mu\}$  over $\bar{F}$ determines an element of $G(\bar{F}\llps t\lrps)$ and hence a point $e_\mu := \mu(t)\in \mathrm{Gr}_G(\bar{F})$. The (affine) Schubert variety $S_\mu$ is the closure of the $G(\bar{F}\lps t\rps)$-orbit of $e_\mu$ in 
$\mathrm{Gr}_G$. The  conjugacy class $\{\mu\}$ has the reflex field $E$ as a minimal field of definition and the Schubert variety $S_\mu\subset \mathrm{Gr}_G$ is defined over $E$.  

\begin{Definition}\label{def:LevinLM}
The local model $\rM_{\Gg, \mu}=\rM_{\Gg^\circ,\mu}$ is defined to be the Zariski closure of $S_\mu$ in $\mathrm{Fl}^{E(u)}_{\underline{\calH}^\circ}\otimes_{\calO_F}\calO_E$.
\end{Definition}

\begin{Remark}\label{rem: local models}{\rm

a) Note that the input for the constructions above is a group scheme $\mathcal{H}$ over $\calO_K$ and a finite extension $K/F$. When $K=F$,  the group scheme $\underline{\calH}^\circ$ and the mixed characteristic affine Grassmannian  $\mathrm{Fl}^{u-\pi}_{\underline{\calH}^\circ}$ agrees with those constructed by in \cite{PZ}.
In this case, it follows from \cite[Thm. 2.7]{HPR} that the local model $\rM_{\Gg, \mu}$ only depends on the local model triple $(G,\{\mu\},\Gg)$ and not on the choice of uniformizer $\pi$. 
 
b) More generally, for an arbitrary $K$ and under some additional assumptions,  we show that the $\rM_{\calG, \mu}$ satisfy Conjecture 21.4.1 of \cite{Schber}, and hence are independent of the choice of $K$, and uniformizer $\pi$ (cf. Theorem \ref{thm:twoLM}).   }
	\end{Remark}

\end{para} 

\begin{para}\label{sec: local models product} In general, if $G$ is quasi-tame,  choose an isomorphism $G\cong\prod_{i=1}^r\mathrm{Res}_{K_i/F}H_i$, with $H_i$ splitting over a tame extension, and   set 
\[
\rM_{\Gg, \mu }:=\prod_{i=1}^r\rM_{\Gg_i, \mu_i}\otimes_{\calO_{E_i}}\calO_E.
\]
 Here $\Gg_i$ with generic fiber $\mathrm{Res}_{K_i/F}H_i$ is determined by $\Gg\cong\prod_{i=1}^r\Gg_i$, $\{\mu_i\}$ is the $\mathrm{Res}_{K_i/F}H_i$-factor of the $G$-conjugacy class $\{\mu\}$, 
 and $E_i$ (resp. $E$) is the field of definition of $\{\mu_i\}$ (resp. $\{\mu\}$). 
 The following theorem follows immediately from \cite[Thm. 4.2.7]{Levin}.

\begin{thm}\label{thm: Levin}
	Suppose $G$   is quasi-tame and that $p$ does not divide the order of $\pi_1(G^{\der})$. Then the  scheme 
$\rM_{\Gg, \mu} $, defined as above, is normal with reduced special fiber. Moreover each geometric irreducible component of $\rM_{\Gg, \mu }\otimes_{\calO_E} k$ is normal and Cohen--Macaulay.\footnote{In fact, $\rM_{\Gg, \mu }$ is Cohen-Macaulay, see \cite[Thm. 2.1]{HR22}.} \qed
\end{thm}
\end{para}

\begin{para}

We now extend this construction of local models to a more general situation.

\begin{prop}\label{prop:PRlsv} Suppose $p>2$ and $(G,\mu)$ is of abelian type. Assume that $\{\mu^\ad\}$
is non-trivial in every $F$-simple factor of $G^\ad$. Then we can find 
 $(G',\mu')$ of Hodge type with an isomorphism $(G'^{\ad},\mu'^{\ad})\simeq (G^\ad,\mu^\ad)$ satisfying the following properties:
 
\begin{altitemize}
\item[1)] $p\nmid|\pi_1(G'^\der)|$,

\item[2)] $G'=\prod_{i=1}^r \Res_{K_i/F}H'_i$ where $K_i/F$ are finite extensions and $H'_i$ is a reductive group over $K_i$ which splits 
over a tame extension.

\item[3)] $E'=E^\ad$, where $E'$ (resp. $E^\ad$) is the reflex field for $\{\mu'\}$ (resp. $\{\mu^\ad\}$).

\item[4)] There are faithful minuscule representations $\rho_i: H'_i\to \GL(V_i)$ over $K_i$, such that, for all $i$,
the compositions
\[
 \Res_{K_i/F}H'_i \xrightarrow{\ \Res_{K_i/F}(\rho_i)\ } \Res_{K_i/F}\GL(V_i)\hook \GL(V_i)
\]
 give Hodge embeddings for $(\Res_{K_i/F}H'_i, \{\mu'_i\})$.
 Here, $(\Res_{K_i/F}H'_i, \{\mu'_i\})$ are the local Shimura pairs determined from $(G',\{\mu'\})$.

\end{altitemize}
 \end{prop}
 
 \begin{proof}
This follows from \cite[Prop. 7.2.1]{PRlsv} and its proof. (A similar argument, in the analogous situation of global Shimura data, also appears in \S \ref{para: Deligne Hodge type lifting}.)
 \end{proof}
 
 \end{para}
 
 \begin{para}\label{sss:generalLM}
Assume now $(G,\{\mu\}, \Gg)$ satisfies the standard assumptions \ref{stdASSumptions}. We construct a local model $\rM^{\rm loc}_{\Gg,\mu}$
for $(G,\{\mu\}, \Gg)$  as follows: We write $G^\ad_1\times G^\ad_2$, where $G^\ad_1$ (resp. $G^\ad_2$) is the product 
of the $F$-simple factors of $G^\ad$ where $\mu^\ad$ is non-trivial (resp. trivial). Let $G_1$ be the kernel of $G\to G^\ad_2$.
Then $\{\mu\}$ factors through $G_1$ and we denote by $\{\mu_1\}$ for the induced conjugacy class of cocharacters. The morphism 
$G_1\to G^\ad_1$ is a central extension and $(G_1,\mu_1)$ is of abelian type and satisfies the assumptions of Proposition \ref{prop:PRlsv} above.
Let $(G',\mu')$ be as in the conclusion of Proposition \ref{prop:PRlsv} applied to  $(G_1,\mu_1)$. Now define
\begin{equation}\label{eq:generalLM}
\rM^{\rm loc}_{\Gg, \mu}:=\rM_{\Gg',\mu'}\otimes_{\O_{E'}}\O_E.
\end{equation}
This is a flat projective $\O_E$-scheme with reduced special fiber, by Theorem \ref{thm: Levin}.

\begin{Remark}
{\rm 
Note that if $G$ is quasi-tame, we also have the ``local model" $\rM_{\Gg,\mu}$ from Definition \ref{def:LevinLM}. However, when $p$ divides $|\pi_1(G^\der)|$, the schemes
 $\rM^{\rm loc}_{\Gg, \mu}$ and $\rM_{\Gg,\mu}$ are not always isomorphic, because $\rM_{\Gg,\mu}$ might not be normal, as was first observed by Haines-Louren\c co-Richarz, see \cite{HLR}.
 }
\end{Remark}

We will show: 
 
\begin{thm}\label{thm:twoLM}
If $(G,\{\mu\}, \Gg)$ satisfies the standard assumptions then $\rM^{\rm loc}_{\Gg, \mu}$, as defined by (\ref{eq:generalLM}) above, satisfies the Scholze-Weinstein conjecture, so
$
\BMloc_{\Gg,\mu}=\rM^{\rm loc}_{\Gg, \mu}.
$
In particular, $\rM^{\rm loc}_{\Gg, \mu}$ is independent, up to unique isomorphism, of all choices made 
in its construction.
\end{thm}

This will follow as a consequence of Theorem \ref{thm: LM embedding} below. This implication is shown in \S \ref{proofoftwoLM}.

\end{para}

\subsection{Embeddings of local models}

\begin{para}\label{FixedLevin} Let  $(G,\{\mu\},\Gg)$ be a local model triple over $F$ with $G\simeq \Res_{K/F}H$, where $H$ splits over a tamely ramified extension of $K$. 
	We fix the isomorphism above and just write $G=\Res_{K/F}H$.
	Assume $\bx\in \B(H, K)=\B(G, F)$ is generic in its facet and let 
	$\Hh=\Hh_{\bx}$, resp. $\Gg=\Gg_{\bx}$, be the Bruhat-Tits stabilizer group 
	schemes for $H$, resp. $G$, over $\O_F$, resp. $\O_K$. We have
	\[
	\Gg\cong\Res_{\O_K/\O_F}\Hh.
	\]
	
	Now suppose that the reductive group $H$ over $K$ and $\bx\in \B(H, K)$ satisfies all the assumptions of Proposition \ref{Fixed}.
	Let $\bx'$, $\ti K/K$, $\Gam=\Gal(\ti K/K)$ be as in the conclusion of Proposition \ref{Fixed}: Then $\ti H:=H\otimes_K \ti K\simeq H_0\otimes_{\Z_p}\ti K$ is split and the point $\bx'$ is hyperspecial over $\ti K$. In this, $H_0$ is the Chevalley form of the split group $\ti H$. Again, $\ti \Hh=\ti\Hh_{\bx'}\simeq H_0\otimes_{\Z_p}\O_{\ti K}$ is the corresponding hyperspecial group scheme for $\ti H$ over $\O_{\ti K}$ and we have
	\[
	\Hh\simeq (\Res_{\O_{\ti K}/\O_K}\ti\Hh)^\Gam.
	\]
	Consider the map
	\begin{equation}\label{map541}
	G=\Res_{K/F}H\to \Res_{{\ti K}/F}(H_0\otimes_{\Z_p}\ti K)=\Res_{K/F}(\Res_{{\ti K}/K}(H_0\otimes_{\Z_p}\ti K)).
	\end{equation}
	given by applying restriction of scalars to  $$H\to \Res_{\ti K/K}(H\otimes_K\ti K)\simeq \Res_{{\ti K}/K}(H_0\otimes_{\Z_p}\ti K).$$
	This extends to the closed immersion of group schemes 
	\begin{equation}\label{map542}
	\Gg=\Res_{\O_{K}/\O_F}\Hh\to  \Res_{\O_{\ti K}/\O_F}(H_0\otimes_{\Z_p}\O_{\ti K}).
	\end{equation}
by Proposition \ref{prop: R-smoothness properties}. We let  $\ti\mu$ be the geometric cocharacter of $\Res_{{\ti K}/F}(H_0\otimes_{\Z_p}\ti K)$ which is given by composing $\mu$ with the map (\ref{map541}).
	Then $$(\Res_{{\ti K}/F}(H_0\otimes_{\Z_p}\ti K), \{\ti\mu\}, \Res_{\O_{\ti K}/\O_F}(H_0\otimes_{\Z_p}\O_{\ti K}))$$ is a local model triple with reflex field $\tilde{E}$ and 
	\begin{equation}\label{map543}
	(G, \{\mu\}, \Gg)\to (\Res_{{\ti K}/F}(H\otimes_K\ti K), \{\ti\mu\}, \Res_{\O_{\ti K}/\O_F}(H_0\otimes_{\Z_p}\O_{\ti K}))
	\end{equation}
	a morphism of local model triples.
\end{para}

\begin{para}\label{par: Levin local models again}We will show \eqref{map543} induces a closed immersion of local models $${\rM}_{\Gg,\mu}\to ({\rM}_{\Res_{\O_{\ti K}/\O_F}(H_0\otimes_{\Z_p}\O_{\ti K}),\ti \mu})\otimes_{\O_{\ti E}}\O_E.$$
	To do this, we recall some aspects of the construction of the group schemes $\underline\calH^\circ$ from \S\ref{para:Levin1}. 
We let $K_0$  (resp. $\ti K_0$) be the maximal unramified extension of $F$ in $K$ (resp. $\ti K$), and  
	we  set 
	\[
	\underline{\ti\Hh}=H_0\otimes_{\Z_p}\O_{\ti K_0}[\ti u].
	\]
	If $e$ is the  ramification degree of the tame extension $\ti K/K$, then, after possibly enlarging $\ti K$, we can find a uniformizer $\pi$ of $\O_{K}$ and a uniformizer $\ti\pi$ of $\O_{\ti K}$ such that $\ti\pi^e=\pi$. We can then identify
	$\Gam=\Gal(\ti K/K)$ with the Galois group of the cover $\O_{\ti K_0}[\ti u^{\pm}]/\O_{K_0}[u^{\pm}]$ given by $u\mapsto \ti u^e$; this identification is compatible with the specializations 
	$u\mapsto\pi$, $\ti u\mapsto\ti\pi$. For typesetting simplicity, in what follows we will write 
	\[
	\O_0=\O_{K_0},\quad \ti \O_0:=\O_{\ti K_0}.
	\]
	According to the construction in \cite{Levin}, \cite{PZ}, there is a semi-linear action of $\Gam$ on the group scheme $H_0\otimes_{\Z_p}\ti\O_{0}[v]$ and one considers
	\[
	\underline\Hh:= (\Res_{\ti \O_{0}[\ti u]/\O_{0}[u]}(H_0\otimes_{\Z_p}\ti\O_{0}[\ti u]) )^\Gam.
	\]
	This is an affine group scheme over $\O_{0}[u]$ which is smooth by Edixhoven's lemma. 
	It now follows from the construction in the proof or by the uniqueness statement in \cite[Thm. 3.3]{Levin}, cf. \cite[\S 4.2.1]{PZ}, that, as the notation suggests, the group scheme $\underline \Hh^\circ$ 
	given by \cite[Thm. 3.3]{Levin} is isomorphic to the neutral connected component of $\underline\Hh$. 
	Then
	\[
	\underline\Hh\to \Res_{\ti\O_{0}[\ti u]/\O_{0}[u]}(H_0\otimes_{\Z_p}\ti\O_{0}[\ti u])  
	\]
	is a closed immersion of group schemes over $\O_{0}[u]$ lifting (\ref{map542}), and 
	\[
	\underline\Hh^\circ\to \Res_{\ti\O_{ 0}[\ti u]/\O_{0}[u]}(H_0\otimes_{\Z_p}\ti\O_{0}[\ti u])  
	\]
	is a locally closed immersion. This gives a natural morphism 
	\begin{equation}\label{imm541}
	{\rm Fl}^{E(u)}_{\underline\Hh^\circ}\to {\rm Fl}^{E(u)}_{\Res_{\ti\O_{ 0}[\ti u]/\O_{ 0}[u]}(H_0\otimes_{\Z_p}\ti\O_{ 0}[\ti u])}
	\end{equation}
	between the Beilinson-Drinfeld style affine Grassmannians of \cite{Levin} over $\O_{F}$.  
	\begin{prop}\label{prop:imm542}
		The natural morphism 
		\begin{equation}\label{imm542}
		{\rM}_{\Gg,\mu}={\rM}_{\Res_{\O_K/\O_F}\Hh,\mu}\to ({\rM}_{\Res_{\O_{\ti K}/\O_F}(H_0\otimes_{\Z_p}\O_{\ti K}),\ti \mu})\otimes_{\O_{\ti E}}\O_E,
		\end{equation}
		induced by (\ref{imm541}), is a closed immersion.
	\end{prop}
	\begin{proof}
	This follows by the above and the argument in the proof of \cite[Prop. 8.1]{PZ}.
	\end{proof}
\end{para}

\begin{para}\label{ss:2lattices}
We now slightly digress   to give a result about minuscule representations which will be useful later. 

	Let $H_0$ be a split reductive group scheme over $\Z_p$. Let $L$ be a field extension of $\Q_p$ and let $\rho: H_0\otimes_{\Z_p}L\to \GL(V)$ be a  representation over $L$.  Choose a maximal torus $T_0\simeq {\mathbb G}_m^r$ and a Borel $B_0$ of $H_0$ containing $T_0$. Let $\{\lambda_1,\ldots , \lambda_n\}$ be the (distinct) highest weights of $T_0$ that appear in the highest weight decomposition of $V$ and denote by $V_{\Z_p}(\lambda_i)$ the Weyl module with highest weight $\lambda_i$ over $\Z_p$. Then there is an $H_0\otimes_{\Z_p}L$-equivariant isomorphism
	\[
	V\simeq \bigoplus_{i=1}^n V_{\Z_p}(\lambda_i)^{\oplus m_i}\otimes_{\Z_p} L
	\]
	where $m_i\geq 1$ are corresponding multiplicities. Set
	\[
	\La_0=\bigoplus_{i=1}^n V_{\Z_p}(\lambda_i)^{\oplus m_i}
	\]
	which supports an $H_0$-representation, i.e. a group scheme homomorphism
	\[
	\rho_0: H_0\to \GL(\La_0).
	\]
	If  $\rho_0\otimes_{\Z_p}L\simeq \rho$ is faithful, by \cite[Cor. 1.3]{PYuGS}, $\rho_0$ is a closed immersion.
	
	\begin{lemma}[{cf. \cite[Prop. 1.10]{KP}}]\label{2lattices}
		Let $H_0$ be a split reductive group over $\Z_p$. Let $R$ be a discrete valuation ring with fraction field $L$ of characteristic $0$ and $\rho: H_0\otimes_{\Z_p} L\to \GL(V)$ a minuscule representation over $L$.  Suppose that $\La$, $\La' $ are two $R$-lattices in $V$ such that $\rho$ extends to group scheme homomorphisms $\rho(\La): H_0\otimes_{\Z_p}R\to \GL(\La)$ and $\rho(\La'): H_0\otimes_{\Z_p}R\to \GL(\La')$. Then, there is $g\in \GL(V)$ centralizing $\rho(H_0\otimes_{\Z_p} L)$ such that $\La'=g\cdot\La$. In particular, $g$ gives an isomorphism $g: \La\xrightarrow{\sim } \La'$ which intertwines $\rho(\La)$ and $\rho(\La')$.
	\end{lemma}
	
	\begin{proof}
		As above, we fix a maximal torus $T_0$ and a Borel subgroup $B_0$ of $H_0$. Let $\{\lambda_1,\ldots , \lambda_n\}$ be the (distinct) highest weights that appear in the highest weight decomposition of $V$. Then, since $\rho$ is minuscule, all the weights appearing in $V$ are of the form $w\cdot \lambda_i$, $w\in W=N_{H_0}(T_0)/T_0$.  Write
		\[
		\La=\bigoplus_{\lambda\in X^*(T_0)} \La_\la,\quad \La'=\bigoplus_{\lambda\in X^*(T_0)} \La'_\la,
		\]
		for the direct sum decompositions induced by the action of the torus $T_0$ via $\rho(\La)$, $\rho(\La')$; in these, $\La_\la$, $\La'_\la\subset V_\la$ are both lattices in the corresponding $L$-vector space  $V_\la$. For each $w$ pick a  representative $n_w\in N_{H_0}(T_0)$. Then
		we have $\La_{w\cdot\la}=\rho(n_w)\La_{\la}$, $\La'_{w\cdot\la}=\rho(n_w)\La'_{\la}$.

		If $g\in \GL(V)$ centralizes $\rho(H_0\otimes_{\Z_p} L)$, then we can consider $g_{|V_\la}\in \GL(V_\la)$ and set $g_i=g_{|V_{\lambda_i}}$. By Schur's lemma, $g\mapsto (g_i)_i$ gives an isomorphism of the centralizer $Z(H):=Z_{\GL(V)}(\rho(H_0\otimes_{\Z_p}F))$ with the group $\prod_{i=1}^n \GL(V_{\lambda_i})$. Choose $g\in Z(H)\subset \GL(V)$ that corresponds to $(g_i)_i$ with $g_{i}:V_{\la_i}\xrightarrow{\sim} V_{\la_i}$ such that $g_i\cdot \La_{\la_i}=\La'_{\la_i}$. Then, since $\La_{w\cdot\la_i}=\rho(n_w)\La_{\la_i}$, $\La'_{w\cdot\la_i}=\rho(n_w)\La'_{\la_i}$, we also have $g\cdot \La=\La'$.
	\end{proof}
\end{para}

\begin{para} Let us now combine this with the set-up of \S \ref{FixedLevin}.  We  consider  a faithful minuscule representation $\rho: H\to \GL(V)$ 
	over $K$ with base change
	\[
	\rho\otimes_K \ti K: H\otimes_K\ti K\to \GL(V\otimes_K\ti K).
	\]
	Recall that  $H\otimes_K\ti K\simeq H_0\otimes_{\Z_p} \ti K$ is split. We assume that the composition of $\ti \mu$ with 
	$\rho\otimes_K \ti K$ is minuscule. We have a group scheme homomorphism
	\[
	\underline \rho_0:= \rho_0\otimes_{\Z_p}\ti\O_{0}[\ti u]: H_0\otimes_{\Z_p} \ti\O_{0}[\ti u]\to \GL(\La_0\otimes_{\Z_p}\ti\O_{0}[\ti u])
	\]
	over $\ti\O_{ 0}[\ti u]$. By restriction of scalars, this induces
	\begin{equation}\label{map54a}
	\Res_{\ti\O_{ 0}[\ti u]/\O_{ 0}[u]}(H_0\otimes_{\Z_p} \ti\O_{0}[\ti u])\to  \Res_{\ti\O_{0}[\ti u]/\O_{0}[u]}(\GL(\La_0\otimes_{\Z_p}\ti\O_{0}[\ti u]))
	\end{equation}
	over $\O_{0}[u]$. Since $\rho_0$ is a closed immersion \cite[Cor. 1.3]{PYuGS}, $ \underline \rho_0$ and $ \Res_{\ti\O_{ 0}[\ti u]/\O_{ 0}[u]}( \underline \rho_0)$  are also closed immersions of group schemes.

	Base changing the morphism (\ref{map54a}) along $\O_{0}[u]\to \O_K$, $u\mapsto \pi$, gives
	\[
	\Res_{\O_{\ti K}/\O_K}(\rho_0\otimes_{\Z_p}\O_{\ti K}):  \Res_{\O_{\ti K}/\O_K}(H_0\otimes_{\Z_p} \O_{\ti K})\to \Res_{\O_{\ti K}/\O_K}\GL(\La_0\otimes_{\Z_p}\O_{\ti K}) 
	\]
	over $\O_K$.
	
	Since (\ref{map54a}) is a closed immersion, it follows
	that the corresponding morphism  
	\begin{equation}\label{imm543}
	{\rm Fl}^{E(u)}_{\Res_{\ti\O_{0}[\ti u]/\O_{0}[u]}(H_0\otimes_{\Z_p}\ti\O_{0}[\ti u])}\to {\rm Fl}^{E(u)}_{\Res_{\ti\O_{0}[\ti u]/\O_{0}[u]}(\GL(\La_0\otimes_{\Z_p}\ti\O_{0}[\ti u]))}
	\end{equation}
	of affine Grassmannians is a monomorphism and hence a closed immersion of ind-projective schemes
	over $\O_{F}$. As above, this implies
	
	\begin{prop}\label{prop:imm544} The morphism 
		\begin{equation}\label{imm544}
		\rM_{  \Res_{\O_{\ti K}/\O_F}(H_0\otimes_{\Z_p} \O_{\ti K}),\ti\mu}\to (\rM_{\Res_{\O_{\ti K}/\O_F}\GL(\La_0\otimes_{\Z_p}\O_{\ti K}),\ti\mu'})\otimes_{\O_{\ti E'}}\O_{\ti E}
		\end{equation}
		of local models obtained from (\ref{imm543}) is  a closed immersion. \endproof 
	\end{prop}
	
	In the above, $\ti\mu'$ is the geometric cocharacter of $\Res_{{\ti K}/F}\GL(V\otimes_K \ti K)$ obtained
	by composing $\Res_{\ti K/F}(\rho\otimes_K\ti K)$
	with $\ti\mu$.
\end{para}

\begin{Remark}
	{\rm 
		Note  that \cite[Cor. 3.6]{HaRi} applied to the finite flat morphism
		$\Spec( \ti\O_0[\ti u])\to\Spec(\O_0[u])$ given by $u\mapsto \ti u^e$, gives a natural isomorphism
		\begin{equation}\label{natural644}
		{\rm Fl}^{E(u)}_{\Res_{\ti\O_{0}[\ti u]/\O_{0}[u]}(\GL(\La_0\otimes_{\Z_p}\ti\O_{0}[\ti u]))}\xrightarrow{\sim} {\rm Fl}^{\ti E(\ti u)}_{\GL(\La_0\otimes_{\Z_p}\ti\O_{0}[\ti u])}
		\end{equation}
		of ind-schemes over $\O_F$.
		Here $\ti E(\ti u)=E(\ti u^e)$ is the Eisenstein polynomial of $\ti\pi$ in $\ti\O_{0}[\ti u]$. This reflects the identifications 
		\[
		\Res_{K/F}(\Res_{\ti K/K}\GL(\La_0\otimes_{\Z_p}\ti K))=
		\Res_{\ti K/F}\GL(\La_0\otimes_{\Z_p}\ti K),
		\]
		\[
		\Res_{\O_K/\O_F}(\Res_{\O_{\ti K}/\O_K}\GL(\La_0\otimes_{\Z_p}\O_{\ti K}))=
		\Res_{\ti K/F}\GL(\La_0\otimes_{\Z_p}\O_{\ti K}).
		\]
		Indeed, since $\ti K/K$ is tame, $\Res_{\ti K/K}\GL(\La_0\otimes_{\Z_p}\ti K)$ splits over the tame extension
		$\ti K/K$ and the two sides in this identification lead to two -a priori different- constructions as in \cite{Levin}.
		However, the 
		isomorphism (\ref{natural644}) above gives an identification between the
		two possible definitions for the local model 
		$\rM_{\Res_{\O_{\ti K}/\O_F}\GL(\La_0\otimes_{\Z_p}\O_{\ti K}),\ti\mu'}$.
		A similar comment applies to the local model $\rM_{  \Res_{\O_{\ti K}/\O_F}(H_0\otimes_{\Z_p} \O_{\ti K}),\ti\mu}$.}
\end{Remark}

 \begin{para}For $S$ an $R$-algebra and $V$ a module over $S$, we write $V^{(R)}$ for the $R$-module obtained by restriction of structure.
	Let $\La$ be any $\O_{K}$-lattice in a finite dimensional $K$-vector space $V$ and $K/F$ a finite extension. (We will eventually apply this to $K$ replaced by $\ti K$, to connect 
	with the previous set-up.)
	Consider  the natural homomorphism
	\begin{equation}\label{BCimm}
	\Res_{\O_{K}/\O_F} \GL(\La)\to \GL(\La^{(\calO_F)})
	\end{equation}
	of group schemes over $\O_F$. 
	We can easily see that this is a closed immersion by writing down the equations giving this morphism. Consider  a geometric minuscule cocharacter
	$\mu$ of $ \Res_{K/F} \GL(V)$ with reflex field $E$.

	\begin{prop}\label{prop:closedLMforGL}
		There is a closed immersion
		\begin{equation}\label{imm55prop}
		\rM_{\Res_{\O_{K}/\O_{F}} \GL(\La ), \mu}\hook   \rM_{\GL(\La^{(\calO_F)} ), \mu} \otimes_{\O_F}\O_{E}
		\end{equation}
		equivariant for the homomorphism (\ref{BCimm}) above which extends the natural morphism between Grassmannians on the generic fibers.
	\end{prop} 
	
	\begin{proof} Lift $\La$ to a finite free $\O_0[u]$-module $\underline{\La}$ and consider the smooth affine group scheme $\underline{\mathcal {GL}}=\GL(\underline{\La})$ over $\O_0[u]$. This is the $\O_0[u]$-group scheme associated to $\GL(\La)$ and the extension $K/F$ as in \S \ref{para:Levin2}. Write $\underline{\mathcal {GL}}_F$
		for the group scheme of linear automorphisms
		of $\underline{\La}$ considered as a $\O_F[v]$-module by restriction of scalars by $\O_F[v]\to \O_0[u]$, $v\mapsto E(u)+\pi_F$.
		The group scheme $\underline{\mathcal {GL}}_F$ is split reductive over $\O_F[u]$ and so the local model $\rM_{\GL(\La ), \mu}$ above is naturally a closed subscheme of 
		${\rm Fl}_{\underline{\mathcal {GL}}_F}^{v-\pi_F}$. Here, ${\rm Fl}_{\underline{\mathcal {GL}}_F}^{v-\pi_F}$ 
		is defined by applying the definition in \S \ref{para:Levin1} with $K=F$. We will show that there is a map
		\[
		{\rm Fl}^{E(u)}_{\underline{\mathcal {GL}}}\to {\rm Fl}_{\underline {\mathcal {GL}}_F}^{v-\pi_F}.
		\]

		Consider the $\O_0$-algebra homomorphism
		\[
		r: \O_0[v]\to \O_0[u],\quad v\mapsto E(u)+\pi_F
		\]
		which lifts the inclusion $\O_0\hook \O_K$, via $v\mapsto \pi_F$, $u\mapsto \pi$.
		Then $r$ is  finite and flat. Let $\underline{\mathcal {GL}}_{K/K_0}$ the group scheme
		obtained by Weil restriction of $\underline{\mathcal {GL}}$ along $r$; then the base change of 
		$\underline{\mathcal {GL}}_{K/K_0}$ along $\O_0[v]\to \O_0=$, $v\mapsto \pi_F$, is identified with
		$\Res_{\O_{K}/\O_0}\GL(\La)$. Denote by $\underline{\mathcal {GL}}_{K_0}$ the group scheme of linear
		automorphisms of $\und\La$ regarded as an $\O_0[v]$-module via $r: \O_0[v]\to \O_0[u]$.
		We will first give a map 
		\[
		{\rm Fl}^{E(u)}_{\underline{\mathcal {GL}},0}\to {\rm Fl}_{\underline {\mathcal {GL}}_{K_0}, 0}^{u-\pi_F}
		\]
		over $\O_0$. See \S \ref{para:Levin1} for the definition of these ind-schemes. (This amounts to constructing the map in the special case $F=K_0$.)
		
		We start by giving a morphism
		\[
		i: \underline{\mathcal {GL}}_{K/K_0}\to \underline{\mathcal {GL}}_{K_0}
		\]
		over $\O_0[v]$ extending the morphism  $\Res_{\O_{K}/\O_0}\GL(\La)\to \GL(\La^{(\O_0)})$ of $\O_0$-group schemes
		under the specialization $v\mapsto \pi_F$. This morphism is obtained by viewing an $\O_0[u]$-automorphism
		of $\underline{\La}$ as an $\O_0[v]$-automorphism of $\underline{\La}$  viewed as an $\O_0[v]$-module via $r$.
		The base change of $i$ to $k\lps v\rps$
		\[
		i_{k\lps v\rps}: \underline{\mathcal {GL}}_{K/K_0, k\lps v\rps}\to \underline{\mathcal {GL}}_{K_0, k\lps v\rps}
		\]
		is a closed immersion since it is induced by restriction of scalars from $k\lps u\rps$-lattices to $k\lps v\rps$-lattices under the map
		$v\mapsto u^{[K:K_0]}$.
		
		By \cite[Cor. 3.6]{HaRi}, the Weil restriction of torsors along $r$ induces an isomorphism
		\[
		{\rm Fl}^{E(u)}_{\underline{\mathcal {GL}},0}\xrightarrow{\sim} {\rm Fl}^{u-\pi_F}_{ \underline{\mathcal {GL}}_{K/K_0},0}.
		\]
		Combining this isomorphism with the map given by taking push-outs of torsors along $i$, we obtained the required map
		\[
		\iota_0:  {\rm Fl}^{E(u)}_{\underline{\mathcal {GL}},0}\simeq {\rm Fl}^{v-\pi_F}_{ \underline{\mathcal {GL}}_{K/K_0},0}\to {\rm Fl}^{v-\pi_F}_{\underline{\mathcal {GL}}_{K_0},0}.
		\]
		Applying $\Res_{\O_0/\O_F}$ we obtain a map
		\[
		\iota:  {\rm Fl}^{E(u)}_{\underline{\mathcal {GL}}}\to \Res_{\O_0/\O_F}{\rm Fl}^{v-\pi_F}_{\underline{\mathcal {GL}}_{K_0}}.
		\]
		A standard argument (\cite[Thm. 1.4]{PR}) shows that $\iota\otimes_{\O_F}k$ is a locally closed immersion. Since the domain of this map is ind-projective, it follows that
		$\iota\otimes_{\O_F}k$ is a closed immersion.
		We now compose this with the map
		\[
		\iota': \Res_{\O_0/\O_F}{\rm Fl}^{v-\pi_F}_{\underline{\mathcal {GL}}_{K_0}}\to {\rm Fl}^{v-\pi_F}_{\underline{\mathcal {GL}}_F}
		\]
		obtained by the construction of \cite{Levin} applied to  $\Res_{\O_0/\O_F}\GL(\La^{(\O_0)})\to \GL(\La^{(\O_F)})$. %Here, the $\O_K$-module $\La$ is considered first as an $\O_0$-module
		%and then as an $\O_F$-module by restriction of scalars. 
		We can easily see that $\iota'\otimes_{\O_F}k$ is a closed immersion, cf. \cite[Prop. 8.1]{PZ}. It follows that 
		the composite map $\iota'\cdot \iota$ is a closed immersion on special fibers. 
		
		Restricting to the local models we obtain a map
		\begin{equation}\label{immLM}
		\rM_{\Res_{\O_{K}/\O_{F}} \GL(\La ), \mu}\hook   \rM_{\GL(\La^{(\calO_F)}), \mu} \otimes_{\O_F}\O_{E}
		\end{equation}
		which is a closed immersion on special fibers. An argument involving Nakayama's lemma as in \cite[Prop. 8.1]{PZ}, 
		shows that (\ref{immLM}) is itself a closed immersion. Finally, it remains to check that (\ref{immLM})
		extends to the natural morphism on generic fibers. This follows from the definitions of local models in \S \ref{sec: Levin local models} and 
		the fact that $r$ takes $v-\pi_F$ to $E(u)$.
	\end{proof}

\end{para}

\begin{para}
We now combine the previous results to show that a suitable Hodge embedding induces a closed immersion of local models. 	

We will consider local model triples $(G, \{\mu\}, \Gg)$  over $F$ with $G$ quasi-tame, $G\simeq \prod_{i=1}^r\Res_{K_i/F}H_i$, with $H_i$ split over a tame extension of $K_i$. Then the local Shimura pair $(G,\{\mu\})$ over $F$ arises as a product of local Shimura pairs
	$(\Res_{K_i/F}H_i, \{\mu_i\})$, $1\leq i \leq r$. Suppose we are given faithful minuscule representations $\rho_i: H_i\to \GL(V_i)$ over $K_i$, such that
	the compositions
	\[
	\Res_{K_i/F}H_i \xrightarrow{\ \Res_{K_i/F}(\rho_i)\ } \Res_{K_i/F}\GL(V_i)\hook \GL(V_{i}^{(F)})
	\]
	give Hodge embeddings for $(\Res_{K_i/F}H_i, \{\mu_i\})$ over $F$, for each $i$.

	We consider 
	\begin{equation}\label{productHodge}
	\rho: G\simeq \prod_{i=1}^r\Res_{K_i/F}H_i\xrightarrow{\ \prod_i \Res_{K_i/F}(\rho_i)\ } \prod_{i=1}^r\Res_{K_i/F}\GL(V_i)\hook \GL(V)
	\end{equation}
	where $V=\oplus_{i=1}^rV_i^{(F)}$ is considered as an $F$-vector space with $F$-structure given by restriction from the $K_i$-structure on each summand.
	Then $\rho$  also gives a 
	Hodge embedding $\rho: (G,\{\mu\})\to (\GL(V),\{\mu_d\})$. In particular, $ (G,\{\mu\})$ is of Hodge type. Note that then, for any $m\geq 1$,  the direct sum representation
	\[
	\rho^{\oplus m}: G\to  \GL(V)\times\cdots \times \GL(V)\hook \GL(V^{\oplus m})
	\]
	also gives a Hodge embedding that factors as in (\ref{productHodge}).
	
	\begin{thm}\label{thm: LM embedding}
		Let $(G, \{\mu\}, \Gg)$ be a local model triple over $F$. Assume $G$ quasi-tame, $G\simeq \prod_{i=1}^r\Res_{K_i/F}H_i$, with $H_i$ split over a tame extension of $K_i$.
	Assume that $p$ is odd and that all the $H_i$ are of classical type.   
	Suppose  $ (G,\{\mu\})$ admits a Hodge embedding $\rho$ of the form (\ref{productHodge}) as above. After replacing the Hodge embedding $\rho$ by a direct sum $\rho^{\oplus m}$ as above, there exists a lattice $\La\subset V$ and a quasi-parahoric  group scheme $\Gg'$ of $G$ with $(\Gg')^\circ=\Gg^\circ$ such that 
		$\rho$ extends to a closed immersion $\Gg'\hook \GL(\La)$ and  there is a closed immersion
		\[
		\rho_*:  \rM_{\Gg', \mu}=\rM_{\Gg,\mu}\to {\rm Gr}(d,\La)_{\O_E}
		\]
		extending the natural map on the generic fiber.
	\end{thm}
 \begin{Remark}
	{\rm If the target $\pi_1(G)_I$ of the Kottwitz homomorphism is a  torsion-free group, then we always have $\Gg'=\Gg=\Gg^\circ$, see \cite{HR}.
		In the course of the proof we will see that if $\Gg=\Gg_\bx$ for $\bx$ generic in its facet, then we can take in the above $\Gg'=\Gg$,
		provided that $G$ does not involve anisotropic factors coming from division algebras
 of degree divisible by $p$. This last condition comes from Proposition \ref{Fixed}.}
\end{Remark}

\begin{proof}			
	We can  reduce 
to the case $G=\Res_{K/F}H$, with $H$ split over a tamely ramified extension of $K$;  the general case is obtained by taking products. We may assume $\Gg=\Gg_\bx=\Res_{\O_K/\O_F}\Hh_{\bx}$ and $\bx\in \B(G, F)=\B(H, K)$ which we can assume is generic in its facet.
	We  have the Hodge embedding $\rho: G\to \GL(V^{(F)})$ given as a composition
	\[
	G=\Res_{K/F}H\xrightarrow{\Res_{K/F}\rho_1} \Res_{K/F}\GL(V)\to \GL(V^{(F)}),
	\]
	starting from $\rho_1: H\to \GL(V)$, cf. (\ref{productHodge}). 	For notational simplicity, in what follows, we will often drop the superscript ${(F)}, {(\calO_F)}$, from the notation for the restriction of structure, as it should be clear from context over which ring the modules are being taken.
	
	We first assume that $H^\ad$ does not involve division algebras of degree divisible by $p$.
		Then the assumption of Proposition \ref{Fixed} for $H$ is satisfied. Hence, we can find a finite tame Galois extension $\ti K/K$ that splits $H$ and a point $\bx'\in \B(H, K)$   with $\Hh=\Hh_{\bx }=\Hh_{\bx'}$ which is hyperspecial in $\B(H, \ti K)$. 
	Now
	we can   apply the construction of \S \ref{FixedLevin}. In this, we consider the composition of the natural map 
	\[
	G=\Res_{K/F}H\to \Res_{K/F}(\Res_{\ti K/K}(H\otimes_K\ti K))=\Res_{\ti K/F}(H\otimes_K\ti K)
	\]
	with
	\[
	\Res_{\ti K/F}(H\otimes_K\ti K)\xrightarrow{\Res_{\ti K/F}(\rho_1\otimes_K\ti K)} \Res_{\ti K/F}\GL(V\otimes_K{\ti K})\xrightarrow{\ } \GL(V\otimes_K\ti K),
	\]
	as a representation over $F$ which is isomorphic to a direct sum of $[\ti K:K]$-copies of
	$\rho$. 
	This extends to a morphism of $\calO_F$-group schemes
	 \[
	\Gg\hook \Res_{\O_{\ti K}/\O_F} (H_0\otimes_{\Z_p} \O_{\ti K})\hook \Res_{\O_{\ti K}/\O_F} \GL(\La_0\otimes_{\Z_p}\O_{\ti K})\hook \GL(\La_0\otimes_{\Z_p}\O_{\ti K})
	.\]
	Here we fix an isomorphism $\ti \calH\cong H_0\otimes_{\bbZ_p} \calO_{\ti K}$, and identify $\rho_1\otimes_K\ti K$ with the base to $\ti K$ of a represention $H_0\hook\GL(\Lambda_0)$ over $\bbZ_p$.  This  morphism is a closed immersion  by Proposition \ref{prop: R-smoothness properties} and \cite[Cor. 1.3]{PYuGS}.
	
Correspondingly, by composing (\ref{imm542}), (\ref{imm544}) and the morphism (\ref{imm55prop}) of Proposition \ref{prop:closedLMforGL} applied to $\ti K/F$ and the lattice $\La_0\otimes_{\Z_p}\O_{\ti K}$, we obtain equivariant maps
	\begin{equation}\label{imm5}
	{\rM}_{\Gg,\mu}\to  \rM_{\Res_{\O_{\ti K}/\O_F}(H_0\otimes_{\Z_p} \O_{\ti K}), \ti\mu}\to (\rM_{\GL(\La_0\otimes_{\Z_p}\O_{\ti K}), \ti\mu'})=\Gr(d,\Lambda).
	\end{equation}
	with $\Lambda=\Lambda_0\otimes_{\bbZ_p}\calO_{\ti K}$ as $\calO_F$-modules. These extend the natural morphisms on the generic fibers and are all closed immersions.  The result follows in this case.

 We now deal with the general case (i.e. when $H^{\ad}$ could involve division algebras of index divisible by $p$).
We may assume $K/F$ is totally ramified; the general case is easily reduced to this. Let $F^\natural/F$ be a finite unramified extension with ring of integers $\calO_{F^\natural}$ such that $H$ is quasi-split after base changing to $K^\natural=KF^\natural$.  We let $\calG^\natural_{\bx}$ denote the stabilizer scheme of $G^\natural:=G\otimes_{F}F^\natural$ for the image of  $\bx$ in $\calB(G,F^\natural)$. Then we have an identification $\calG^{\natural}_{\bx}\cong \calG_{\bx}\otimes_{\calO_F}\calO_{F^\natural}$. By construction, we also have an isomorphism 
$$
\rM_{\calG,\mu}\otimes_{\calO_F}\calO_{F^\natural}\cong\rM_{\calG^\natural,\mu^\natural}
$$
 where
$	\rM_{\calG^\natural,\mu^\natural}$ is the local model associated to the local model triple 	
\[
(G^\natural, \{\mu^\natural\}, \Gg^\natural):=(G\otimes_FF^\natural,\{\mu\otimes_{F}{F^\natural}\}, \Gg\otimes_{\O_F}\O_{F^\natural}),
\] 
over $F^\natural$.  

Let $\Omega\subset\calB(G^\natural,F^\natural)$ be the facet containing $\bf x$ and ${\bf y}\in \Omega$ a point which is generic in $\Omega$. Then $\Omega$ is stable under $\Gamma^\natural=\Gal(F^\natural/F)$  and $\calG^\natural_{\bf y}$ has the same neutral component as $\calG^\natural$. Since $G^\natural$ is quasi-split, its adjoint group does not involve division algebras of degree divisible by $p$, and so the above argument applied to the base changed embedding $\rho^\natural$  gives (upon replacing $\rho^\natural$ by a direct sum)  closed immersions 
$$
\calG^{\natural}_{\bf y}\hook \GL(\Lambda^\natural),\ \ \ \ \rM_{\calG^{\natural}_{\bf y},\mu^\natural}=\rM_{\Gg , \mu }\otimes_{\O_F}\O_{F^\natural}\hook {\rm Gr}( d, \La^\natural)_{\O^\natural_E}
$$ 
for $\Lambda^{\natural}\subset V\otimes_{\calO_F}\calO^\natural$ an $\calO^\natural$-lattice. By \'etale descent, the natural morphisms 
\[
\rM_{\Gg , \mu } \to \Res_{\O_{F^\natural}/\O_F}(\rM_{\Gg , \mu }\otimes_{\O_F}\O_{F^\natural}),
\quad
\Res_{\O_{F^\natural}/\O_F}{\rm Gr}( d, \La^\natural)_{\O^\natural_E}\to {\rm Gr}(d f, \Lambda)_{\O^\natural_E},
\] 
are closed immersions where $f=[F^\natural:F]$ and $\Lambda$ is $\Lambda^\natural$ considered as an $\calO_F$-module. 
We thus obtain a closed immersion $\rM_{\Gg , \mu }\hook {\rm Gr}(d f, \Lambda^\natural)_{\O^\natural_E}$.

Now note that $\calG^\natural_{\bf y}$ is equal to  the stabilizer $\widehat{\calG}_\Omega$ of $\Omega$, hence $\calG^\natural_{\bf y}$ is $\Gamma^\natural$-invariant, hence arises as the base change to $\O_{F^\natural}$ of a quasi parahoric $\calG'$ of $G$ with $\calG'^\circ=\calG^\circ$. Thus the composition 
$$
\calG'\rightarrow \Res_{\O_{F^\natural}/\calO_F}\calG_{\O_{F^\natural}}'=\Res_{\O_{F^\natural}/\calO_F}\calG^\natural_{\bf y}\rightarrow \Res_{\O_{F^\natural}/\calO_F}\GL(\Lambda^\natural)\rightarrow \GL(\Lambda)
$$ is a closed immersion as desired.
	\end{proof}
\end{para}

\subsection{Proof of Theorem \ref{thm:twoLM}}

We can now complete the proof. We make use of the following lemma.
		\begin{lemma}\label{lemma:goodLM}
	Let $(G, \{\mu\}, \Gg)$ be a local model triple over $\O_F$. Suppose 
	$\rho: (\Gg,\mu)\hook(\GL(\La),\mu_d)$ is an integral Hodge embedding.
	Let $\ov X_\mu$ be the (reduced) Zariski closure of  
	$X_\mu=G/P_\mu\hook {\rm Gr}(d,V)_{E}$ in ${\rm Gr}(d,\La)_{\O_E}$.
	If $\ov X_\mu$  is normal and has reduced special fiber, then 
	$\ov X_\mu$ is the unique scheme over $\O_E$ that satisfies the Scholze-Weinstein conjecture \cite[Conj. 21.4.1]{Schber} for 
	$(G, \{\mu\}, \Gg)$, i.e. we have $\ov X_\mu=\BMloc_{\Gg, \mu}$. In fact,  then the
	closed immersion 
	\[
	\rho_*: \ov X_\mu=\BMloc_{\Gg,\mu}\to 
	{\rm Gr}(d,\La)_{\O_E}
	\]
	is the unique morphism of normal schemes which gives, after applying the diamond functor,
	the  morphism
	$ {\rm M}^v_{\Gg, \mu}\to {\rm M}^v_{\GL(\La),\mu_d}$
	of $v$-sheaves  over ${\rm Spd}(\O_E)$ obtained from 
	$\rho: (\Gg,\mu)\hook(\GL(\La),\mu_d)$  by functoriality, cf. \S \ref{vsheafmorphisms}.
\end{lemma}

\begin{proof} As above, ${\rm M}^v_{\GL(\La),\mu_d}={\rm Gr}(d,\La)^\Diamond$, and so $\BMloc_{\GL(\La),\mu_d}={\rm Gr}(d,\La)$.
	The $v$-sheaf $(\bar X_\mu)^\Diamond$ over ${\rm Spd}(\O_E)$  given by the Zariski closure $\ov X_\mu$ of $X_\mu$ in ${\rm Gr}(d,\La)_{\O_E}$ agrees with the $v$-sheaf closure
	$(X_\mu^\Diamond)^-$ of $X_\mu^\Diamond$ in 
	\[
	({\rm Gr}(d,\La)_{\O_E})^\Diamond={\rm M}^v_{\GL(\La), \mu}\times_{{\rm Spd}(\O_F)}{\rm Spd}(\O_E).
	\]
	But $(X_\mu^\Diamond)^-$ is  also the $v$-sheaf closure of $X_\mu^\Diamond$  in the $v$-sheaf Beilinson-Drinfeld Grassmannian ${\rm Gr}_{\GL(\La), {\rm Spd}(\O_E)}$. By definition, this last closure  is ${\rm M}^v_{\Gg,\mu}$.  The result follows, cf. 
	\cite[Thm 2.15]{HPR}.
\end{proof}
 \end{para}
 
\begin{para}\label{proofoftwoLM} 
	\begin{proof}[Proof of Theorem \ref{thm:twoLM}.]

		Since $\Mloc_{\Gg, \mu}$ is flat and projective with reduced special fiber, it suffices to show that 
		$(\Mloc_{\Gg,\mu})^\Diamond$ can be identified with $\rM^v_{\Gg,\mu}:=\Gr_{\calG,\mathrm{Spd}(\calO_E),\mu}$. We use the notation of \S\ref{sss:generalLM}, so that $G^{\ad}=G^{\ad}_1\times G^{\ad}_2.$

		By \cite[Prop. 21.4.3]{Schber}, \cite[Prop. 21.5.1]{Schber}, there are natural isomorphisms 
		$$
		\Gr_{\calG,\mathrm{Spd}(\calO_E),\mu}  \cong  \Gr_{\calG^{\ad}, \mathrm{Spd}(\calO_E), \mu^{\ad}},\qquad 
		\Gr_{\calG', \mathrm{Spd}(\calO_E), \mu'}\cong \Gr_{\calG^{\ad}_1, \mathrm{Spd}(\calO_E), \mu_1^{\ad}},
		$$
		induced by the surjective morphisms $G\rightarrow G^{\ad}$ and $G'\rightarrow G_1^{\ad}$. Since $\calG^{\ad}=\calG_1^{\ad}\times \calG_2^{\ad}$, we have an isomorphism $$\Gr_{\calG^{\ad},\mathrm{Spd}(\calO_E),\mu^{\ad}}\cong\Gr_{\calG_1^{\ad},\mathrm{Spd}(\calO_E),\mu_1^{\ad}}\times_{\mathrm{Spd}\calO_E}\Gr_{\calG^{\ad}_2,\mathrm{Spd}(\calO_E),\mu_2^{\mathrm{ad}}},$$
		where for $i=1,2$,  $\{\mu_i^{\ad}\}$ is the factor of $\{\mu^{\ad}\}$ in $G_i$. By assumption, $\mu_2^{\ad}$ is trivial and hence $\Gr_{\calG^{\ad}_2,\mathrm{Spd}(\O_E),\mu_2^{\mathrm{ad}}}\cong \mathrm{Spd}(\O_E)$. It follows that
		$\Gr_{\calG^{\ad},\mathrm{Spd}(\calO_E),\mu^{\ad}}\cong\Gr_{\calG_1^{\ad},\mathrm{Spd}(\calO_E),\mu_1^{\ad}}$ and hence we obtain an isomorphism 
		$$
		\Gr_{\calG,\mathrm{Spd}(\calO_E),\mu}\cong  \Gr_{\calG',\mathrm{Spd}(\calO_E),\mu'}.
		$$
		Since the local model $\bbM^{\mathrm{loc}}_{\calG,\mu}$ is defined using the auxiliary group $G'$ from Proposition \ref{prop:PRlsv}, it suffices to prove the result in the case $(G,\{\mu\},\calG)=(G',\{\mu'\},\calG')$. By Theorem \ref{thm: LM embedding}, upon possibly replacing $\calG'$ with a different quasi-parahoric, we may find an integral Hodge embedding $(\calG',\mu')\hook(\GL(\Lambda),\mu_d)$ such that the natural map $X_{\mu'}\rightarrow \Gr(d,V)_E$ extends to a closed immersion $$\Mloc_{\calG' ,\mu'}\rightarrow \Gr(d,\Lambda)_{\calO_E}.$$It follows that we have an isomorphism $\Mloc_{\calG',\mu'}\cong \overline{X}_{\mu'}$, and hence $\overline{X}_{\mu'}$ is normal and has reduced special fiber by Theorem \ref{thm: Levin}. Thus the result follows by  Lemma \ref{lemma:goodLM}.
	\end{proof}

\end{para}
 \begin{para}\label{par:variantGood}
We introduce some definitions that are needed for later applications. 

\begin{Definition}\label{def:Good} Let $(G,\mu,\calG)$ be a local model triple and $\rho: (\Gg,\mu)\hook(\GL(\La),\mu_d)$ an integral Hodge embedding.
We say that $\rho$  is \emph{good}, if the morphism
	\[
	\rho_*:  \BMloc_{\Gg,\mu}\to {\rm Gr}(d,\La)_{\O_E}=\BMloc_{\GL(\La),\mu_d}\otimes_{\O}\O_E
	\]
	is a closed immersion.
\end{Definition}

Often, we need to consider a variant of the above definition: Let  $\calL=\{\Lambda_i\}_{i\in \Z}$ be a periodic lattice chain in $V$, see \S \ref{lattices}.
Let $\GL(\calL)$ be the parahoric group scheme of $\GL(V)$ which corresponds to the stabilizer of  $\calL$.
Suppose that $\rho: (G, \mu)\hook (\GL(V),\mu_d)$ extends to a closed immersion of group schemes $\Gg\hook \GL(\calL)$. Then 
we say that the integral Hodge embedding $\rho: (\Gg,\mu)\hook(\GL(\calL),\mu_d)$ 
is good, if the natural morphism
\[
\rho_*:  \BMloc_{\Gg,\mu}\to  \BMloc_{\GL(\calL),\mu_d}\otimes_{\O}\O_E
\]
is also a closed immersion. 

Assume $\calL=\{\Lambda_i\}_{i\in \Z}$ has a determining segment 
\[
p\La_0=\La_r\subset \La_{r-1}\subset \cdots \subset\La_0.
\]
As in \S \ref{lattices}, we set ${\rm tot}(\calL)=\La_0\oplus\La_1\oplus \cdots \oplus\Lambda_{r-1}\subset V^{\oplus r}$, a lattice well-determined up to homothety.
The natural morphisms 
\[
\GL(\calL)\to \GL({\rm tot}(\calL)),\quad \BMloc_{\GL(\calL),\mu_d}\to \BMloc_{\GL({\rm tot}(\calL)),\mu_{rd}},
\]
are both closed immersions, resp. by Lemma \ref{BTGL} and the standard construction of parahoric local models for the general linear group.  Hence, $\rho: (\Gg,\mu)\hook(\GL(\calL),\mu_d)$ is a good integral Hodge embedding, if and only if 
$\rho^{\oplus r}: (\Gg,\mu)\hook(\GL({\rm tot}(\calL)),\mu_{rd})$ is a good integral Hodge embedding.

\end{para}
\begin{para}
Now let $(G,\{\mu\},\Gg)$ be a local model triple with $G\cong \Res_{K/F}H$ with $H$ split over a tamely ramified extension. We assume  that $p\nmid|\pi_1(G^{\der})|$, $\calG=\calG_{\bx}$ for some $\bx\in \calB(G,F)$ generic in its facet and that $H^{\ad}$ does not have factors involving division algebras with index divisible by $p$. The proof of Theorem \ref{thm:twoLM} shows that if there is a faithful minuscule representation $\rho_1: H\to \GL(V)$ over $K$, such that
the composition
\[
\Res_{K/F}H \xrightarrow{\ \Res_{K/F}(\rho_1)\ } \Res_{K/F}\GL(V)\hook \GL(V)
\]
give Hodge embeddings, then $(G,\{\mu\},\Gg)$ admits  good Hodge embeddings. These are given by the composition \[
\Gg\hook \Res_{\O_{\ti K}/\O_F} (H_0\otimes_{\Z_p} \O_{\ti K})\hook \Res_{\O_{\ti K}/\O_F} \GL(\La_0\otimes_{\Z_p}\O_{\ti K})\hook \GL(\La_0\otimes_{\Z_p}\O_{\ti K})
.\] where $\ti K/K$ is a tame extension over which $\bx$ becomes hyperspecial and $\La_0\otimes_{\Z_p}\O_{\ti K}
\subset V\otimes_K\ti K$ is considered as an $\calO_F$-lattice. The next proposition shows that we can replace $\Lambda_0\otimes_{\bbZ_p}\calO_{\ti K}$ with a $\Gamma$-stable lattice. This will be a key property that is needed in \S\ref{sec: NE cases}.

By Proposition 	 \ref{prop:findlattice} there exists a $\Gamma$-invariant lattice $\ti\La\subset V\otimes_K\ti K$ such that $\rho \otimes_K \ti K$ extends to a closed immersion \[
\rho_{1,\ti \La}: H_0\otimes_{\Z_p}\O_{\ti K}\to \GL(\ti\La)
\]
We thus obtain a closed immersion 
\[
\rho_{\ti \La}:\Gg\hook \Res_{\O_{\ti K}/\O_F} (H_0\otimes_{\Z_p} \O_{\ti K})\hook \Res_{\O_{\ti K}/\O_F} \GL(\ti\La)\hook \GL(\ti \La)\] where in the last term we consider $\La$ as an $\calO_F$-module. We let $\ti \mu'$ denote the image of the conjugacy class of cocharacters $\mu$.
\begin{prop}\label{prop: better lattice} Under the assumptions above, 
$$
\rho_{\ti \Lambda}:(\calG,\mu)\hook (\GL(\ti\La),\ti \mu')
$$
is a good integral Hodge embedding. 

Moreover we have  an equality:
\[
\Gg=\Res_{\calO_{\ti K}/\calO_F}(H_0\otimes_{\bbZ_p}\calO_{\ti K})\cap \{g\in \GL(\ti\La)\ |\ g\cdot t_a=t_a\cdot g, \forall a\}
\]
where $t_a:\ti\La\to\ti\La$ are the following $\O_F$-linear endomorphisms: $t_\gamma:\ti\La\to \ti\La$ given by the action of $\ga\in \Gamma$ on $\ti\La\subset V\otimes_K\ti K$, and $t_x:\ti\La\to\ti\La$
given by the multiplication by a set of generators $x\in  \O_{\ti K}$ of the $\O_F$-algebra $ \O_{\ti K}$. 
	\end{prop}
\begin{proof}
We apply Lemma \ref{2lattices} to $L=\ti K$ and the lattices $\ti\La$, $\La_0\otimes_{\Z_p}\O_{\ti K}$: It follows that there is $g\in \GL(V\otimes_K\ti K)$   centralizing the image of $H_0\otimes_{\Z_p} \ti K$, such that $g\cdot (\La_0\otimes_{\Z_p}\O_{\ti K})=\ti\La$. Conjugation by $g$ gives an isomorphism 
\[
{\rm ad}_g: \GL(\La_0\otimes_{\Z_p}\O_{\ti K})\xrightarrow{\sim} \GL(\ti\La)
\]
such that
\[
\rho_{1, \ti\La}= {\rm ad}_g\circ \rho_1.
\]
Using this, combined with the fact that $\rho_{\La_0\otimes_{\Z_p}\O_{\ti K}, *}$ is a closed immersion shows that 
\[
\rho_{\ti\La, *}: \BMloc_{\Gg,\mu}\to (\BMloc_{\GL(\ti\La), \ti\mu'})\otimes_{\O_F}\O_{E}
\]
is also a closed immersion. Then $\rho_{\ti\La}: (\Gg,\mu)\hook (\GL(\ti\La), \ti\mu')$  is also a good integral Hodge embedding.

For the ``moreover'' part, note that we have an equality
$$
\calG\cong (\Res_{\calO_{\ti K}/\calO_F}(H_0\otimes_{\bbZ_p}\calO_{\ti K}))^\Gamma=\Res_{\calO_{\ti K}/\calO_F}(H_0\otimes_{\bbZ_p}\calO_{\ti K})\cap (\Res_{\calO_{\ti K}/\calO_F}\GL(\ti \La))^\Gamma
$$
where the last term is a scheme-theoretic intersection. The result then follows since
${\Res}_{\O_{\ti K}/\O_F}\GL(\ti\La)^\Gamma\subset \GL(\ti\La)$ is the scheme-theoretic stabilizer of the $t_a$.
\end{proof}

\begin{Remark}\label{rem:diagram}{\rm Let $\L$ be the lattice chain in $V$ given by $\{(\ti\pi^i\ti\La)^\Ga\}_{i\in\Z}$. Then
there is a commutative diagram with arrows  the natural morphisms between local models
\begin{equation}\label{eq:diagram}
\xymatrix{ \BMloc_{\Res_{\O_{\ti K}/\O_F}(H_0\otimes\O_{\ti K}), \ti\mu}\ar[r]& \BMloc_{\GL(\ti\La), \ti\mu'}\otimes_{\O_F}\O_{E}\\
	\BMloc_{\Gg,\mu}\ar[u]\ar[r] &\BMloc_{\GL(\L), \mu'}\otimes_{\O_F}\O_{E}\ar[u]\ar[r] &\BMloc_{\GL({\rm tot}(\L)), \mu'}\otimes_{\O_F}\O_{E}.}
\end{equation}
In this, the composition of the left vertical with the top horizontal morphism  is $\rho_{\ti\La, *}$ which, by the above, is a closed immersion. The morphism
$\BMloc_{\GL(\L), \mu'}\to \BMloc_{\GL({\rm tot}(\L)), \mu'}$ is easily seen to be a closed immersion. It follows that all the arrows in the diagram
are closed immersions.}
\end{Remark}

\end{para}

	\section{Root curves and spanning tangent spaces}\label{s:RootCurves}
	
	In this section, we study the tangent spaces of certain Schubert varieties inside the affine Grassmannian. We show that in most cases which are related to Shimura varieties, the tangent space can be spanned by the images of tangent spaces to smooth curves.
	
\subsection{Tangent spaces of affine Schubert varieties}
\begin{para}\label{par:1.1.1b}
	Let $\bk$ be an algebraically closed field of characteristic $p$ and  $G$ a (split, connected) reductive group over $\bk$.  Recall the affine Grassmannian 
	$\Gr_G=LG/L^+G$  defined as in \S\ref{par:1.1.1}.

	We fix $T$ a maximal torus of $G$ and $B$ a Borel subgroup containing $T$, and we write $X_*(T)^+$ for the set of dominant cocharacters with respect to $B$. For any $\mu\in X_*(T)$, we let $t^\mu$ denote the $\bk$-point of $LG$ determined by the $\bk\llps t\lrps$-point of $G$ induced by $\mu$. For simplicity, we also let $t^\mu$ denote the image of $t^\mu$ in $\Gr_G$.
	
	For $\mu\in X_*(T)^+$, we let $S_\mu\subset \Gr_G$ denote the affine Schubert variety corresponding to $\mu$. By definition, this is the reduced orbit closure of the $G(\bk\lps t\rps)$-orbit of $t^\mu$.
	We let $\lleq$ denote the dominance ordering on $X_*(T)^+$ so that $\lambda\lleq\mu $ if and only if $\mu-\lambda$ is an integral linear combination of positive coroots with non-negative coefficients. Then we have $S_\lambda\subset S_\mu$ if and only if $\lambda\lleq\mu$. We sometimes write $S_\mu^G$ for $S_\mu$ if we want to make clear the group $G$ that appears.
	
	We will mainly be interested in the cases when the pair $(G,\mu)$ is related to the special fiber of a local model for a Shimura variety of abelian type.

	\begin{Definition}
		Let $(G,\mu)$ be a pair as above. We say $(G,\mu)$ is of \textit{mod $p$ abelian type} if each 
		simple factor $(H_i,\mu_i)$ of $(G^{\ad},\mu^{\ad})$ satisfies one of the following two conditions:
		\begin{altenumerate}
		\item
			$H_i$ is of type $A,B,C$ and $\mu_i$ is a sum of minuscule coweights, 
			
			\item  $H_i$ is of type $D_n$ and $\mu_i=r\varpi_1^\vee$ (type $D^{\mathbb R}_n$) or $\mu_i=s\varpi_{n-1}^\vee+t\varpi_n^\vee$ (type $D^{\mathbb H}_n$), with $r,s,t\in \Z_{\geq 0}$.
		\end{altenumerate}
		
		Here, $\varpi^\vee_j$ denotes the $j^{\mathrm{th}}$-fundamental  coweight, and we use the labeling of roots as in \cite{Bour}. 
\end{Definition}

\begin{Remark} 
{\rm	Let $(G,\{\mu\},\calG)$ be a local model triple over $F$ of abelian type satisfying the standard assumptions as in 
	\S\ref{s:LM} and
		with $\Gg\simeq \prod_{i=1}^r\Res_{\O_{K_i}/\O_F}H_i$, where $H_i$ is a  split reductive group scheme over $\O_{K_i}$. Then there is a pair $(G',\mu')$ over $k$ of mod $p$ abelian type such that $\BMloc_{\calG,\mu}\otimes_{\calO_E}k\cong S_{\mu'}^{G'}$, see Lemma \ref{lemma:LMtomodp}. } 
\end{Remark}

	\begin{Definition}\label{def:span}
	For a scheme $X$ over $\bk$ and $x\in X(\bk)$, we say that the tangent space $T_x(X)$  of $X$ at $x$ is 
	\emph{spanned by smooth formal curves} if the images of the tangent spaces
	by $\bk$-morphisms $\Spec(\bk\lps t\rps)\to X$ with the closed point mapping to $x$ generate the $\bk$-vector space $T_x(X)$.
	\end{Definition}	
\begin{Remark}\label{remSpanHilbert}

{\rm \begin{altenumerate}
\item    Suppose $X$ is of finite type over $\bk$. A necessary condition for $T_x(X)$  to be 
	spanned by smooth formal curves is that $T_x(X)$ is spanned as a $\bk$-vector space by the $\bk$-points of the reduced subscheme $TC_x(X)_{\rm red}$ of the (affine) tangent cone 
	$TC_x(X)$ of $X$ at $x$. 
	
\item Consider the normal surface 
	$
	X=\Spec \bk[x,y,x]/(x^2 +xyz).
	$
	When $\bk=\bar{\mathbb F}_p$, $X$ gives an open affine chart of the local model for the reduction modulo $p$ of a Hilbert modular surface with Iwahori level at an odd prime $p$ which ramifies in the real quadratic field, see \cite[Ex. 4.5]{Phmv}.
	The tangent space $T_0X$ at the origin is 3-dimensional. If $f(t),g(t),h(t)\in t\bk\lps t\rps$ is such that $f(t)^2+f(t)g(t)h(t)=0,$ then the coefficient of $t$ in $f(t)$ is equal to $0$. Thus for $(f(t),g(t), h(t))$ a $\bk\lps t\rps$-point of $X$ lifting $0$, the image of the tangent space of this formal curve lies in the 2-dimensional subspace of $T_0X$ given by $ x=0$.
	Here, the  tangent cone $TC_0(X)$ is $\Spec(\bk[x,y,z]/(x^2))$ and its reduced subscheme $\Spec(\bk[y,z])$ only spans this $2$-dimensional subspace of $T_0(X)$. \end{altenumerate}
	}
	\end{Remark}
	
The main theorem of this section is the following.

	\begin{thm}\label{thm:curve span}
	Assume $(G, \mu)$ is of  \textit{mod $p$ abelian type} with $p\nmid|\pi_1(G^{\der})|$ and has no factors of type $D^{\mathbb H}_n$. Then the tangent space of the affine Schubert variety $S_\mu$ at each $\bk$-valued point is spanned by smooth formal curves. 
	\end{thm}
	
	This will be shown as a consequence of the combination of Theorem \ref{thm: root direction equality} and Theorem \ref{thm: spanning cartan directions}.
	These statements  provide more precise results and include information about cases with factors of type $D^{\mathbb H}_n$.
	 
	\end{para}
	
\begin{para}
We begin by recalling the description of the tangent space of $\Gr_{G}$ at the points $t^\lambda$. Let $\fkg$ denote the Lie algebra of $G$ and let $\fkt$ denote the Lie algebra of $T$.  We write $R$ for the set of roots for $G$ and $R_+$ (resp. $R_-$) the set of positive (resp. negative) roots for $G$, and let $\Delta\subset R_+$  be the set of simple roots. We fix a Chevalley system $(x_\alpha)_{\alpha\in R}$ for $G$, which determines  a set of root vectors $X_\alpha\in \fkg$, for $\alpha\in R$. 
	Then $X_\alpha$ generate  the weight space of $\fkg$ corresponding to $\alpha$. 
	
	Let $L^-G$ denote the negative loop group for $G$. Thus $L^-G$ represents the functor $R\mapsto G(R[t^{-1}])$ on $k$-algebras $R$, and let $L^{--}G$ denote $\ker(L^-G\rightarrow G) , t\mapsto 0$. 	For $\lambda\in X_*(T)^+$, the map 
	$$
	t^\lambda L^{--}Gt^{-\lambda}\to \Gr_G,\quad g\mapsto g t^{\lambda},
	$$
	is representable by an open immersion which maps $1$ to $t^\lambda\in \Gr_G$ (cf. \cite[Lemma 3.1]{HRTestFn}). We thus have an isomorphism
\begin{align*}
 T_{t^\lambda}\Gr_G& \cong t^\lambda\Lie L^{--}G t^{-\lambda}\\ & \cong \fkg\otimes_\bk \bk(\!(t)\!)/t^\lambda(\fkg\otimes_\bk \bk[\![t]\!])t^{\lambda^{-1}}\\&\cong\bigoplus_{\alpha\in R}t^{\langle \lambda,\alpha\rangle-1}\bk[t^{-1}]X_\alpha\oplus t^{-1}\bk[t^{-1}]\fkt.
 \end{align*}

For $\mu \in X_*(T)^+$ with $\lambda\lleq \mu$, we have the subspace $T_{t^\lambda}S_\mu\subset T_{t^\lambda}\Gr_G$. Then  $t^\lambda\in S_\mu$ and hence $T_{t^\lambda}S_\mu$ is   preserved under the action of the torus $\tilde T=\Gm\times T$, with the first factor $\Gm$ acting on $\Gr_G=LG/L^+G$ by `rotations' $t\mapsto at$. Hence, $T_{t^\lambda}S_\mu$ has a basis given by elements of the form $t^{-r}X_\alpha$ together with elements of the form $t^{-r}H$ for $H\in \fkt$.  We let $\Phi^{\tan}_{\lambda,\mu}\subset R\times \bbZ$ denote the subset of pairs $(\alpha, r)$ such that $t^rX_\alpha\in T_{t^\lambda}S_\mu$, and we set
		$$
	\fkT^{\tan}:=T_{t^\lambda}S_\mu\cap t^{-1}\bk[t^{-1}]\fkt.
	$$
	Then we have a decomposition
 $$
	T_{t^\lambda}S_\mu\cong \bigoplus_{(\alpha,k)\in \Phi^{\tan}_{\lambda,\mu}}t^{k}X_\alpha\oplus\fkT^{\tan} .
	$$
	
We now fix $\mu,\lambda\in X_*(T)^+$ with $\lambda\lleq \mu$. We will show that in most cases when $(G,\mu)$ is of mod $p$ abelian type, the tangent space $T_{t^\lambda}S_\mu$ is spanned by smooth formal curves. We deal separately with the subspace spanned by $t^{k}X_\alpha$,
$(\alpha,k)\in \Phi^{\tan}_{\lambda,\mu}$ (the ``root directions") and $\fkT^{\tan}$ (the ``Cartan directions") in the next two sections.

\end{para}

\subsection{Root curves and root tangent directions}

\begin{para}
We first consider the tangent directions along the root vectors $X_\alpha$. In this case, we can span many tangent directions using curves coming from the unipotent root groups as follows.
	
	For $\alpha \in R$, we define  
	$$
k_\alpha^{(\lambda,\mu)}:=\max\{k\in \bbZ\ |\ (\lambda-k\alpha^\vee)_{\dom} \lleq \mu\}.
	$$ 
Here, for $\nu\in X_*(T)$, we denote by $\nu_{\dom}\in X_*(T)^+$ its dominant representative. We will often fix coweights $\mu,\lambda$ as above and write $k_\alpha$ for $k^{(\lambda,\mu)}_\alpha$ when there is no risk of confusion.

 The following is essentially contained in \cite[Proposition 3.6]{PappasZhou}.

\begin{prop}\label{prop: curve bound tangent vector} Let $\lambda,\mu\in X_*(T)^+$ with $\lambda\lleq \mu$.
	\begin{enumerate}\item We have $k_\alpha=k_{-\alpha}+\langle \lambda,\alpha\rangle$.
		\item	Let $1\leq k\leq k_\alpha$. Then the tangent vector $X_\alpha t^{-k+\langle\lambda,\alpha\rangle}\in T_{t^\lambda}\Gr_G$ lies in the subspace $T_{t^\lambda}S_\mu$.
		\end{enumerate}
\end{prop}
\begin{proof}
Part (1) is \cite[Prop. 3.6]{PappasZhou}. 	

For (2),	we consider the  map
$$
f_{\alpha,k}:\bbA^1\rightarrow\Gr_G
$$  given by $a\mapsto t^\lambda x_\alpha(t^{-k}a)$ whose image lies in $S_\mu$ by \cite[Prop. 3.6]{PappasZhou}. Moreover, by loc. cit., we have  $f(0)=t^{\lambda}$ and the image of the tangent space of $\bbA^1$ at $0$ contains the vector $X_\alpha t^{-k+\langle\lambda,
	\alpha\rangle}$. 
\end{proof}

	 Proposition \ref{prop: curve bound tangent vector} shows that the tangent vector $t^{\langle\lambda,\alpha\rangle-k}X_\alpha$ lies in the subspace of $T_{t^\lambda}S_\mu$ spanned by smooth formal curves. We set  
	 $$
	 \Phi_{\lambda,\mu}^{\cur}=\{(\alpha,k)\, |\, \alpha\in R,\, \langle\lambda,\alpha\rangle-k_\alpha\leq k\leq \langle\lambda,\, \alpha\rangle-1\}\subset R\times \bbZ.
	 $$
	
	Then we have  inclusions $\Phi_{\lambda,\mu}^{\cur}\subset \Phi_{\lambda,\mu}^{\tan}$. The first main result is the following.
	
	\begin{thm}\label{thm: root direction equality}
		Let $(G,\mu)$ be  of mod $p$ abelian type with $p\nmid |\pi_1(G^{\der})|$. Then we have $$\Phi_{\lambda,\mu}^{\cur}= \Phi_{\lambda,\mu}^{\tan}.$$
	\end{thm}
	\end{para}
	\begin{para}\label{para: reduction almost simple}We first explain how to reduce to proving this Theorem in the case when $G^\ad$ is  simple and $G^{\der}=G^{\mathrm{sc}}$. 
	
	Fix $G,\mu,\lambda$ as above. Let $G^{\ad}\cong \prod_{i=1}^rH_i$ be the decomposition of $G^{\ad}$ into simple factors, and $\mu_i$ (resp. $\lambda_i$) the component of $\mu^{\ad}$ (resp. $\lambda^{\ad}$) in $H_i$. Choose a $z$-extension $$1\rightarrow Z\rightarrow\tilde{H}_i\rightarrow H_i\rightarrow 1$$ so that $\tilde{H}_i^{\der}=\tilde{H}_i^{\mathrm{sc}}$ and $\mu_i,\lambda_i$ lift to cocharacters $\tilde{\mu}_i,\tilde{\lambda}_i$ of $\tilde{H}_i$ (see \cite[Prop. 3.1]{MilneShih}).  The maximal torus and Borel $T,B$ of $G$ determine corresponding pairs $T_i,B_i$ in each $\tilde{H}_i$, and we have $\tilde{\mu}_i,\tilde{\lambda}_i\in X_*(T_i)^+$.
	
	We let $S_{\tilde{\mu}_i}$ denote the affine Schubert variety in $\Gr_{\tilde{H}_i}$ corresponding to $\tilde{\mu}_i$. Then as in \cite[Prop. 2.2.7]{KP}, there is an isomorphism 
	$$
	\prod_{i=1}^rS_{\tilde{\mu}_i}\xrightarrow{\sim} S_{\mu}.
	$$
	This induces natural decompositions 
	$$
	\Phi^{\cur}_{\lambda,\mu}=\prod_{i=1}^r\Phi^{\tilde{H}_i,\cur}_{\tilde{\lambda}_i,\tilde{\mu}_i},\qquad
	\Phi^{\tan}_{\lambda,\mu}=\prod_{i=1}^r\Phi^{\tilde{H}_i,\tan}_{\tilde{\lambda}_i,\tilde{\mu}_i}.
	$$
	Thus in order to prove Theorem \ref{thm: root direction equality}, we may and do assume until further notice that $G^{\der}=G^{\mathrm{sc}}$, and that $G^{\ad}$ is simple.
\end{para}

\begin{para}
	The set $\Phi_{\lambda,\mu}^{\tan}$ has a description in terms of Demazure modules for the associated affine Kac--Moody  algebra (cf. \cite[Cor. 4.3, Lem. 5.9]{HLR}). However, it seems difficult to compare this description of $\Phi_{\lambda,\mu}^{\tan}$ with $\Phi_{\lambda,\mu}^{\cur}$. Instead, we will consider a set $$\Phi^{\FM}_{\lambda,\mu}\subset R\times \bbZ$$ which contains $\Phi_{\lambda,\mu}^{\tan}$, but which is more amenable to computation, and can therefore be compared more easily with $\Phi_{\lambda,\mu}^{\cur}$. 
 The definition of the set $\Phi^{\FM}_{\lambda,\mu}$ is inspired by a conjectural modular description of Schubert varieties which is due to Finkelberg-Mirkovic when $\bk$ has characteristic 0 \cite[\S10.3]{FM}. For general fields, such a description is considered in the forthcoming work of Haines--Jin. (If this conjectural description holds, then  $\Phi^{\FM}_{\lambda,\mu}=\Phi^{\tan}_{\lambda,\mu}$. Here, this equality will be shown directly.)

Let $\Rep_\bk G$ denote the category of  finite dimensional representations of $G$ over $\bk$.	For $V\in \Rep_\bk G$, we write $V^*$ for the contragredient representation. For $\nu\in X_*(T)$, we also write $\nu$ for the representation of $B$ obtained by composing $\nu$ with the projection $B\rightarrow T$. We let $W$ denote the Weyl group for $G$ and $w_0\in W$ the longest element of $W$.

For $\nu\in X_*(T)^+$, we let 
$$
V(\nu):=\Ind_B^G(-w_0(\nu))^*
$$
	 denote the Weyl module associated to $\nu$ (cf. \cite[II, Chapter 2]{Jantzen}). We set $d_\nu:=\dim V(\nu)$.

	 Recall, that $\Gr_G$ represents the functor on $\bk$-algebras $R$ classifying isomorphism classes of pairs $(\calE,\phi)$ as in \ref{par:1.1.1}. If $\calE$ is a $G$-torsor, we denote by $\calE(\nu)$ the vector bundle 
	 $$\calE(\nu)=\calE\times^{G}V(\nu)$$ obtained by pushing out the structure group by the representation $\rho(\nu): G\to \GL(V(\nu))$.

	\begin{Definition}\label{def: FM Schubert variety}
 We define the subfunctor $S_\mu^{\FM}$ of $\Gr_G$ as follows.
For a $\bk$-algebra $R$, an $R$-point of $S_\mu^{\FM}$ consists of a pair $(\calE,\phi)\in \Gr_G(R)$ such that for every dominant weight $\nu\in X_*(T)^+$, 	we have
\begin{equation}	\label{eqn: FM bound defn}
	t^{\langle\mu, -w_0\nu\rangle}\phi({\nu})(\calE(\nu))\subset \calE^0(\nu)   \subset t^{\langle\mu, -\nu\rangle}\phi({\nu})(\calE(\nu)),
	\end{equation}
	as subsheaves of $\calE(\nu)[1/t]$.

		\end{Definition}
	
It is easy to see that $S_\mu^{\FM}$ is  represented by a closed ind-subscheme of $\Gr_G$ which is in fact a projective scheme over $\bk$. We can also see that $t^\mu\in S_\mu^{\FM}(\bk)$ and that $S_\mu^{\FM}$ is $G(\bk\lps t\rps)$-invariant. 
Hence, $S_\mu$ is a closed subscheme of $S_\mu^{\FM}$ and we have 
	$$
	S_\mu\hook (S_\mu^{\FM})_{\mathrm{red}}.
	$$
	\begin{Remark}\label{rem: FM conjecture}
		{\rm \begin{altenumerate}
		\item In fact, the above closed immersion induces an identification 
		\[
		S_\mu= (S_\mu^{\FM})_{\mathrm{red}}.
		\] This is shown, when $\bk$ has characteristic $0$, by \cite[Prop. 6.4]{HainesPlucker}, and, for a general perfect field, in forthcoming work of Haines-Jin. However, we will not need this  in what follows.
				\item 
When $\bk$ has characteristic $0$, it is conjectured that $S_\mu^{\FM}$ is reduced and so $S^\FM_\mu=S_\mu$. This is proved when $G$ is of type $A$; see \cite{KMWY}.		
\item In what follows we will only need to use the inclusion $
	t^{\langle\mu, -w_0\nu\rangle}\phi({\nu})(\calE^0(\nu))\subset \calE(\nu)  $ in \eqref{eqn: FM bound defn}. In fact, the right inclusion even follows from this by applying it to the dual representation $V(-w_0\nu)$.

	\end{altenumerate}}
	\end{Remark}
	
	It follows that taking tangent spaces at $t^\lambda$, gives inclusions:
	$$
	T_{t^\lambda}S_\mu\subset T_{t^\lambda}S_\mu^{\FM}\subset T_{t^\lambda}\Gr_G.
	$$
	The subspace $T_{t^\lambda}S_\mu^{\FM}$ is preserved by the action of the  torus $\tilde T$, hence, like $T_{t^\lambda}{S_\mu}$, it admits a basis consisting of elements of the form $t^{-r}X_\alpha$ together with elements of the form $t^{-r}H$ for $H\in \fkt$. We define 
	$$
	\Phi^{\FM}_{\lambda,\mu}:=\{(\alpha,r)\in R\times\bbZ\,|\, t^rX_\alpha\in T_{t^\lambda}S^{\FM}_\mu \}\subset R\times \bbZ,
	$$
	  and we set
	  $$
	  \fkT^{\FM}_{\lambda,\mu}:=T_{t^\lambda}S^{\FM}_\mu\cap t^{-1}\bk[t^{-1}]\fkt.
	$$
\end{para}

\begin{para}
	We can obtain a more explicit description of $\Phi^{\FM}_{\lambda,\mu}$ as follows. 
	For $\alpha\in R$, we let $\calW(\alpha)$ denote the set of pairs $(\varpi,\varpi')$ where $\varpi\in X_*(T)^+$ is a dominant cocharacter and $\varpi'$ is a weight of $V(\varpi)$ such that $X_{\alpha} v_{\varpi'}\neq 0$ for some weight vector $v_{\varpi'}\in V(\varpi)$ of weight $\varpi'$. Equivalently, $(\varpi,\varpi')\in \calW(\alpha)$ if and only if $\varpi'$ and  $\varpi'+\alpha$ are weights of $V(\varpi)$. In particular, we have $(\varpi,\varpi')\in \calW(\alpha)$ if and only if $(\varpi,\varpi'+\alpha)\in \calW(-\alpha)$.
	 
	 We set 
	$$
l^{(\lambda,\mu)}_\alpha=\min_{(\varpi,\varpi')\in \calW(\alpha)}\langle \mu,\varpi \rangle - \langle\lambda, \varpi'\rangle.
	$$
	As with $k_\alpha$, we will often drop the $(\lambda,\mu)$ from the notation and just write $l_\alpha$ when there is no risk of confusion.
	\begin{prop}\label{prop: FM bound relation} Let $\lambda,\mu\in X_*(T)^+$ with $\lambda\lleq\mu$.
		\begin{altenumerate}
			\item  We have $l_\alpha=l_{-\alpha}+\langle\lambda,\alpha\rangle$.
			
			\item Let $(\alpha,\langle\lambda,\alpha\rangle-l)\in \Phi_{\lambda,\mu}^{\FM}$. Then $1\leq l\leq l_\alpha$.
		\end{altenumerate}
	\end{prop}
	\begin{proof}
	(1)	Let $\alpha\in R$.  Then $(\varpi,\varpi')\in \calW(\alpha)$  if and only if $(\varpi,\varpi'+\alpha)\in \calW(-\alpha)$, and we have  $$\langle \mu,\varpi \rangle - \langle\lambda, \varpi'\rangle-\langle\lambda,\alpha\rangle=\langle \mu,\varpi \rangle - \langle\lambda, \varpi'+\alpha\rangle .$$ It follows $l_\alpha-\langle\lambda,\alpha\rangle=l_{-\alpha}$.
		
	(2) 	Consider the element $t^{\langle\lambda,\alpha\rangle-l}X_\alpha\in T_{t^\lambda}S_\mu^{\FM}$; this corresponds to a $\Spec\bk[\epsilon]/\epsilon^2$-valued point of $S_\mu^{\FM}$. Let $\nu\in X_*(T)^+$ be a dominant weight and $v\in V(\nu)$ a weight vector of weight $\nu'$. Then consider  
		$$
		(1+\epsilon t^{\langle\lambda,\alpha\rangle-l}X_\alpha )t^\lambda v= t^{\langle\lambda,\nu'\rangle}v+\epsilon t^{\langle\lambda,\alpha+\nu'\rangle-l}X_\alpha v\in V(\nu)\otimes_\bk\bk\llps t\lrps\otimes_{\bk}\bk[\epsilon]/\epsilon^2.
		$$
		By the definition of $\Phi_{\lambda,\mu}^{\FM}$, specifically the left inclusion in  \eqref{eqn: FM bound defn}, this element has worst pole $-\langle \mu,-w_0\nu\rangle$. Thus if $X_\alpha v\neq 0$,  we have  $\langle \lambda,\alpha+\nu'\rangle -l \geq-\langle \mu,-w_0\nu\rangle$. 
		
		We set $\varpi=-w_0\nu\in X_*(T)^+$ a dominant weight and $\varpi'=-\nu'$.  If $\nu'+\alpha$ is a weight of $V(\nu)$, $\varpi'-\alpha$ is a weight of $V(\varpi)$.  It follows that $(\varpi,\varpi')\in \calW(-\alpha),$ or equivalently $(\varpi,\varpi'-\alpha)\in \calW(\alpha)$ and $$ l\leq \langle \mu,\varpi\rangle -\langle\lambda,\varpi'-\alpha\rangle.$$
		
		If we let $\nu$ and $v\in V(\nu)$ range over all such pairs with $X_\alpha v\neq 0$, then $(\varpi,\varpi'-\alpha)$ range over all elements of $\calW(\alpha)$. It follows that $1\leq l \leq l_\alpha$.
	\end{proof}

\end{para}

\begin{para}
	
	Note that we have  the following inclusions 
	$$
	\Phi^{\cur}_{\lambda,\mu}\subset \Phi^{\tan}_{\lambda,\mu}\subset\Phi^{\FM}_{\lambda,\mu}.
	$$
	 It follows that we have an inequality 
	\begin{equation}\label{eqn: k,l inequality}
	 k_\alpha \leq l_\alpha,\ \ \forall \alpha \in R,
	\end{equation}
	with  equality if and only if $\Phi^{\cur}_{\lambda,\mu}=\Phi^{\tan}_{\lambda,\mu}=\Phi^{\FM}_{\lambda,\mu}.$ 
The following proposition gives a criterion for when \eqref{eqn: k,l inequality} is an equality; to state it we introduce some notation.
	
 Let $P$ (resp. $P^\vee$) denote the weight (resp. coweight) lattice for $G$ and $P^+$ (resp. $P^{\vee,+})$ the set of dominant weights (resp. coweights). Thus $P$ is the $\bbZ$-dual of the coroot lattice and $P^\vee$ is the $\bbZ$-dual of the root lattice, and there are natural maps $X^*(T)\rightarrow P$ and $X_*(T)\rightarrow P^\vee$. Since $G^{\der}=G^{\mathrm{sc}}$, the map $X^*(T)\rightarrow P$ is surjective.
 
 Let $\Delta=\{\alpha_1,\dotsc,\alpha_n\}$ be the set of simple roots and $\Omega=\{\varpi_1,\dotsc,\varpi_n\}\subset P^+$ the corresponding set of fundamental weights. For each $\varpi_i$, we fix a lift to $X^*(T)^+$ also denoted $\varpi_i$, which we use to identify $\Omega$ with a subset of $X^*(T)^+$.  Recall a weight $\varpi$ is said to be minuscule if $|\langle \alpha^\vee,\varpi\rangle|\leq 1$ for all $\alpha\in R$. We let $\Omega^{\min}\subset \Omega$ denote the subset of minuscule fundamental weights. 
	\begin{prop}\label{prop: criterion equality k, l bounds}
 Let $\mu\in X_*(T)^+$ and $S\subset \Omega^{\min}$ a subset which satisfies the following property:		
 
 {\rm($\ast$)}	For all $\nu\in X_*(T)^+$ such that $\mu-\nu$ lies in the coroot lattice, we have $\nu \lleq\mu$ if and only if $$\langle\mu-\nu,\varpi\rangle\geq 0, \text{ for all $\varpi\in S$}.$$

		Then, for every $\lambda \in X_*(T)^+$ with $\lambda \lleq \mu$ and every $\alpha\in R$, we have $$k^{(\lambda,\mu)}_{\alpha}=\min_{(\varpi,\varpi')\in \calW(\alpha), \varpi\in S} \langle\mu,\varpi\rangle-\langle\lambda,\varpi'\rangle=l_\alpha^{(\lambda,\mu)}.$$
\end{prop}
	\begin{Remark}\label{rem: dominance order}
	{\rm 		\begin{altenumerate}
			\item For $\alpha\in R$ and $\nu\in X_*(T)$, the pairing $\langle \nu,\alpha\rangle$ only depends on the image of $\nu $ in $P^\vee$. Using this fact, one can check that the statement of Proposition \ref{prop: criterion equality k, l bounds} is independent of the choice of lifting of $\Omega$ to $X^*(T)^+$. For example, let $(\varpi,\varpi')\in \calW(\alpha)$.
			If $\omega$ and $\varpi$ have the same image in $P$, then $\omega=\varpi+\gamma$ where $\gamma\in X^*(G^{\ab})$,  and we have $(\omega,\omega')\in \calW(\alpha)$ where $\omega'=\varpi'+\gamma$. Then \begin{align*}\langle \mu,\omega\rangle -\langle\lambda,\omega'\rangle&=\langle\mu-\lambda,\omega\rangle-\langle\lambda,\omega'-\omega\rangle\\
			&=\langle\mu-\lambda,\varpi\rangle-\langle\lambda,\varpi'-\varpi\rangle\\
			&=\langle \mu,\varpi\rangle -\langle\lambda,\varpi'\rangle.
			\end{align*}
			\item Note that for $\mu,\nu\in X_*(T)^+$ with $\mu-\nu\in X_*(T_{\mathrm{sc}})$, we have 
	$$
	\mu-\nu=\sum_{i=1}^n m_i  \alpha^\vee_i,\  m_i\in \bbZ.
	$$
	 Then $m_i=\langle\mu-\nu,\varpi_i\rangle$, and hence $\nu\lleq\mu$ if and only if $\langle\mu-\nu,\varpi_i\rangle\geq 0$ for all $i=1,\dotsc,n.$ The point of ($\ast$) is that the condition  $\langle\mu-\nu,\varpi\rangle$ for $\varpi\in S$ forces this condition  for all $i$.  Note that the choice of $S$ satisfying $(\ast)$ depends on $\mu$, and that not all $\mu$ affords such an $S$.  For example, in the notation of the proof of Theorem \ref{thm: root direction equality} below, there does not exist an $S$ for the cocharacter $\mu=(2,1,1)$ in the case of type $B_n$.
\end{altenumerate}}
	 \end{Remark}
	\begin{proof}
For $\alpha\in R$, we 	write 
	$$
	j_\alpha:=\min_{(\varpi,\varpi')\in \calW(\alpha), \varpi\in S} \langle\mu,\varpi\rangle-\langle\lambda,\varpi'\rangle.
	$$
	 Then by definition, we have $j_\alpha\geq l_\alpha$. It suffices to prove $k_{\alpha}\geq j_\alpha$, since then $$k_\alpha\geq j_{\alpha}\geq l_\alpha,$$ and hence since $k_\alpha\leq l_\alpha$, we have equality throughout.
		
	By  Proposition \ref{prop: curve bound tangent vector} (1) and (the proof of) Proposition \ref{prop: FM bound relation} (1),  it suffices to prove $k_\alpha\geq j_\alpha$ for $\alpha\in R_-$
		 or equivalently, that 
		 $$
		 (\lambda-j_{\alpha}\alpha^\vee)_{\dom}\lleq \mu
		 $$ 
		 for all $\alpha \in R_-$. 
	We therefore fix $\alpha \in R_-$. Then by ($\ast$), we need to check that 
		$$
		\langle \mu-(\lambda-j_{\alpha}\alpha^\vee)_{\dom},\varpi\rangle\geq 0
		$$
		 for all $\varpi \in S$. Let $w\in W$ be such that $w(\lambda-j_\alpha\alpha^\vee)=(\lambda-j_{\alpha}\alpha^\vee)_{\dom}.$
		Then for any $\varpi\in S$, we have
		\begin{align*}\langle \mu-(\lambda-j_{\alpha}\alpha^\vee)_{\dom},\varpi\rangle&= \langle \mu-w(\lambda-j_{\alpha}\alpha^\vee),\varpi\rangle\\
		&	=\langle \mu,\varpi\rangle-\langle \lambda,w^{-1}(\varpi)\rangle +j_\alpha\langle w(\alpha^\vee),\varpi\rangle\\
		&\geq \begin{cases}
		\langle \mu,\varpi\rangle-\langle \lambda,w^{-1}(\varpi)\rangle & \text{ if $\langle \alpha^\vee,w^{-1}(\varpi)\rangle\geq 0$}\\ \langle \mu,\varpi\rangle-\langle \lambda,w^{-1}(\varpi)\rangle- j_\alpha& \text{ if $\langle \alpha^\vee,w^{-1}(\varpi)\rangle <0$,}
		\end{cases}
		\end{align*}
	where the last inequality follows from the fact that $\varpi$ is minuscule.
		
		Note that since $\lambda,\mu \in X_*(T)^+$ and $\mu \ggeq \lambda$, we have $$\langle \mu,\varpi\rangle \geq \langle \lambda,\varpi\rangle\geq \langle\lambda, w^{-1}(\varpi)\rangle$$
		and hence we are done if $\langle \alpha^\vee,w^{-1}(\varpi)\rangle \geq 0.$
		
		If $\langle \alpha^\vee,w^{-1}(\varpi)\rangle<0$, then we have $(\varpi,w^{-1}(\varpi))\in \calW(\alpha)$, and hence by definition of $j_\alpha$, we have $$\langle \mu,\varpi\rangle-\langle \lambda,w^{-1}(\varpi)\rangle- j_\alpha\geq 0$$
		as  desired.
	\end{proof}
	
\end{para}
\begin{para} We  now use the previous  proposition to prove Theorem \ref{thm: root direction equality}. 

	\begin{proof}[Proof of Theorem \ref{thm: root direction equality}] We prove that if $(G,\mu)$ is of mod $p$ abelian type  with $G^{\ad}$ simple and $G^{\der}=G^{\mathrm{sc}}$, then for all $\alpha\in R$, we have $k_\alpha= l_\alpha$.
		
By Proposition \ref{prop: criterion equality k, l bounds}, it suffices to find $S\subset \Omega^{\min}$ satisfying condition ($\ast$) in the statement of Proposition \ref{prop: criterion equality k, l bounds}. Note that condition ($\ast$) only depends on the image of $\mu$ and $\nu$ in $P^{\vee,+}$;  we will also use $\mu$ and $ \nu$ to denote their respective images in $P^{\vee,+}$. We verify ($\ast$)  case-by-case depending on the type of $(G^{\ad},\mu^{\ad})$ using the standard representations of $P$ and $P^\vee$.

		In what follows we let $e_1,\dotsc,e_n$ be the standard basis of $\bbZ^n$ and we equip $\bbZ^n$ with the bilinear pairing $\bbZ^n\times \bbZ^n\rightarrow \bbZ$ given by $\langle e_i,e_j\rangle=\delta_{ij}$.

		\textit{Type $A_n$.}
		Let $\mu\in P^{\vee,+}$  be any dominant coweight. Then we may take $S=\Omega^{\min}=\Omega$. Then ($\ast$) is clearly satisfied (cf. Remark \ref{rem: dominance order}).
		\smallskip
		
		\textit{Type $B_n$.} We identify $P^\vee$ and $P$ with $\bbZ^n$  equipped with the usual pairing, so that  we have 
		$$
		R=\{a_{\pm i,\pm j}=\pm e_i\pm e_j|1\leq i < j \leq n\}\cup\{a_{\pm i}:=\pm e_i|1\leq i \leq n\},
		$$
		$$
		P^{\vee,+}=\{(\lambda_1,\lambda_2,\dotsc,\lambda_n)\in P^\vee|\lambda_1\geq\lambda_2\geq\dotsc\geq \lambda_n\geq 0\}.
		$$
		The simple roots $\Delta=\{\alpha_1,\dotsc,\alpha_n\}$ are given by $\alpha_i = e_i - e_{i+1}$ for $i = 1, \ldots, n-1$, and $\alpha_n = e_n$,
		and we have 
		$$
		\varpi_i=\sum_{j=1}^i e_i, \text{ for } i=1,\dotsc,n-1,\quad  \varpi_{n}=\left(\frac{1}{2},\dotsc,\frac{1}{2}\right).
		$$

		In this case, the only minuscule coweight is $\varpi_1^\vee=(1,0,\dotsc,0)$ so if $(G,\mu)$ is of mod $p$ abelian type, we have $\mu = (r, 0, \ldots, 0)$. We take
		$$
		S=\left\{\varpi_n=\left(\frac{1}{2},\dotsc,\frac{1}{2}\right)\right\}=\Omega^{\min}.
		$$

		Let $\nu=(\nu_1,\dotsc,\nu_n)\in P^{\vee,+}$ with $\mu-\nu$ in the coroot lattice and suppose $\langle\mu-\nu,\varpi_n\rangle\geq 0$; thus $$r-\sum_{i=1}^n\nu_i\geq 0.$$
		Since $\nu\in P^{\vee,+}$, we have $\nu_i\geq 0$ for  all $i$, and hence 
		$$
		\langle\mu-\nu,\varpi_i\rangle=r-\sum_{j=1}^i\nu_j\geq 0, \text{ for all $i$.}
		$$ 
		Thus $\mu\ggeq\nu$ and ($\ast$) is satisfied. 
		\smallskip
		
		\textit{Type $C_n$.} We identify $P^\vee$ and $P$ with submodules of $\frac{1}{2}\bbZ^n$, so that $P^\vee$ is the submodule generated by $\bbZ^n$ and $(\frac{1}{2},\dotsc,\frac{1}{2})$. Then we have $$R=\{a_{\pm i,\pm j}:=\pm e_i \pm e_j, 1\leq i< j\leq n\} \cup \{a_{\pm i}:=\pm 2e_{i}, 1 \leq i \leq n\}$$
		$$ P^{\vee,+}=\{(\lambda_1,\dotsc,\lambda_n)\in P^\vee|\lambda_1\geq\lambda_2\geq\dotsc\geq\lambda_n\geq 0\}.$$
		The simple roots are given by $\alpha_i=e_i-e_{i+1}, i=1,\dotsc,n-1$ and $\alpha_n=2e_n$, and we have  $\varpi_{i}=\sum_{j=1}^ie_j$.
		
		The only minuscule coweight  is $\varpi_n^\vee=(\frac{1}{2},\dotsc,\frac{1}{2})$. Thus if $(G,\mu)$ is of mod $p$ abelian type, we have $\mu=(\frac{r}{2},\frac{r}{2},\dotsc,\frac{r}{2})$ for $r$ a positive integer. We take $$S=\{\varpi_1=(1,0,\dotsc,0)\}=\Omega^{\min}.$$Let $\nu=(\nu_1,\dotsc,\nu_n)\in  P^{\vee,+}$ with $\mu-\nu$ in the coroot lattice and suppose $\langle\mu-\nu,\varpi_1\rangle\geq 0$. Then $\frac{r}{2}-\nu_1\geq 0$, and hence $\frac{r}{2}-\nu_i\geq 0$ since $\nu\in  P^{\vee,+}$. Thus $\nu\lleq\mu$ and ($\ast$) is satisfied. 
		\smallskip
		
		\textit{Type $D_n^{\R}$ and $D_n^{\mathbb H}$.} We identify $P^\vee$ and $P$ with submodules of $\frac{1}{2}\mathbb{Z}^n$ so that $P^\vee$ is generated by $\bbZ^n$ and $(\frac{1}{2},\dotsc,\frac{1}{2})$. We have 
		$$R=\{a_{\pm i, \pm j} := \pm e_i \pm e_j| 1 \leq i < j \leq n \}$$
		$$ P^{\vee,+}=\{(\lambda_1,\dotsc,\lambda_n)\in P^\vee|\lambda_1\geq\dotsc\geq\lambda_{n-1}\geq|\lambda_n|\}$$
		The simple roots are given by $\alpha_i=e_i-e_{i+1}$ for $i=1,\dotsc,n-1$ and $\alpha_n=e_{n-1}+e_{n-1}$. We have 
		$$
		\varpi_i=\sum_{j=1}^i e_i, \ i=1,\dotsc,n-2,\quad 
		\varpi_{n-1}=\left(\frac{1}{2},\dotsc,\frac{1}{2},-\frac{1}{2}\right),\ \varpi_n=\left(\frac{1}{2},\dotsc,\frac{1}{2},\frac{1}{2}\right).
		$$
		For type $D_n^{\R}$, we have $\mu=r\varpi_1^\vee=(r,0,\dotsc,0)$. We take $$S=\{\varpi_{n-1},\varpi_n\}\subset \Omega^{\min}=\{\varpi_1,\varpi_{n-1},\varpi_n\}.$$
		Let $\nu=(\nu_1,\dotsc,\nu_n)\in  P^{\vee,+}$ with $\mu-\nu$ in the coroot lattice and suppose $\langle\mu-\nu,\varpi\rangle \geq 0$ for $\varpi\in S$. Then we have $$r-\sum_{j=1}^{n-1}\nu_i=\langle\mu-\nu,\varpi_{n-1}+\varpi_n\rangle\geq0.$$
		Hence since $\nu_j\geq0$ for $j=1,\dotsc,n-1$, we have 
		$$
		r-\sum_{j=1}^i\nu_j=\langle\mu-\nu,\varpi_i\rangle\geq0,\ \text{for all}\ i=1,
		\dotsc , n-2.
		$$
		It follows that $\nu\lleq\mu$ and property ($\ast$) is satisfied.
		
		For type $D_n^{\mathbb H}$, we have $\mu=s\varpi_{n-1}^\vee+t\varpi_n^\vee$, $s,t\in \bbZ_{\geq0}$.  We write $ s-t=q, s+t=r$; then we have $\mu=(\frac{r}{2},\dotsc,\frac{r}{2},\frac{q}{2})$. We take $$S=\{\varpi_1,\varpi_{n-1},\varpi_n\}=\Omega^{\min}.$$  Let $\nu=(\nu_1,\dotsc,\nu_n)\in  P^{\vee,+}$ with $\mu-\nu$ in the coroot lattice and suppose $\langle\mu-\nu,\varpi\rangle \geq 0$ for $\varpi\in S$. Then we have $\langle\mu-\nu,\varpi_1\rangle=\frac{r}{2}-\nu_1\geq0$.  Since $\nu_1\geq \nu_j$, $j=1\dotsc,n-2$, we have 
		$$
		\langle\mu-\nu,\varpi_i\rangle=\sum_{j=1}^i\frac{r}{2}-\nu_j\geq0,\ \text{for all}\ i=1,\dotsc,n-2,
		$$
		 and hence ($\ast$) is satisfied.
		\end{proof}
	
	\begin{Remark}\label{rem: bound realized by varpi in S}{\rm 
 By Proposition \ref{prop: criterion equality k, l bounds}, we have that in each case $$k_{\alpha}=\min_{(\varpi,\varpi')\in \calW(\alpha), \varpi\in S} \langle\mu,\varpi\rangle-\langle\lambda,\varpi'\rangle=l_\alpha.$$
}
	\end{Remark}
\end{para}

\begin{para}
In what follows, we will need a more explicit description of $k_{-\alpha}$ for $\alpha$ a simple root. For $\alpha,\alpha'\in \Delta$, a geodesic from $\alpha$ to $\alpha'$ is a sequence of simple roots $\alpha=\alpha_0,\alpha_1,\dotsc,\alpha_r=\alpha'$ such that $\alpha_i,\alpha_{i+1}$ are adjacent in the Dynkin diagram, and all the $\alpha_i$ are distinct. Since $G^{\ad}$ is simple, its Dynkin diagram is connected, so geodesics always exist and it is clear that they are unique. 
	
	Let $\alpha\in \Delta$ and $\varpi\in \Omega$ a fundamental weight corresponding to $\alpha'\in \Delta$.  Let $\alpha=\alpha_0,\alpha_1,\ldots ,\alpha_r=\alpha'$ be a geodesic. We set $$\gamma:=\begin{cases}0 & \text{if $r=0$}\\ 
	s_{\alpha_1} s_{\alpha_2}\cdots s_{\alpha_{r-1}}\alpha'&\text{ if $r>0$} \end{cases}$$ so that  $\varpi_\alpha:=s_{\alpha_1} s_{\alpha_2}\dotsc s_{\alpha_{r-1}}s_{\alpha_r}\varpi=\varpi-\gamma$. Then $\langle\alpha^\vee,\varpi_\alpha\rangle>0$; this is clear if $r=0$, and for $r>0$ we have $\gamma=\sum_{i=1}^{r}m_i\alpha_i$ with $m_i>0$, so that  $\langle \alpha^\vee,\gamma\rangle <0$.
It follows that  $(\varpi,\varpi_\alpha)\in \calW(-\alpha).$
	
	\begin{lemma}\label{lem: varpi_alpha} Assume $\varpi\in \Omega^{\min}$. Then we have 
		$$
		\langle\mu,\varpi\rangle-\langle\lambda,\varpi_\alpha\rangle=\min_{\{\varpi'|(\varpi,\varpi')\in\calW(-\alpha) \}}\langle\mu,\varpi\rangle-\langle \lambda,\varpi'\rangle
		$$
	\end{lemma}
	\begin{proof}Suppose $\varpi$ corresponds to $\alpha'\in \Delta$. If $\alpha=\alpha'$, i.e. $\langle\alpha^\vee,\varpi\rangle =1$, then we have  $\varpi_\alpha=\varpi$ and the result is clear since $\langle \lambda,\varpi\rangle\geq \langle\lambda,\varpi'\rangle$ for any $\varpi'$ a weight of $V(\varpi)$.
			
	Now assume $\alpha\neq \alpha'$.	Note that $\langle \alpha_i^\vee,s_{\alpha_{i+1}}\cdots s_{\alpha_r}\varpi\rangle \geq1$ for any $i$, and hence since $\varpi$ is minuscule we have equality. It follows that $\gamma=\sum_{i=1}^r\alpha_i$. Now suppose $(\varpi,\varpi')\in \calW(-\alpha)$  so that $\langle\alpha^\vee, \varpi'\rangle=1$. We write $\varpi-\varpi'=\sum_{\beta\in \Delta}c_\beta \beta$, where $c_\beta\geq0$ since $\varpi$ is dominant. Then it is clear that  the subset $$\mathrm{supp}(\varpi-\varpi'):=\{\beta\in \Delta|c_\beta>0\}$$ is connected and contains $\alpha'$. Indeed,  let $w\in W$ be a minimal length element with $w(\varpi)=\varpi'$ and let  $w=s_{\alpha_{i_1}}\cdots s_{\alpha_{i_n}}$ be a reduced word decomposition. Then $\langle\alpha_{i_j}^\vee,s_{\alpha_{i_{j+1}}}\cdots s_{\alpha_{i_n}}\varpi\rangle>0$ for all $j$, and hence $\alpha_{i_n}=\alpha'$ and $\alpha_{i_j}$ is adjacent to an element of $\mathrm{supp}(\varpi-s_{\alpha_{i_{j+1}}}\cdots s_{\alpha_{i_n}}\varpi)$.  Thus $\mathrm{supp}(\varpi-\varpi')$ is connected and contains $\alpha'$ by induction.
		
		Since $\langle\alpha^\vee,\varpi'\rangle>0$, it follows that $\mathrm{supp}(\varpi-\varpi')$ contains a neighbour of $\alpha^\vee$. Since $\mathrm{supp}(\varpi-\varpi')$ contains $\alpha'$ and is connected, we have $\alpha_1,\dotsc,\alpha_r\in \mathrm{supp}(\varpi-\varpi')$. Thus $\varpi_\alpha-\varpi'$ is a linear combination of positive roots with non-negative coefficients. It follows that $$\langle\mu,\varpi\rangle-\langle\lambda,\varpi'\rangle\geq \langle\mu,\varpi\rangle-\langle\lambda,\varpi_\alpha\rangle$$ since $\lambda$ is dominant.
	\end{proof}

\begin{cor}\label{cor: k_-alpha simple root}
Let $(G,\mu)$ be of mod $p$ abelian type with $G^{\ad}$  simple and $G^{\der}=G^{\mathrm{sc}}$, and let $\lambda\in X_*(T)^+$ with $\lambda\lleq\mu$. Let $S\subset \Omega^{\min}$ be the subset as in the proof Theorem \ref{thm: root direction equality}, then for $\alpha\in \Delta$, we have 
$$
k_{\alpha}=\min_{\varpi\in S} \langle\mu,\varpi\rangle-\langle\lambda,\varpi_\alpha\rangle=l_\alpha.
$$
\end{cor}
\begin{proof}
	This follows from Lemma \ref{lem: varpi_alpha}, cf. Remark \ref{rem: bound realized by varpi in S}.
\end{proof}
\end{para}

\subsection{Cartan tangent directions}

\begin{para}
We now consider the directions along the Cartan.
	We fix $\mu,\lambda \in X_*(T)^+$ with $\lambda\lleq\mu$ as before. For an element $\alpha\in \Delta$, we write $\mathrm{d}\alpha^\vee:\Lie \bbG_m\rightarrow \fkt$ for the map on Lie algebras induced by $\alpha^\vee$. We set $H_\alpha=\mathrm{d}\alpha^\vee(1)$. Then $H_\alpha, X_\alpha,X_{-\alpha}$ form an $\fks\fkl_2$-triple in $\fkg$.
	
	Let $1\leq k \leq k_\alpha$ and consider $t^{\langle\lambda,\alpha\rangle-k}X_\alpha\in T_{t^\lambda}\Gr_\mu$. Note that $T_{t^\lambda}\Gr_\mu$ is equipped with a natural action of $G(\bk\lps t\rps)\cap t^\lambda G(\bk\lps t\rps)t^{-\lambda}$. Set $u_{-\alpha}=x_{-\alpha}(1)\in G(\bk\lps t\rps)\cap t^\lambda G(\bk\lps t\rps)t^{-\lambda}$. Then we have 
	$$
	u_{-\alpha} t^{\langle\lambda,\alpha\rangle-k}X_\alpha u^{-1}_{-\alpha}=t^{\langle\lambda,\alpha,\rangle-k}(X_\alpha+H_\alpha+X_{-\alpha})\in T_{t^{\lambda}}S_\mu.
	$$
	In particular, we have
	 $t^{\langle\lambda,\alpha\rangle-k}H_\alpha\in T_{t^\lambda}S_\mu$ via the torus action. Moreover, conjugating the curve $a\mapsto t^\lambda U_\alpha(t^{-k}a)$ by $u_{-\alpha}$ gives a smooth formal curve whose tangent space generates the subspace spanned by $t^{\langle\lambda,\alpha\rangle-k}(X_\alpha+H_\alpha+X_{-\alpha})$.
	
	We set $\fkT^{\cur}_{\lambda,\mu}\subset\fkT^{\tan}_{\lambda,\mu}$ to be the subspace spanned by $t^{\langle\lambda,\alpha\rangle-k}H_\alpha$ for $\alpha \in \Delta$ and $1\leq k\leq k_\alpha$.
	
	\begin{thm}\label{thm: spanning cartan directions}
		Let $(G,\mu)$ be of mod $p$ abelian type with $p\nmid|\pi_1(G^{\der})|$.
		
		\begin{altenumerate}\item Assume $(G,\mu)$ has no factors  of type $D^{\mathbb H}$. Then for any $\lambda\in X_*(T)^+$ with $\lambda \lleq \mu$, we have
		 $$
			\fkT^{\cur}_{\lambda,\mu}= \fkT^{\tan}_{\lambda,\mu}.
			$$
			\item If $G^{\ad}$ is simple and  $(G,\mu)$ is  of type $D^{\mathbb H}$, and $\lambda$ satisfies $\langle \lambda,\alpha_{n-1}\rangle =0$ or $\langle\lambda,\alpha_n\rangle=0$; here we use the labelling of the roots as in Theorem \ref{thm: root direction equality}. Then 
			$$\fkT^{\cur}_{\lambda,\mu}=\fkT^{\tan}_{\lambda,\mu}.$$  In particular, this holds when $\lambda$ is the minimal  element in  $\{\nu\in X_*(T)_+|\nu \lleq\mu\}$.
		\end{altenumerate}
	\end{thm}
	
As in \ref{para: reduction almost simple}, we can reduce to proving this in the case  when $G^{\ad}$ is  simple and $G^{\der}=G^{\mathrm{sc}}$. 
\end{para}

\begin{para}
We assume for the rest of the section that  $G^{\ad}$ is simple and $G^{\der}=G^{\mathrm{sc}}$. To prove Theorem \ref{thm: spanning cartan directions}, we again use the series of inclusions 
$$
\fkT^{\cur}_{\lambda,\mu}\subset \fkT^{\tan}_{\lambda,\mu}\subset \fkT^{\FM}_{\lambda,\mu}.
$$
The theorem will then follow if we can show  $\fkT^{\cur}_{\lambda,\mu}=\fkT^{\FM}_{\lambda,\mu}$. 

For an element $H \in \fkt$, we write $\calW(H)$ for the set of pairs $(\varpi,\varpi')$ with $\varpi\in X^*(T)^+$ and  $\varpi'$ a weight of $V(\varpi)$ such that  $Hv_{\varpi'} \neq 0$ for some weight vector $v_{\varpi'}$ of weight $\varpi'$. The latter condition is equivalent to $\mathrm{d}{\varpi'}(H)$ being non-zero. We set
\[ 
l_H := l_H^{(\lambda,\mu)}=\min_{(\varpi,\varpi')\in \calW(H)} \langle\mu, \varpi\rangle - \langle\lambda, \varpi'\rangle.
\]

A similar computation to Proposition \ref{prop: FM bound relation} gives the following.

\begin{prop}
	Let $H\in \fkt$ and assume $t^{-l}H\in \fkT^{\FM}_{\lambda,\mu}$ with $l\geq 1$. Then $1\leq l\leq l_H$.\qed
\end{prop} 
Note that $t^{\langle\lambda,\alpha\rangle-k_{\alpha}}H_\alpha=t^{-k_{-\alpha}}H_\alpha\in\fkT_{\lambda,\mu}^{\cur}\subset \fkT_{\lambda,\mu}^{\FM}.$ Thus the previous proposition implies we have the inequality $$k_{-\alpha}\leq l_{H_\alpha}.$$
\end{para}

\begin{para}
Fix $\mu,\lambda$ as in the statement of Theorem \ref{thm: spanning cartan directions}. We will show the inclusion \begin{equation}\label{eqn: tangent space Cartan}
\bigoplus_{\alpha\in\Delta}\left(\bigoplus_{i=1}^{k_{-\alpha}}t^{-i}\bk  H_\alpha\right)=\fkT^{\cur}\subset \fkT^{\FM}
\end{equation} is an equality.  However, unlike the case of root directions, it is not a priori clear that $\fkT^{\FM}$  will decompose as a direct sum over $\alpha$ as is the case for $\fkT^{\cur}$ in \eqref{eqn: tangent space Cartan}. We will instead prove this directly by computing $l_H$ for all $H\in \fkt^{\der}$.

Let $T^{\der}=T\cap G^{\der}$, a maximal torus of $G^{\der}$, and write $\fkt^{\der}=\Lie T^{\der}$. We first show there are no non-trivial elements of $\fkT^{\FM}$ outside of $\fkt^{\der}$.

\begin{lemma}\label{lem: Cartan tangent space contained in derived group} 
	Let $H\in \fkt\setminus\fkt^{\der}$. Then $l_H=0$.
\end{lemma}

\begin{proof}
	Let $G^{\ab}$ denote the quotient of $G$ by $G^{\der}$, and $\fkg^{\ab}$ its Lie algebra. Then we have an exact sequence \[ \xymatrix{0\ar[r]  & T^{\der}\ar[r] & T\ar[r] &G^{\ab}\ar[r]&0}\] and hence an exact sequence
	\[\xymatrix{0\ar[r] &\fkt^{\der}\ar[r]& \fkt\ar[r]^{\psi} &\fkg^{\ab}\ar[r]& 0.}\]
	It follows that the image $\psi(H)$ of $H$ in $\fkg^{\ab}$ is non-zero. Since $G^{\ab}$ is a split torus, we may choose a character $\nu$ of $G^{\ab}$ such that  $\mathrm{d}\nu(\psi(H))\neq 0$. Its composition with $G\rightarrow G^{\ab}$ gives rise to a dominant weight  $\varpi\in X^*(T)^+$ with $\mathrm{d}\varpi(H)\neq 0$. Then we have $(\varpi,\varpi)\in \calW(H)$, and $\langle \mu,\varpi\rangle -\langle\lambda,\varpi\rangle =0$, since $\mu-\lambda$ is a sum of coroots. 
	It follows that $l_H=0$.
\end{proof}

\end{para}

\begin{para}

We now consider directions along $\fkt^{\der}$. Note that $\{H_\beta\}_{\beta\in \Delta}$ is a basis for $\fkt^{\der}$, so that any $H\in \fkt^{\der}$ can be written uniquely as $\sum_{\beta\in \Delta} m_\beta H_\beta$, $m_\beta\in \bk$.

\begin{prop}\label{prop: cartan spanning 2}
 Let $H=\sum_{\beta\in \Delta} m_\beta H_\beta\in \fkt^{\der}$,  with $H\neq 0$. Assume $(G,\mu)$ is of mod $p$ abelian type and is not of type $D_n^{\bbH}$. Then  for any $\lambda\in X^*(T)_+$ with $\lambda\lleq\mu$, we have 
	$$
	l_H=\min_{\beta\in\Delta,m_\beta\neq 0} k_{-\beta}.
	$$ 
\end{prop}
\begin{proof}Note that for $k=\min_{\beta\in\Delta,m_\beta\neq 0} k_{-\beta}$, we have $t^{-k}H\in \fkT_{\lambda,\mu}^{\tan}\subset \fkT^{\FM}_{\lambda,\mu}$, and hence $	l_H\geq\min_{\beta\in\Delta,m_\beta\neq 0} k_{-\beta}.$ Thus it suffices to show the reverse inequality.

Let $S\subset \Omega^{\min}$ be the subset of fundamental weights in the proof of Theorem \ref{thm: root direction equality}. Then by Corollary \ref{cor: k_-alpha simple root}, we have 
$$k_{-\alpha}
		=\min_{\varpi\in S}\langle\mu,\varpi\rangle-\langle\lambda,\varpi_\alpha\rangle$$
		for any $\alpha\in \Delta$. We verify in each case that there exists $\alpha\in \Delta$ and $\varpi\in S$ satisfying   	\begin{enumerate}
			\item[(a)] $m_\alpha\neq 0$ and $k_{-\alpha}=\min_{\beta\in\Delta,m_\beta\neq 0} k_{-\beta}.$
			\item[(b)] $k_{-\alpha}=\langle \mu,\varpi\rangle-\langle\lambda,\varpi_\alpha\rangle$.

			\item[(c)] $\sum_{\beta\in \Delta}m_\beta\langle\beta^\vee,\varpi_\alpha\rangle\neq 0$.
		\end{enumerate}
	
		In this case, the last condition implies for $0\neq v\in V(\varpi)$ a weight vector of weight $\varpi_\alpha$, we have $Hv=\sum_{\beta\in \Delta}m_\beta\langle\beta^\vee,\varpi_\alpha\rangle v\neq 0$, and hence 
	$(\varpi,\varpi_\alpha)\in\calW(H)$. It follows that 
	$$
	k_{-\alpha}=\langle \mu,\varpi\rangle-\langle\lambda,\varpi_\alpha\rangle\geq l_H
	$$as desired. For types $B_n,C_n$ and $D_n$,  we use the same notation for root systems and fundamental weights as in the proof of Theorem \ref{thm: root direction equality}.
		
	\textit{Type $A_{n-1}$:} In this case, we may take $G=GL_n$ and we identify $X_*(T)$ and $X^*(T)$ with $\bbZ^n$ under the usual pairing. Then the roots are given by $\pm e_i\mp e_j$, for $i<j$, with positive roots $e_i-e_j , i<j.$  The simple roots are given by $\Delta=\{\alpha_1,\dotsc,\alpha_{n-1}\}$, where $\alpha_{i}=e_i-e_{i+1}$. In this case, we take $S=\Omega^{\min}=\Omega$.
	
Let 	$\mu=(\mu_1,\dotsc,\mu_n)$, $\lambda=(\lambda_1,\dotsc,\lambda_n)$, and let $H=\sum_{i=1}^{n-1}m_{\alpha_i}H_{\alpha_i}.$ Choose $\alpha=\alpha_j\in \Delta$ and $\varpi=\varpi_k\in \Omega_{\min}$ with $|k-j|$ minimal satisfying (a) and (b). We will show that (c) is also satisfied.

If $k=j$, then  $\langle\alpha^\vee_i,\varpi_\alpha\rangle=\langle \alpha_i^\vee,\varpi_k\rangle=0$ for $i\neq k$, and hence $\sum_{\beta\in \Delta} m_\beta \langle \beta^\vee,\varpi_\alpha\rangle=m_\alpha \langle\alpha^\vee,\varpi_\alpha\rangle\neq 0$. Thus $(c)$ is satisfied.  We therefore  assume $k\neq j$.

We assume $k<j$; the case $j<k$ is symmetric. Note that the only possible $\beta\in \Delta$ such that $\langle\beta^\vee,\varpi_{\alpha}\rangle \neq0$  are $\beta=\alpha_{j-1},\alpha_j, \alpha_{k-1}$; the last case only occuring when $k>1$. Thus it suffices to show that $m_{\alpha_{j-1}},m_{k-1}=0$.

If $m_{\alpha_{j-1}}\neq 0$, note that $\varpi_{\alpha_{j-1}}+\alpha_{j-1}=\varpi_{\alpha_j}$. It follows that $$k_{-\alpha_{j-1}}\leq \langle \mu,\varpi\rangle-\langle \lambda,\varpi_{\alpha_{j-1}}\rangle\leq \langle \mu,\varpi\rangle-\langle \lambda,\varpi_{\alpha_{j}}\rangle=k_{-\alpha_{j}}$$
contradicting minimality of $|k-j|$; thus $m_{\alpha_{j-1}}=0$.

If $m_{\alpha_{k-1}}\neq 0$, then we  have 
\begin{align*}
\sum_{i=1}^{k-1}\mu_i-\lambda_i&=\langle\mu,\varpi_{k-1}\rangle-\langle\lambda,\varpi_{k-1}
\rangle\\
&>\langle\mu,\varpi\rangle-\langle\lambda,\varpi_{\alpha_j}\rangle\\
&=\sum_{i=1}^k\mu_i-(\sum_{i=1}^{k-1}\lambda_i+\lambda_j),
\end{align*}
where the inequality follows from the minimality of $|k-j|$. It follows that $\lambda_j>\mu_k$.

Similarly, we have by minimality that 
\begin{align*}
\sum_{i=1}^{j}\mu_i-\lambda_i&=\langle\mu,\varpi_{j}\rangle-\langle\lambda,\varpi_{j}
\rangle\\
&>\langle\mu,\varpi\rangle-\langle\lambda,\varpi_{\alpha_j}\rangle\\
&=\sum_{i=1}^k\mu_i-(\sum_{i=1}^{k-1}\lambda_i+\lambda_j),
\end{align*}
and hence $\sum_{i=k+1}^j\mu_i>\sum_{i= k}^{j-1}\lambda_{i}$. But since $\mu$ and $\lambda$ are dominant, we have $$\lambda_k\geq \dotsc \geq \lambda_j>\mu_k\geq\dotsc\geq\mu_j,$$ which is a contradiction. It follows that $m_{\alpha_{k-1}}=0$. 
	
	\textit{Type $B_n$:} Let $\mu=(r,0,\dotsc,0)$, $\lambda=(\lambda_1,\dotsc,\lambda_n)\in P^{\vee,+}$, with $r\in \bbZ_{>0}$, and set $\delta= \frac{r-\sum_{i=1}^n\lambda_i}{2}$. We have $\Delta=\{\alpha_1,\dotsc,\alpha_n\}$ and $S=\{\varpi_n\}$, and hence
	$$k_{-\alpha_i}=\begin{cases}\delta+\lambda_{i+1} & \text{ for }i=1,\dotsc,n-1\\ \delta & i=n \end{cases}$$ since $\varpi_{n,\alpha_i}=\varpi_n-e_{i+1}$.
	In particular, we have $k_{-\alpha_1}\geq\dotsc\geq k_{-\alpha_n}$. For $H=\sum_{i=1}^nm_{\alpha_i}H_{\alpha_i}\in \fkt^{\der}$, let $j\in \{1,\dotsc,n\}$ be largest such that $m_{\alpha_j}\neq 0$. Then $\alpha_j$ satisfies (a) and (b) (for $\varpi=\varpi_n$). Since $\langle\alpha_i^\vee,\varpi_{\alpha_j}\rangle= 0$ for $i<j$, we have $\sum_{i=1}^nm_{\alpha_i}\langle\alpha_i^\vee,\varpi_{\alpha_i}\rangle=m_{\alpha_j}\langle\alpha_j^\vee,\varpi_{\alpha_j}\rangle\neq 0$  and hence $(c)$ is satisfied.

	\textit{Type $C_n$:} Let $\mu=(\frac{r}{2},\dotsc,\frac{r}{2}), \lambda=(\lambda_1,\dotsc,\lambda_n)\in P^{\vee,+}$, with $r\in \bbZ_{>0}$. We have $\Delta=\{\alpha_1,\dotsc,\alpha_n\}$ and $S=\{\varpi_1\}$, and hence $$k_{-\alpha_i}=\frac{r}{2}-\lambda_i$$ since $\varpi_{1,\alpha_i}=e_i$. In particular, we have $k_{-\alpha_1}\leq\dotsc\leq k_{-\alpha_n}$.  For $H=\sum_{i=1}^nm_{\alpha_i}H_{\alpha_i}\in \fkt^{\der}$, let $j$ be smallest such that $m_{\alpha_j}\neq 0$. Then $\alpha_j$ satisfies (a) and (b) (for $\varpi=\varpi_1$). Since $\langle\alpha_i^\vee,\varpi_{\alpha_j}\rangle=0$ for $i>j$, we have $\sum_{i=1}^nm_{\alpha_i}\langle\alpha_i^\vee,\varpi_{\alpha_i}\rangle=m_{\alpha_j}\langle\alpha_j^\vee,\varpi_{\alpha_j}\rangle\neq 0$  and hence $(c)$ is satisfied.
	
	\textit{Type $D_n^{\bbR}$:} Let $\mu=(r,0,\dotsc,0),\lambda=(\lambda_1,\dotsc,\lambda_n)\in P^{\vee,+}$ with $r\in \bbZ_{>0}$. Upon applying the automorphism of the Dynkin diagram switching $\alpha_{n-1}$ and $\alpha_n$, we may assume without loss of generality that $\lambda_n \geq 0$. Let $\delta=\frac{r-\sum_{i=1}^n \lambda_i}{2}$. We have $\Delta=\{\alpha_1,\dotsc,\alpha_n\}$ and $S=\{\varpi_{n-1},\varpi_n\}$. Then  we compute that
	 \begin{align*}k_{-\alpha_i}&=\begin{cases}\langle\mu,\varpi_{n-1}\rangle-\langle\lambda, (\varpi_{n-1})_{\alpha_i}\rangle &\text{ if $i=1,\dotsc,n-1$}
	\\ \langle\mu,\varpi_{n}\rangle -\langle\lambda,\varpi_{n}\rangle &\text{ if $i=n$}\end{cases}
\\&=\begin{cases}\delta+\lambda_{i+1}&\text{ if $i=1,\dotsc,n-1$}\\
	\delta &\text{ if $i=n$}.
	\end{cases}
	\end{align*}
In particular, we have $k_{-\alpha_1}\geq\dotsc \geq k_{-\alpha_n}$.  For $H=\sum_{i=1}^nm_{\alpha_i}H_{\alpha_i}\in \fkt^{\der}$, let $j$ be largest such that $m_{\alpha_j}\neq 0$. Then $\alpha_j$ satisfies (a) and (b) for $$\varpi=\begin{cases}\varpi_{n-1} &j=1,\dotsc,n\\
\varpi_n &j=n\ .\end{cases}$$

We compute that $\langle\alpha_i^\vee,\varpi_{\alpha_j}\rangle=0$ for $i<j$, and hence  $\sum_{i=1}^nm_{\alpha_i}\langle\alpha_i^\vee,\varpi_{\alpha_i}\rangle=m_{\alpha_j}\langle\alpha_j^\vee,\varpi_{\alpha_j}\rangle$ is non-zero, i.e. (c) is satisfied.
\end{proof}

			\begin{prop}\label{prop: cartan spanning DH}Let $H=\sum_{\beta\in \Delta} m_\beta H_\beta\in \fkt^{\der}$,  with $H\neq 0$. Assume $(G,\mu)$ is of type $D_n^{\bbH}$ and that either $\langle\lambda,\alpha_{n-1}\rangle=0$ or $\langle\lambda,\alpha_n\rangle=0$. Then we have
				$$
				l_H=\min_{\beta\in\Delta,m_\beta\neq 0} k_{-\beta}.
				$$ 
				\end{prop}
			\begin{proof}
As in Proposition \ref{prop: cartan spanning 2}, it suffices to prove $	l_H\leq\min_{\beta\in\Delta,m_\beta\neq 0} k_{-\beta}.$ Let $$\mu=s\omega_{n-1}+t\omega_{n-1}=(\frac{r}{2},\dotsc,\frac{r}{2},\frac{q}{2})$$ with $s-t=q, s+t=r$, and let $\lambda=(\lambda_1,\dotsc,\lambda_n)$. We have $\Delta=\{\alpha_1,\dotsc,\alpha_n\}$ and $S=\{\varpi_1,\varpi_{n-1},\varpi_n\}$. Let  $\alpha=\alpha_j\in \Delta$ and $\varpi=\varpi_k\in S$ such that the length of the geodesic between $\alpha_j$ and $\alpha_k$ is minimal for those pairs satisfying the following properties:
 \begin{enumerate}
 	\item[(a)] $m_\alpha\neq 0$ and $k_{-\alpha}=\min_{\beta\in\Delta,m_\beta\neq 0} k_{-\beta}.$
 	\item[(b)] $k_{-\alpha}=\langle \mu,\varpi\rangle-\langle\lambda,\varpi_\alpha\rangle$.
 
 \end{enumerate}
	
	If $k=j$, then as in the case of Type $A_{n-1}$ in Proposition \ref{prop: cartan spanning 2}, we have 
	$$
	\sum_{\beta\in \Delta}m_\beta\langle\beta^\vee,\varpi_\alpha\rangle=m_\alpha\langle\alpha^\vee,\varpi_\alpha\rangle\neq 0
	$$
	 and hence we obtain the bound $	l_H\leq\min_{\beta\in\Delta,m_\beta\neq 0} k_{-\beta}$. Thus assume $k\neq j$. Let $\alpha=\gamma_0,\dotsc,\gamma_m=\alpha_k$ be the geodesic from $\alpha$ to $\alpha_k$ so that $\varpi_\alpha=\varpi-\sum_{i=1}^m\gamma_i$. Then we compute that if $\langle\beta^\vee,\varpi_\alpha\rangle\neq 0$ for $\beta\in \Delta$, we have $\beta=\gamma_0,\gamma_1,\gamma_m$ or $\gamma_0'$, where $\gamma_0'$ is a neighbor of $\gamma_1$ not equal to $\gamma_0$ or $\gamma_2$. Note that $\gamma_0'$ only occurs if $\gamma_1=\alpha_{n-2}$. 
	
	By minimality, we have $m_{\alpha_k}=0$. And similar to the Type $A_{n-1}$ case in Proposition \ref{prop: cartan spanning 2}, we have $$k_{-\gamma_1}\leq\langle\mu,\varpi\rangle-\langle\lambda,\varpi_{\gamma_1}\rangle\leq \langle\mu,\varpi\rangle-\langle\lambda,\varpi_{\gamma_0}\rangle=k_{-\gamma_0},$$ and hence $m_{-\gamma_1}=0$ by minimality. If $m_{\gamma_0'}=0$, then   $\sum_{\beta\in \Delta}\langle\beta^\vee,\varpi_{\alpha}\rangle= m_\alpha\langle \alpha^\vee,\varpi_\alpha\rangle\neq 0$ and hence $	l_H\leq\min_{\beta\in\Delta,m_\beta\neq 0} k_{-\beta}$ as desired.
	
	 Now assume $m_{\gamma_0}\neq 0$ and $m_{\gamma_0'}\neq 0$. We consider separate cases depending on the choice of $\varpi$.

	Case (1): $\varpi=\varpi_{n-1}$ or $\varpi_n$. It suffices to consider $\varpi=\varpi_{n-1}$ as the other case is obtained by applying the non-trivial automorphism of the Dynkin diagram. Then we have $\gamma_1=\alpha_{n-2}$ and $\{\gamma_0,\gamma_0'\}=\{\alpha_{n-3},\alpha_n\}.$
	
Note that  $(\varpi_{n-1})_{\alpha_{n-3}}=(\varpi_{n-1})_{\alpha_n}=\varpi_{n-1}-\alpha_{n-1}-\alpha_{n-2}$. By minimality, we have 
\begin{align*}
k_{-\gamma_0'}\geq k_{-\gamma_0}&=\langle \mu,\varpi_{n-1}\rangle-\langle\lambda,(\varpi_{n-1})_{\gamma_0}\rangle
\\&=\langle \mu,\varpi_{n-1}\rangle-\langle\lambda,\varpi_{n-1}-\alpha_{n-1}-\alpha_{n-2}\rangle
\\&=\frac{(n-1)r-q}{4}-\frac{1}{2}(\sum_{i=1}^n\lambda_i) +\lambda_{n-2}\ .
\end{align*} 

In particular, since $\varpi_1\in S$, we  have \begin{align}\label{eqn: ineq DH1}\begin{split}
\frac{r}{2} -\lambda_{n-3}&=\langle\mu,\varpi_1\rangle-\langle\lambda,(\varpi_1)_{\alpha_{n-3}}\rangle  \\& \geq k_{-\alpha_{n-3}}
\\&\geq \frac{(n-1)r-q}{4}- \frac{1}{2}(\sum_{i=1}^n\lambda_i) +\lambda_{n-2}\\&\geq \frac{3r-q}{4}-\frac{1}{2}(\lambda_{n-3}-\lambda_{n-2}+\lambda_{n-1}+\lambda_{n}) \end{split}\end{align}
where the last inequality follows from the fact that $r\geq 2\lambda_i$ for all $i$. This gives
\begin{equation}
\label{eqn: ineq DH2}0\geq (r-q)+2(\lambda_{n-3}+\lambda_{n-2}-\lambda_{n-1}- \lambda_{n}).
\end{equation}
On the other hand, we have $ r\geq q$, and $\lambda_{n-3}\geq\lambda_{n-2}\geq\lambda_{n-1}\geq\lambda_{n}$, so that \eqref{eqn: ineq DH2} is an equality. It follows that every inequality in \eqref{eqn: ineq DH1} is also an equality so, in particular,  $$\frac{r}{2}-\lambda_{n-3}=k_{-\alpha_{n-3}}=k_{-\gamma_0}.$$

We now replace $\varpi$ by $\varpi_1$ and $\alpha$ by $\alpha_l$, where $l\in \{1,\dotsc,n\}$ is least such that $m_{\alpha_l}\neq 0$. Then $l\leq n-3$, and $$k_{-\alpha_l}\leq \langle \mu,\varpi_1\rangle -\langle \lambda,(\varpi_1)_{\alpha_l}\rangle=\frac{r}{2}-\lambda_l\leq k_{-\alpha_{n-3}}, $$ and hence we have equality throughout since $\alpha_{n-3}$ satisfies (a). Thus (a) and (b) are also satisfied for $\alpha=\alpha_l$ and $\varpi=\varpi_1$. Moreover, for $i> l$, we have $\langle \alpha_i^\vee,(\varpi_1)_{\alpha_l}\rangle=0$. It follows that $\sum_{\beta\in \Delta}\langle\beta^\vee,\varpi_{\alpha}\rangle= m_\alpha\langle \alpha^\vee,\varpi_\alpha\rangle\neq 0$ and hence $	l_H\leq\min_{\beta\in\Delta,m_\beta\neq 0} k_{-\beta}$.

	Case (2): $\varpi=\varpi_1$. Then $\gamma_0\in \{\alpha_{n-1},\alpha_n\}$. If $m_{-\alpha_{n-1}}\neq -m_{-\alpha_{n}}$, then we have $$\sum_{\beta\in \Delta} m_\beta\langle \beta^\vee, \varpi_\alpha \rangle = m_{\alpha_{n-1}}+m_{\alpha_n}\neq 0,$$ and we are done. Otherwise assume $m_{-\alpha_{n-1}}=-m_{-\alpha_n}$. By assumption, we have  either $\langle\lambda,\alpha_{n-1}\rangle=0$ or $\langle\lambda,\alpha_n\rangle=0$. We set $$\varpi'=\begin{cases}\varpi_\alpha-\alpha_{n-1} &\text{ if $\langle\lambda,\alpha_{n-1}\rangle=0$ }\\
	\varpi_\alpha-\alpha_n &\text{ if $\langle\lambda,\alpha_{n}\rangle=0$ }.
	\end{cases}$$
	Then $$\sum_{\beta\in \Delta}m_\beta\langle\beta^\vee,\varpi'\rangle= \begin{cases}-2m_{\alpha_{n-1} }& \text{ if $\langle\lambda,\alpha_{n-1}\rangle=0$ }\\
		-2m_{\alpha_n} &\text{ if $\langle\lambda,\alpha_{n}\rangle=0$\ , }
		\end{cases}$$
which is non-zero in either case.	On the other hand,  we have $$\langle \mu,\varpi\rangle -\langle \lambda,\varpi'\rangle=\langle \mu,\varpi\rangle -\langle \lambda,\varpi_\alpha\rangle=k_{-\alpha}$$ and hence $l_H\leq k_{-\alpha}=\min_{\beta\in\Delta,m_\beta\neq 0} k_{-\beta}$ as desired.
\end{proof}

\end{para}
\begin{para}\begin{proof}[Proof of Theorem \ref{thm: spanning cartan directions}]Fix $(G,\mu)$ and $\lambda$ as in the statement, and let $t^{-l}H\in \fkT^{\FM}_{\lambda,\mu}$, with $H\in \fkt$ and $l\geq 1$. Then we have $l\leq l_H$ by Proposition \ref{prop: curve bound tangent vector}.
			By Lemma \ref{lem: Cartan tangent space contained in derived group}, we have $H\in \fkt^{\der}$, and hence we can write $H=\sum_{\beta\in \Delta}m_\beta H_\beta$, for some $m_\beta\in \bk$.  We show that $H\in \fkT_{\lambda,\mu}^{\cur}$ by induction on the number of non-zero $m_\beta$.
		
			Let $\alpha\in \Delta$ with $k_{-\alpha}=\min_{\beta\in \Delta, m_\beta\neq 0}k_{-\beta}$.	By Proposition \ref{prop: cartan spanning 2} for case (1) and Proposition \ref{prop: cartan spanning DH} for case (2), we have $k_{-\alpha}\geq l$. It follows that $t^{-l}H_\alpha\in \fkT^{\cur}_{\lambda,\mu}$. By induction, $H-m_\beta t^{-l}H_\alpha\in \fkT^{\cur}_{\lambda,\mu}$, and hence $H\in \fkT^{\cur}_{\lambda,\mu}$ as desired.
	\end{proof}
	
	\begin{Remark}{\rm We give an example where $(G,\mu)$ is of type $D_4^{\mathbb{H}}$ and $\lambda\in X_*(T)_+,$ with $\lambda\lleq \mu$ for which $\fkT_{\lambda,\mu}^{\cur}\subset \fkT^{\FM}_{\lambda,\mu}$ is not an equality. Let $\mu=3\varpi_{n-1}+3\varpi_n=(3,3,3,0)$ and $\lambda=(1,1,1,0)$. We take $H=H_{\alpha_{n-1}}-H_{\alpha_n}\in \fkt^{\der}$. Then we compute that  $$k_{-\alpha_{n-1}}=k_{-\alpha_n}=\langle\mu,\varpi_1\rangle-\langle\lambda,\varpi_1-\alpha_1-\alpha_2\rangle=2$$ using Corollary \ref{cor: k_-alpha simple root}. On the other hand, we compute that $l_H=3$, and hence $t^{-3}H\in \fkT^{\FM}_{\lambda,\mu}\setminus \fkT^{\cur}_{\lambda,\mu}$.}
		\end{Remark}
\end{para}

\begin{para}\begin{proof}[Proof of Theorem \ref{thm:curve span}]
Theorem \ref{thm: root direction equality} and Theorem \ref{thm: spanning cartan directions} together then imply that for $(G,\mu)$ and $\lambda$ as in Theorem \ref{thm: spanning cartan directions} (1),  the tangent space $T_{t^\lambda}S_\mu$ is spanned by smooth formal curves. The same is then true for any point lying in the $G(\bk\lps t\rps)$-orbit of some $t^\lambda$. In particular, if $(G,\mu)$ has no factors of type $D^{\bbH}$, the tangent space $T_xS_\mu$ is spanned by smooth formal curves for all $x\in S_\mu(\bk)$.
\end{proof}

\begin{Remark}{\rm As mentioned in Remark \ref{rem: FM conjecture}, it is conjectured that $S^{\FM}_\mu=S_\mu$. Theorem \ref{thm:curve span} provides some evidence for this conjecture for $(G,\mu)$ of mod $p$ abelian type without factors of type $D^{\bbH}$. Indeed the theorem implies that $S_{\mu}^{\FM}$ and $S_\mu$ have the same tangent spaces. It may be possible to use similar methods to understand the jet schemes of $S_\mu^{\FM}$, but we do not pursue this here.}
	\end{Remark}
\end{para}

\subsection{Tangent spaces of certain local models}\label{ss:tangentspacesLM}
	 	\begin{para}
		Let us now return to the set-up of \S\ref{ss:LM}.
	Let $(G,\{\mu\},\Gg)$ be a local model triple over $ \O_F$ which satisfies our standard assumptions. In addition, we assume   
	there is a finite extension $K/F$ and a  reductive group scheme $H$ over $ \O_{K}$ such that 
		$$
		\calG\cong \mathrm{Res}_{\O_K/\O_F}H.
		$$

\begin{lemma}\label{lemma:LMtomodp}
	Let $(G,\{\mu\},\calG)$ be a local model triple satisfying the assumptions above.  Then there is a pair $(\underline{G},\underline \mu )$, where $\underline{G}$ is a reductive group over $k$ and $\underline\mu$ a cocharacter of $\underline{G}$, which is of mod $p$ abelian type and with $p\nmid|\pi_1(\underline{G}^{\der})|$, such that there is an isomorphism 
	$$
	\BMloc_{\Gg,\mu}\otimes_{\O_E}k\cong S_{\underline \mu},
	$$
	where $S_{\underline \mu}\subset \Gr_{\underline{G}}$ is the corresponding affine Schubert variety.
\end{lemma}
\begin{proof}
Under the above assumptions, we have $\BMloc_{\Gg,\mu}=\Mloc_{\calG,\mu}$, by Theorem \ref{thm:twoLM}. Since $H$ splits after an unramified base change we can easily see that it is enough to show the statement under the additional assumption that $H$ is split reductive over $\O_K$.	Now remark that the group
	$G'$  used in  the construction of $\Mloc_{\Gg,\mu}$ in \S \ref{sss:generalLM}
	 is such that $p\nmid|\pi_1(G'^{\der})|$ and is again of the form $G'=\Res_{K/F}H'$. Denote by $\Gg'$ the stabilizer group scheme
	 of $G'$ which corresponds to $\Gg$. This is  also of the same form $\Gg'=\mathrm{Res}_{\O_K/\O_F}H'$, with $H'$ split and reductive. By the definition of $\Mloc_{\Gg,\mu}$, we have
	 \[
	 \Mloc_{\Gg,\mu}\otimes_{\O_E}k\simeq {\rm M}_{\Gg',\mu'}\otimes_{\O_{E'}}k.
	 \]
Recall that $\Mloc_{\Gg',\mu'}={\rm M}_{\Gg',\mu'}$ and is given  via a Beilinson-Drinfeld affine Grassmannian, as in Definition \ref{def:LevinLM}. This allows us to reduce proving the 
statement for ${\rm M}_{\Gg,\mu}\otimes_{O_E}k$ when $\Gg= \mathrm{Res}_{\O_K/\O_F}H$, with $H$ split reductive, under the additional assumption $p\nmid|\pi_1(G^{\der})|$. The rest of the proof is a case of unpacking the constructions in \cite{Levin} and above.

	We may assume $H^{\ad}$ is simple. Let $K_0$ denote the maximal unramified extension of $F$ contained in $K$, and let $\pi$ be a uniformizer of $K'$. Since $H$ is split, we can take $\underline{\calH}_0=H\otimes \O_{K_0}[u]$  in \cite[\S3.3]{Levin}. Here, by slightly abusing notation, we also write $H$ for the split Chevalley form of $H$. 
	 	
	Let $k_0$ be the residue field of $K_0$, and let 
	$$
	\underline{G}=\prod_{\varphi:k_0\rightarrow k} H\otimes k
	$$
	 a split reductive group scheme over $k$. Then $\Mloc_{\Gg,\mu}\otimes_{\O_E}k$ can be identified with a Schubert variety $S_{\underline{\mu}}\subset \Gr_{\underline{G}}$ for $\underline{\mu}$ a dominant cocharacter of $\underline{G}$. The cocharacter $\underline \mu$ of  $\underline{G}$ can be computed from the cocharacter $\mu$ as follows. We have an isomorphism 
	$$
	G_{\bar {K}}\cong \prod_{\theta: K\rightarrow \bar {K}} H_{\bar {K}}
	$$
	where the product is taken over $F$-algebra morphisms of $K$ into the algebraic closure $\bar K$. We write $\mu_{\theta}$ for the cocharacter of $H_{\bar {K}}$ in the factor corresponding to $\theta$, and similarly we write $\underline{\mu}_{\varphi}$ for the factor of $\underline{\mu}$ corresponding to $\varphi$. We may identify dominant cocharacters of $H_{\bar {K}}=H\otimes \bar  K$ with dominant cocharacters of $H\otimes k$.  Then under this identification, we have 
	$$
	\underline{\mu}_{\varphi}=\sum_{\theta \text{ s.t. } \theta|_{k_0}=\varphi}\mu_\theta.
	$$
	Since $(G, \mu)$ is of    abelian type, the classification of such pairs (cf. \cite[Prop. 7.2.1]{PRlsv} and its proof)
	 implies that $\mu^\ad_{\theta}$ is minuscule, and if $H$ is of type $D_n$, we have either $\mu_\theta^{\ad}\in\{\varpi_1, 1\}$ for all $\theta$, or $\mu_\theta^{\ad}\in\{\varpi_{n-1}, \varpi_{n}, 1\}$ for all $\theta$. The result follows.
\end{proof}

The following Theorem now is immediate from the preceding lemma and Theorems \ref{thm: root direction equality} and \ref{thm: spanning cartan directions}, see also Theorem \ref{thm:curve span}.

\begin{thm}\label{corLMSpan}
	Let $(G,\{\mu\},\calG)$ be a local model triple 
 over $ \O_F$ which satisfies our standard assumptions. In addition, we assume that 
	there is a finite extension $K/F$ and a reductive group scheme $H$ over $ \O_{K}$ such that 
		$
		\calG\cong \mathrm{Res}_{\O_K/\O_F}H.
		$
 \begin{altenumerate}
		\item  If the point $x\in \BMloc_{\calG,\mu}(k)$ lies in the minimal stratum, then the tangent space of $\BMloc_{\calG,\mu}\otimes_{\O_E}k$ at $x$ is spanned by smooth formal curves.
		\item  If $G$ has no factors of type $D_n^{\mathbb H}$, then, for every point $x\in \BMloc_{\calG,\mu}(k)$, the tangent space of $\BMloc_{\calG,\mu}\otimes_{\O_E}k$ at $x$ is spanned by smooth formal curves.\qed
	\end{altenumerate} 
\end{thm}
\end{para}

\section{Displays and very good   embeddings}\label{s:Displays}

In this section, we revisit the theory of \cite{KP} about deformations of Dieudonn\'e displays equipped with tensors,
give the key definition of a very good integral Hodge embedding, and prove various properties of very good embeddings.

\subsection{Displays and deformations}\label{1.1} 

We will mostly use the  notations of \cite[\S 3.1]{KP}.   Suppose $R$ is a Noetherian complete local ring with residue field $k$ and maximal ideal $\frakm$. Fix integers $0\leq d\leq n$.  We let $W(R)$ denote the Witt vectors of $R$. We consider the subring $\widehat{W}(R)\subset W(R)$ given by 
$$
\widehat{W}(R)=W(k)\oplus\widehat{W}(\mathfrak{m})\subset W(R),
$$
 where $\widehat{W}(\mathfrak{m})\subset W(R)$ consists of Witt vectors $(w_i)_{i\geq1}$ with $w_i\in\mathfrak{m}$ and $w_i\rightarrow 0$ in the $\mathfrak{m}$-adic topology. We have $\widehat{W}(R)=\varprojlim_a \widehat{W}(R/\mathfrak{m}^a)$, and $\widehat{W}(R/\mathfrak{m}^a)$, for each $a$, is a (non-Noetherian) complete local ring with residue field $k$; see \cite{ZinkCFT}, \cite{ZinkWindows}, for details. We let $\hat I_R=I_{\whW(R)}$ be the kernel of the ring homomorphism $\whW(R)\to R$ given by projection to the first Witt coordinate.
 Also, we denote by $\phi: \whW(R)\to \whW(R)$ the Frobenius and by $V^{-1}: \hat I_R\to \whW(R)$ the inverse of the Verschiebung (see \cite{ZinkDisplay}, \cite{Lau2}). 
 
 The data $(\whW(R),\hat I_R, \phi, V^{-1})$ give an example of a ``frame" in the sense of Zink and Lau, as we will see next.

 \begin{para}
 By \cite[Definition 2.1]{Lau}, \cite[2.A]{Lau2}, a {\sl frame} is a quadruple ${\mathscr F}=(S, I, \phi, \phi_1)$ consisting of a ring $S$, an ideal $I$ of $S$, a ring endomorphism $\phi: S\to S$, and a $\phi$-linear endomorphism $\phi_1: I\to S$, such that the following hold:
 \begin{itemize}
 \item[i)] $I+pS\subset {\rm Rad}(S)$,
 \item[ii)] $\phi(a)\equiv\,a^p\,{\rm mod}\,pS$, i.e. $\phi$ is a lift of the Frobenius on $S/pS$,
 \item[iii)] $\phi_1(I)$ generates $S$ as an $S$-module.
 \end{itemize}
For our purposes, we will also assume that $\phi(a)=p\phi_1(a)$, for all $a\in I$, i.e. $\theta=p$ in the notation of \cite[Lem. 2.2]{Lau}.  (Recall, we assume $p>2$ throughout.)

By definition, a  morphism of frames $\alpha: {\mathscr F}=(S, I, \phi, \phi_1)\to {\mathscr F}'=(S', I', \phi', \phi'_1)$   
is a ring homomorphism  $\alpha: S\to S'$ such that 
 $\alpha(I)\subset I'$, and 
 $\phi'\cdot \alpha=\alpha\cdot \phi$, $\phi'_1\cdot \alpha=\alpha\cdot \phi_1$. (These are called ``strict morphisms" in \cite[2.A]{Lau2}.)  

We give some examples of frames that we will use:

a) Suppose $R$ is a Noetherian complete local ring with residue field $k$. Then, as above, we have the ``Dieudonn\'e--Witt frame" 
\[
{\mathscr D}_R:=(\whW(R),  \hat I_R, \phi, V^{-1}).
\]

b) Suppose that $B$ and $R$ are Artin local rings with residue field $k$ and $B\to R$ is a surjection whose kernel $\fkb$ is equipped with divided powers. Then, we also have the ``relative 
Dieudonn\'e--Witt frame" 
\[
 {\mathscr D}_{B/R}:=(\whW(B), \hat I_{B/R}, \phi, V^{-1}).
\]
 Here, $\hat I_{B/R}$ is the kernel of the composition
$\whW(B)\to \whW(R)\to R$, and $V^{-1}$ is defined by extending $V^{-1}: \hat I_B\to \whW(B)$ to $\hat I_{B/R}=\hat I_B\oplus [\fkb]$ by setting $V^{-1}([\fkb])=0$; this construction uses Zink's log coordinates, cf. \cite[Lemma 38]{ZinkDisplay}, \cite[2.C, 2.D]{Lau2}. There are natural  morphisms of frames,
\[
 {\mathscr D}_{B}\to  {\mathscr D}_{B/R}\to  {\mathscr D}_{R}.
\]

c) Later, we will also consider the frame $(\gS, (p), \phi, p^{-1}\phi)$, where $\gS=W(k)\lps u\rps$ and $\phi:\gS\to \gS$ is the standard lift of Frobenius with $\phi(u)=u^p$.

\end{para}

\begin{para}\label{par:frametilde}
 Fix a frame  ${\mathscr F}=(S, I, \phi, \phi_1)$ as above. We can  consider the category $P_{\mathscr F}$ of pairs  $(M, M_1)$, where $M$ is a finite free $S$-module and $M_1$ is an $S$-submodule of $M$ such that there is a \emph{normal decomposition} $M=L\oplus T$, $M_1=L\oplus IT$, with $L$, $T$, finite free $S$-modules of rank $d$ and $n-d$ respectively.  Morphisms $(M, M_1)\to (M', M'_1)$ in the category are $S$-homomorphisms 
 $M\to M'$ which take $M_1$ to $M'_1$. We will call an object of the category $P_{{\mathscr F}}$ a pair over the frame ${\mathscr F}$, or simply a pair over $S$, if the frame structure on $S$ is understood.

If $\alpha: {\mathscr F}\to {\mathscr F'}$ is a   frame  morphism then there is a corresponding base change functor $\alpha_*: P_{{\mathscr F}}\to P_{{\mathscr F}'}$ which in terms of normal decompositions 
is given by $(L, T)\mapsto (S'\otimes_SL, S'\otimes_ST)$.

We now  
define a functor 
\[
\tau_{\mathscr F}: P_{{\mathscr F}}\to {\rm Mod}^{\rm ff}_S,\quad (M, M_1)\mapsto \wtM_1,
\] 
 into the category  of finite free $S$-modules 
 as follows:

Choose a basis $\B=(e_1,\ldots, e_n)$ of $M$, such that $(e_1,\ldots, e_d)$ is a basis of $L$ and $(e_{d+1},\ldots, e_n)$ is a basis of $T$.
(We say that such a basis $\B$ of $M$ is {\sl adapted} to the normal decomposition $M=L\oplus T$.) We set $\wtM_1$ to be the free $\whW(R)$-module of rank $n$ with basis $\ti \B=(\ti e_1,\ldots, \ti e_n)$. Let $(M', M'_1)$ be a second pair, with $M'=L'\oplus T'$, $M'_1=L'\oplus IT'$ and $\B'=(e'_1,\ldots, e'_{n})$ an adapted basis. Suppose $f: (M, M_1)\to (M', M'_1)$
is a morphism of pairs. We can write $f$ in terms of $\B$, $\B'$, as a matrix in block form
\[
\begin{pmatrix}
A & B\\
C & D
\end{pmatrix}
\]
with the entries of $C$ in $I$.  Then the functor associates to $f$ the  map $\ti f: \wtM_1\to \wtM_1'$ which, in the bases $\ti \B$ and $\ti \B'$,
is given by
\begin{equation}\label{phiphi1Matrix}
\begin{pmatrix}
\phi(A)& p\phi(B)\\
\phi_1(C) & \phi(D)
\end{pmatrix}.
\end{equation}
We can check, using that $\phi_1$ is $\phi$-linear and $\phi_{|I}=p\phi_1$, that $f\mapsto \ti f$ respects composition.
The functor $\tau_{\mathscr F}$ is, up to natural equivalence, independent
of the choices of bases, cf. \cite[2.3]{BueltelPappas}.

  The functors $\tau_{\mathscr F}$, for variable ${\mathscr F}$, are compatible with base change in the sense that, given a  morphism $\alpha: {\mathscr F}\to {\mathscr F'}$, there are natural isomorphisms
\begin{equation}\label{bctilde}
\wtM'_1\simeq \wtM_1\otimes_S S',
\end{equation}
where we denote by $\wtM'_1$ the finite free $S'$-module associated by $\tau_{\mathscr F'}$ to the object $(M', M'_1)=\alpha_*(M, M_1)$ of $P_{{\mathscr F}'}$.
\end{para}

\begin{para}
We can apply this construction to the Dieudonn\'e--Witt frame ${\mathscr D}_R:=(\whW(R),  \hat I_R, \phi, V^{-1})$, where $R$ is as above. Then pairs $(M, M_1)$ over ${\mathscr D}_R$ amount to pairs $(M, M_1)$ of a finite free $\whW(R)$-module $M$ of rank $n$ and a $\whW(R)$-submodule $M_1\subset M$ such that $M/M_1$ is a finite free $R$-module of rank $n-d$. Indeed, assuming $M/M_1$ is finite free of rank $n-d$ we can write $M=L\oplus T$, $M_1=L\oplus \hat I_RT$, where $L$ and $T$ are finite free $\whW(R)$-modules of rank $d$ and $n-d$ (cf. \cite[2.C]{Lau2}). The above functor 
\[
(M, M_1)\mapsto \wtM_1,
\]
generalizes a construction of \cite[\S 3.1]{KP}, see \cite{Hoff}. 

The base change compatibility (\ref{bctilde}) now gives the following. 
Let $R'\to R$ be a  local homomorphism of complete local rings as before. This induces a frame morphism ${\mathscr D}_{R'}\to {\mathscr D}_{R}$ and a  base change from pairs $(M', M'_1)$ over $\whW(R')$ to pairs $(M, M_1)$ over $\whW(R)$ as in \cite[3.1.6]{KP}. This base change is compatible with the functor above, so  we have natural isomorphisms
\begin{equation}\label{bc5.1}
\wtM_{\whW(R), 1}\simeq \wtM'_{\whW(R'), 1}\otimes_{\whW(R')}\whW(R).
\end{equation}
Here we write $M_{\whW(R)}$, $M_{\whW(R), 1}$ and $\wtM_{\whW(R), 1}$ instead of $M$, $M_1$, $\wtM_1$, to emphasize the ring $\whW(R)$ over which these are modules.

Starting from a pair $(M, M_1)$ over $\whW(R)$ as above, we will denote by $(M_0, M_{0,1})$ the pair of $W(k)$-modules obtained from $(M, M_1)$ by base change by $R\to R/\frakm=k$.
By (\ref{bc5.1}), we have a natural isomorphism
\[
\wtM_{0,1}\simeq \wtM_1\otimes_{\whW(R)}W(k).
\]

\end{para}

\begin{para}\label{par:torsionfree}
There are functorial $S$-homomorphisms
\[
\wtM_1\to \phi^*M=S\otimes_{\phi, S}M,
\]
defined as follows: Suppose $(L, T)$ gives a normal decomposition of $(M, M_1)$. Let $(e_1,\ldots, e_d)$ be a basis of $L$, 
$(e_{d+1},\ldots, e_n)$ a basis of $T$, and let $(\ti e_1,\ldots , \ti e_n)$ be the corresponding basis of $\ti M_1$.
The homomorphism sends $\ti e_i$ to $\phi^*e_i$  for $1\leq i\leq d$ and $\ti e_i$ to $p\phi^*e_i$ for $d+1\leq i\leq n$.
We can see that this does not depend on the choice of bases.

Suppose now that the ring $S$ of the frame $\mathscr F$ is $p$-torsion free. Then the above homomorphism
$
\wtM_1\to \phi^*M 
$
is injective. Using this, we identify
\[
\wtM_1={\rm Im}(\phi^*(i): \phi^*M_1\to \phi^*M)=\phi^*L\oplus p\phi^*T\subset \phi^*L\oplus \phi^*T=\phi^*M
\]
where $i:M_1\rightarrow M$ is the inclusion. Hence
 \[
\wtM_1[1/p]=(\phi^*M)[1/p].
\]
In particular, this applies to $S=\whW(R)$, when $R$ is a Noetherian complete local ring with residue field $k$ and $R$ is $p$-torsion free.
\end{para}

\begin{para}\label{par:relativeconnection}
Suppose $\alpha: {\mathscr F}=(S,I,\phi,\phi_1)\to {\mathscr F}'=(S, I', \phi, \phi'_1)$ is a morphism of frames
underlying the identity ${\rm id}: S\to S$, so $I\subset I'$ and $(\phi'_1)_{|I}=\phi_1$. It induces 
$P_{{\mathscr F}}\to P_{{\mathscr F}'}$ given as $(M, M_1)\mapsto (M, M'_1)$.

\begin{lemma}\label{masterlemma}
The   functor
\[
\tau_{{\mathscr F}}: P_{{\mathscr F}}\to {\rm Mod}^{\rm ff}_{S}; \quad (M, M_1)\mapsto \wtM_1,
\]
is naturally equivalent to the composition
\[
P_{{\mathscr F}}\to P_{ {\mathscr F}'}\xrightarrow{\tau_{{\mathscr F}'}}{\rm Mod}^{\rm ff}_{S}.
\]
\end{lemma}
\begin{proof}
This quickly follows from the definition of the functors $P_{{\mathscr F}}\to P_{ {\mathscr F}'}$, $\tau_{{\mathscr F}}$ and $\tau_{{\mathscr F}'}$ via (\ref{phiphi1Matrix}), by
using $(\phi'_1)_{|I}=\phi_1$. 
\end{proof}

Note that this Lemma applies to the natural frame morphism 
$
 {\mathscr D}_B\to  {\mathscr D}_{B/k}
$
when $B$ is an Artin local ring with $\fkm_B^2=0$ and residue field $B/\fkm_B=k$.

\end{para}

\begin{para}

For the following statement, we let ${\mathscr F}=(S,I,\phi,\phi_1)$ be a frame together with frame morphisms
\[
{\mathscr D}_k \to  {\mathscr F}\xrightarrow{\iota} {\mathscr D}_k
\]
with composition  the identity of $ {\mathscr D}_k$. We denote by $I'$ the kernel of the composition $S\xrightarrow{\iota} W(k)\to k$.

If $(M, M_1)$ is a pair over the frame ${\mathscr F}$ we set
\[
(M_0, M_{0,1})=\iota_*(M, M_1)
\]
for the pair over ${\mathscr D}_k$ obtained by base change.

\begin{lemma}\label{masterlemma2}
Suppose that $I\subset I'$ and that $\phi_1: I\to S$ 
extends to $\phi'_1: I'\to S$, so that ${\mathscr F}'=(S, I', \phi, \phi'_1)$ is a frame.
Let  $(M, M_1)$ be a pair over $ {\mathscr F}$, together with an isomorphism $\psi: M_0\otimes_{W(k)}S\xrightarrow{\sim} M$  whose base change by $S\xrightarrow{\iota} W(k)$ is the identity. Then, there is an isomorphism
\begin{equation}\label{masterconnection}
c_S:  \wtM_{0,1}\otimes_{W(k)}S \xrightarrow{\sim } \wtM_1 
\end{equation}
of finite free $S$-modules which is functorial in pairs $(M, M_1)$ equipped with an isomorphism $\psi$ and is compatible with base change.
\end{lemma}

\begin{proof} Note that the identity of $S$ induces a frame morphism $\alpha: {\mathscr F}\to  {\mathscr F}'$ to which we can apply Lemma \ref{masterlemma}.
We will denote by $(M, M_1)' $ the image of $(M, M_1)$ in $P_{{\mathscr F}'}$ given by the  functor $\alpha_*: P_{\mathscr F}\to P_{{\mathscr F}'}$. Note that $P_{{\mathscr F}'}$    fully embedds in the category of pairs $(N, N_{0,1})$, where $N$ is a finite free $S$-module     and $pN_0\subset N_{0,1}\subset N_0=N\otimes_{S}W(k)$ is a $W(k)$-submodule.
The choice of $\psi$ determines an isomorphism 
\[
\underline \psi: ((M_0, M_{0,1})\otimes_{W(k)}{S})'\xrightarrow{ \sim }(M, M_1)'%  =(M, M_{0,1})
\]
in $P_{ {\mathscr F}'}$. Here, the left hand side   is the image of
$(M_0, M_{0,1}) $ given by applying the base change
$P_{ {\mathscr D}_{k}}\to P_{ {\mathscr F}}$ followed by $\alpha_*: P_{ {\mathscr F} }\to P_{ {\mathscr F}'}$. We now consider the isomorphism $\tau_{ {\mathscr F}'}(\underline\psi)$. By applying Lemma \ref{masterlemma} for $\alpha_*: P_{\mathscr F}\to P_{{\mathscr F}'}$ and combining with base change (\ref{bctilde}), we see that $\tau_{ {\mathscr F}'}(\underline\psi)$ produces
 the desired isomorphism 
\[
c_{S}: \wtM_{0,1}\otimes_{W(k)}S \xrightarrow{\sim } \wtM_1.
\]
 This is functorial in pairs $(M, M_1)$ equipped with   $\psi$ and is compatible with base change.
 \end{proof}
 
 \begin{Remark}\label{applymaster}
 {\rm Lemma \ref{masterlemma2} can be applied to the frame ${\mathscr F}=
 {\mathscr D}_B$ where $B$ is an Artin local $k$-algebra with $\fkm_B^2=0$ and residue field $B/\fkm_B=k$ with ${\mathscr D}_B\to  {\mathscr D}_k$, ${\mathscr D}_k\to  {\mathscr D}_B$   the natural morphisms. Then $\alpha: {\mathscr F}\to {\mathscr F}'$ is the natural 
 frame morphism $ {\mathscr D}_B \to  {\mathscr D}_{B/k}$.  }
 \end{Remark}
 \end{para}
  
\begin{para}\label{par:mainlemma}
We now return to the set-up in the beginning of \S \ref{1.1}. 
Set $\fraka=\frakm^2+pR\subset R$ and consider the quotient $R/\fa$. We also have the pair
$(M_{\whW(R/\fa) }, M_{\whW(R/\fa), 1 })$ over $\whW(R/\fa)$ obtained by base change from $(M, M_1)$. We  fix an isomorphism 
\[
M=M_0\otimes_{W(k)}\whW(R)
\]
 reducing to the identity modulo $\fkm_R$.

 \begin{lemma}\label{319}
(cf. \cite[Lem. 3.1.9]{KP}.) 
There is a canonical commutative diagram
 \begin{equation}\begin{aligned} 
 \xymatrix{
\wtM_1\otimes_{\whW(R)} \whW(R/\fa ) \ar[r] & \varphi^*(M_{\whW(R/\fa) })\ar@{=}[d]\\
\wtM_{0,1}\otimes_{W(k)}\whW(R/\fa ) \ar[r]\ar[u]^c & \varphi^*(M_0)\otimes_{W(k)}\whW(R/\fa ).
}\end{aligned}
\end{equation} 
In this, the left vertical map is an isomorphism and the horizontal maps are induced 
by base changing $\wtM_1\to \phi^*M$ and $\wtM_{0,1}\to \phi^*M_0$.
\end{lemma}

We will call
\[
c: \wtM_{0,1}\otimes_{W(k)}\whW(R/\fa)\xrightarrow{\sim} \wtM_1\otimes_{\whW(R)}\whW(R/\fa)
\] 
the ``connection isomorphism". 

\begin{Remark}\label{remark:Hoff}
{\rm a) As was pointed out to the authors by M. Hoff, the isomorphism which is given by the construction of \cite[Lem. 3.1.9]{KP} is not canonical and hence not ``correct". (The construction there is given 
 using a normal decomposition $M=L\oplus T$, but the resulting map depends on that choice.) In particular, 
\cite[Lem. 3.1.12]{KP} does not hold when $c$ is defined as in the proof of \cite[Lem. 3.1.9]{KP}.  Note that the diagram does not determine  $c$ since the horizontal maps are not always injective, and this can occur even if $\whW(R)$ is $p$-torsion free.

 b) Our corrected construction of $c$ follows \cite{Hoff}; the main idea already appears in Zink's work, see \cite[Theorem 3]{ZinkCFT}.} 
\end{Remark}

\begin{proof}
We apply Lemma \ref{masterlemma2} to the frame ${\mathscr D}_{ R/\fa}$, cf. Remark \ref{applymaster}, and define  $c$ as  the composition of the 
  isomorphism (\ref{masterconnection})
  \[
c_{\whW(R/\fa)}: \wtM_{0,1}\otimes_{W(k)}\whW(R/\fa) \xrightarrow{\sim } \wtM_{\whW(R/\fa), 1},
\]
 with the base change isomorphism (\ref{bctilde})
\[
 \wtM_{\whW(R/\fa), 1} \xrightarrow{\sim} \wtM_1\otimes_{\whW(R)} \whW(R/\fa).
\]  
 Note that the essential ingredient for this construction is the relative Dieudonn\'e frame which uses Zink's logarithmic coordinates, see also \cite[Lemma 1.21]{Hoff}.

The commutativity of the diagram follows from the construction of $c$ and \S\ref{par:torsionfree}. It is also instructive to deduce it from a useful explicit description of $c$ as follows. Fix a normal decomposition $M=L\oplus T$, $M_1=L\oplus \hat I_RT$, and a basis $\B=(e_1,\ldots , e_n)$ adapted to this decomposition as above. Reduce $\B$ modulo $\whW(\frakm)$ to obtain a basis $\B_0$ of $M_0$. In turn, this gives a new basis $\B_0\otimes 1$ of $M$ by base changing by $W(k)\to \whW(R)$. Denote by 
 \[
\begin{pmatrix}
X & Y\\
Z & U
\end{pmatrix}
\]
 the change of basis matrix between $\B$ and $\B_0\otimes 1$. Since $\B$ reduces to $\B_0$, we have 
  \[
\begin{pmatrix}
X & Y\\
Z & U
\end{pmatrix}\equiv I_{n}\, {\rm mod}\, \whW(\frakm),
\]
with $I_n$ the $n\times n$ identity matrix.
 In particular, $Z$ is a matrix with coefficients in $\whW(\frakm)$. 
 Set $B=R/\fa$ and observe that $$\whW(\fkm_B)=I_{\whW(\fkm_B)}\oplus [\fkm_B]\subset \hat I_{B/k}=\hat I_B\oplus [\fkm_B].$$ The isomorphism $c$ is now given, in terms of the bases $\ti \B\otimes_{\whW(R)}\whW(B)$ and $\ti \B_0\otimes_{W(k)}\whW(B)$ of 
 $\wtM_1\otimes_{\whW(R)}\whW(B)$ and $\wtM_{0, 1}\otimes_{W(k)}\whW(B)$, by the matrix
 \begin{equation}\label{uni}
 \begin{pmatrix}
I_{d} & 0\\
V^{-1}(\overline Z) & I_{n-d}
\end{pmatrix},
 \end{equation}
with entries in  $\whW(\frakm_B)$. Here we write   $\overline Z$ for the reduction of $Z$ modulo $\fa$. The commutativity of the diagram now  follows 
by using the description of $\wtM_1\to \phi^*M$ and $\wtM_{0,1}\to \phi^*M_0$ in \S\ref{par:torsionfree} and combining it with the above, together with $\phi(\whW(\frakm_B))=0$.
\end{proof}
 \end{para}

\begin{para}\label{par:Psi}
Suppose now we have a Dieudonn\'e display $(M, M_1, \Phi, \Phi_1)$ over the $p$-torsion free $\whW(R)$ with corresponding $(M, \wtM_1, \Psi)$ as in \cite[Lem. 3.1.5]{KP}. Denote by $(M_0, M_{0,1}, \Phi_0, \Phi_{0,1})=(\DD,\DD_1, \Phi_0,\Phi_{0,1})$ the Dieudonn\'e display over $W(k)$ obtained  by base change by $R\to R/\frakm=k$ as in \cite[3.1.6]{KP}. This has corresponding $(M_0, \wtM_{0,1}, \Psi_0)$. 

As in \cite[3.1.1]{KP}, we say ``$\Psi$ is constant modulo $\fa$'' if the composite 
\[
\wtM_{0,1}\otimes_{W(k)}\whW(R/\fa)\simeq \wtM_1\otimes_{\whW(R)}\whW(R/\fa)\xrightarrow{\Psi}  M\otimes_{\whW(R)}\whW(R/\fa)=M_0\otimes_{W(k)}\whW(R/\fa)
\]
 is $\Psi_0\otimes 1$, where the first map  in the composition
is the isomorphism $c$ of Lemma \ref{319}. Then, with this definition, \cite[Lem. 3.1.12]{KP} holds, see also \cite[Thm. 1.28]{Hoff}.
\end{para}
 
\subsection{Very good embeddings: definition}\label{ss:VG definition}

 Suppose that $\Gg\subset \GL(\La)$ is a closed immersion of group schemes over the $p$-adic discrete valuation ring $\O$, where $\La$ is a finite free $\O$-module. 
Set $\La^\otimes:=\oplus_{m, n\geq 0}\La^{\otimes m}\otimes_\O (\La^\vee)^{\otimes n}$ for  the total tensor algebra of $\La$, where $\La^\vee={\rm Hom}_{\O}(\La,\O)$.
As usual, we say that {\sl $\Gg$ is cut out in $\GL(\La)$ by a set of tensors $(s_a)\subset \La^{\otimes}$}, if for all $\O$-algebras $R$, we have
\[
\Gg(R)=\{g\in \GL(\La\otimes_{\Z_p} R)\ |\ g\cdot (s_a\otimes 1)=s_a\otimes 1,\forall a\}.
\]
Here $s_a\otimes 1$ is the image of $s_a$ under $\La^{\otimes}\to \La^{\otimes}\otimes_{\O}R=(\La\otimes_{\Z_p}R)^\otimes$. 
%If $\Gg$ is $\O$-flat, it is enough to require this property for all $\O$-flat $R$.

\begin{para} We now consider  a local model triple $(G,\{\mu\},\Gg)$ and assume that  $\rho: (\Gg,\mu)\hook(\GL(\La),\mu_d)$ is a good integral Hodge embedding. We suppose that $\O$ is unramified over $\Z_p$.

We first assume $\O=\Z_p$.  
Suppose   $x\in \BMloc_{\Gg,\mu}(k)$, where we now take $k=\bar k_E=\bar \mF_p$. Following \cite{KP}, we will denote by $R_{G,x}$, or simply $R_G$, the completion of the local ring of $ \BMloc_{\Gg,\mu}$ at $x$ and by $R_E$ the completion of the local ring of the Grassmannian ${\rm Gr}(d,\Lambda)_{\O_E}$ at the image of the point $x$ under 
the embedding $\BMloc_{\Gg,\mu}\hook {\rm Gr}(d,\La)_{\O_E}$. Then $R_G$ is a quotient of $R_E$ and $R_E$ is non-canonically isomorphic to a power series ring over the integers $\O_EW(k)$ of the completion of the  unramified extension of $E$ with residue field $k$.

Set $M=\La\otimes_{\Z_p}\whW(R_E) $ and denote by $\hat I_{R_E}M\subset M_1\subset M$ the unique $\whW(R_E)$-submodule corresponding to the universal $R_E$-valued point of the Grassmannian. Then $(M, M_1)$ is a pair over $\whW(R_E)$ as considered in the previous paragraph. Usually, we will denote for simplicity also by $(M, M_1)$ the pair of $\whW(R_G)$-modules which is obtained by restricting along $R_E\to R_G$. (If  noting the specific pair $(\Gg,\mu)$ is important, we will denote this by $(M^G,M_1^G)$.) To this pair, we  associate the finite free $\whW(R_G)$-module $\wtM_1$ with 
\[
\wtM_1[1/p]=(\phi^*M)[1/p].
\]
 
Choose  $(s_a)\subset \La^\otimes$ that cut out $\Gg$, cf. \cite[Prop. 1.3.2]{KP}, \cite[3.2.1]{PCan} and set   \[
\ti s_a:=s_a\otimes 1=\phi^*(s_a\otimes 1)\in \La^{\otimes}\otimes_{\Z_p} \whW(R_G)=\phi^*M^{\otimes}\subset (\phi^*M)^{\otimes}[1/p]=\wtM_1^{\otimes}[1/p]
.\]

 Observe that the tensors
\[
\ti s_{a,0}=s_a\otimes 1\in \La^{\otimes}\otimes_{\Z_p} W(k)[1/p]=\wtM_{0,1}^{\otimes}[1/p]
\]
lie in $\wtM_{0,1}^{\otimes}\subset \wtM_{0,1}^{\otimes}[1/p]$: Indeed,
by  
(\ref{diagram:Witt}),
\[
\BMloc_{\Gg,\mu}(k)\subset \frac{G(W(k)[1/p])}{\Gg(W(k))}\subset \frac{\GL(\La\otimes_{\Z_p} W(k)[1/p])}{\GL(\La\otimes_{\Z_p} W(k))}.
\]
This implies that we have $\wtM_{0,1}=g\cdot (\La\otimes_{\Z_p} W(k))$ for some $g\in G(W(k)[1/p])$. Since $g$ preserves the tensors $s_a\otimes 1$, we obtain $\ti s_{a,0}\in \wtM_{0,1}^{\otimes}$ (cf. the proof of Lemma \ref{perfect} below).

By the argument of \cite[Cor. 3.2.11]{KP} (now using also the main result\footnote{under our standard assumptions, a simpler proof of this is given in \cite[Appendix]{PRShtuka}.} of \cite{An}), we then also have  $\ti s_a\in\wtM_1^{\otimes}$ and the scheme
\[
\calT=\underline{\rm Isom}_{(\ti s_a), (s_a)}(\wtM_1, \La\otimes_{\Z_p}\whW(R_G))
\]
of $\whW(R_G)$-linear isomorphisms that preserve the tensors is a trivial $\Gg$-torsor over $\whW(R_G)$.
The scheme $\calT$ is independent of the choice of the set of tensors $(s_a)\subset \La^\otimes$ that cut out $\Gg$.

Set $\fa_G=\frakm_G^2+\pi_ER_G\subset R_G$. Then, by Lemma \ref{319}, there is a canonical  isomorphism
\begin{equation}\label{canonicaliso}
c:  \wtM_{0,1}\otimes_{W(k)}\whW(R_G/\fa_G)\xrightarrow{\sim} \wtM_1\otimes_{\whW(R_G)}\whW(R_G/\fa_G).
\end{equation}
We say that ``the  tensors $\ti s_a$ are horizontal at $x$'' if they are preserved by $c$, i.e. if
\[
c(\ti s_{a,0}\otimes 1)=\ti s_a\otimes 1.
\]
Note here that $\wtM_{0,1}={\rm Im}(\phi^*M_{0,1}\to \phi^*M_0)$. 
Suppose this is the case  for a finite set of tensors $(s_a)\subset \La^{\otimes}$ cutting out $\Gg\hook\GL(\La)$. Then the isomorphism $c$ uniquely descends to an isomorphism of $\Gg$-torsors
\[
c^\Gg: \calT_0\otimes_{W(k)}\whW(R_G/\fa_G)\xrightarrow{\sim} \calT\otimes_{\whW(R_G)}\whW(R_G/\fa_G).
\]

 \begin{lemma}\label{extratensor} 
Suppose $(\Gg,\mu)\hook (\GL(\La),\mu)$ and that $(s_a)\subset \La^{\otimes  }$ cuts out $\Gg$, such that $(\ti s_a)$ are horizontal. If a tensor $t\in \Lambda^{\otimes}$ is fixed by $\Gg$, then $\ti t$ is horizontal. 
\end{lemma}

\begin{proof}
This follows by the discussion above, since $\calT$ is independent of the choice of the set $(s_a)$ that cuts out $\Gg$.
\end{proof}
\end{para}

\begin{para}\label{par:VG}
The following notion plays a central role.

\begin{Definition}\label{def:vgood}
Let $\iota:(\Gg,\mu)\hook (\GL(\La),\mu_d)$ be a good integral Hodge embedding. We say that $\iota$ is \emph{very good at $x$} (or just \emph{very good}
 if the point $x$ is understood), if there are tensors $(s_a)\subset \La^{\otimes }$ cutting out $\Gg$ in $\GL(\La)$ such that $(\ti s_a)$ are horizontal at $x$. This is equivalent to asking that the canonical isomorphism $c$ descends to an isomorphism of $\Gg$-torsors $c^\Gg$, as above.
 \end{Definition}
\end{para}

More generally, suppose that $\GL(\calL)$ is the parahoric group scheme 
determined by a periodic lattice chain $\calL$ in $V$. We will say that a good integral Hodge embedding  $(\Gg,\mu)\hook (\GL(\calL),\mu_{d})$  is  very good if the good integral Hodge embedding  $$(\Gg,\mu)\hook (\GL({\rm tot}(\calL)),\mu_{rd})$$ given by composing with the diagonal, is very good in the sense
of Def. \ref{def:vgood} above. Here ${\rm tot}(\calL)$ is the direct sum of the lattices in a determining segment of $\calL$, cf. \S \ref{par:variantGood}.
By Lemma \ref{diagonal} (b) below, this notion does not depend on the choice of determining segment.

\begin{para}
The above definitions extend to the case that $\O\simeq W({\mathbb F}_q)$ is finite unramified over $\Z_p$. In this case, the arguments of \cite[\S 3.2]{KP} show that we have $\phi^*(s_a\otimes 1)\in \wtM_1^\otimes$ and we say that a good integral Hodge embedding is very good when $\phi^*(s_a\otimes 1)$ are horizontal.   
\end{para}

\begin{lemma}\label{unramified} Assume $(\Gg,\mu)\hook (\GL(\La),\mu_d)$ is a good  integral Hodge embedding over $\Z_p$. Let $\O/\Z_p$ be a finite unramified extension. Then there is a natural isomorphism $\BMloc_{\Gg\otimes_{\Z_p}\O,\mu\otimes_{\Z_p}\O}=\BMloc_{\Gg,\mu}\otimes_{\O_E} \O_E\O$ and
$(\Gg,\mu)\hook   (\GL(\La),\mu_d)$ is very good at $x\in \BMloc_{\Gg, \mu}(k)$ if and only if the
 base change $(\Gg\otimes_{\Z_p}{\O},\mu\otimes_{\Z_p}\O)\hook   (\GL(\La\otimes_{\Z_p}{\O}),\mu_d\otimes_{\Z_p}\O)$ is very good at 
$x\in \BMloc_{\Gg\otimes_{\Z_p}\O,\mu\otimes_{\Z_p}\O}(k)$.
\end{lemma}

\begin{proof}
In the above, $\O_E\O$ is the integers of the join of $E$ with ${\rm Fr}(\O)$.
The isomorphism $\BMloc_{\Gg\otimes_{\Z_p}\O,\mu\otimes_{\Z_p}\O}=\BMloc_{\Gg,\mu}\otimes_{\O_E} \O_E\O$ is standard
 and follows from the construction of the local models, cf. \cite[Prop. 2.14]{HPR}, or by their characterization via $v$-sheaves in \cite{Schber}. The rest of the statement follows from the definitions. 
\end{proof}

\subsection{Very good embeddings: properties} We now give various results regarding very good integral Hodge embeddings.

\begin{para} We start with the following.

\begin{lemma}\label{endo}
Assume $(\Gg,\mu)\hook (\GL(\La),\mu_d)$ is a good  integral Hodge embedding. Let $t:\La\to \La$ be an endomorphism which is fixed by $\Gg\hook \GL(\La)$. Then $\ti t: \wtM_1\to \wtM_1$   is horizontal at $x$.
\end{lemma}

\begin{proof}
Since $t$ is fixed by $\Gg$, we have $t(\F)\subset \F$ for the universal point of ${\rm Gr}(d, \La)$ which corresponds to $\BMloc_{\Gg,\mu}\hook {\rm Gr}(d, \La)_{\O_E}$. Indeed, it is enough to check this on the generic fiber and this follows from  \cite[3.2.5]{KP}. Hence, $$t\otimes 1: M=\La\otimes\whW(R_G)\to M=\La\otimes\whW(R_G)$$ preserves the submodule $M_1.$ Then we see that $\ti t: \wtM_1\to \wtM_1$ is preserved by $c$, i.e. is horizontal. Indeed, this  follows from the functoriality of the isomorphism $c$ for homomorphisms of pairs $(M, M_1)$ which respect an isomorphism $M=M_0\otimes_{W(k)}\whW(R_G)$, see Lemma \ref{319} and its proof.
 \end{proof}

\begin{cor}\label{EL}
Assume $(\Gg,\mu)\hook (\GL(\La),\mu_d)$ is a good integral Hodge embedding. If $\Gg\hook\GL(\La)$ is cut out by a set of endomorphisms $t_a:\La\to \La$, then $(\Gg,\mu)\hook (\GL(\La),\mu_d)$ is very good at all $x$. \qed
\end{cor}

\begin{cor}\label{endoplus}  Suppose we have $(\Gg,\mu)\hook (\Gg',\mu')\hook (\GL(\La),\mu_d)$, where both $(\Gg',\mu')\hook (\GL(\La),\mu_d)$ and the composition are good integral Hodge embeddings.
Suppose that there are endomorphisms $t_a:  \La\to \La$
such that 
\[
\Gg=\Gg'\cap \{g\in \GL(\La)\ |\ g\cdot t_a=t_a\cdot g, \forall a\}
\]
as closed subschemes of $\GL(\La)$. Consider $x\in\BMloc_{\Gg,\mu}(k)$. 

Suppose that 
$(\Gg',\mu')\hook (\GL(\La),\mu_d)$ is very good at the image $x'\in \BMloc_{\Gg',\mu'}(k)$ of $x$. Then $(\Gg,\mu)\hook   (\GL(\La),\mu_d)$ is very good at $x$.
\end{cor}
\begin{proof}
 Let $s_{a'}\in \Lambda^\otimes$  a set of tensors which cut out $\calG'$. The module $\wtM^G_1$ over $\whW(R_{G,x})$ which corresponds to $(\Gg,\mu)\hook   (\GL(\La),\mu_d)$ is the base change by the surjection $R_{G',x'}\to R_{G,x}$ of 
the module $\wtM^{G'}_1$ over $\whW(R_{G',x'})$ which corresponds to $(\Gg',\mu')\hook   (\GL(\La),\mu_d)$. The same is true for the corresponding connection isomorphisms, and hence $\ti s_{a'}\in\wtM^{G,\otimes}_1$ is horizontal at $x$.
Since $\calG$ is cut out by the union of the tensors $s_{a'}$ and $t_a$, the result  follows from Lemma \ref{endo}.
\end{proof}
\begin{Remark}
{\rm In the applications, we will apply the above corollary to the case  $\calG'=\Res_{\calO_K/\calO}\calH$ where $K/F$ is a field over which $G$ splits and $\calH$ is a hyperspecial subgroup of $G_K$. This will allow us to produce very good Hodge embeddings for general parahorics from those coming from Weil restrictions of hyperspecial subgroups. }
\end{Remark}
\end{para}
\begin{para} The next two lemmas show that very good embeddings behave well with respect to taking direct sums, and projections onto direct summands.
\begin{lemma}\label{diagonal}
Consider a good integral Hodge embedding $(\Gg,\mu)\hook (\GL(\La_1),\mu_{d_1})$ and a map of local model pairs  $(\Gg,\mu)\to (\GL(\La_2),\mu_{d_2})$. Set $\La=\La_1\oplus \La_2$, $\mu_d=\mu_{d_1}\times \mu_{d_2}$, and consider the diagonal embedding
\[
(\Gg,\mu)\hook (\GL(\La),\mu_d).
\]
\begin{altenumerate}
\item[a)] If the diagonal embedding is very good at $x$, then $(\Gg,\mu)\hook (\GL(\La_1),\mu_{d_1})$ is very good at $x$.

\item[b)] Suppose that there is an isomorphism $h: \La_1\xrightarrow{\sim} \La_2$ which intertwines the embeddings $(\Gg,\mu)\hook (\GL(\La_1), \mu_{d_1})$ and $(\Gg,\mu)\to (\GL(\La_2),\mu_{d_2})$. Suppose that $(\Gg,\mu)\hook (\GL(\La_1),\mu_{d_1})$ is very good at $x$. Then,  both $(\Gg,\mu)\hook (\GL(\La_2),\mu_{d_2})$  and the diagonal above are  very good at $x$.
\end{altenumerate}
\end{lemma}

\begin{proof}
The diagonal immersion gives
\[
\BMloc_{\Gg,\mu}\hook {\rm Gr}(d_1, \La_1)_{\O_E}\times_{\O_E}{\rm Gr}(d_2, \La_2)_{\O_E}\subset {\rm Gr}(d, \La)_{\O_E}
\]
with $d=d_1+d_2$. Hence, the module $\wtM_1$ over $\whW(R_G)$ obtained from the diagonal immersion  is a direct sum 
$\wtM_{1,1}\oplus \wtM_{2,1}$ and we can see that we have $c=c_1\oplus c_2$, with obvious notation. 
Suppose that $\Gg\hook\GL(\La_1)$ is cut out by $(s_{a,1})\subset \La_1^\otimes$. Since $\La=\La_1\oplus\La_2$, we have 
$\La_1^{\otimes}\subset \La^\otimes$ and $\Gg$ fixes $s_{a,1}$ considered as tensors in $\La^\otimes$. The tensors $\widetilde{s}_{a,1}\in \wtM_{1,1}^\otimes\subset \wtM_{1}^\otimes$ are then horizontal at $x$ by Lemma \ref{extratensor}. Then   (a) follows.

Now let us show (b). It is easy to see that $(\Gg,\mu)\hook (\GL(\La_2),\mu_{d_2})$ is very good at $x$.  To discuss the diagonal, suppose that $(\Gg,\mu)\hook (\GL(\La_2),\mu_{d_2})$ is cut out by 
$(s_{b,2})$. Then $\Gg\subset \GL(\La)$ is cut out by the following tensors: the union of the sets $(s_{a,1} )$,  $(s_{b,2})$,   the isomorphism 
\[
\La_1\oplus \La_2\xrightarrow{h\times h^{-1}}\La_2\oplus \La_1\xrightarrow{\iota} \La_1\oplus \La_2
\]
(where $\iota$ is the obvious ``reflection'' map), and the
 tensors that cut out $\GL(\La_1)\times \GL(\La_2)$ in $\GL(\La)$. Since $\wtM_1=\wtM_{1,1}\oplus \wtM_{2,1}$ and the isomorphism $c$ is functorial and decomposes over the direct sum as above, this last set of tensors is  horizontal (we can also apply Lemma \ref{endo} to the corresponding projections). The result then follows as before.
\end{proof}

 \begin{Remark}
{\rm Suppose that $(\Gg,\mu)\hook (\GL(\La),\mu_{d})$ is a good integral Hodge embedding. Part (b) amounts to the statement that if this embedding is good at $x$, then the diagonal embedding $(\Gg,\mu)\hook (\GL(\La^{\oplus 2}),\mu_{2d})$ is also very good at $x$. On the other hand, suppose that $(\Gg,\mu)\hook (\GL(\La_i),\mu_{d_i})$ are two good integral Hodge embeddings that are very good at $x$, but are, in general, unrelated. It does not appear easy to show that the diagonal embedding  $(\Gg,\mu)\hook (\GL(\La_1\oplus\La_2),\mu_{d_1+d_2})$ is also very good at $x$.}
\end{Remark}

\begin{lemma}\label{lem: product vg}
			Let $\rho_i:(\Gg_i,\mu_i)\rightarrow (\GL(\La_i),\mu_{d_i})$, $i=1,\dotsc,r$, be very good integral Hodge embeddings. 	Set $\La=\oplus_{i=1}^r\La_i$, $\Gg=\prod_{i=1}^r\Gg_i$, $\mu=\prod_{i=1}^r\mu_i$, $d=\sum_{i=1}^r d_i$, and consider
	\[
	\rho: (\Gg, \mu)\to (\GL(\La), \mu_d)
	\]
	given as the composition of the product of $\rho_i$ with the standard group scheme embedding $\prod_{i=1}^r\GL(\La_i)\to \GL(\La)$. Then $\rho$ is also a very good integral  Hodge embedding. 
	\end{lemma}

	\begin{proof}
We fix tensors $(s_{\alpha,i})\in \Lambda^\otimes_i$ which cut out the group $\calG_i$. Via the inclusion $\Lambda^\otimes_i\subset \Lambda^\otimes$, we may consider the $s_{\alpha,i}$ as tensors in $\Lambda^\otimes$. Then $\calG$ is cut out by the tensors $(s_{\alpha,i}), i=1,\dots,r$, and the tensors corresponding to the endomorphisms $p_i:\Lambda\rightarrow \Lambda$ defined by projection to the direct summand $\La_i$. By our assumption and Lemma \ref{endo}, these are horizontal at all points of $\BMloc_{\calG,\mu}=\prod_{i=1}^r\BMloc_{\calG_i,\mu_i}$, hence $(\calG,\mu)\rightarrow (\GL(\Lambda),\mu_{d})$ is very good.
 \end{proof}

\end{para}

\begin{para} The following proposition is a key result, which combined with the results in \S 4 allows us to produce very good Hodge embeddings in many  cases when the parahoric $\calG$ arises as the Weil restriction of a split reductive group.

\begin{prop}\label{span}
Assume $(\Gg,\mu)\hook (\GL(\La),\mu_d)$ is a good integral Hodge embedding. If the tangent space of the special fiber $\BMloc_{\Gg,\mu}\otimes_{\O_E} k$ at $x$ is  spanned by smooth formal curves (see Definition \ref{def:span}), then $(\Gg,\mu)\hook (\GL(\La),\mu_d)$ is very good at $x$.
\end{prop}
\newcommand{\Su}{k\lps u\rps}
\newcommand{\Wu}{W \lps u  \rps}

\begin{proof}
  We thank the referee for suggesting the following presentation of the proof. The main point of the proof is showing the following statement: 
\smallskip
 
Let $v\in \BMloc_{\Gg,\mu}(k[ \ep ])$ be a tangent vector at $x$ and let 
$(M_{\whW(k[\epsilon])}, M_{\whW(k[ \ep ]), 1})$ be the pair obtained by base changing $(M, M_1)$ by the corresponding map $\whW(R_G)\to \whW(k[\ep])$. If $v$
  can be lifted to a $k\lps u\rps$-valued point of $\BMloc_{\Gg,\mu}$, then the isomorphism $c_{\whW(k[ \ep ])}$ for the pair $(M_{\whW(k[ \ep ])}, M_{\whW(k[\ep ]), 1} )$ preserves the tensors in $ M_{\whW(k[\ep ]), 1}^\otimes$ 
  obtained from $\ti s_a\in M_{\whW(R_G), 1}^\otimes$ by base change. (Here,  we write  for simplicity $k[\ep]=k\oplus k \ep$, $\ep^2=0$, for the ring of dual numbers.) 
 \smallskip

Before we start the proof, we will discuss certain frames and frame morphisms that we will use. 

We set $\gS = W(k) \lps u  \rps=W\lps u  \rps$.
  We equip $\gS$ with the standard Frobenius lift $\varphi$ 
given by $\phi(u )= u ^p$. The Frobenius lift $\phi$ gives a homomorphism $\lambda: \gS=W\lps u\rps \to W(\gS)$ with $\lambda(u)=[u]$. We can compose $\lambda$ with $ W(\gS)\to  W(\Su)$ given by the reduction $\gS\to \Su=\gS/p\gS$ to obtain
 \[
\bar\lambda: \gS=W\lps u\rps\to \wh W(\Su)\subset W(\Su).
\]
(The image indeed lands in Zink's ring $\wh W(\Su)$).  This induces a frame morphism
\[
\bar\lambda: (\gS, (p), \phi, p^{-1}\phi)\to {\mathscr D}_{\Su}=(\whW(k\lps u\rps),\hat I_{\Su},\phi, V^{-1}).
\]

Set $k[\ep]=\Su/(u^2)$, $W[\ep]=\Wu/(u^2)$, with $\ep\mapsto u\,{\rm mod}\, (u^2)$. 
We also have the frames $(W[\ep], (p), \phi, p^{-1}\phi)$ and $(W[\ep], (p,\ep), \phi, p^{-1}\phi )$; here $\phi(\ep)=0$. For clarity, we note that in the latter frame $$\phi_1=p^{-1}\phi :(p,\ep)=pW\oplus W\ep\to W[\ep]$$ is the sum of $p^{-1}\phi$ on the first component with the zero map on the second. In what follows, for simplicity, we will omit the notation of the maps $\phi$ and $\phi_1$ and write frames simply as pairs.  

Composing $\bar\lambda$ with   $\whW(\Su)\to \widehat W(\Su/(u^2))$ gives
\[
\bar\lambda_\ep: W[\ep]=\Wu/(u^2)\to \widehat W(\Su/(u^2))=\widehat W(k[\ep]).
\]
  This induces frame morphisms
\[
\bar\lambda_\ep: (W[\ep], (p))\to  {\mathscr D}_{k[\ep]}=(\whW(k[\ep]), \hat I_{k[\ep]}),
\]
\[
 \bar\lambda_\ep: (W[\ep], (p,\ep))\to   {\mathscr D}_{k[\ep]/k}=(\whW(k[\ep]), \hat I_{k[\ep]/k}),
\]
where the last frame is the relative Dieudonn\'e-Witt frame for $k[\ep]\to k$.
All together, we have a commutative diagram of  frame morphisms
\begin{equation}\label{framescd}
 \begin{aligned}
 \xymatrix
{(W\lps u\rps, (p))  \ar[r]\ar[d]_{\bar\lambda} &  (W[\ep], (p)) \ar[r]\ar[d]_{\bar\lambda_\ep} & (W[\ep], (p,\ep))  \ar[d]_{\bar\lambda_\ep} 
 \\
(\whW(\Su), \hat I_{\Su})\ar[r] &(\whW(k[\ep]), \hat I_{k[\ep]}) \ar[r]  &  (\whW(k[\ep]), \hat I_{k[\ep]/k}).}
\end{aligned}
\end{equation}

We now proceed with the main argument.

  Suppose  that $a$ is an $\Su$-valued point  of the local model $\BMloc_{\Gg,\mu}$ which lifts the $k$-valued point $x$. This gives an $\Su$-valued point of the  Grassmannian ${\rm Gr}(d, \Lambda)$ and, hence, a direct summand $M_{\Su,1}\subset  M_{\Su}=\Lambda\otimes_{\Z_p}\Su$. 
  
   Set $M_\gS=\Lambda\otimes_{\Z_p}\gS$, $M_{\Su}=\Lambda\otimes_{\Z_p}\Su$.  Denote by $M_{\gS,1}$ the inverse image of $M_{\Su,1}$ under the map
\[
\Lambda\otimes_{\Z_p}\gS\to \Lambda\otimes_{\Z_p}\Su
\]
given by reduction modulo $p$. Then 
\[
pM_\gS=p\Lambda\otimes_{\Z_p} \gS\subset M_{\gS,1}\subset M_\gS=\Lambda\otimes_{\Z_p}\gS
\]
and $M_{\gS,1 }$ is a free $\gS$-module with
\[
M_{\gS,1 }\otimes_\gS \gS[1/p]= \Lambda\otimes_{\Z_p}\gS[1/p].
\]
Then $(M_\gS, M_{\gS,1})$ is a pair over the frame $(\gS, (p))=(W\lps u\rps, (p))$.

\begin{lemma}\label{perfect}
The tensors \[s_a\otimes 1\in \Lambda^{\otimes }\otimes_{\Z_p}\gS[1/p]=M^{\otimes}_{\gS,1}\otimes_\gS \gS[1/p]\] lie in the submodule $M^{\otimes}_{\gS,1}$.  
\end{lemma}

\begin{proof}
Let $\K$ be an algebraic closure of $k\llps u\lrps$ and consider the natural map 
\[
\tau: \gS=W\lps u\rps\xrightarrow{\bar\lambda} W(\Su)\to W(k\llps u\lrps)\to W(\K),
\]
where $\bar\lambda$ is as above. Set $\calO_{\calE}=\varprojlim_n (W/p^nW)\llps u\lrps$
which identifies with the $p$-adic completion of the localization $\gS_{(p)}$ and   is a dvr with residue field $k\llps u\lrps$ and uniformizer $p$.
The map $\tau$ factors as
\[
\tau: \gS=W\lps u\rps\to\calO_\calE \to W(\K)
\]
and it is injective and flat. We have $W\lps u\rps[1/p]\subset W(\K)[1/p]$ via $\tau$ and
\[
W\lps u\rps[1/p]\cap W(\K)=W\lps u\rps,
\]
with the intersection taking place in $W(\K)[1/p]$.

Set 
\[
M_{W(\K), 1}=M_{\gS,1}\otimes_\gS W(\K),
\]
for which
\[
p\Lambda\otimes_{\Z_p} W(\K)\subset M_{W(\K), 1}\subset  \Lambda\otimes_{\Z_p} W(\K).
\]
Using the above, we see that $M^{\otimes}_{\gS,1}[1/p]\cap M_{W(\K), 1}^\otimes =M^{\otimes}_{\gS,1}$, and so it is enough to show that  $s_a\otimes 1\in M_{W(\K),1}^\otimes$. 

Now observe that $M_{W(\K), 1}$ is the $W(\K)$-lattice corresponding to a $\K$-point of the Witt vector affine Grassmannian ${\rm Gr}^W_{\GL(\La)}=L^W\GL(\La)/L^{W+}\GL(\La)$ for $\GL=\GL(\Lambda)$. This $\K$-point comes from 
\[
\Spec(\K)\to \Spec(\Su)\xrightarrow{a} \BMloc_{\Gg,\mu}\otimes_{\O_E}k\hook {\rm Gr}(d, \La)_k.
\]
Using (\ref{diagram:Witt}), we obtain 
\[
M_{W(\K), 1}=g\cdot (\Lambda\otimes_{\Z_p} W(\K))
\]
for some $g\in \Gg(W(\K)[1/p])=G(W(\K)[1/p])$. Since $g$ preserves the tensors $s_a\in \La^\otimes\subset \La^{\otimes}[1/p]$, we obtain that $s_a\otimes 1\in M^\otimes_{W(\K),1}$. Hence, by the above, we also have $s_a\otimes 1\in M_{\gS,1}^{\otimes}$.
\end{proof}
 \smallskip 

We now continue with the proof. We can apply the tilde functor of \S \ref{par:frametilde} to the pair $(M_\gS,M_{\gS,1})$ over the frame $(\gS, (p))$ to obtain $\widetilde M_{\gS,1}$. The ring $\gS$ is $p$-torsion free, and we can easily see that $\phi^*(M_{\gS,1})\to \phi^*(M_{\gS})$ is injective. Hence, by \S\ref{par:torsionfree}, we can identify
\[
\widetilde M_{\gS,1} = \varphi^*(M_{\gS,1})\subset \varphi^*(M_\gS)=\varphi^*(\Lambda\otimes_{\Z_p}\gS)\cong \Lambda\otimes_{\Z_p}\gS.
\]
Let us consider the base change $(M_0, M_{0,1})$ of the pair $(M_\gS, M_{\gS,1})$ by the frame morphism
$(\gS, (p))\to (W(k), (p))$ given by $u\mapsto 0$. We have
\[
M_{\gS, 1}\otimes_{\gS}W=M_{0,1}, \quad\hbox{\rm and}\quad \widetilde M_{\gS, 1}\otimes_\gS W\simeq \wtM_{0,1}.
\]
We can also consider the base change $(M_{W[\ep]}, M_{W[\ep],1})$ of the pair $(M_\gS, M_{\gS,1})$ by the frame morphism
$(\gS, (p))\to (W[\ep], (p))$. 

Lemma \ref{masterlemma2} now applies to the frame 
$(W[\ep], (p))$ with the natural morphisms ${\mathscr D}_k=(W, (p))\to (W[\ep], (p))\to 
{\mathscr D}_k$ and to the pair $(M_{W[\ep]}, M_{W[\ep],1})$ over $(W[\ep], (p))$ together with the identification $M_0\otimes_W W[\ep]=\Lambda\otimes_{\Z_p}W[\ep]=M_{W[\ep]}$. We obtain an isomorphism
\[
{ c}_{W[\ep]}: \wtM_{0,1}\otimes_W W[\ep]\xrightarrow{\sim}\widetilde M_{W[\ep],1}.
\]
\quash{Lemma \ref{masterlemma}  applies to $(W[\ep], (p))\to  (W[\ep], (p,\ep))$ and, hence, the tilde functor $\tau_{(W[\ep], (p))}: P_{(W[\ep], (p))}\to {\rm Mod}^{\rm ff}_{W[\ep]}$ factors 
\[
P_{(W[\ep], (p))}\to P_{(W[\ep], (p,\ep))}\to {\rm Mod}^{\rm ff}_{W[\ep]}.
\]
By following the proof of Lemma \ref{masterlemma2}, we see that by using this fact together with the identification $M_\gS=\Lambda\otimes_{\Z_p}\gS=M_0\otimes_W \gS$,
we obtain an isomorphism
\[
{ c}_{W[\ep]}: \wtM_{0,1}\otimes_W W[\ep]\xrightarrow{\sim}\widetilde M_{W[\ep],1}.
\]
}
Composing  ${ c}_{W[\ep]}$ with the base change isomorphism $\widetilde M_{W[\ep],1}\xrightarrow{\sim} \widetilde M_{\gS,1}\otimes_{\gS}W[\ep]$, 
gives a ``connection isomorphism"
\[
{ c}_\gS:  \wtM_{0,1}\otimes_W W[\ep]  \xrightarrow{\sim}  \widetilde M_{\gS,1}\otimes_\gS W[\ep].
\]
As in the proof of Lemma \ref{319} (b), we see that this fits in a commutative diagram
\begin{equation}\label{gSconnectiondiagram}
 \begin{aligned}
\xymatrix
{\widetilde M_{\gS,1}\otimes_\gS W[\ep]  \ar[r] & \varphi^*( M_\gS) \otimes_\gS W[\ep] \\
\wtM_{0,1}\otimes_W W[\ep]  \ar[r]\ar[u]^{{ c}_\gS} & \varphi^*(M_0)\otimes_W W[\ep]\ar[u]^\sim .}
 \end{aligned}
\end{equation}
Here, the vertical map on the right is inducing the identity on $\varphi^*(M_0)$ and the horizontal maps are induced by
tensoring the natural maps $$\widetilde M_{\gS,1}\to \varphi^*( M_\gS),\ \ \wtM_{0,1}\to \varphi^*(M_0).$$
The rings $\gS$ and $W[\ep]$ are $p$-torsion free and we have $\widetilde M_{\gS,1} = \varphi^*(M_{\gS,1})$, $\widetilde M_{0,1} = \varphi^*(M_{0,1})$, as above. Hence, the horizontal maps give isomorphisms after inverting $p$.
By Lemma \ref{perfect},
$s_a\otimes 1\in M_{\gS,1}^{\otimes}$. It follows that the connection isomorphism ${c}_\gS$ preserves the tensors $\ti s_a=s_a\otimes 1=\phi^*(s_a\otimes 1)\in \widetilde M^\otimes_{\gS,1} = \varphi^*(M_{\gS,1}^\otimes)$.

\begin{Remark}
{\rm  The isomorphism ${c}_{\gS}$ can be given more directly as follows: 
For any $\gS$-module $H,$ if we set $H_0 = H\otimes_\gS W,$ then we have canonically 
\[
 \varphi^*(H_0) \otimes_W W[\ep]     \xrightarrow{\sim} \varphi^*(H)\otimes_\gS W[\ep],
\]
because $\varphi$ induces a lift of Frobenius on $W[\ep]$ which factors through $W.$ 
If $H$ itself has the form $H_0\otimes_W\gS,$ then the above isomorphism induces the identity on $\varphi^*(H_0).$ 
Since the isomorphism is functorial, this holds for any isomorphism $H \simeq H_0\otimes_W\gS,$ which lifts 
the identity on $H_0.$ Applying this 	discussion to $H=M_{\gS, 1} \subset  M_\gS$ gives
\[
{ c}_\gS:  \wtM_{0,1}\otimes_W W[\ep]  \xrightarrow{\sim}  \widetilde M_{1,\gS}\otimes_\gS W[\ep].
\]
We can immediately see that this fits in the commutative diagram (\ref{gSconnectiondiagram}), hence it agrees with the map denoted by $c_\gS$ above.}
\end{Remark}
 
Let us set $M_{\whW(\Su)}=\Lambda\otimes_{\Z_p} \widehat W(\Su)$ and let $M_{\whW(\Su), 1}$ be the $\whW(\Su)$-module
\[
\hat I_{\Su}\otimes_{\Z_p}\Lambda \subset M_{\whW(\Su), 1}\subset M_{\whW(\Su)}=\Lambda\otimes_{\Z_p} \widehat W(\Su)
\]
obtained by lifting $ M_{\Su, 1}\subset M_{\Su}=\Lambda\otimes_{\Z_p}\Su$. The pair 
$(M_{\whW(\Su)}, M_{\whW(\Su), 1})$ is the base change of $(M_\gS, M_{\gS,1})$ under $\bar\lambda: (\gS, (p))\to (\whW(k\lps u\rps),\hat I_{\Su})$.
It is also the base change of the pair $(M, M_1)$ over $\whW(R_G)$ under the frame morphism underlying the map
$\whW(R_G)\to \whW(\Su)$ induced by $a: \Spec(\Su)\to \BMloc_{\Gg,\mu}$. (Recall $R_G$ is  the completion of the local ring at $x$ of $\BMloc_{\Gg,\mu}$.) We can apply the frame morphism $\bar\lambda_\ep: (W[\ep], (p))\to (\whW(k[\ep]), \hat I_{k[\ep]})$ appearing in the diagram (\ref{framescd}). By the functoriality of the construction, we have
\[
{ c}_{\whW(k[\ep])}={c}_\gS\otimes_{W[\ep],\bar\lambda_\ep}\whW(k[\ep]).
\]
But, by the above, $c_\gS$ preserves the tensors $\ti s_a\in \wtM_{\gS, 1}^\otimes$. Hence, ${c}_{\whW(k[\ep])}$ preserves the tensors $\ti s_a\in \wtM_{\whW(k[\ep]), 1}^\otimes$.\footnote{We often abuse notation and use the same symbol $\ti s_a$ to also denote a tensor obtained from $\ti s_a$ by base change.}
However, $(M_{\whW(k[\ep])}, M_{\whW(k[\ep]),1})$ is the base change of $(M, M_1)$ by the frame morphism underlying the map
\[
\whW(R_G)\to \whW(\Su)\to \whW(k[\ep])
\]
given by $a^*_{(2)}: R_G\to \Su\to k[\ep]$. It follows that
\[
{ c}_{\whW(k[\ep])}=c\otimes_{\whW(R_G/\fa_G), a^*_{(2)}}\whW(k[\ep]).
\]
Hence, the base change of $c$ by the map
\[
 \widehat W(R_G/\fa_G)\to \widehat W(\Su)\to \whW(k[\ep])
 \]
 induced by $a^*_{(2)}: R_G\to \Su\to k[\ep]$ preserves the tensors $\ti s_a$. This implies the statement given in the beginning of the proof.

We can now complete the proof of the Proposition.
Our assumption that the tangent space at $x$ is spanned by smooth formal curves gives the following: There are $a_i: \Spec(k\lps u_i\rps)\to \Spec(R_G)$, $i=1,\ldots , r$, lifting $x$, which give an injective map
 \[
\mathfrak J_G:= \mathfrak m_G/\fa_G=\mathfrak m_G/\mathfrak m_G^2+(\pi_E)\xrightarrow{\oplus_i a^*_i}  \oplus_{i=1}^r (u_i)/(u_i)^2.
 \]
Applying Zink's log coordinates to $\whW(\mathfrak J_G)$ for the square zero ideal $\mathfrak J_G$, 
we obtain
 \[
 \whW(R_G/\fa_G)\simeq W(k)\oplus (\bigoplus_{m\geq 0} \mathfrak J_G)\subset \bigoplus_{i=1}^r (W(k)\oplus (\bigoplus_{m\geq 0} \mathfrak (u_i)/(u_i)^2)).
 \]
Notice that, by using Zink's log coordinates, we see that $\whW(\mathfrak J_G)$ is a square zero ideal in  $\whW(R_G/\fa_G)$ with $p\cdot \whW(\mathfrak J_G)=0$. 
The ideal $\cal  I\subset \whW(R_G/\fa_G)$ cutting out the locus where $c(s_a\otimes 1)=\ti s_a$ is contained in $\whW(\mathfrak J_G)$.
 The modules $\widetilde M_1$ are free and the  connection  isomorphism  of  Lemma \ref{319} is compatible with base change. We can verify that the connection isomorphism respects the tensors after pulling back by all $a_i$, $i=1,\ldots ,r$; this was shown above. This implies that the $\cal I$ maps to $0$ under each $\whW(\mathfrak J_G)=\bigoplus_{m\geq 0} \mathfrak J_G\to \bigoplus_{m\geq 0} \mathfrak (u_i)/(u_i)^2$, induced by $a_i$ above, hence $\cal I=0$.
\end{proof}

\begin{cor}\label{smooth}
Assume $(\Gg,\mu)\hook (\GL(\La),\mu_d)$ is a good integral Hodge embedding. If $\BMloc_{\Gg,\mu}\otimes_{\O_E} k$ is smooth at $x$, then $(\Gg,\mu)\hook (\GL(\La),\mu_d)$ is very good at $x$.
\end{cor}

\begin{proof}
The smoothness assumption easily implies that each tangent vector extend to a smooth formal curve and so this follows from Prop. \ref{span}.
\end{proof}

 \end{para}

\begin{para}
We mention an interesting variant of Prop. \ref{span} which is not used in the rest of this paper.

Suppose that $X$ is a scheme over $\br \O_E$ and let $x\in X(k)$.
By definition, the mod $\pi_E$ tangent space of $X$ at $x$ is $ \bar T_x X=T_x (X\otimes_{\br \O_E}k):=(\frakm_{X, x}/\frakm_{X, x}^2+(\pi_E))^*$. We say that the mod $\pi_E$ tangent space $ \bar T_x X$ is \emph{spanned by arithmetic curves}, if the images of $\bar T_{0}\Spec(\O_K)\to \bar T_x X$   by all morphisms $a: \Spec(\O_K)\to X$ that map the closed point $0$ of $\Spec(\O_K)$ to $x$, where $K/\br E$ runs
over all finite extensions, generate the $k$-vector space $\bar T_x X$. Here, $\Spec(\O_K)$ is considered as an $\br\O_E$-scheme.

\begin{prop}\label{spanArithmetic}
Assume $(\Gg,\mu)\hook (\GL(\La),\mu_d)$ is a good  integral Hodge embedding. If the mod $\pi_E$ tangent space of  $\BMloc_{\Gg,\mu}\otimes_{\O_E} \br \O_E$ at $x$ is  spanned by arithmetic curves, then $(\Gg,\mu)\hook (\GL(\La),\mu_d)$ is very good at $x$.
\end{prop}
\end{para}

\begin{proof}
This is similar to the proof of Prop. \ref{span}: The crucial point is to show that the base change of $\ti s_a\in \wtM_1^\otimes$ by the map $\whW(R_G)\to \whW(\O_K)$ given by a
local $\br \O_E$-algebra homomorphism $a^*: R_G\to \O_K$ with $K/\br E$ finite, is horizontal over $\whW(\O_K/((\pi^2_K)+(\pi_E)))$. To prove that we write $\O_K=W(k)\lps x\rps/({\rm E}(x))$, where ${\rm E}(x)$ is an Eisenstein polynomial for a uniformizer $\pi_K$ of $K$ and then use $\bar\lambda: \gS\to \whW(\O_K)$ determined by $\bar\lambda(x)=[\pi_K]$, cf. \cite[Proof of Lem. 3.2.9]{KP}. The rest of the argument follows along the   lines of the proof of Proposition \ref{span} and, in fact, is somewhat simpler: the analogue of Lemma \ref{perfect} is now provided by 
\cite[Lem. 3.2.6]{KP}.  
\end{proof}

 \begin{para}
For future use, we will need a result for groups which are not connected, mainly to deal
 with orthogonal groups. We assume that $\Gg$ is smooth and affine of finite type over $\Z_p$, $G=\Gg\otimes_{\Z_p}\Q_p$. Denote by $G^\circ$ the neutral component. Assume that $G^\circ$ is reductive and that the Zariski closure $\Gg^0$ of $G^\circ$
  in $\Gg$ is a smooth stabilizer group scheme for $G^\circ$. Note  that $\Gg^0$ is not necessarily equal to the neutral component of the group scheme $\Gg$.

Let $\{\mu\}$ be the $G(\bar\Q_p)$-conjugacy class of $\mu:\Gm_{\bar\Q_p}\to G_{\bar\Q_p}$, with reflex field $E$. The coweight $\mu: \Gm_{\bar\Q_p}\to G_{\bar \Q_p}$ automatically factors through the neutral 
 component giving $\mu^\circ: \Gm_{\bar\Q_p}\to G^\circ_{\bar\Q_p}$. We assume that $\mu^\circ$ is minuscule. Hence, we have a local Shimura pair $(G^\circ, \mu^\circ)$. The reflex field $E^\circ$ of $(G^\circ, \mu^\circ)$ is an extension of $E$. 
 
Suppose now that we have an integral Hodge embedding $(\Gg,  \mu )\hook (\GL(\La), \mu_d)$ (with the obvious generalization of the definition to non-connected $G$).
 Since $\Gg^0\hook \Gg$ is  a closed immersion, we also have an integral Hodge embedding $(\Gg^0, \mu^\circ)\hook (\GL(\La), \mu_d)$. As usual, we assume that this is good, i.e. it induces a closed immersion $\BMloc_{\Gg^0, \mu^\circ}\hook {\rm Gr}(d,\La)_{\O_{E^\circ}}$. Consider $x\in \BMloc_{\Gg^0,\mu^\circ}(k)$ with $k=\bar k_{E^\circ}$.
 
 \begin{prop}\label{neutral}
 Suppose that   $\Gg\hook\GL(\La)$ is cut out by a set of tensors $(s_a)\subset \La^{\otimes}$ such that $\ti s_a\in \wtM_1^\otimes$   are horizontal at $x$. Then,   
 $(\Gg^0, \mu^\circ)\hook (\GL(\La), \mu_d)$ is very good at $x$.
 \end{prop}

\begin{proof} 
Recall that the construction in \cite[\S 3]{KP} which was reviewed above, when applied to 
$(\Gg^0, \mu^\circ)\hook (\GL(\La), \mu_d)$ gives a $\Gg^0$-torsor
$\calT^0$ over $\whW(R_{G^\circ})
$. In fact, the arguments in \emph{loc. cit.} extend to show that \cite[Cor. 3.2.11]{KP} also holds for $\Gg$: We have $\ti s_a\in \wtM_1^\otimes$ and the $\Gg$-scheme $\calT'$ of isomorphisms between  
$\wtM_1$ and $\La$ that take $\ti s_a$ to $s_a$ is a $\Gg$-torsor. To see this one observes that, since the coweight takes values in $G^\circ$, the $\Gg$-scheme $\underline{\rm Isom}_{(s_a)}(F, M_{\frakS})_{|D^*}$ over $D^*=\Spec(\frakS)-\{(0,p)\}$ as in the proof of \cite[Lem. 3.2.6]{KP}, is actually induced from a similar
$\Gg^0$-scheme which then comes from a trivial $\Gg^0$-torsor. This produces isomorphisms $F\xrightarrow{\sim} M_{\frakS}$ over $D$ which preserve the tensors $s_a$. This proves the claims of \cite[Lem. 3.2.6]{KP} in the current situation, and the rest of the argument goes through. It now follows that the natural map between $\calT'$ and the $\Gg$-torsor $\calT$ obtained
by pushing out $\calT^0$ by $\Gg^0\to\Gg$ is an isomorphism. 

Set $S=R_{G^\circ}/\frakm^2_{G^\circ}+(\pi_{E^\circ})$.
By assumption, the connection homomorphism 
$c$ on $\wtM_1$ over $\whW(S)$ respects the tensors $\ti s_a$. Hence, $c$ descends to a $\Gg$-torsor isomorphism
\[
c^\Gg: \calT_0\otimes_{W(k)}\whW(S)\xrightarrow{\sim} \calT\otimes_{\whW(R_{G^\circ})}\whW(S).
\]
We have to show that $c^{\Gg}$ further descends to an isomorphism 
\[
c^{\Gg^0}: \calT^0_0\otimes_{W(k)}\whW(S)\xrightarrow{\sim} \calT^0\otimes_{\whW(R_{G^0})}\whW(S)
\]
of the underlying $\Gg^0$-torsors. Since $\whW(S)$ is henselian with residue field $k$ and $\Gg^0$ is smooth, we can choose a section of $\calT^0$; this also induces a section of $\calT$. Since $c^{\Gg}$ is the identity modulo $\whW(\frakm_S)$, it is given by an element in 
\[
 {\rm ker}(\Gg(\whW(S))\to \Gg(W(k))).
\]
Recall that $\whW(S)=W(k)\oplus \whW(\frakm_S)$, and by Zink's log coordinates, $\whW(\frakm_S)$ is a square zero ideal. Since $\Gg^0\to\Gg$ is a closed immersion and  $\Gg$ and $\Gg^0$ are both smooth of the same dimension, 
\[
 {\rm ker}(\Gg^0(\whW(S))\to \Gg^0(W(k)))={\rm ker}(\Gg(\whW(S))\to \Gg(W(k))).
\]
This shows that the isomorphism $c^{\Gg}$ is given by a point of $\Gg^0(\whW(S))$ and so it descends to an isomorphism of the underlying $\Gg^0$-torsors.
\end{proof}

 \end{para}

\section{The construction of very good embeddings}\label{sec: very good embeddings}

In this section, we apply the previous results to construct very good integral Hodge embeddings 
for many local model triples $(G, \{\mu\}, \Gg)$.

\subsection{The non-exceptional cases} \label{sec: NE cases}
  
Let $(G, \{\mu\}, \Gg)$ be a local model triple over $\Q_p$ which satisfies the standard  assumptions.
We often need to assume the following condition on the pair $(G,\mu)$:
\begin{itemize}

\item[(NE)]
$(G^\ad,\mu^\ad)$ does not contain a simple factor of type $D^{\mathbb H}$, or a simple factor of type $A$ with adjoint group $\Res_{F/\Q_p}{\rm PGL}_m(D)$, with $D$ a central division $F$-algebra 
of index divisible by $p$. 
\end{itemize}
We will sometimes call $(G, \{\mu\}, \Gg)$ that satisfy (NE), non-exceptional.

We will now apply the results of the previous sections to show:

\begin{thm}\label{thm:main}
Let $(G, \{\mu\}, \Gg)$ be a local model triple over $\Q_p$ which satisfies the standard  assumptions and (NE), i.e. it is non-exceptional. Suppose that $\Gg$ is the stabilizer group scheme for a point $\bx$ in $\B(G, \Q_p)$ which is generic in its facet, that the centralizer of a maximal $\br\Q_p$-split torus of $G$ is $R$-smooth and that $p$ does not divide $|\pi_1(G^\der)|$.  Suppose $G\subset G'=\prod_{i=1}^s \Res_{F_i/\Q_p}H_i$, $H_i$ split over a tamely ramified extension of $F_i$, and   $G^\der=G'^\der$. Let $\rho'_i: H_i\to \GL(W_i)$ be faithful minuscule representations over $F_i$, such that the composition
\[
\rho':  G'=\prod_i \Res_{F_i/\Q_p}H_i\xrightarrow{\prod_i\Res_{F_i/\Q_p}\rho'_i }\prod_i \Res_{F_i/\Q_p}\GL(W_i)\to \GL(\oplus_i V_i)=\GL(V),
\]
where in the target $V_i$ is $W_i$ regarded as a $\Q_p$-vector space, gives a (local) Hodge embedding $(G',\mu')\hook (\GL(V),\mu_d)$, where $\mu'$ is the composition of $\mu$ with $G\subset G'$. Assume the restriction $\rho:=\rho'_{|G}$ also gives a Hodge embedding $(G,\mu)\hook (\GL(V),\mu_d)$.  

Then there is   a periodic $\Z_p$-lattice chain $\L$ in $V$ and an integral Hodge embedding $(\Gg,\mu)\hook (\GL(\L),\mu_{d})$ extending $\rho$ which is very good at all points of $\BMloc_{\Gg,\mu}$. 
\end{thm}

As before, set  $\La:={\rm tot}(\L)\subset V^{\oplus r}$ where $r$ is the number of lattices in 
a determining segment of $\L$.The conclusion means
 that $\rho^{\oplus r}: G\hook \GL(V^{\oplus r})$ extends to
an integral Hodge embedding $(\Gg,\mu)\hook (\GL(\La),\mu_{rd})$ which is very good at all points of $\BMloc_{\Gg,\mu}$, see \S \ref{par:VG}.

\begin{proof}
 Fix an algebraic closure $\bar\Q_p$ of $\Q_p$. If $F\subset \bar\Q_p$ is a finite extension of $\Q_p$, we will denote by $F^t$ the maximal field extension of $F$ which is contained in $\bar\Q_p$ and is tamely ramified over $F$. 
  
  \begin{lemma}\label{lemma:tame}
  Let $F$ be a finite field extension of $\Q_p$ contained in $\bar\Q_p$. Then $F^t$  is the compositum $F\Q_p^t$ in $\bar\Q_p$.
  \end{lemma}

\begin{proof}
A similar statement holds for the maximal unramified extensions, i.e. $F^{\rm un}=F\Q_p^{\rm un}$. Now $F^t=\cup_{e}F^\un(\pi^{1/e})$, where $\pi$ is a uniformizer of $F$ and $e$ ranges over all integers prime to $p$; by Hensel's lemma this holds for all uniformizers $\pi$ and any choice of $\pi^{1/e}$. Similarly, $\Q_p^t=\cup_e\Q_p^\un(p^{1/e})$. We will show that for each $e$ prime to $p$, $\pi^{1/e}$ belongs to $F\Q_p^t$. Let $1/a$ be the $p$-adic valuation of $\pi$ so that $(p)=(\pi)^a$ in $\O_F$.
Write $a=p^mb$ with $b$ prime to $p$ and $1=up^m+ve$, with $u$, $v\in \Z$. Then $\varpi:=\pi^vp^{u/be}\in F\Q_p^t$ has valuation $1/p^mbe=1/ae$. Hence $\varpi^{e}\in F\Q_p^t$ has the same valuation as $\pi$ and so $\pi=\varpi^e \cdot \alpha$, $\alpha$ a unit of $F\Q_p^t$.
Then $\pi^{1/e}=\varpi\cdot \alpha^{1/e}$. Since  $\alpha^{1/e}$ is in $F\Q_p^t$ by Hensel's lemma, $\pi^{1/e}\in F\Q_p^t$. \end{proof}
\smallskip

We now fix embeddings $F_i\hook\bar\Q_p$, for all $i$. Using Lemma \ref{lemma:tame} and 
Proposition \ref{Fixed} applied to  $H_i$, for all $i$, we see that, under our assumptions, there is a finite tame Galois extension $\ti\Q_p/\Q_p$ with $\ti\Q_p\subset \bar\Q_p$  such that 

\begin{altitemize}
\item for each $i$, $\ti\Q_p$ contains the maximal tame subextensions of $F_i/\Q_p$,

\item for each $i$, the group $H_i$ splits over the compositum $\ti F_i:=F_i\ti\Q_p \subset \bar\Q_p$,  

\item 
for
\[
\bar \bx=(\bar\bx_i)_i\in \bar\B(G,\Q_p)= \B(G^\der, \Q_p)=\prod_i \B(H^\der_i, F_i),
\]
all the points $\bar\bx_i$ are hyperspecial in $\B(H^\der_i, \ti F_i)$.
\end{altitemize}

Set $\Gamma=\Gal(\ti\Q_p/\Q_p)$.
For each $i$, $$\Ga_i=\Gal(\ti F_i/F_i)=\Gal(F_i\ti\Q_p/F_i)\simeq \Gal(\ti\Q_p/F_i\cap \ti\Q_p )$$ is identified with a subgroup of $\Ga$. We have
\[
F_i\otimes_{\Q_p}\ti\Q_p\simeq \prod_{\ga\in \Ga/\Ga_i} F_i\ti\Q_p=\prod_{\ga\in \Ga/\Ga_i} \ti F_i.
\]
Now, using the standard argument of taking fixed points by a tame action, we can write 
\[
\Gg=\Gg_\bx=(\Res_{\ti\Z_p/\Z_p}\ti\Gg)^\Ga,
\]
with $\ti\Gg=\ti\Gg_{\bx}$, in which $\bx$ is considered as a point of $\B(G,\ti\Q_p)$. In particular, the natural morphism
\[
\Gg\hook \Res_{\ti\Z_p/\Z_p}\ti\Gg
\]	
is a closed immersion.  Consider the image $\bx'\in \B(G',\Q_p)$ of $\bx$ under the natural map $\B(G,\Q_p)\to\B(G',\Q_p)$; we have similar statements for the corresponding stabilizer group schemes $\Gg'=\Gg'_{\bx}$ over $\Z_p$ and $\ti\Gg'=\ti\Gg'_{\bx}$ over $\ti\Z_p$.
Using the $R$-smoothness condition, by Proposition \ref{prop: R-smoothness properties} (3), we see that the natural morphisms
\[
\Gg\to\Gg',\qquad \ti\Gg\to\ti\Gg',
\]
are closed immersions. Note that both $G\to G'$ and $\ti G\to \ti G'$ induce isomorphisms on adjoint groups.
 Write $\bx'=(\bx'_i)_i$ in $\B(G',\Q_p)=\prod_i\B(H_i, F_i)$. By the above, 
\[
G'\otimes_{\Q_p}\ti\Q_p=\prod_{i}
\Res_{F_i\otimes_{\Q_p}\ti\Q_p/\ti\Q_p}(H_i\otimes_{F_i}(F_i\otimes_{\Q_p}\ti\Q_p))=\prod_i\prod_{\ga\in \Ga/\Ga_i}\Res_{\ti F_i/\ti\Q_p}(H_i\otimes_{F_i}\ti F_i)
\]
with $ H_i\otimes_{F_i}\ti F_i$ split and $\bx'_i$  hyperspecial in $\B(H_i , \ti F_i)$. Note
 \[
\B(G',\ti\Q_p)=\prod_i \B(\Res_{F_i/\Q_p} H_i, \ti\Q_p)=\prod_{i}\prod_{\ga\in \Ga/\Ga_i}\B(H_i, \ti F_i).
\] 
Let $\ti\Hh_i$ be the reductive group schemes over $\ti \O_i:=\O_{\ti F_i} $ corresponding to $\bx'_i$ with generic fibers the split groups $H_i\otimes_{F_i}\ti F_i$. Then, 
\[
\ti\Gg'\simeq\prod_i\prod_{\ga\in \Ga/\Ga_i} \Res_{\ti \O_i/\ti\Z_p }\ti\Hh_i
\]
as group schemes over $\ti\Z_p$. Under the above isomorphism, the semi-linear action of $\Ga$ on $\ti\Gg'$ preserves the $i$-factors and corresponds, on each $i$-factor, to the action obtained by inducing the action of the subgroup $\Ga_i$ on $\Res_{\ti O_i/\ti\Z_p }\ti\Hh_i$ to the whole Galois group $\Ga$.
By Proposition \ref{prop:findlattice}, there are $\ti \O_i$-lattices $\ti\La_i\subset W_i\otimes_{F_i}\ti F_i$, which are $\Ga_i=\Gal(\ti F_i/F_i)$-stable, such that $\rho_i'\otimes_{F_i}\ti F_i$ extend to
\[
\ti\Hh_i\hook \GL(\ti \La_i)
\]
which are closed immersions. We combine these to get a closed immersion
\begin{equation}\label{imm612}
\ti\rho': \ti\Gg'\hook \GL(\ti\La)
\end{equation}
which extends $\rho'\otimes_{\Q_p}\ti\Q_p: G'\otimes_{\Q_p}\ti\Q_p \to \GL(V\otimes_{\Q_p}\ti\Q_p)$ and factors as
\[
\ti\Gg'=\prod_i\prod_{\ga\in \Ga/\Ga_i} \Res_{\ti \O_i/\ti\Z_p}\ti\Hh_i\hook \prod_i\prod_{\ga\in \Ga/\Ga_i} \Res_{\ti \O_i/\ti\Z_p}\GL( \ti \La_i)\hook \GL(\oplus_{i, \ga}  \ti\La_i)=\GL(\ti\La).
\]
Here, $\ti\La=\oplus_{i}\oplus_{\ga\in\Ga/\Ga_i}\ti\La_i$ is a $\ti\Z_p $-lattice in 
\[
V\otimes_{\Q_p}\ti\Q_p =(\oplus_i W_i)\otimes_{\Q_p}\ti\Q_p =\oplus_i (W_i\otimes_{F_i}(F_i\otimes_{\Q_p}\ti\Q_p ))=\oplus_i\oplus_{\ga\in\Ga/\Ga_i} W_i\otimes_{F_i}\ti F_i.
\]
(The action of $G'\otimes_{\Q_p}\ti\Q_p $ on $V\otimes_{\Q_p}\ti\Q_p$ given via $\rho'\otimes_{\Q_p}\ti\Q_p $ is also  induced from the subgroups $\Ga_i$ as above.) Note that $\ti\La$ is a $\Ga=\Gal(\ti\Z_p /\Z_p)$-stable lattice
and   (\ref{imm612}) is compatible with the $\Ga$-action on $\ti\Gg'$.

We now consider the composition
\begin{equation}\label{imm613}
\Gg\hook \Res_{\ti\Z_p /\Z_p}\ti\Gg\hook
\Res_{\ti\Z_p /\Z_p}\ti\Gg'\hook \GL(\ti\La).
\end{equation} 
This is a sequence of closed immersions given, more precisely, as
\[
\Gg=(\Res_{\ti\Z_p /\Z_p}\ti\Gg)^\Ga\hook \Res_{\ti\Z_p /\Z_p}\ti\Gg\hook
\Res_{\ti\Z_p /\Z_p}\ti\Gg'\xrightarrow{\Res_{\ti\Z_p /\Z_p}(\ti\rho')} \Res_{\ti\Z_p/\Z_p}\GL(\ti\La)\hook\GL(\ti\La).
\]
Here, in the target, $\ti\La$ is considered as a $\Z_p$-lattice by restriction of scalars.
 On the generic fibers, the composition gives
\[
G\to \Res_{\ti\Q_p/\Q_p}(G\otimes_{\Q_p}\ti\Q_p)\xrightarrow{\Res_{\ti\Q_p/\Q_p}(\rho\otimes_{\Q_p}\ti\Q_p)}\Res_{\ti\Q_p/\Q_p}\GL(V\otimes_{\Q_p}\ti\Q_p)\hook \GL(V\otimes_{\Q_p}\ti\Q_p).
\]
where in the target $V\otimes_{\Q_p}\ti\Q_p$ is considered as a $\Q_p$-vector space by restriction of scalars.

We can then see, using the same argument as in Proposition \ref{prop: better lattice}, that the group scheme $\Gg$ is cut out 
in $\Res_{\ti\Z_p /\Z_p}\ti\Gg\hook \GL(\ti\La)$ by a set of $\Z_p$-linear endomorphisms 
$e_a: \ti\La\to \ti\La$.

It now follows from Proposition \ref{prop: better lattice}, that the integral Hodge embeddings induced by $\Gg'\hook \GL(\ti\La)$ and by
\begin{equation}\label{imm615}
\Res_{\ti \Z_p/\Z_p}\ti\Gg'\hook \Res_{\ti\Z_p/\Z_p}\GL(\ti\La)\hook \GL(\ti\La),
\end{equation}
(i.e. the partial composition appearing in (\ref{imm613})), give closed immersions
\[
 \BMloc_{\Gg',\mu'}\to \BMloc_{\Res_{\ti \Z_p/\Z_p}\ti\Gg', \ti\mu'}\otimes_{\O_{\ti E'}}\O_{E'},\qquad \BMloc_{\Res_{\ti \Z_p/\Z_p}\ti\Gg', \ti\mu'}\to \BMloc_{\GL(\ti\La), \ti\mu'}\otimes_{\Z_p}\O_{\ti E'},
\] 
between the corresponding local models. Hence, these are good integral Hodge embeddings. 
Recall that the morphism of local model triples $(G,\{\mu\},\Gg)\to (G',\{\mu'\}, \Gg')$ induces an isomorphism
of local models
\[
\BMloc_{\Gg,\mu}\xrightarrow{\sim}\BMloc_{\Gg',\mu'}\otimes_{\O_{E'}}\O_{E}.
\]
Indeed, $G\to G'$ induces an isomorphism on adjoint groups, cf. \cite[Prop. 21.5.1]{Schber}, \cite[Pro. 2.14 (c)]{HPR}. Similarly, we have
\[
\BMloc_{\Res_{\ti \Z_p/\Z_p}\ti\Gg, \ti\mu} \xrightarrow{\sim}\BMloc_{\Res_{\ti \Z_p/\Z_p}\ti\Gg', \ti\mu'}\otimes_{\O_{\ti E'}}\O_{\ti E}.
\]
It follows that
\[
 \BMloc_{\Gg,\mu}\to \BMloc_{\Res_{\ti \Z_p/\Z_p}\ti\Gg, \ti\mu}\otimes_{\O_{\ti E}}\O_{E},\qquad \BMloc_{\Res_{\ti \Z_p/\Z_p}\ti\Gg, \ti\mu}\to \BMloc_{\GL(\ti\La), \ti\mu}\otimes_{\Z_p}\O_{\ti E}
\] 
are also closed immersions. Hence, the integral Hodge embeddings induced by $\Gg\hook \GL(\ti\La)$ of (\ref{imm613}) and by
\begin{equation}\label{imm614}
\Res_{\ti \Z_p/\Z_p}\ti\Gg\hook  \GL(\ti\La),
\end{equation}
are also good integral Hodge embeddings.

Now consider $x\in \BMloc_{\Gg, \mu}(k)$. 
 Set, for simplicity, $\calJ=\Res_{\ti \Z_p/\Z_p}\ti\Gg$, $\calJ'=\Res_{\ti \Z_p/\Z_p}\ti\Gg'$. The second group scheme is isomorphic to
a product of restriction of scalars of the (split) reductive group schemes $\ti\Hh_i$. Hence, since we exclude factors of type $D^{\mathbb H}_n$, Theorem \ref{corLMSpan} (2) implies that, at all points of
$\BMloc_{\calJ',\ti\mu'}(k)$, the tangent space of the special fiber of $\BMloc_{\calJ',\ti\mu'}$ is spanned by smooth formal curves.
Since $\BMloc_{\calJ,\ti\mu}\simeq \BMloc_{\calJ',\ti\mu'}\otimes_{\O_{E'}}\O_E$, the same holds for the tangent spaces of the special fiber of
$\BMloc_{\calJ,\ti\mu}$. Proposition \ref{span} then implies that the integral Hodge embedding given by $\calJ \hook  \GL(\ti\La)$ of (\ref{imm614}) is  very good  
at the image $x'\in\BMloc_{\calJ,\ti\mu}(k)$ of $x$ under $\BMloc_{\Gg,\mu}\to \BMloc_{\calJ,\ti\mu}\otimes_{\O_{\ti E}}\O_{E}$.
Since, as we have seen above, $\Gg$ is cut out in $\calJ \hook \GL(\ti\La)$ by endomorphisms of $\ti\La$, Corollary  \ref{endoplus} now implies that the embedding given by $\Gg\hook \GL(\ti\La)$ of (\ref{imm613}) is 
very good at $x$. 

Finally, we let $\L$ be the lattice chain in $V=(V\otimes_{\Q_p}\ti \Q_p)^\Ga$ which is given by $\{(\ti\pi^i\ti\La)^\Ga\}_{i\in \Z}$, see Lemma \ref{fixedptsGL}. Then ${\rm tot}(\L)\subset V^{\oplus r}$, where $r$ is the number of lattices in a determining segment of $\L$. Set $\La={\rm tot}(\L)$. We now have a diagram of closed group scheme immersions
\begin{equation}\label{eq:diagram2}\begin{aligned}
\xymatrix{\Res_{\ti \Z_p/\Z_p}\ti\Gg\ar[r]& \GL(\ti\La ) &\\
 \Gg\ar[u]\ar[r] & \GL(\L)\ar[u]\ar[r] & \GL({\rm tot}(\L))=\GL(\La)}\end{aligned}
\end{equation}
inducing a corresponding diagram of local model triples which are all good integral Hodge embeddings, cf. \S \ref{rem:diagram}. It remains to deduce that $\Gg\hook \GL(\La)$ is also   very good at $x$. Observe that, after an unramified extension, $\Gg\hook \GL(\La)$ becomes a direct summand in 
 $\Gg\hook \GL(\ti\La)$, cf. (\ref{generalDirectSum}). Then, since $\Gg\hook \GL(\ti\La)$ gives a very good integral Hodge embedding at $x$, the argument in \S\ref{2immersions} together with Lemmas  \ref{diagonal} (a) and \ref{unramified}, implies that $\Gg\hook \GL(\La)$ gives a very good integral Hodge embedding at $x$.
 \end{proof}

  \begin{para}\label{sss:simil1} Here we present a variant
   of Theorem \ref{thm:main} in the presence of alternating forms. This is needed   in the final section of the paper when we consider Hodge embeddings in symplectic groups.
  
  We continue with the same notation. Suppose that there are perfect alternating $F_i$-bilinear forms $\psi_i: W_i\times W_i\to F_i$ such that $\rho_i: H_i\to \GL_F(W_i)$ factors through $\GSp_{F_i}(W_i)$, for all $i$. Recall that $V_i$ is $W_i$ regarded as a $\Q_p$-vector space by restriction of scalars. For each $i$, equip $V_i$ with a perfect alternating $\Q_p$-bilinear form  given by 
 \[
 \psi^{0}_i(v, v')={\rm Tr}_{F_i/\Q_p}(\delta_{F_i/\Q_p}^{-1} \psi_i(v, v'))
 \]
where $\delta_{F_i/\Q_p}$ is a generator of the different ideal of the extension $F_i/\Q_p$. (The form depends on this choice.)
Then the sum
\[
\psi(v, v')=\sum_i \psi_i^{0}(v_i, v'_i), \qquad v=(v_i)_i,\quad v'_i=(v'_i)_i,
\]
 gives a perfect alternating $\Q_p$-bilinear form $\psi$ on $V=\oplus_i V_i$.  We use the superscript $^\vee$ to denote the $\psi$-dual of a $\bbZ_p$-lattice (resp. $\ti\Z_p$-lattice) in $V$ (resp. $V\otimes_{\bbQ_p}\ti\bbQ_p$). If $\calL$ is a periodic lattice chain in $V$, we let $\calL^\vee$ denote the periodic lattice chain whose constituent lattices are given by $\Lambda^\vee$ for $\Lambda\in \calL$.
 
  \begin{thm}\label{thm:mainGSp} Suppose that $(G, \{\mu\}, \Gg)$ is a local model triple over $\Q_p$ satisfying the assumptions of Theorem \ref{thm:main}. With the notations of that Theorem, we assume there are perfect alternating $F_i$-bilinear forms $\psi_i: W_i\times W_i\to F_i$ such that $\rho_i: H_i\to \GL_F(W_i)$ factors through $\GSp_{F_i}(W_i)$, for all $i$. We define $\psi: V\times V\to F$ as in the paragraph above and suppose that the image $\rho(G)$ lies in the 
 	symplectic similitude group $\GSp(V)=\GSp(V,\psi)$.  
 	
 	Then there is a periodic lattice chain $\L$ in $V$ such that $\rho$ extends 
	to closed immersions $
	\Gg\hook \GL(\L)$, $ \Gg\hook \GL(\L^\vee)$
	which both give very good integral Hodge embeddings \[
	(\Gg,\mu)\hook (\GL(\L),\mu_{d}), \quad (\Gg,\mu)\hook (\GL(\L^\vee),\mu_{d}).
	\]
	 In addition, the direct sum  
	$
	(\Gg,\mu)\hook (\GL(\L\oplus \L^\vee),\mu_{2d})
	$
	is a very good integral Hodge embedding.
	\end{thm}
 \begin{proof}We choose $\tilde{\calO}_i$-lattices $\tilde{\Lambda}_i$ as in Theorem \ref{thm:main}, and let $\ti\Lambda_i^*$ denote the $\psi_i$ dual of $\ti\Lambda_i$.  We have $\Ga$-invariant $\ti\Z_p$-lattices $\ti\La:=\oplus_i \oplus_{\in \Ga/\Ga_i}\ti\La_i$ and 
 	$\ti\La^*:=\oplus_i \oplus_{\in \Ga/\Ga_i}\ti\La^*_i$
 	in 
 	$V\otimes_{\Q_p}\ti\Q_p$. If we consider $\ti\Lambda^*$ as a $\ti\Z_p$-lattice in $V\otimes_{\Q_p}\ti\Q_p$, then we have $\ti\Lambda^*=\ti\Lambda^\vee$.
 	
 	 Let $\calL$ denote the lattice chain $\{(\ti\pi^i\ti\La)^\Ga\}_{i\in \Z}$ in $V$. Then $\calL^\vee=\{(\ti\pi^i\ti\La^*)^\Ga\}_{i\in \Z}$ and
	$\L\oplus \calL^\vee=\{(\ti\pi^i(\ti\La\oplus \ti\La^*))^\Ga\}_{i\in \Z}$.
	  Indeed, $
 	 \ti\pi^{-m}\delta^{-1}\ti\La^\vee
 	 $  is the $\psi$-dual of the $\ti\Z_p$-lattice $\ti\pi^{m}\ti\La$, and hence  $(\ti\pi^{-m}\delta^{-1}\ti\La^\vee)^\Gamma$ is the $\psi$-dual of $(\ti\pi^{-m}\delta^{-1}\ti\La)^{\Gamma}.$ Here, the element $\delta $ generates the different of the extension $\ti\Q_p/\Q_p$. Then the argument in the proof of Theorem \ref{thm:main} applies to $\ti\La$, $\ti\La^*$ and $\ti\La\oplus\ti\La^*$, and shows that   $(\calG,\mu)\hookrightarrow (\GL(\calL),\mu_d), (\calG,\mu)\hookrightarrow (\GL(\calL^\vee),\mu_d)$ and their direct sum are all very good integral Hodge embeddings.
\end{proof}

 \begin{Remark}\label{exclude}
 {\rm When $G^{\ad}$ is simple over $\Q_p$, the assumptions of Theorem \ref{thm:main} exclude:

 \begin{altitemize}
 \item[1)]Types $D^{\mathbb H}_n$, i.e. with $G^\ad=\Res_{F/\Q_p}H^\ad$, 
 $H^\ad\otimes_F\bar \Q_p\simeq {\rm PSO}_{2n}$, such that $\mu^\ad\neq 1$ and, for each $\phi: F\to \bar\Q_p$, $\mu^\ad_{\phi}: \Gm_{\bar\Q_p}\to 
 H^\ad\otimes_F\bar\Q_p\simeq {\rm PSO}_{2n}$ is of type $\varpi^\vee_n$, $\varpi^\vee_{n-1}$, or is trivial. Here, $n\geq 4$.

 \item[2)] Types $A_n$ with adjoint group $G^\ad=\Res_{F/\Q_p}{\rm PGL}_m(D)$, where $D$ is a division $F$-algebra such that $p$ divides the index of $D$. 
 \end{altitemize}
 
We will handle such cases by explicit ad hoc arguments and give ``sufficient" local model triples with very good Hodge embeddings. Roughly, the main idea is that in these cases there are enough Hodge embeddings which are (essentially) of PEL type. This is discussed in the next paragraphs.}
  \end{Remark}

  \end{para}
  
\subsection{$D^{\mathbb H}_n$ types}\label{ss:DnH} 

\begin{para} 
Let $V$ be a $K$-dimensional vector space of even dimension $2n$, equipped with a perfect symmetric $K$-bilinear  form $h: V\times V\to K$. For a $K$-algebra $R$, we set $V_R=V\otimes_KR$. The group of orthogonal similitudes ${\rm GO}(V)={\rm GO}(V, h)$ has $R$-valued points
\[
{\rm GO}(V, h)(R)=\{g\in \GL_R(V_R)\ |\ \exists\ c(g)\in R^\times,\ h(gv, gv')=c(g)h(v, v'),\ \forall v, v'\in V_R \}.
\]
This group has two connected components; the neutral component is the subgroup
${\rm GO}^+(V)$ of $g\in {\rm GO}(V)(R)$ with $c(g)^n=\det(g)$.
 \end{para}

\begin{para}\label{par:DnH}
Suppose $G^\ad$ is  simple over $\Q_p$ and $(G^\ad,\mu^\ad)$ is of type $D^{\mathbb H}_n$, as above. As in \cite[\S 5.3.8]{PZ}, \cite{Gross}, we see that $G^\ad\simeq \Res_{K/\Q_p}G'^\ad$, with 
$G'$ as in one of the following cases:

 a) There is a $K$-vector space $V\simeq K^{2n}$ and a perfect symmetric $K$-bilinear $h: V\times V\to K$ such that $G'=\GO^+(V, h)$.
 
 In this case, we can obtain (symplectic) representations of $G'$ that give local Hodge embeddings as follows.
 Let $\tau: G'\hook \GL(V)$ be the natural embedding. Suppose $V_0\simeq K^{2s}$ is equipped  with a perfect alternating $K$-bilinear form 
 $S: V_0\times V_0\to K$ and set $W=V_0\otimes_K V$. This is an ${\rm End}_K(V_0)$-module and 
 supports the perfect alternating form $\psi$ given by
 \[
 \psi(x_1\otimes v_1, x_2\otimes v_2)=S(x_1, x_2)h(v_1, v_2).
 \]
We have the intersection
\[
\GO(V, h)=\GL_{{\rm End}_K(V_0)}(W)\cap \GSp(W)
\]
 and an embedding 
\[
\sigma_{V_0}: G'\subset \GO(V, h)\hook\GSp(W)\hook \GL(W).
\]
 (Note that ${\rm SO}(V, h)$ and ${\rm Sp}(V_0, S)$ form a dual pair in ${\rm Sp}(W, \psi)$.)
 
 Since $(G, \mu)$ is of type $D^{\mathbb H}_n$, both $\tau$ and $\sigma_{V_0}$, followed by taking restriction of scalars, give (local) Hodge embeddings 
 \[
 (G,\mu)\hook (\GL(V), \mu'),\quad (G,\mu)\hook (\GL(W), \mu'')
 \]
 where $V$ and $W$ as considered as $\Q_p$-vector spaces   and $\mu'$, $\mu''$ are the corresponding (minuscule) coweights 
 obtained by composing $\Res_{K/\Q_p}\tau$ and $\Res_{K/\Q_p}\sigma_{V_0}$ with $\mu$.
 
Note that we can choose a Lagrangian basis $\{e_1,\ldots, e_{2s}\}$ of $(V_0, S)$, i.e. such that $S(e_i, e_{2s+1-i})=1$, if $1\leq i\leq s$, and $S(e_i, e_j)=0$ if $1\leq i, j\leq s$, or $s+1\leq i, j\leq 2s$. The representation $\sigma_{V_0}: G'\to \GSp(W)\subset \GL(W)$ is isomorphic to a direct sum 
of $s$ copies of $\sigma_{K^2}$ obtained from $V_0=K^2$ with its standard alternating form; the resulting alternating form on $
W$ is identified with the corresponding orthogonal direct sum.

 b) There is a (left) $D$-module $T\simeq D^n$ for a division quaternion $K$-algebra $D$ and a non-degenerate quaternionic anti-hermitian form $\phi: T\times T \to D$ 
 for the main involution $d\mapsto \bar d$ on $D$, such that $G'={\rm GU}^+(T, \phi)$, where ${\rm GU}(T, \phi)$ is the corresponding unitary similitude group, and ${}^+$ signifies taking the neutral component. 
 Here ${\rm GU}(T, \phi)$ can also be given as follows: Consider the alternating $K$-bilinear form $\psi: T\times T\to K$ given by 
 \[
 \psi(t_1, t_2)={\rm Tr}_{D/K}(\phi(t_1, t_2))
 \]
where ${\rm Tr}_{D/K}: D\to K$ is  the reduced trace
(cf.    \cite[\S 5.3.8]{PZ}, \cite[Prop. A.53]{R-Z}, applied to $n=1$.) 
For a $K$-algebra $R$, ${\rm GU}(T, \phi)(R)$ is given by $D\otimes_K R$-linear automorphisms of $T\otimes_KR$  that respect $\psi$ up to a similitude in $R^*$. Hence,
\[
{\rm GU}(T, \phi)=\GL_{D}(T)\cap \GSp(T, \psi).
\] 
This gives an embedding $\sigma: G'\hook  \GSp(T, \psi)\hook \GL(T)$ which produces a local Hodge embedding  for $(G,\mu)$.

We can obtain more symplectic representations of $G'$ that give local Hodge embeddings as follows.
Let $T_0\simeq D^s$ be a right $D$-module with a non-degenerate quaternionic hermitian form $S: T_0\times T_0\to D$, again for the main involution. We can consider
the $K$-vector space $W=T_0\otimes_D T$ with $K$-bilinear alternating form 
\[
\psi(t_0\otimes t, t'_0\otimes t')={\rm Tr}_{D/K}( S(t_0', t_0)\phi(t, t')).
\]
Then we have
\[
{\rm GU}(T, \phi)=\GL_{{\rm End}_D(T_0)}(W)\cap \GSp(W, \psi).
\]
This gives an embedding $\sigma_{T_0}: G'\hook  \GSp(W, \psi)\hook \GL(W)$ which produces a 
local Hodge embedding  for $(G,\mu)$. Taking $T_0=D$ as a right $D$-module with the standard  hermitian form $S(d, d')=\bar d d'$ gives $W=T$ and the embedding $\sigma$ as above. In fact, there is always a $D$-basis $T_0= D^s$ for which $S$ is the standard hermitian form $S((d_i), (d'_i))=\sum_{i=1}^s \bar d_id'_i$, cf. \cite{ShimuraQ}. Hence, the representation $\sigma_{T_0}: G'\to \GSp(W)\subset \GL(W)$ is isomorphic to a direct sum 
of $s$ copies of $\sigma_{D}=\sigma$ obtained from $T_0=D$ with its standard hermitian form; the resulting alternating form on $
W$ is identified with the corresponding orthogonal direct sum.

Let $L/K$ be a degree $2$ unramified extension with $L\subset D$ as $K$-algebras. Then we can write $D=L\oplus L\cdot \Pi$ with $\Pi^2=\pi$. Base changing from $K$ to $L$ splits $D$: $D\otimes_K L\simeq  {\rm M}_2(L)$. Morita equivalence then gives $T_L=L^2\otimes_L V_L$ for a $2n$-dimensional $L$-vector space $V_L$. The base change $\phi\otimes_KL$ is determined by a symmetric $L$-bilinear form $h_L: V_L\times V_L\to L$ as in case (a) above,
cf. \cite[Prop. A.53]{R-Z}. We can see that the base change of the pair of the group $G'={\rm GU}^+(T, \phi)$ with its representation $\sigma=\sigma_{D}$ in case (b), is isomorphic to 
$ {\rm GO}^+(V_L, h_L)$ with the representation $\sigma_{L^2}$  in case (a). 

For a lattice chain $\calL$  of $\bbZ_p$-lattices in $W$ (in cases (a) or (b)), we write $\calL^\vee$ for the dual lattice chain with respect to the alternating form  $\Tr_{K/\bbQ_p}\circ (\delta^{-1}_{\rmK/\bbQ_p}\psi).$
\end{para}

  \begin{thm}\label{thm:DnH} Let $G= \Res_{K/\Q_p}G'$ with $G'$ as in   \S\ref{par:DnH}   and let $(G, \{\mu\}, \Gg)$ be a local model triple 
 of  $D^{\mathbb H}_n$ type.  Assume that $\Gg$  is the stabilizer group scheme for a point $\bf x$ in $\B(G, \Q_p)= \B(G', K)$ which is generic in its facet.
 Let $\rho'=\sigma_{V_0}: G'\hook \GL(W)$ (in case (a)) and $\rho'=\sigma_{T_0}: G'\hook \GL(W)$ (in case (b)) be as above. 
 
 Then there is a periodic lattice chain $\calL$ of $\bbZ_p$-modules in $W$ which is self-dual (i.e. $\calL=\calL^\vee$) such that $\rho'$ extends to a very good Hodge embedding $(\calG,\mu)\rightarrow (\GL(\calL),\mu'')$.
 \end{thm}

  \begin{proof}
  Let us  discuss case (a). Since $\sigma_{V_0}$ is isomorphic to a direct sum of copies of $\sigma_{K^2}$, we see, using Lemma \ref{diagonal}, that it is enough to show the statement of $\sigma_{K^2}$.  
  By Prop. \ref{Fixed}, there is a tame Galois extension $\ti K/K$ such that $G'\otimes_K\ti K$ splits and the stabilizer group scheme for $\bx\in \B(G',\ti K)$ is hyperspecial. Hence, it is the stabilizer of an $\ti\O$-lattice $\ti\La$ in $V\otimes_K\ti K$ which is $\Ga=\Gal(\ti K/K)$-stable and is self-dual  up to homothety, i.e. $\ti\La^\vee=\ti\pi^a\ti\La$, for $h_{\ti K}$ (see \cite{BTclassII}, \cite[15.2]{KalethaPrasad}). By further enlarging $\ti K$ to allow a square root of $\ti\pi$, we can change  $\ti\La$ in its homothety class and assume it is  self-dual $\ti\La^\vee=\ti\La$. We set $\ti\Gg'=\GO^+(\ti\La, h)$. 
 Now set   $\ti M:=\O^2\otimes_\O \ti\La\subset W_{\ti K}=V_{\ti K}\oplus V_{\ti K}$ which is 
 $\Ga$-stable and $\psi$-self-dual.

The argument in the proof of Thm \ref{thm:main} produces
\begin{equation}\label{eqn: morphism D_n}
\Res_{\ti \O/\Z_p}\ti\Gg'\hook \Res_{\O/\Z_p}\GL(\ti M)\hook \GL(\ti M)
\end{equation}
which gives a good integral Hodge embedding. The proof of the conclusion of Theorem \ref{thm:main} applies provided we can ensure that this gives a very good embedding. Note the self-duality of the resulting lattice chain $\calL$ follows from the $\psi$-self-duality of $\ti M$.

Observe that we have
\begin{equation}\label{M2action}
\GO(\ti\La, h)=\GL_{{\rm End}_\O(\O^2)}(\ti M)\cap \GSp(\ti M,\psi)
\end{equation}
 as a scheme-theoretic intersection.
Indeed, this situation falls in case (II) considered in \cite[App. to Ch. 3]{R-Z} and (\ref{M2action}) follows from 
 \emph{loc. cit.} Prop A. 18, Prop. A. 19. In what follows, we will omit the notation of the forms $h$ and $\psi$.
 Using (\ref{M2action}) we see that $\Res_{\ti \O/\Z_p}\GO(\ti\La)$ is cut out in $\Res_{\ti \O/\Z_p}\GSp(\ti M)$ by a set of endomorphisms $ \ti M\to\ti M$. On the other hand, the integral Hodge embedding given by
$\Res_{\ti \O/\Z_p}\GSp(\ti M)\hook \GL(\ti M)$ is very good by an application of Theorem \ref{thm:main} to the symplectic 
similitude group. Hence,  as in the argument of Proposition \ref{prop: better lattice},  Cor. \ref{endoplus} implies that the composition
\begin{equation}\label{integralDn}
\Res_{\ti \O/\Z_p}\GO(\ti\La)\hook \Res_{\ti \O/\Z_p}\GSp(\ti M)\hook \GL( \ti M)
\end{equation}
 is cut out by a set of tensors $(s_a)\in \ti M^{\otimes}$ such that $\ti s_a$ are horizontal. 
  Now $\Res_{\ti O/\Z_p}\GO^+(\ti\La)$ is the Zariski closure of $\Res_{\ti K/\Q_p}\GO^+(V)$ in 
 $ \Res_{\ti \O/\Z_p}\GO(\ti\La)$. Hence, we can apply Prop. \ref{neutral} and conclude that the restriction 
\begin{equation}\label{integralDn2}
\Res_{\ti \O/\Z_p}\ti\Gg'=\Res_{\ti \O/\Z_p}\GO^+(\ti\La)\hook \GL( \ti M)
 \end{equation}
 of (\ref{integralDn})  gives a  very good integral Hodge embedding. 
This is now enough to deduce the result by using the argument in the proof of Theorem \ref{thm:main}, as we mentioned above. 
This completes the proof in case (a).

Case (b) is now similar: First, we reduce to the case of $\sigma$, using Lemma \ref{diagonal}.  
  By Prop. \ref{Fixed}, there is a tame Galois extension $\ti K/K$ such that $G'\otimes_K\ti K$ splits and the stabilizer group scheme for $\bx\in \B(G',\ti K)$ is hyperspecial. In fact, by possibly enlarging $\ti K$, we can also make sure that the base change $\sigma\otimes_K\ti K$ is isomorphic to $\sigma_{\ti K^2}$ as obtained from the standard split symmetric form on $\ti K^{2n}$ in case (a). The same argument as in case (a) now goes through. (Note that $\sigma$ and $\sigma_{K^2}$ are forms of each other, so the action of the Galois group $\Ga$ is different in the two cases.)
 \end{proof}

\begin{para}\label{para:modification D_n}
 For global applications later, we will need to consider a modification of the groups $G$ and $G'$ above. 
 
  In case (a) we let $\sigma_{K^2}:G'\rightarrow \GL(W)$ be the representation above where we take $V_0=K^2$ with the standard alternating form. 
  Set $G'_1$ to be the subgroup of $\GL(W)$ generated  by $G'$ and $K^\times\times K^\times$ acting on the first factor  $V_0=K^2=Ke_1\oplus Ke_2$ of $W=V_0\otimes_K V$ by $(a,b)\cdot e_1=ae_1$, $(a,b)\cdot e_2=be_2$.
 
  In case (b), we let $\sigma=\sigma_{D}:G'\rightarrow \GL(W)$ be the representation above. Let $L/K$ be the degree $2$ unramified extension; we assume $L\subset D$. Let $L^\times$ act diagonally on the left on $T_0=D$ and hence on $W=D\otimes_D T$. Set $G'_1$ to be the subgroup of $\GL(W)$ generated by $G'$ and $L^\times$ acting as above. 
  
After base changing to $L$,  these groups are identified under the isomorphism induced by Morita equivalence. 
  
 We set $G_1:=\Res_{K/\Q_p}G'_1$, and $\sigma_1:G'_1\rightarrow \GL(W)$ with $W=K^2\otimes_K V$ or $W=T$, to be the canonical representation obtained as above from $\sigma_{K^2}$ or $\sigma$ in cases (a) and (b) respectively.

\begin{Remark}
{\rm The reason for considering the modification $G_1'$ is that this is the group which naturally arises when applying Deligne's construction of Hodge type liftings for abelian type Shimura datum of type $D_n^{\bbH}$. The extra factor of $K^\times\times K^\times$ or $L^\times $ in cases (a)  and (b) respectively is needed to modify the Hodge cocharacter so that the dimensions of the weight 0 and weight 1 spaces are equal in the representation $W$. This modification becomes necessary when some of the cocharacters $\mu_\phi$, $\phi: K\hook \Q_p$, that constitute $\mu$, are trivial.}
	\end{Remark}

\begin{cor}\label{cor:DnH} With notations as above, let $(G_1,\mu,\calG_1)$ be a local model triple of $D^{\mathbb H}_n$ type with $\calG_1$ a stabilizer group scheme for a point $\mathbf{x}\in \calB(G_1,\bbQ_p)$ which is generic in its facet. Let $\rho_1:G'_1\rightarrow \GL(V'')$ a direct sum of $s$ copies of $\sigma_1:G'_1\rightarrow \GL(W)$, $s\geq 1$. Then the conclusion of Theorem \ref{thm:DnH} holds for $(G_1,\mu,\calG_1)$ and $\rho_1$.
\end{cor}

\begin{proof}By Lemma \ref{diagonal}, it suffices to prove the result for $\rho_1=\sigma_1:G'_1\rightarrow \GL(W)$. Upon modifying $\mathbf{x}$ by an element of the center, we may assume it lies in the image of $\calB(G,\bbQ_p)$. 

We only discuss case (a), as case (b) is similar. As in the proof of Theorem \ref{thm:DnH}, we let $\tilde{\Lambda}\subset V_{\tilde{K}}$ be an $h$-self-dual $\Ga$-stable $\ti\calO$-lattice corresponding to the image of $\mathbf{x}$ in $\calB(\GL(V),\tilde{K})$, and set $\tilde{M}=\tilde{\Lambda}\oplus\tilde{\Lambda}$.
	
	We let $\tilde{\calG}_1'$ denote the hyperspecial parahoric for $G'_{1,\ti K}$ corresponding to the image of $\mathbf{x}$ in $\calB(G'_1,\tilde{K})$.  Then we have a scheme theoretic intersection
	$$
	\tilde\calG'_1=[\GO^+(\tilde{\Lambda})\times\GO^+(\tilde{\Lambda})]\cap \GSp(\tilde{M}).$$
As in the proof of Theorem \ref{thm:DnH}, the group scheme homomorphism
\[
\Res_{\ti\O/\bbZ_p}\GO^+(\tilde{\Lambda})\hook \GL(\ti\La)\times \GL(\ti\La)\hook \GL(\ti M)
\]
extending $\sigma_{K^2}\otimes_K\ti K$ gives a  very good Hodge embedding. Hence, by Lemma \ref{diagonal} and Lemma \ref{lem: product vg}, the embeddings $\Res_{\ti \O/\bbZ_p}\GO^+(\tilde{\Lambda})\hook \GL(\ti\La)$ and then
$$
\Res_{\ti\calO/\bbZ_p}\GO^+(\tilde{\Lambda})\times\Res_{\ti\O/\bbZ_p}\GO^+(\tilde{\Lambda})\hook \GL(\ti\La)\times\GL(\ti\La)\hook \GL(\tilde{M})
$$
are very good. By Theorem \ref{thm:main},
$$
\Res_{\ti\O/\bbZ_p}\GSp(\tilde{M})\hook \GL(\tilde{M})
$$  
also gives a very good Hodge embedding. Hence, $\Res_{\ti\O/\bbZ_p}\tilde{\calG}'_1\hook \GL(\tilde{M})$ is cut out by horizontal tensors, and hence is very good. The argument as before proves the result.
\end{proof}
\end{para}

 \subsection{Exceptional $A_n$ types}\label{ss:A_n}
 \begin{para}
 Here we give a result covering some $A_n$ types which are excluded in Theorem \ref{thm:main}, cf. Remark \ref{exclude}.

Let $G=A^*=\Res_{K/\Q_p}\GL_m(D)$, where $A={\rm M}_m(D)$  with $D$ a division central $K$-algebra. Let $V=D^m$ considered as a $\bbQ_p$-vector space and let $\rho:G\rightarrow \GL(V)$ denote the representation given by left multiplication of $A$ on $D^m$.  Similarly,  let $\overline{V}=D^{\mathrm{opp},m}$ and let $\overline{\rho}:G\rightarrow \GL(\overline{V})$ be the representation where $x\in A$ acts on $\overline{V}$ via left multiplication by $x^{-1}$.

Now let $(G, \mu, \Gg)$ be a local model triple. Write $\mu'=\rho\circ \mu$ and $\bar\mu'=\bar\rho\circ \mu$. The representations $\rho$ and $\overline{\rho}$ give local Hodge embeddings $(G,\mu)\hook (\GL(V),\mu')$, resp. $(G,\mu)\hook (\GL(\overline{V}),\bar\mu')$. By \cite{BTclass}, each point $\bx$ in the building of $G=A^*$ corresponds to a graded periodic (right) $\O_D$-lattice chain $(\L, c)$ in $V$.
By \cite[3.6, Thm]{BTclass}, the stabilizer group scheme $\Gg=\Gg_\bx$ is given as the group scheme of $\O_D$-automorphisms of the
$\O_D$-lattice chain $\L$. Thus there is a corresponding closed group scheme immersion $\Gg\hook \GL(\L)$. Similarly, there is a lattice chain $\overline{\calL}$ of right $\O_D$-modules  in $\overline{V}$ such that $\calG$ is the group scheme stabilizer of $\overline{\calL}$ under the representation $\overline{\rho}$. Then $\overline{\calL}$ has the property that there is bijection $\Lambda_i\mapsto \overline{\Lambda}_i$ between determining segments for $\calL$ and $\overline{\calL}$ such that the stabilizer of $\Lambda_i$ and $\overline{\Lambda}_i$ are identified. 
Then we obtain   a closed immersion $\Gg\hook \GL(\overline{\L})$.

\begin{prop}\label{cor:ELdiv}
The integral Hodge embeddings \[
\rho: (\Gg,\mu)\hook (\GL(\calL), \mu'),\quad \overline{\rho}: (\Gg,\mu)\hook (\GL(\overline{\calL}),\bar\mu'),
\]  are very good.
\end{prop}
\begin{proof}
Set $\Lambda=\mathrm{tot}(\calL)$ and write $\mu'=\mu_d$. Then	by Theorem \ref{thm: LM embedding} and its proof, cf. \cite[Prop. 8.1, \S 8.2.3]{PZ}, the group scheme homomorphism $\Gg\hook \GL(\La)$ induces an equivariant closed immersion $\BMloc_{\Gg, \mu}\hook \Gr(\Lambda,rd)_{\calO_E}$ and  so   $\rho^r: (\Gg,\mu)\hook (\GL(\La), \mu_{rd})$ is a good integral Hodge embedding. The fact that it is very good follows  by applying Corollary \ref{EL}. The result for $\overline{\rho}$ is proved in the same way.
\end{proof}

\begin{Remark}
	{\rm   Prop. \ref{cor:ELdiv}  is not covered by the previous results when $p$ divides the index of $D$. Note though that  this statement is restricted to   ``standard" Hodge embeddings and does not cover 
		Hodge embeddings for central quotients $(A^*/C, \mu)$ which are given by other fundamental weights. For example, these can occur when, for each $\phi$, the cocharacter $\mu_\phi$ is either of type $\varpi^\vee_1$ or is trivial. }
	\end{Remark}

\end{para}

\begin{para}\label{para: modification A-n}As in the case of type $D_n^{\bbH}$, we prove a modified version of this result in the presence of an alternating form which is needed in the global applications. 

We set $W=V\oplus \overline{V}\cong (D\times D^{\mathrm{opp}})^m$, and we let $G_1$ denote the subgroup of $\GL(W)$ generated by the image of $G$ under $\rho\oplus \overline\rho$ and $K^\times\times K^\times$, where the first and second factors of $K^\times$ correspond to scalar multiplication on $V$ and $\overline{V}$ respectively. We write $\rho_1:G_1\rightarrow \GL(W)$ for the natural representation. We define an alternating form
 $$
 \psi:W\times W\rightarrow K
 $$  
 as follows. Consider the involution $\tau$ of $D\times D^{\mathrm{opp}}$ given by $(d,d')\mapsto (d', d)$. Choose
 $\xi \in K^\times \times K^{\times}$ such that $\tau(\xi)=-\xi$, so  $\xi=\pi^a\cdot (u,-u)$, for $u\in \O^\times$, $a\in \Z$. For $x=(x_1,\dotsc,x_m)\in (D\times D^{\mathrm{opp}})^m$, $y=(y_1,\dotsc,y_m)\in (D\times D^{\mathrm{opp}})^m$, we set 
 $$
 \psi(x,y)=\sum_{i=1}^m\mathrm{Tr}_{D\times D^{\mathrm{opp}}/K}(\xi \tau(x_i) y_i)=\pi^a\cdot \sum_{i=1}^m\mathrm{Tr}_{D\times D^{\mathrm{opp}}/K}((u, -u) \tau(x_i) y_i).
 $$
 Then we have $G_1=(G\times G)\cap\mathrm{GSp}(W, \psi)$.

For a lattice chain $\calL'$ of $\bbZ_p$-modules in a direct sum $W^s$ of $W$, we let $\calL'^\vee$ denote the lattice chain whose constituent lattices are given by the dual of those in $\calL'$ with respect to the form $[\Tr_{K/\bbQ_p}\circ\delta_{K/\bbQ_p}^{-1}\psi]^s$.

\begin{cor}\label{cor:PELdiv}
	Consider $\mathbf{x}\in\calB(G_1,\bbQ_p)$ with  corresponding stabilizer group scheme $\calG_1$, and let $(G_1,\mu,\calG_1)$ be a local model triple. Then there is a self-dual lattice chain $\calL'$ 	in $W^s$	such that   $\rho_1^s$ extends to a very good Hodge embedding $(\calG_1,\mu)\rightarrow (\GL(\calL),\mu_{ds})$.
\end{cor}
\begin{proof}
	By Lemma \ref{diagonal}, it suffices to prove this for the representation $\rho_1$. Upon modifying $\mathbf{x}$ by an element of the center of $G_1$, we may assume it lies in the image of $\calB(G,\bbQ_p)$. Then, as above, $\mathbf{x}$ corresponds to a  lattice chain $\calL$ in $V$ and a lattice chain $\overline{\calL}$ in $\overline{V}$. We let $\calL'$ denote the (periodic) lattice chain in $W$ whose constituent lattices are the scalar multiples of $\Lambda_i':=\Lambda_i\oplus\overline{\Lambda}_i$ for $\Lambda_i$, resp. $\overline{\Lambda}_i$, members of a determining segment for $\calL$, resp. $\overline{\calL}$. We can choose $\overline{\Lambda}_i$ so that $\Lambda_i'$ is self-dual for $\psi$. Then $\calL'$ is a self dual lattice chain in $W$, and  for $\Lambda'=\mathrm{tot}(\calL')\subset W^r$, the embedding $\calG\times \calG\rightarrow \GL(\Lambda')$  is a very good Hodge embedding by Corollary \ref{cor:ELdiv} and Lemma \ref{lem: product vg}.  
	
We let $\psi'$ denote the alternating form on $W^r$ given by the sum of those on $W$; then $\Lambda'$ is self dual for $\psi'$.
	We have a scheme-theoretic intersection $\calG_1=(\calG\times\calG)\cap \mathrm{GSp}(\Lambda')$. Hence, 
	by Theorem \ref{thm:main} applied to $\GSp(\La')\hook \GL(\La')$ and the above, we see that $\calG_1\rightarrow \GL(\Lambda')$ gives a very good Hodge embedding. 
	\end{proof}

\end{para}

\section{Shimura varieties}\label{sec:SV}

In this section, we use the local results of \S \ref{sec: very good embeddings} to obtain our main results on integral models of Shimura varieties. 

\subsection{Integral models}\label{subsec: integral models}
\begin{para}  Let $(\bfG,X)$ be a Shimura datum in the sense of \cite{DelBour} so that $\bfG$ is a reductive group over $\bbQ$ and $X$ is a $\bfG_{\bbR}$-conjugacy class of homomorphisms $\bbS:=\mathrm{Res}_{\bbC/\bbR}\bbG_m\rightarrow \bfG_{\bbR}$.
We fix a prime $p>2$ and write $G$ for the base change of $\bfG$ to $\bbQ_p$. Let $\bbA_f$ denote the ring of finite adeles and $\bbA_f^p$ the ring of prime-to-$p$ adeles which we consider as the subgroup of $\bbA_f$ with trivial $p$-component. Let $\rmK_p\subset \bfG(\bbQ_p)$ and $\rmK^p\subset\bfG(\bbA_f)$ be  compact open subgroups  and write $\rmK:=\rmK_p\rmK^p$. Then if $\rmK^p$ is sufficiently small, we have the associated Shimura variety $\Sh_{\rmK}(\bfG,X)$ defined over the reflex field $\bfE\subset \bbC$ whose complex points are given by the double quotient
$$
\Sh_{\rmK}(\bfG,X)(\bbC)=\bfG(\bbQ)\backslash X\times \bfG(\bbA_f)/\rmK;
$$
see  \cite{DeligneCorvallis} for the construction in the case of Shimura varieties of abelian type, which is all that we consider in this paper. Here, $\bfE$ is  defined to be the field of definition of the conjugacy class of Hodge cocharacters $\{\mu_h\}$ associated to $h$.

We also define the pro-variety $$\Sh_{\rmK_p}(\bfG,X):=\lim_{\leftarrow\rmK^p}\Sh_{\rmK_p\rmK^p}(\bfG,X)$$ 

\end{para}
\begin{para}We now assume that there is an embedding of Shimura data $$\iota:(\bfG,X)\rightarrow (\mathbf{GSp}(V),S^{\pm})$$ with $\mathbf{GSp}(V)$ the group of symplectic similitudes of a $\bbQ$-vector space $V$ of dimension $2d$ equipped with a perfect alternating bilinear
form $\psi$, and $S^\pm$ is the Siegel double space. We call $\iota$ a Hodge embedding.

Let  $v|p$ be a prime of $\bfE$ and let   $E$ denote the completion of $\bfE$ at $v$. 
We let $k_E$ denote the residue field at $v$ and we fix an algebraic closure $k$ of $k_E$. Let $\calG$ be the Bruhat--Tits stabilizer group scheme corresponding to some $\bx\in \calB(G,\bbQ_p)$ which is generic in its facet. We obtain  a local model triple  $(G,\{\mu_h\},\calG)$ with attached local model $\BMloc_{\calG,\mu_h}$. We now make the following assumptions.

\begin{altenumerate}
	\item[(A)] $\rmK_p=\calG(\bbZ_p)$.
	\item[(B)] $G$ is $R$-smooth and $p\nmid |\pi_1(G^{\der})|$. 
	\item[(C)] $\iota_{\bbQ_p}:G\rightarrow \GL(V_{\Q_p})$ extends to a very good Hodge embedding $(\calG,\mu_h)\rightarrow (\GL(\Lambda),\mu_d)$ where $\Lambda\subset V_{\bbQ_p}$ is a  $\bbZ_p$-lattice which is contained in its $\psi$-dual.
\end{altenumerate}
We write $\rmK_p'$ for the stabilizer in $\mathbf{GSp}(V_{\bbQ_p})$ of the lattice $\Lambda$ and we fix $\rmK'^p\subset\bfG(\bbA_f^p)$ a compact open subgroup containing $\rmK^p$. We set $\rmK'=\rmK'_p\rmK'^p$. We then obtain a morphism of Shimura varieties 
$$
\Sh_{\rmK}(\bfG,X)\rightarrow \Sh_{\rmK'}(\bfGSp(V),S^\pm)_{\bfE}
$$ 
which  is a closed immersion if $\rmK'^p$ is sufficiently small.

We set $V_{\bbZ_{(p)}}:=V\cap \Lambda$ which is a $\bbZ_{(p)}$-submodule of $V$, and we let $\G_{\bbZ_{(p)}}$ denote the Zariski closure of $\G$ in $\GL(V_{\bbZ_{(p)}})$. The choice of $V_{\bbZ_{(p)}}$ gives rise to an interpretation of $\Sh_{\rmK'}(\bfGSp(V),S^\pm)$ as a moduli space of polarized abelian varieties, and hence to an integral model $\SSh_{\rmK'}(\bfGSp(V),S^\pm)$	over $\bbZ_{(p)}$, cf. \cite[\S 6.3]{Zhou}. We define the integral model $\SSh_{\rmK}(\bfG,X)$ over $\calO_E$  to be the normalization of the Zariski closure of $\Sh_{\rmK}(\bfG,X)$ in $\SSh_{\rmK'}(\bfGSp(V),S^\pm)_{\calO_E}$.  
Under these assumptions, the following theorem summarizes the main results concerning $\SSh_{\rmK}(\bfG,X)$.  

\begin{thm}[cf. \cite{KP},\,\cite{KZhou}]\label{thm: main SV Hodge}
Under the assumptions (A), (B) and (C), the schemes $\SSh_{\rmK}(\bfG,X)$ satisfy the following properties.
\begin{altenumerate}\item For $R$ a discrete valuation ring of mixed characteristic $(0,p)$, we have a bijection 
$$
\varprojlim_{\rmK^p}\SSh_{\rmK_p\rmK^p}(\bfG,X)(R)=\Sh_{\rmK_p}(\bfG,X)(R[1/p]).
$$

\item There exists a local model  diagram 
	\[
	\xymatrix{ & \widetilde{\SSh}_{\rmK}(\bfG,X)\ar[ld]_{\pi}\ar[rd]^q&\\ \SSh_{\rmK}(\bfG,X)& &\BMloc_{\calG,\mu_h}}
	\] 
	where $\pi$ is a $\calG$-torsor and $q$ is  $\calG$-equivariant and smooth of relative dimension $\dim G$.
	
	\item If in addition, we have $\calG=\calG^\circ$, i.e. the stabilizer group scheme is connected, then for each $x\in \SSh_{\rmK}(\bfG,X)(k')$ with $k'/k_E$ finite, there is a point $y\in \BMloc_{\calG,\mu_h}(k')$ such that we have an isomorphism of henselizations 
	$$
	\calO^{\rm h}_{\SSh_{\rmK}(\bfG,X),x}\simeq \calO^{\rm h}_{ \BMloc_{\calG,\mu_h},y}.
	$$
	\end{altenumerate}
\end{thm}

\begin{Remark} 
	{\rm\begin{altenumerate}

\item In the reference \cite{KP} and  previous versions of \cite{KZhou}, the assumption (C) concerning the property of a \emph{very good} (as opposed to just \emph{good}) embedding was erroneously omitted. With this assumption in place, the result follows from the proofs in \emph{op. cit.}. We recall  the argument and the role played by assumption (C) below.
	\item 	The results in \S\ref{sec: very good embeddings} shows that Assumption (C) is satisfied in many cases. In the following subsection, we will show that the cases covered by those  results are sufficient to construct good integral models in all abelian type settings.
\end{altenumerate}
	}
\end{Remark}

\begin{proof}[Proof of Theorem \ref{thm: main SV Hodge}]Property (1) follows by the construction of the models and the N\'eron--Ogg--Shafarevich criterion.	For (2) and (3), we fix a collection of tensors $s_\alpha\in V_{\bbZ_{(p)}}^\otimes$ whose stabilizer is  $\G_{\bbZ_{(p)}}$. The Betti-\'etale comparison isomorphism gives  corresponding  tensors $s_{\alpha,\et}\in \calV_p^\otimes$, where $\calV_p$ is the $\bbZ_p$-local system on $\Sh_{\rmK}(\bfG,X)$ corresponding to the  dual of the $p$-adic Tate-module of the pullback of the  universal abelian variety $\calA$ obtained by pullback from $\SSh_{\rmK'}(\mathbf{GSp}(V),S^\pm)_{\calO_E}$.

For $x\in\SSh_{\rmK}(\bfG,X)(k)$, we let $\scrG_x:=\calA_x[p^\infty]$ denote the $p$-divisible group over $k$ associated to the pullback $\calA_x$ of $\calA$ along $x$, and $\bbD$ the Dieudonn\'e module of $\scrG_x$. Then for $K/\breve \bbQ_p$ finite and  $\tilde{x}\in\SSh_{\rmK}(\bfG,X)(\calO_K)$ a point lifting $x$, the $p$-adic comparison isomorphism gives rise to tensors $s_{\alpha,0}\in \bbD[1/p]^\otimes$, which lie in the submodule $\bbD^\otimes$ by the argument in \cite[\S 3.3]{KP} and are independent of the choice of lift $\tilde{x}$. Moreover, the scheme of tensor preserving isomorphisms $\underline{\Isom}_{s_\alpha,s_{\alpha,0}}(V^\vee_{\bbZ_p},\bbD)$ is a trivial $\calG$-torsor. Here, one needs to use the purity result \cite[Prop. (10.3)]{An} or \cite[Thm. A.3.2]{PRShtuka}, instead of \cite[Prop. 1.4.3]{KP}. This construction globalizes to give the $\calG$-torsor $\widetilde{\SSh}_{\rmK}(\bfG,X)$  by considering the scheme of tensor preserving trivializations of the de Rham cohomology of $\calA$, and the $\calG$-equivariant morphism $q$ is induced by pulling back the Hodge filtration along this trivialization; see \cite[Thm. 4.2.7]{KP}.

The assumption (C) is used in showing (3) and the smoothness of $q$ in (2). More precisely, given $x\in \SSh_{\rmK}(\bfG,X)(k)$, the filtration on $\bbD\otimes_{\breve \bbZ_p}k$ corresponds to a point $y\in \BMloc_{\calG,\mu_h}(k)$. We let $R_G$ (resp. $R$) denote the completion of  local ring of $\BMloc_{\calG,\mu_h}$ (resp. $\Gr(d, \Lambda)$) at $y$. Under assumption (C), the construction in \cite[3.2.12]{KP} goes through and it produces a versal $p$-divisible group $\scrG$ over $\Spf R_E$, see  \cite[Lem. 3.1.12]{KP} and \S\ref{par:Psi}. The Dieudonn\'e display of the restriction of $\scrG$ to $\Spf R_G$ carries tensors  that lift $s_{a,0}$ and \cite[Prop. 2.3.17]{KP} gives a crucial property of $\scrG$, see also \S \ref{par:adapted} below.   The argument in \cite[Prop. 4.2.2, Thm. 4.2.7]{KP} now
shows that we have an isomorphism of completions $\widehat{\calO}_{\SSh_{\rmK}(\bfG,X),x}\cong R_G$, and that  $q$ is smooth.  The isomorphism of henselizations in (3)  then follows formally using (2) and the fact that the torsor $\widetilde{\SSh}_{\rmK}(\bfG,X)$ is for a connected group scheme. 
\end{proof}

\end{para}

\begin{para}\label{par:adapted}  The versal $p$-divisible group $\scrG$ over $\Spf R_E$, which is constructed in the course of the above proof, satisfies the following property: For $K/\breve\bbQ_p$ finite, a local ring homomorphism $u:R\rightarrow \calO_K$  factors through $R_G$ if and only if $\scrG_u$ is $(\calG,\mu_h)$-adapted in the sense of \cite[Def. 3.2.4]{KZhou}, cf. \cite[\S7.1]{PRlsv}. Hence, as a byproduct of the above  argument, we also obtain the following deformation theoretic description of the formal neighbourhood $\widehat{U}_x$ of $x\in \SSh_{\rmK}(\bfG,X)(k)$.
	\begin{prop}
	 Let $K/\breve\bbQ_p$ be finite. Then  
	 a deformation $\scrG_{\calO_K}$ of $\scrG_x$  over $\O_K$ corresponds to an $\O_K$-point of $\widehat{U}_x$ if and only if $\scrG_{\calO_K}$ is $(\calG,\mu_h)$-adapted.
		\end{prop}
	\begin{proof} This follows from the above, and from \cite[Prop. 2.3.17]{KP} and its proof. See \cite[Prop. 4.1.9]{KZhou}.
		\end{proof}
 
	\end{para}

\begin{para}\label{par:sh}
Before continuing, let us mention that if we are willing to replace henselization by strict henselization in Theorem \ref{thm: main SV Hodge} (3), there is a more general result available which does not require  assuming (B) or ``very good" in (C). The proof of this result uses, in addition to the above, results on $p$-adic shtukas.

\begin{thm}\label{thm:sh}
Let $(\G, X)$ be a Shimura datum of Hodge type. Suppose $p>2$ and let $\Gg$ be a stabilizer group scheme for $G=\G_{\Q_p}$.   
Let $\iota:(\bfG,X)\rightarrow (\mathbf{GSp}(V),S^{\pm})$ be a Hodge embedding and suppose there is a 
 self dual periodic $\bbZ_p$-lattice chain $\L$ in $V_{\bbQ_p}$ such that 
\[
\Gg(\br\Z_p)=\iota^{-1}_{\br\Q_p}(\GSp(\L)(\br\Z_p))\cap G(\br\Q_p).
\]
Let $\SSh_{\rmK}(\bfG,X)$  for $\rmK_p=\Gg(\Z_p)$, be the normalization of the Zariski closure of $\Sh_{\rmK}(\bfG,X)$ in the Siegel moduli scheme
with parahoric level given by $\L$, as above. Then for each $x\in \SSh_{\rmK}(\bfG,X)(k)$, there exists  $y\in \BMloc_{\calG,\mu_h}(k)$ such that there is an isomorphism of (strict) henselizations 
$$
\calO^{\rm sh}_{\SSh_{\rmK}(\bfG,X),x}\simeq \calO^{\rm sh}_{ \BMloc_{\calG,\mu_h},y}.
$$
\end{thm}

\begin{proof}
Given $x\in \SSh_{\rmK}(\bfG,X)(k)$, a point $y\in \BMloc_{\calG,\mu_h}(k)$ is provided as above. By \cite[Cor. 2.6]{Artin}, it is enough to show that 
there is an isomorphism 
\[
\widehat{\calO}_{\SSh_{\rmK}(\bfG,X),x}\simeq \widehat{\O}_{\BMloc_{\calG,\mu_h}, y}
\]
between the completions of the local rings of $\SSh_{\rmK}(\bfG,X)\otimes_{\O_E}\O_{\breve E}$ and 
$\BMloc_{\calG,\mu_h}\otimes_{\O_E}\O_{\breve E}$ at $x$ and $y$ respectively. Note that both these rings are normal. 

If $\Gg=\Gg^\circ$, i.e. the stabilizer $\Gg$ is parahoric, then \cite[Thm. 1.3.2 (c)]{PRShtuka} implies that the $v$-sheaf associated to $\widehat{\calO}_{\SSh_{\rmK}(\bfG,X),x}$ is isomorphic to the $v$-sheaf given by the ``formal completion" of a corresponding integral moduli of $\Gg$-shtuka. For stabilizers $\Gg$ which are not necessarily connected, the same result follows by \cite[Thm. 4.2.3]{DvHKZ} and its proof (this extends \cite[Thm. 1.3.2]{PRShtuka}).
By \cite[Thm. 2.5.5]{PRlsv}, this formal completion is in turn isomorphic to the $v$-sheaf represented by $\widehat{\O}_{\BMloc_{\calG,\mu_h}, y}$. The result then follows by the full-faithfulness of the diamond functor, \cite[Prop. 18.4.1]{Schber}.
\end{proof}

\begin{Remark}
{\rm a) The proof of \cite[Thm. 2.5.5]{PRlsv} and hence of Theorem \ref{thm:sh} relies on the results in the present paper and, in particular, on the results about very good embeddings in \S \ref{s:Displays} and \S\ref{sec: very good embeddings}. 

b) Under the assumptions of Theorem \ref{thm:sh}, \cite[Thm. 1.3.2]{PRShtuka} and \cite[Thm. 4.2.3]{DvHKZ} imply that $\SSh_{\rmK}(\bfG,X)$ is the canonical integral model of ${\rm Sh}_{\rmK}(\G,X)$ in the sense of \cite{PRShtuka}. Hence, by \emph{loc. cit.}, $\SSh_{\rmK}(\bfG,X)$ is independent of the choice of Hodge embedding and lattice.  

c) The stronger result of Theorem \ref{thm: main SV Hodge} (3) concerning henselizations, as well as the local model diagram in (2), is needed in applications towards determining the local zeta factors of the Shimura variety over $p$ via the Langlands-Kottwitz method, cf. \cite{HZZ}.}
\end{Remark}
\end{para}

	\begin{para}\label{para: ab type reduction assumptions}We now deduce  corresponding results for Shimura varieties of abelian type and for parahoric level (as opposed to stabilizer level). We continue to fix $p>2$ and let  $(\bfG,X)$ be a Shimura datum of Hodge type with reflex field $\bfE$ as above and we assume that it satisfies assumptions (A), (B) and (C). We also introduce two further assumptions. As before, for a group scheme $\bfH$ over $\bbQ$, we write $H$ for its base change to $\bbQ_p$. We also write $\bfC$ for the kernel of the morphism $\bfG^{\mathrm{sc}}\rightarrow \bfG^{\der}$, where $\bfG^{\mathrm{sc}}$ is the simply-connected cover of the derived group $\bfG^{\der}$.
		
		\begin{altenumerate}
			\item[(D)] If $c\in\rmH^1(\bbQ,\bfC)$ satisfies $c_\ell=0$ for all $\ell\neq p$, then $c_p=0$, cf. \cite[(4.3.4)]{KP}.
			\item[(E)] The center of $Z_G$ of $G$ is an $R$-smooth torus.
		\end{altenumerate}
		
	We set $\rmK_p^\circ=\calG^\circ(\bbZ_p)$ and $\rmK^\circ=\rmK_p^\circ\rmK^p$.	There is a natural finite map of Shimura varieties $\Sh_{\rmK^\circ}(\bfG,X)\rightarrow \Sh_{\rmK}(\bfG,X)$, and we define the integral model $\SSh_{\rmK^{\circ}}(\bfG,X)$ to be the normalization of $\SSh_{\rmK}(\bfG,X)$ in $\Sh_{\rmK^\circ}(\bfG,X)_E$.
	The discussion in \cite[\S4.3]{KP} extends verbatim to our setting and we obtain the following proposition, cf. \cite[Prop. 4.3.9]{KP}.
		
		\begin{prop}
			Assume (A)--(D) are satisfied. 
			
			\begin{altenumerate}
				\item The natural map $\SSh_{\rmK^\circ}(\bfG,X)\rightarrow \SSh_{\rmK}(\bfG,X)$ is \'etale.
				\item The geometric connected components of $\SSh_{\rmK^\circ}(\bfG,X)$  are defined over the maximal extension $\bfE^p$  of $\bfE$ unramified at all places lying above $p$.
			\end{altenumerate}
		\end{prop}

	\end{para}	
	
	\begin{para}\label{par:7.1.12}
	Now let $(\bfG_2,X_2)$ be  a Shimura datum which is equipped with a central isogeny 
		$\alpha:\bfG^{\der}\rightarrow \bfG_{2}^{\mathrm{der}}$ inducing an isomorphism $(\bfG^{\ad},X^{\ad})\cong (\bfG_2^{\mathrm{ad}},X_2^{\mathrm{ad}})$. Let $\bx^{\ad}$ be the image of $\bx$ 
		in $\calB(G^{\ad},\bbQ_p)$ and we fix $\bx_2\in \calB(G_2,\bbQ_p)$ lifting $\bx^{\ad}$. We write  $\calG_2$ (resp.   $\calG_2^{\circ}$) for the stabilizer group scheme (resp. parahoric) corresponding to the point $\bx_2$. 
		In this case, we say that the stabilizer group scheme $\calG$ lifts $\calG_2$. We also set $\calG^{\ad}:=\calG/\calZ$ where $\calZ$ is the Zariski closure of $Z_G$ inside $\calG$, and we let $\calG^{\ad,\circ}$ denote its neutral component. 
		Note that in general, $\calG^{\ad}$ is not necessarily the Bruhat--Tits stabilizer group scheme associated to $\bx^{\ad}$. 
		However, assumption (E) implies that $\calG^{\ad}$ is smooth and $\calG^{\ad,\circ}$ is equal to the parahoric group scheme associated to $\bx^{\ad}$, cf. \cite[Lemma 4.6.2]{KP}, \cite[Prop. 2.4.13]{KZhou}. 
		We set $\rmK_{2,p}:=\calG_2(\bbZ_p)$ and $\rmK_{2,p}^\circ=\calG_2^{\circ}(\bbZ_p)$. We write $\bfE_2$ for the reflex field of $(\bfG_2,X_2)$ and we let $\bfE':=\bfE.\bfE_2.$ 
		We fix a place $v'$  of  $\bfE'$  lying above $v$  and we set $E':=\bfE'_{v'}$ to be the completion at $v'$.
		
		Fix  a connected component $X^+\subset X$. By real approximation, upon modifying the isomorphism $\bfG^{\ad}\cong \bfG_2^{\mathrm{ad}}$ by an element of $\bfG^{\ad}(\bbQ)$,  we may assume that the image of  $X_2\subset X_{2}^{\mathrm{ad}}$ contains the image of $X^+.$ We write $X_2^+$ for $X^+$ viewed as a subset of $X_2.$  
		We denote by $\Sh_{\rmK^\circ_p}(\bfG, X)^+ \subset \Sh_{\rmK^\circ_p}(\bfG, X)$ and 
		$\Sh_{\rmK_{2^\circ,p}}(\bfG_2, X_2)^+  \subset \Sh_{\rmK^\circ_{2,p}}(\bfG_2, X_2)$ the geometrically connected components corresponding 
		to $X^+$ and $X_2^+$. These are defined over extensions  of $\bfE$ and $\bf E'$ respectively, which are unramified at primes above $p$ by Assumption (D). The identification $X_2^+ \simeq X^+$ induces a finite map 
		\begin{equation}\label{eqn:mapconncomps}
			\Sh_{\rmK^\circ_p}(\bfG, X)^+ \rightarrow \Sh_{\rmK^\circ_{2,p}}(\bfG_2, X_2)^+ 
		\end{equation}
		
		We then have the following generalization of \cite[Cor. 4.6.18]{KP}.
		
		\begin{prop}\label{prop: auxiliary SV construction}Under the assumptions (A)--(E), there is a $\bfG_2(\bbA^p_f)$-equivariant extension of $\Sh_{\rmK^\circ_{2,p}}(\bfG_2, X_2)$ to an $\calO_{E'}$-scheme with 
			$\bfG_2(\bbA^p_f)$-action 
			$\mathscr{S}_{\rmK^\circ_{2,p}}(\bfG_2,X_2)$ such that
			\begin{altenumerate}
				\item For  $R$  a discrete valuation ring of mixed characteristic  $(0,p)$, the map 
				$$\mathscr{S}_{\rmK^\circ_{2,p}}(\bfG_2,X_2)(R)\rightarrow\Sh_{\rmK^\circ_{2,p}}(\bfG_2,X_2)(R[1/p])$$ 
				is a bijection.
				\item The map (\ref{eqn:mapconncomps}) induces a finite map of $\O_{E^{\prime\ur}}$-schemes
				$$\mathscr{S}_{\rmK^\circ_p}(\bfG, X)^+ \rightarrow \mathscr{S}_{\rmK^\circ_{2,p}}(\bfG_2, X_2)^+, $$ 
				where $\mathscr{S}_{\rmK^\circ_{2,p}}(\bfG_2, X_2)^+$ denotes the closure of $\Sh_{\rmK^\circ_{2,p}}(\bfG_2, X_2)^+$ 
				in the $\calO_{E'^{\mathrm{ur}}}$-scheme  $\mathscr{S}_{\rmK^\circ_{2,p}}(\bfG_2, X_2)_{\calO_{E'^{\mathrm{ur}}}},$ and similarly for $\mathscr{S}_{\rmK^\circ_{p}}(\bfG, X)^+.$
				\item There exists a diagram 
				\begin{equation}\label{eqn: local model diagram abelian type}
				\begin{aligned}\xymatrix{ &\widetilde{\mathscr{S}}^{\mathrm{ad}}_{\rmK^\circ_{2,p}}(\bfG_2,X_2)\ar[dr]^q\ar[dl]_\pi&\\
						\mathscr{S}_{\rmK^\circ_{2,p}}(\bfG_2,X_2) & &\bbM^{\mathrm{loc}}_{\Gg^\circ_2, \mu_{h_2}}\otimes_{\calO_E}\calO_{E'}}
						\end{aligned}\end{equation}
						where $\pi$ is a $\bfG_2(\bbA_f^p)$-equivariant ${\calG}^{\ad}$-torsor and $q$ is $\Gg^\ad$-equivariant, smooth of relative dimension				 $\dim \bfG^{\ad},$ and 
				$\bfG_2(\bbA_f^p)$-equivariant for the trivial $\bfG_2(\bbA_f^p)$-action on $\bbM^{\mathrm{loc}}_{\calG_2, \mu_{h_2}}$. 
				If in addition, we have $\calG=\calG^\circ$, then $\pi$ reduces to a $\calG^{\ad,\circ}$ torsor. 
			\end{altenumerate}
		\end{prop}
		\begin{proof}  This is deduced from Theorem \ref{thm: main SV Hodge} by following the arguments in \cite[\S 4.4-\S 4.6]{KP}  
		and noting that we have an equivariant isomorphism $\BMloc_{\Gg, \mu_h}\otimes_{\O_E}\O_{E'}\cong \BMloc_{\Gg^\circ_2,\mu_{h_2}}\otimes_{\O_{E_2}}\O_{E'}$
		obtained by combining the isomorphisms induced from $G_2\to G_2^\ad$ and $G\to G^\ad\cong G_2^\ad$ by \cite[Prop. 21.5.1]{Schber} and the full-faithfulness of the diamond functor. We sketch some details, enough to explain how we use assumption (E). 
		
		Let $\G^\circ_{\Z_{(p)}}$ (resp. $\G^{\ad,\circ}_{\Z_{(p)}}$) denote the $\Z_{(p)}$-model of $\G$ (resp. $\G^\ad$) associated to
$\Gg^\circ$ (resp. $\Gg^{\ad,\circ}$) via the construction in \cite[\S 4.6.1]{KP}. 
Let ${\mathbf Z}_{\bfG}$ denote the center of $\G$
and ${\mathbf Z}_{\Z_{(p)}}$ the closure of ${\mathbf Z}_{\G}$ in $\G^\circ_{\Z_{(p)}}$. The assumption of
$R$-smoothness on the torus $Z_G$ and descent implies that ${\mathbf Z}_{\Z_{(p)}}$ and $\G^{\ad,\circ}_{\Z_{(p)}}$ are smooth and that the 
 $p$-adic completion of $\G^{\ad,\circ}_{\Z_{(p)}}$  agrees with the 
parahoric group scheme of $G^\ad$ associated to $\bx^\ad$. 
This gives us the analogue of \cite[Lem. 4.6.2(2)]{KP}
and allows us to carry out the constructions of \S 4.6 of loc. cit.
	\end{proof}

		\end{para}
\subsection{Existence of very good Hodge type liftings}

\begin{para}\label{para: Deligne}
In order to obtain  unconditional results, we show in this subsection that given an abelian type Shimura datum $(\bfG_2,X_2)$, we can find a Hodge type Shimura datum $(\bfG,X)$ satisfying assumptions (A)-(E).	We carry this out in two steps. First we consider the case when $\bfG_2^{\ad}$ is simple; this case is divided into two parts, the non-exceptional (NE) case and the exceptional type $A$ and $D_n^{\bbH}$ cases. The last step consists of deducing the case of general $\bfG_2$ from the case where $\bfG_2^{\ad}$ is simple using a modified product construction.

We begin by recalling Deligne's construction of Hodge type liftings in \cite{DeligneCorvallis}. Let $H$ be a simple, adjoint, reductive group over $\mathbb R,$ which is of classical type, and is 
associated to a Hermitian symmetric domain; in particular $H(\mathbb R)$ is not compact. 
Thus $H$ is of type $A, B, C, D^{\mathbb R}, D^{\mathbb H}$ in the classification 
of \cite[1.3.9]{DeligneCorvallis}, with the type $A$ case including unitary groups of any signature $U(p,q)$ with $p,q \neq 0.$
We set $H^\sharp = H^{\mathrm{sc}},$ the simply connected cover of $H,$ unless $H$ is of type $D^{\mathbb H},$ 
in which case we set $H^\sharp$ equal to the image of $H^{\mathrm{sc}}$ in the representation corresponding to the standard 
representation of the orthogonal group. 

Now let $\rmF$ be a totally real field, and $\bfH$ a simple, adjoint reductive group of classical type over $\rmF.$ 
Assume that
\begin{altitemize} 
	\item for every embedding $\sigma: \rmF \hookrightarrow \mathbb R,$ $\bfH\otimes_{\sigma,\rmF} \mathbb R$ is either compact or associated to a Hermitian symmetric domain.
	\item $\bfH\otimes_{\sigma,\rmF} \mathbb R$ is non-compact for some $\sigma.$ 
	\item If $\bfH$ is of type $D,$ then for those $\sigma$ such that $\bfH\otimes_{\sigma,F} \mathbb R$ is non-compact, the associated Hermitian symmetric domain does not depend on $\sigma.$ That is, it is always of type $D^{\mathbb R}$ or always of type $D^{\mathbb H}.$
\end{altitemize}
We define $\bfH^{\sharp}$ to be $\bfH^{\mathrm{sc}}$ unless $\bfH$ is of type $D,$ in which case we define $\bfH^{\sharp}$ 
to be the unique quotient of $\bfH^{\mathrm{sc}}$ such that 
$\bfH^{\sharp}\otimes_{\sigma,F} \mathbb R = (\bfH\otimes_{\sigma,F} \mathbb R)^{\sharp}$ whenever 
$\bfH\otimes_{\sigma,F} \mathbb R$ is non-compact.

Now suppose $\bfH$ is a reductive group over $\rmF,$ with $\bfH^{\ad} = \prod_{i=1}^s \bfH_i$ where each $\bfH_i$ is 
a simple, adjoint reductive group of classical type over $F$ satisfying the three conditions above. 
Then we set $\bfH^{\sharp} = \prod_{i=1}^s \bfH_i^{\sharp}.$

Now let $(\bfH,Y)$ be a Shimura datum such that $(\bfH^{\ad},Y^{\ad})$ is of abelian type. 
Recall \cite[1.3.10, 2.3.10] {DeligneCorvallis} that in this case the three conditions above are satisfied, so $\bfH^{\sharp}$ is well defined
\footnote{In \cite[4.6.21]{KP} it is incorrectly asserted that $\bfH^{\sharp}$ is defined for any $(H,Y)$ with $H$ of classical type, 
	however $H$ may not satisfy the third condition above. This is however satisfied if $(\bfH^{\ad},Y^{\ad})$ is of abelian type.}, and $(\bfH,Y)$ is of abelian type if and only if $\bfH^{\der}$ is a quotient of 
$\bfH^\sharp$.

\end{para}\begin{para}

Let $(\bfG_2,X_2)$ be a Shimura datum of abelian type such that $\bfG_2^{\ad}$ is  $\bbQ$-simple. Then $\bfG_2^{\ad}\cong \Res_{\rmF/\bbQ}\bfH$ for $\bfH$ an absolutely simple group over $\rmF$. Let $I$ be the set of real places of $\rmF$, and   $I_{nc}$ (resp. $I_c$) the set of places where $\bfH$ is non-compact (resp. compact). 
	
	For $v\in I$, we write $D_v$ for the Dynkin diagram of $H_v:=\bfH\otimes_{\rmF,v}\bbC$; then the Dynkin diagram $D$ of $\bfG_{\bbC}$ is the union of the $D_v$.  We write $D_{nc}$ (resp. $D_c$) for the union of the $D_v$ for $v\in I_{nc}$ (resp. $v\in I_{c}$).
	
	Let $S\subset D$ be a subset of vertices of $D$ such that
	\begin{altenumerate}
		\item $S$ is stable under $\Gal(\overline{\bbQ}/\bbQ)$.
		\item $S\cap D_{nc}$ is a subset of the underlined vertices in Deligne's table \cite[1.3.9]{DeligneCorvallis}.
	\end{altenumerate}

	For $s\in S$, let $W(s)$ be the irreducible complex representation of $\bfG^{\mathrm{sc}}$ with highest weight corresponding to $S$. Then for suitable $n$, there is a representation $W$ of $\bfG^{\mathrm{sc}}$ defined over $\bbQ$ such that the representation $ \oplus_{s\in S}W(s)^n\cong W_{\bbC}$. Let  $W_s\subset W_{\bbC}$ denote the subspace $W(s)^n$. As in \cite{DeligneCorvallis}, we identify $S$ with $\Hom(\rmK_S,\bbC)$ for $\rmK_S$ a suitable product of totally real or CM fields, and we obtain an action of $\rmK_S$ on $W$ via the decomposition $W_{\bbC}\cong \oplus_{s\in S}W(s)^n$. 
	
\end{para}\begin{para}\label{para: Deligne Hodge type lifting}	In what follows, we choose $S$ as follows:
	\begin{altitemize}
		\item If $(\bfG_2^{\ad},X_2^{\ad})$ is not of type $A$ or of type $D_n^{\bbH}$, then we choose $S$ maximal satisfying the two conditions above (this is the choice used in \cite[Proposition 2.3.10]{DeligneCorvallis}). 
		
		\item If $(\bfG_2^{\ad},X_2^{\ad})$ is of type $A_n$, we choose $S$ to be $S=\{\varpi_{v,1},\varpi_{v,n}|v\in I\}$ i.e. the union of the leftmost and rightmost vertices in $D_v$ in \cite[Table 1.3.9]{DeligneCorvallis} for each $v$. Then $S$ is a single orbit for the action of $\Gal(\overline{\bbQ}/\bbQ)$,  since complex conjugation acts on $D_v$ via the opposition involution. Thus $\rmK_S$ is a CM extension of $F$.
		
		\item If $(\bfG_2^{\ad},X_2^{\ad})$ is of type $D_n^{\bbH}$, then we choose $S=\{\varpi_{v,1}|v\in I\}$, i.e. in each $D_v$ we choose the leftmost vertex in \cite[Table 1.3.9]{DeligneCorvallis}. Then $\rmK_S=\rmF$.
	\end{altitemize}
	In each case we find that the largest quotient of $\bfG^{\mathrm{sc}}$ through which the representation $\bfG^{\mathrm{sc}}\rightarrow \mathbf{GL}(W)$ factors is $\bfG^{\sharp}:=\Res_{\rmF/\bbQ}\bfH^{\sharp}$.

	Let $\rmK$ be a CM extension of $\rmF$ disjoint from $\rmK_S$ such that every prime of $\rmF$ lying above $p$ splits in $\rmK$, and we fix a set $T$ of embeddings $\rmK\rightarrow \bbC$ satisfying the same conditions in \cite[Lemma 4.6.22]{KP}.
	We let $V=W\otimes_{\rmF}\rmK$  which we consider as a vector space over $\bbQ$ and let $\bfG''\subset \mathbf{GL}(V)$ be the subgroup generated by $\rmK_S^\times, \mathrm{Res}_{\rmF/\bbQ}\bfH^{\sharp}$ and $\rmK^\times$ (this is the group $G_3$ in Deligne's notation). We let $\bfG'\subset \bfG''$ be the subgroup generated by $\Res_{\rmF/\bbQ}\bfH^\sharp$, $\rmF^\times$ and the maximal compact subtorus of the center of $\bfG''$. Then $\bfG'$ is of the form $\Res_{\rmF/\bbQ}\bfH'$ for $\bfH'$ a group over $\rmF$ which is tamely ramified at all places lying above $p$, and the morphism $\bfG'\rightarrow \mathbf{GL}(V)$ arises from a morphism of $\rmF$-group schemes  $\bfH'\rightarrow \mathbf{GL}_{\rmF}(W\otimes_{\rmF}\rmK);$ here the subscript $\rmF$ means we consider automorphisms of $W\otimes_{\rmF}\rmK$ as an $\rmF$-vector space.  The vector space  $V$ is equipped with a Hodge structure of type $(0,-1), (-1,0)$  which arises from a homomorphism $h':\bbS\rightarrow \bfG'_{\bbR}$. We then obtain via \cite[Corollaire 2.3.3]{DeligneCorvallis}  a Shimura datum $(\bfG,X)$ with $\bfG\subset \bfG'$ and an alternating form $\psi:V\times V\rightarrow \bbQ$ such that there is a Hodge embedding $(\bfG,X)\rightarrow (\bfGSp(V),S^\pm)$. Explicitly, $\bfG$ is generated by $\bfG'^{\der}=\Res_{\rmF/\bbQ}\bfH^\sharp$, the maximal compact subtorus of $\bfZ_{\bfG'}$ and the scalars $\bbG_m$; equivalently, $\bfG$ is given by the neutral component  $(\bfG'\cap\mathbf{GSp}(V))^0$ of $\bfG'\cap\mathbf{GSp}(V)$.
\end{para}
\begin{para}\label{para: kernel multiplier} Now let $(\bfG,X)$ be a Shimura datum of Hodge type with $\bfG^{\ad}$ simple.  The center $\bfZ_{\bfG}$ of $\bfG$ splits over a CM field, and hence the largest compact subtorus $\bfZ_{\bfG,0}$ of $\bfZ_{\bfG}$ is defined over $\bbQ$.
	We let $\bfG^c$  denote the subgroup of $\bfG$ generated by $\bfG^{\der}$ and $\bfZ_{\bfG,0}$.
 Similarly, we let $\bfZ_{\bfG}^c$ denote the subgroup of $\bfZ_{\bfG}$ generated by $\bfZ_{\bfG^{\der}}$ and $\bfZ_{\bfG,0}$. As before, we let $G^c$ and $Z_{G}^c$ denote the base change of these groups to $\bbQ_p$.
	\begin{lemma}\label{lem: symp multiplier kernel}We  have exact sequences
		\[\xymatrix{1\ar[r]& \bfG^c\ar[r]&\bfG\ar[r] &\bbG_m\ar[r]&1}\]and 
						\[\xymatrix{1\ar[r]& \bfZ_{\bfG}^c\ar[r]&\bfZ_{\bfG}\ar[r] &\bbG_m\ar[r]&1},\]
		where the maps $\bfG\rightarrow \bbG_m$ and $\bfZ_{\bfG}\rightarrow\bbG_m$ are induced by the symplectic multiplier homomorphism induced by some (equivalently any) Hodge embedding for $(\bfG,X)$.
	\end{lemma}
	\begin{proof}Let $c:\bfG\rightarrow \bbG_m$ be the symplectic multiplier homomorphism associated to  some Hodge embedding $\iota$. Then it is clear that $\bfG^{\der}$ and $\bfZ_{\bfG,0}$ are contained in $\ker(c)$, and hence $\bfG^c$ and $\bfZ_{\bfG}^c$ are contained in $\ker(c)$.  
		
		Note that $\bfG$ is generated by $\bfG^c$ and  $w_h(\bbG_m)$.  By \cite[\S1.1.18]{DeligneCorvallis}),  $\bfG_{\bbR}^c$ contains $h(U^1)$, where $U_1=(\Res_{\bbC/\bbR}\bbG_m)^{\mathrm{Nm}_{\bbC/\bbR}=1}$ is the unit circle, and hence $\bfG^c$ contains $w_h(\mu_2)\subset h(U_1)$. Thus $\ker(c|_{w_h(\bbG_m)})=w_h(\mu_2)$ is contained in $\bfG^c$, and hence $\bfG^c=\ker(c)$, so that we obtain the first exact sequence.

For the second exact sequence, we have  $\bfZ_{\bfG}^c=\bfZ_{\bfG}\cap\bfG^c$ and hence $w_h(\mu_2)\subset \bfZ_{\bfG}^c$. Then since $\bfZ_{\bfG}$ is generated by $\bfZ_{\bfG}^c$ and $w_h(\bbG_m)$, it follows as above that $\ker(c|_{\bfZ_{\bfG}})=\bfZ_{\bfG}^c$.
		\end{proof}
\end{para}
\begin{para}We now introduce a technical condition on a Hodge embedding for $(\bfG,X)$ which is needed to ensure the assumptions of Theorem \ref{thm:mainGSp} are satisfied.	We assume the following property:
\[
\text{$(\dagger)$ \ \ \ $\bfG^c\cong\mathrm{Res}_{\rmF/\bbQ}\bfH^c$ for an $\rmF$-group $\bfH^c$ with $\bfH^{c,\ad}$ absolutely simple.}
\]

\begin{Definition}
Let $\iota:(\bfG,X)\rightarrow (\mathbf{GSp}(V),X^\pm)$ be a Hodge embedding. We say that $\iota$ is \emph{fundamental} 
if $V$ has the structure of an $\rmF$-vector space such that $i|_{\bfG^c}$ factors as $$\mathrm{Res}_{\rmF/\bbQ}\bfH^c\rightarrow \Res_{\rmF/\bbQ}\mathbf{GL}_{\rmF}(V)\rightarrow \mathbf{GL}(V)$$ where the first map arises via Weil restriction from a morphism of group schemes over $\rmF$, and the second map is restriction of structure. Here, $\mathbf{GL}_{\rmF}(V)$ denotes the group of $\rmF$-linear automorphisms of $V$.
\end{Definition}
If $(\bfG,X)$ satisfies $(\dagger)$ as above, and $\iota:(\bfG,X)\rightarrow (\mathbf{GS}(V),S^\pm)$ is any Hodge embedding, then we obtain a fundamental Hodge embedding $$\iota':(\bfG,X)\rightarrow (\mathbf{GSp}(V'),S'^\pm),$$ where $V'=V\otimes_{\bbQ}\rmF$ considered as an $F$-vector space equipped with the alternating form $\Tr_{\rmF/\bbQ}\circ(\psi\otimes\rmF)$, and $\iota'$ is the composition of $\iota$ with the diagonal map $\mathbf{GSp}(V)\rightarrow \mathbf{GSp}(V')$.

	Given such a fundamental Hodge embedding, we let $\bfH'$ denote the subgroup of $\mathbf{GL}_{\rmF}(V)$ generated by $\bfH^c$ and the homotheties $\rmF^\times$, and we set $\bfG':=\Res_{\rmF/\bbQ}\bfH'$. We thus have an inclusion $\bfG\subset \bfG'$, and the embedding $\bfG\rightarrow \mathbf{GSp}(V)$ extends to an embedding $\bfG'\rightarrow \mathbf{GL}(V)$, which arises via restriction of structure from an $\rmF$-morphism $\bfH'\rightarrow \mathbf{GL}(V)$. The Hodge type liftings discussed in the last subsection are easily seen to satisfy $(\dagger)$, and the Hodge embeddings constructed   are fundamental.  Morever, the definition of the groups $\bfH',\bfG'$ coincide in the two discussions. 
	\begin{lemma}\label{lem:form over F}
	Let $(\bfG,X)\rightarrow (\mathbf{GSp}(V),S^\pm)$ be a fundamental Hodge embedding.  Then	 the alternating form $\psi$ on $V$ may  be chosen to satisfy the following properties:
		\begin{altenumerate}
			\item $\psi$ is of the form $\mathrm{Tr}_{\rmF/\bbQ}\circ \Psi$, where $\Psi:V\times V\rightarrow \rmF$ is an $\rmF$-bilinear alternating form.
			\item The morphism $\bfH'\rightarrow \mathbf{GL}_{\rmF}(V)$ factors through an $\rmF$-morphism to $\mathbf{GSp}_{\rmF}(V,\Psi)$.
		\end{altenumerate} 
	\end{lemma}
	\begin{proof}
Let $\mathrm{Bil}_{\bfH^c}(V)$ denote the $\rmF$-vector space of $\bfH^c	$-invariant $\rmF$-bilinear maps $V\times V\rightarrow \rmF$. Then we have  an isomorphism $$\mathrm{Bil}_{\bfH^c}(V)\otimes_{\bbQ}\rmF \cong \prod_{\sigma:\rmF\rightarrow \bbR}\mathrm{Bil}_{\bfH^c_{\bbR,\sigma}}(V_{\bbR,\sigma}),$$ where $\mathrm{Bil}_{\bfH^c_{\bbR,\sigma}}(V_{\bbR,\sigma})$ is the $\bbR$-vector space of $\bfH^c_{\bbR,\sigma}(:=\bfH^c\otimes_{\rmF,\sigma}\bbR)$-invariant bilinear maps $V_{\bbR,\sigma}\times V_{\bbR,\sigma}\rightarrow \bbR$. We also have an isomorphism $$\bfG'_{\bbR}\cong \prod_{\sigma:\rmF\rightarrow \bbR}\bfH'_{\bbR,\sigma}.$$ Let $h\in X$; then considering $h$ as a morphism $h:\bbS\rightarrow \bfG'_{\bbR}$, we have $h=\prod_{\sigma:\rmF\rightarrow \bbR}h_\sigma$, for some $h_\sigma:\bbS\rightarrow \bfH'_{\rmF,\sigma}$. Then $h_\sigma(i)$ is a Cartan involution of $\bfH'_{\bbR,\sigma}/w_\sigma(\bbR)^\times$; here 	$w_\sigma:\bbG_m\rightarrow \bfH'_{\bbR,
		\sigma}$ is the weight homomorphism for $h_\sigma$.	We let $U_\sigma\subset \mathrm{Bil}_{\bfH^c_{\bbR,\sigma}}(V_{\bbR,\sigma})$ denote the subset consisting of polarizations forms $V_{\bbR,\sigma}\times V_{\bbR,\sigma}\rightarrow \bbR(-1)$ for $h_\sigma(i)$ in the sense of \cite[1.1.10]{DeligneCorvallis}. Then $U_\sigma$ is open and non-empty by \cite[1.1.18 (a)]{DeligneCorvallis}. 
		
		We choose $\Psi\in \mathrm{Bil}_{\bfH^c}\cap \prod_{\sigma\in \rmF}U_\sigma$. Then $\bfH'\rightarrow \mathbf{GL}_{\rmF}(V)$  factors through a  morphism $\bfH'\rightarrow \mathbf{GSp}_{\rmF}(V,\Psi)$. Moreover, if we set $\psi=\mathrm{Tr}_{\rmF/\bbQ}\circ\Psi$, then $\psi$ is a polarization form for $h(i)$ and the result follows.
	\end{proof}
	
\end{para}
\begin{para} We now prove the existence of the desired Hodge type liftings in the non-exceptional (NE) cases.

	\begin{prop}\label{prop: HT lift simple} Let $(\bfG_2,X_2)$ be a Shimura datum of abelian type with $\bfG_2^{\ad}\cong \mathrm{Res}_{\rmF/\bbQ}\bfH$ for $\bfH$ an absolutely simple group over $\rmF$ and $\calG_2^\circ$ a parahoric group scheme of $G_2$. Assume  $p>2$ and  that the pair $(G_2,\mu_{h_2})$ is (NE).

		Then there exists a Shimura datum $(\bfG,X)$ of Hodge type together with a central isogeny $\bfG^{\der}\rightarrow \bfG_{2}^{\der}$ which induces an isomorphism $(\bfG^{\ad},X^{\ad})\cong (\bfG_{2}^{\ad},X_{2}^{\ad})$. Moreover, $(\bfG,X)$ may be chosen to satisfy the following conditions.
		
		\begin{altenumerate}
			\item  $\bfG^{\der}\cong \Res_{\rmF/\bbQ}\bfH^{\sharp}$.
			\item Any prime $v_2|p$ of $\bfE_2$  splits in the composite $\bfE':=\bfE.\bfE_2$.
	
			\item 	$\bfG$ satisfies $(\dagger)$, and there exists a fundamental Hodge embedding $
			\iota:(\bfG,X)\rightarrow(\mathbf{GSp}(V),S^\pm)$, such that there is a stabilizer group scheme $\calG$ for $G$ lifting $\calG^\circ_2$ and a self-dual  lattice $\Lambda\subset V_{\bbQ_p}$  such that  $\iota$ extends to a very good Hodge embedding $$(\calG,\mu_h)\rightarrow (\GL(\Lambda),\mu_d).$$

			\item $Z_G^c$ is a quasi-tame torus, and  $X_*(Z_G^c/Z_{G^{\der}})_I$ is torsion free, where $I$ is the inertia subgroup of $\Gal(\overline{\bbQ}_p/\bbQ_p)$.

		\end{altenumerate}
	\end{prop}
	\begin{proof} We follow the proof of \cite[Lem. 4.6.22]{KP}. We choose $S, \rmK$ and $T$ as in \S\ref{para: Deligne Hodge type lifting}. Then we obtain a Shimura datum $(\bfG,X)$ with  $\bfG^{\der}=\Res_{\rmF/\bbQ}\bfH^{\sharp}$ and hence (1) is satisfied. Moreover the choice of $T$ implies that any prime $v_2|p$ of $\bfE_2$ splits in $\bfE'$; thus (2) is satisfied.  As explained above,  $\bfG$ satisfies $(\dagger)$ and the  Hodge embedding $$(\bfG,X)\rightarrow (\bfGSp(V),S^\pm)$$ is fundamental, so the first part of (3) is satisfied. 
	
	To arrange so that condition (4) is satisfied, we argue as in \cite[Lem. 4.6.22]{KP}. Note that we have a containment of $\rmF$-groups $\bfH^c\subset \bfH'$, and so by the discussion in \S\ref{para: Deligne Hodge type lifting}, $\bfH^c$ splits over an extension which is tamely ramified at all $p$-adic places of $\rmF$. In particular $G^c$ is quasi-tame. Let $\fkp_1,\dotsc,\fkp_r$ denote the primes of $\rmF$ above $p$ and  $F_i$ the completion of $\rmF$ at $\fkp_i$. We set $H'_i:=\bfH'_{F_i}$ and let $S'_i\subset H'_i$ be the centralizer of a maximal $\breve F_i$-split torus.
	 Arguing as in \cite[Prop. 2.2.4]{KisinJAMS}, we may choose a maximal torus $\bfS'$ in $\bfH'$ such that the following two conditions are satisfied:
	\begin{enumerate}
		\item $\bfT':=\mathrm{Res}_{\rmF/\bbQ}\bfS'\subset \bfG'$ contains the image of some $h\in X$.
		\item $\bfS'_{\rmF_i}$ is $H_i(F_i)$ conjugate to $S'_i$.
	\end{enumerate}
	Let $\bfT=\bfG\cap \bfT'$ which is a maximal torus in $\bfG$. Then its maximal compact subtorus $\bfT_0$ is of the form $\Res_{\rmF/\bbQ}\bfS_0$ for an $\rmF$-torus $\bfS_0$, and its base change to $\bbQ_p$ is quasi-tame.
As in \cite[Lemma 4.6.22]{KP}, we set $\bfG_1=\bfG\times^{\bfZ_{\bfG}}\bfT$  and  let $X_1$ be the $\bfG_{1,\bbR}$-conjugacy class of Deligne homomoprhisms of $\bfG_1$ induced by $h\times 1$. 
	As in \emph{loc. cit.}, $(\bfG_1,X_1)$ is of Hodge type and  satisfies condition (1) and (2). We also have $$\bfG_1^c=\bfG^c\times^{\bfZ_{\bfG,0}}\bfT_0=\mathrm{Res}_{\rmF/\bbQ}\bfH_1^c$$ for some $\rmF$-group $\bfH_1^c$ and hence $\bfG_1$ satisfies ($\dagger$).   By construction, we have $\bfZ_{\bfG_1}=\bfT$ and  $\bfZ_{\bfG_1^{\der}}=\bfZ_{\bfG^{\der}}\subset \bfT$. It follows that $\bfZ_{\bfG_1}^c=\bfT_0$ and hence $Z_{G_1}^c$ is a quasi-tame torus. Upon replacing $(\bfG,X)$ by $(\bfG_1,X_1)$, we may assume $Z_{G}^c$ is a quasi-tame torus.
	
	We may further modify $(\bfG,X)$ as in \cite[Lemma 4.6.22]{KP} to ensure that in addition $X_*(Z_G^c/Z_{G^{\der}})_I$ is torsion free.  The modification in \emph{loc. cit.} is given by $\bfG_1=(\bfG\times \bfT'\times\bfT'')/(\bfZ_{\bfG^{\der}}\times \bfZ_{\bfG,0})$ for certain  tori $\bfT'$ and $\bfT''$ which are Weil restrictions of $\rmF$-tori whose base change to $\bbQ_p$ are quasi-tame. In particular $\bfG_1^c=(\bfG^c\times \bfT'\times\bfT'')/(\bfZ_{\bfG^{\der}}\times \bfZ_{\bfG,0})$ is the Weil restriction of an $\rmF$-group and hence satisfies $(\dagger)$. The other previously arranged conditions continue to be satisfied as in \cite[Lemma 4.6.22]{KP}. We may therefore assume that $(\bfG,X)$ satisfies (1), (2),  (4) and the condition $(\dagger)$.

	It remains to verify the last part of (3). 	We fix a fundamental Hodge embedding $\iota:(\bfG,X)\rightarrow (\mathbf{GSp}(V),S^\pm)$, so that $V$ is a vector space over $\rmF$. By Lemma \ref{lem:form over F}, we may assume the alternating form $\psi$ on $V$ is of the form $\mathrm{Tr}_{\rmF/\bbQ}\circ \Psi$ for $\Psi:V\times V\rightarrow\rmF$ an alternating $\rmF$-bilinear form on $V$, and that $\bfG'\rightarrow \mathbf{GL}(V)$ arises from an morphism $\bfH'\rightarrow \mathbf{GL}_{\rmF}(V)$ via restriction of structure.
	
	Let $\mathbf{x}\in \calB(G,\bbQ_p)$ be a point which is generic in its facet and  whose image in $\calB(G^{\ad},\bbQ_p)$ is  the image of a point  $\bx_2\in \calB(G_2,\bbQ_p)$ corresponding to  $\calG^\circ_2$. We let $\calG=\calG_{\bx}$ be the associated stabilizer group scheme.
	As above, let $H'_i=\bfH'_{F_i}$. Then we have $G\subset G'\cong \prod_{i=1}^r\Res_{F_i/\bbQ_p}H'_i$. Since $\iota$ is fundamental, and by our assumption on $\psi$, the conditions of Theorem \ref{thm:mainGSp} are satisfied (up to modifying the local forms $\Psi_{F_i}:V_{F_i}\times V_{F_i}\rightarrow F_i$ by the different). 
Condition (4) and Lemma \ref{lem: symp multiplier kernel} imply that $Z_G$ is an $R$-smooth torus (cf. Proposition \ref{prop: R-smoothness properties}), and hence $G$ is $R$-smooth by Lemma \ref{lem: Z R-sm implies T R-sm} below. Thus  by Theorem \ref{thm:mainGSp}, $\iota$ extends to very good Hodge embeddings $(\calG,\mu_h)\rightarrow (\GL(\calL),\mu_d),$ $(\calG,\mu_h)\rightarrow (\GL(\calL^\vee),\mu_d)$ for some lattice chain $\calL$ in $V_{\bbQ_p}$,   and the direct sum $(\calG,\mu_h)\rightarrow (\GL(\calL\oplus \L^\vee),\mu_{2d})$ is also very good.

We can choose the determining segments for $\L$ and $\L^\vee$ so that $\mathrm{tot}(\L^\vee)$ is a lattice in $V^r_{\Q_p}$ which is obtained from the dual $\La'^\vee$ of $\La':=\mathrm{tot}(\L)$ by permuting the constituent direct summands. Here $\Lambda'^\vee$ is the dual of $\La'$ 
with respect to the alternating form on $V^r_{\Q_p}$ given by the sum of $\psi$. It follows, by using Lemma \ref{diagonal}, that $(\calG,\mu_h)\rightarrow (\GL(\Lambda'),\mu_{rd})$,
$(\calG,\mu_h)\rightarrow (\GL(\Lambda'^\vee),\mu_{rd})$ are very good and a similar argument shows that  $(\calG,\mu_h)\rightarrow (\GL(\La'\oplus \La'^\vee),\mu_{2rd})$ is also very good.

 In order to obtain an embedding into a self-dual lattice, we apply Zarhin's trick \cite{Zarhin}. Thus we  replace $\iota$ by $\iota^{8r}$ and set $\Lambda=\Lambda'^4\oplus\Lambda'^{\vee,4}\subset V^{8r}$. Then the group-theoretic formulation of Zarhin's trick implies that there is an alternating form  on $V^{8r}$ for which $\Lambda$ is self-dual, we refer to \cite[\S4.5.9]{Madapusi-v2} for the explicit description of this form. The embedding  $\iota$ extends to a very good Hodge embedding $(\calG,\mu)\rightarrow (\GL(\Lambda),\mu_{8rd})$ by Lemma \ref{diagonal} and the above.
		\end{proof}
			\begin{lemma}\label{lem: Z R-sm implies T R-sm}
			Let $p>2$ and $(\bfG,X)$ a Shimura datum of abelian  type, and let $\bfZ_{\bfG}$ denote the center of $\bfG$. Suppose  $Z_G$ is an $R$-smooth torus.  Then $G$ is $R$-smooth.
		\end{lemma}
		\begin{proof}
			If $T$ is the centralizer of a maximal $\brQ$-split torus, then we have an exact sequence \[\xymatrix{1\ar[r] & Z_G\ar[r] & T\ar[r] & T^{\ad}\ar[r] &1}\] where $T^{\ad}$ is the image of $T$ in $G^{\ad}$. Since $(\bfG,X)$ is abelian type, $G^{\ad}$ is quasi-tame, cf. Remark \ref{rem: standard assumption}, and hence  $T^{\ad}$ is quasi-tame. Thus $T^{\ad}$ is  $R$-smooth by Proposition \ref{prop: R-smoothness properties} (1), and since $Z_G$ is $R$-smooth, $T$ is $R$-smooth by Proposition \ref{prop: R-smoothness properties} (2).
		\end{proof}
	\end{para}
	\begin{para}
We now use the results of \S\ref{ss:DnH}, \ref{ss:A_n} to deduce corresponding results in the exceptional type $A$ and type $D_n^{\bbH}$ cases. We first need the next two lemmas, which apply to general reductive groups over $\bbQ_p$.	 

\begin{lemma}\label{lem: intersection vg} Let $(G',\{\mu'\},\calG')$ be a local model triple and $(\calG',\mu')\rightarrow (\GL(\Lambda),\mu_d)$ a very good local Hodge embedding with $\Lambda_{\bbQ_p}=V$, and suppose $V$ is equipped with an  alternating perfect bilinear form $\psi$. Let $G$ be the neutral component of $G'\cap \GSp(V)$ and assume $G$ is $R$-smooth. Assume $G^{\der}\cong G'^{\der}$ and $\mu'$ arises from a cocharacter $\mu$ of $G$. Let $\Gg$ be the stabilizer group scheme of $G$ that corresponds to $\Gg'$. Assume in addition that $\La$ is a self-dual  lattice for $\psi$, i.e. $\La=\La^\vee$, and that the scheme theoretic intersection
	$\calG'\cap \GSp(\La)$ is smooth.  Then the embedding $(\calG,\mu)\rightarrow (\GL(\Lambda),\mu_d)$ is very good.
	
\end{lemma}
\begin{proof}
	By $R$-smoothness of $G$ and Proposition \ref{prop: R-smoothness properties} (3), $G\hook G'$ extends to  a closed immersion $\Gg\hook\Gg'$. Since $\La$ is self-dual, the parahoric $\GSp(\La )$ is reductive over $\Z_p$ and is the closed subgroup scheme
	of $\GL(\La)$ given as the Zariski closure of $\GSp(V)$ in $\GL(\La)$. Hence, the smooth  $\ti\calG:= \calG'\cap \GSp(\La )$ contains the   Zariski closure   of $G$ in $\GL(\La)$ which is $\calG$.   Then $\calG$ is a union of connected components of $\ti\calG$. The result now follows from Prop. \ref{neutral} and Theorem \ref{thm:main} applied to the (local) Hodge embedding given by $\GSp(V)\hook \GL(V)$.
\end{proof}

\begin{lemma}\label{lem:sp smooth}
	Suppose  that $\Gg$ is a smooth group scheme over $\Z_p$ and $\Gg\hook \GSp(\La)$  is a closed immersion, where $\La=\La^\vee$. Suppose $p>2$ and $\Gg$ contains the central diagonal torus  ${\rm diag} :\Gm\hook \GSp(\La)$. Then  the similitude $c: \Gg\xrightarrow{\ }\Gm$ is a smooth morphism.
\end{lemma}

\begin{proof} Since  $c({\rm diag}(\lambda))=\lambda^2$, the sequence
	\[\xymatrix{1\ar[r]& \ker(c)\ar[r]&\calG\ar[r]^c &\Gm\ar[r]&1}\]
	is fppf exact. Its pull-back by the \'etale $[2]: \Gm\xrightarrow{x\mapsto x^2}\Gm$ gives a split exact sequence. 
	If $\ti\Gg=\Gg\times_{\Gm, [2]}\Gm$ is the fiber product, then $\ti\Gg\to \Gg$ is \'etale and so $\ti\Gg$ is also smooth. The base change of $c$
	by $[2]$ is the split projection  $\ti\Gg\to \Gm$, hence it is smooth. By \'etale descent $c$ is smooth.
\end{proof}

\end{para}
\begin{para}We now assume $(\bfG_2,X_2)$ is a Shimura datum of abelian type with $\bfG^\ad_2=\mathrm{Res}_{\rmF/\bbQ}\bfH$ simple.
		
		\begin{prop}\label{prop: HT lift A/D}Assume that   either:
			
			\begin{altenumerate}	\item $(G_2^{\ad},\mu_2^{\ad})$ contains a simple factor of type $D_n^{\bbH}$.
				\item $G_2^{\ad}$ contains a simple factor of type $A$ of the form $\Res_{F/\bbQ_p}\mathrm{PGL}_m(D)$, with $D$ a central division $F$-algebra of index divisible by $p$.
	
			\end{altenumerate}
			Then the conclusion of Proposition \ref{prop: HT lift simple} holds, apart from $X_*(Z_G^c/Z_{G^{\der}})_I$ being torsion free in case (2).
			\end{prop}
			
			\begin{proof}		
	We choose $S,\rmK$ and $T$ as in \S\ref{para: Deligne Hodge type lifting} and let $(\bfG,X)$  be the Shimura datum thus obtained with $\bfG^{\der}=\Res_{\rmF/\bbQ}\bfH^{\sharp}$.  
				As before, properties (1) and (2)  are satisfied and there is a  fundamental Hodge embedding $\iota:(\bfG,X)\rightarrow (\bfGSp(V),S^\pm)$. As before, we choose the alternating form $\psi$ to be given by $\Tr_{\rmF/\bbQ}\circ\Psi:V\times V\rightarrow \bbQ$. We now verify the remaining properties.
				
				Let $\fkp_i,i=1,\dotsc,r$ denote the primes of $\rmF$ lying above $p$ and $F_i:=\rmF_{\fkp_i}$ the completion of $F$ at $\fkp_i$. As before, $G'\rightarrow \GL(V_{\bbQ_p})$ arises as a product of representations $$\rho_i:G_i':=\mathrm{Res}_{F_i/\bbQ_p}H_i'\rightarrow \GL(V_i)$$ where $H_i'=\bfH_{F_i}$.  Let $\mu_{i}'$ denote the factor of $\mu'$ in $G_i':=\Res_{F_i/\bbQ_p}H_i'$. The alternating form $\Psi_{\bbQ_p}$ decomposes as a sum of forms $\Psi_i:V_i\times V_i\rightarrow F_i$.
				
 			(1) Type $D_n^{\bbH}$. Recall that $\rmK_S=\rmF$ and $\rmK$ is a CM extension of $\rmF$. Thus $\bfZ_{\bfG}$ is generated by  $\bfZ_{\bfG^{\der}}$, $(\Res_{\rmK/\bbQ} \bbG_m)^{\mathrm{Nm}_{\rmK/\rmF}=1}$ and $\bbG_m$ considered as subgroups of $\mathbf{GL}(V)$, and its maximal compact subtorus $\bfZ_{\bfG,0}$ is given by $(\Res_{\rmK/\bbQ} \bbG_m)^{\mathrm{Nm}_{\rmK/\rmF}=1}$. We find that $\bfZ_{\bfG^{\der}}=\Res_{\rmF/\bbQ}\mu_2\subset (\Res_{\rmK/\bbQ} \bbG_m)^{\mathrm{Nm}_{\rmK/\rmF}=1}$, and hence  $$\bfZ^c_{\bfG}\cong (\Res_{\rmK/\bbQ} \bbG_m)^{\mathrm{Nm}_{\rmK/\rmF}=1}.$$
Since $\rmK/\rmF$ is split at all primes lying above $p$,  we have $Z_G^c=\prod_{i=1}^r\Res_{F_i/\bbQ_p}\bbG_m$ is a quasi-tame torus, and $Z_{G^{\der}}$ is identified with the subgroup $\prod_{i=1}^r\Res_{F_i/\bbQ_p}\mu_2$. Then we have $$Z_{G^c}/Z_{G^{\der}}\cong\prod_{i=1}^r\Res_{F_i/\bbQ_p}\bbG_m$$  and hence $X_*(Z_{G}^c/Z_{G^{\der}})_I$ is torsion free so that (4) is satisfied. It remains to verify the last part of (3). 

We first show each $\rho_i:G_i'\rightarrow \GL(V_i)$ extends to a very good Hodge embedding $(\calG'_i,\mu_i')\rightarrow(\GL(\calL),\mu_d)$ for $\calL$ a self-dual lattice chain. We may also restrict to those factors for which $\mu_i$ is non-trivial as otherwise the local model is 0-dimensional.
Thus we	may assume $G'^{\der}_i\cong \Res_{F_i/\bbQ_p}\mathrm{SO}^+(V^{\mathrm{st}}_i)$  in case \ref{par:DnH} (a)  or  $\Res_{F_i/\bbQ_p}\mathrm{SU}^+(W^{\mathrm{st}},\varphi)$ in case \ref{par:DnH} (b). By our choice of $S$, we have $G_i'$ is isomorphic to the group $G_1$ considered in
 \ref{para:modification D_n}, and the representation $\rho_i:G_i'\rightarrow \GL(V_i)$ is a direct sum of the representation denoted $\sigma$ in loc. cit.. The  discussion in \cite[2.2]{Satake} implies that the alternating form $\Psi_i$  is of the form considered in \ref{par:DnH}. Thus the result follows by Corollary \ref{cor:DnH}. 

By an argument as in the proof of  Proposition \ref{prop: HT lift simple}, upon replacing $\iota$ by $\iota^{8r}$, we obtain a Hodge embedding and a self-dual lattice $\Lambda\subset V$ for which $\iota$ extends to a very good Hodge embedding $(\calG',\mu')\rightarrow (\GL(\Lambda),\mu_d)$. Note that, by the construction, the lattice $\Lambda$ with its alternating form is obtained as a direct sum of lattices $\Lambda_i$ with forms ${\rm Tr}_{F_i/\bbQ_p}\circ \Psi_i$, for $i=1,\ldots, r$.  Then the scheme theoretic intersection $ \calG'\cap\mathrm{GSp}(\Lambda)$ arises from the pullback of the map
 \begin{equation}\label{eqn: product similitude}
 \calG'=\prod_{i=1}^r\calG_i'\rightarrow \prod_{i=1}^r \Res_{\calO_{F_i}/\bbZ_p}\bbG_m
 \end{equation} 
 induced by the product of the similitude factors of the forms $ \Psi_i$, along the diagonal map
  $$
  \bbG_m\rightarrow \prod_{i=1}^r \Res_{\calO_{F_i}/\bbZ_p}\bbG_m.
  $$
   By Lemma \ref{lem:sp smooth}, the map in \eqref{eqn: product similitude} is smooth, and hence the intersection $\calG'\cap\mathrm{GSp}(\Lambda)$ is smooth. The result then follows from Lemma \ref{lem: intersection vg}.

		(2) Type $A$.  
			 Recall that $\rmK_S$ and $\rmK$ are disjoint CM extensions of $\rmF$. Then the center $\bfZ_{\bfG}$ is generated by  $\bfZ_{\bfG^{\der}}$, $(\Res_{\rmK/\bbQ} \bbG_m)^{\mathrm{Nm}_{\rmK/\rmF}=1}$, $(\Res_{\rmK_S/\bbQ} \bbG_m)^{\mathrm{Nm}_{\rmK_S/\rmF}=1}$ and the scalars $\bbG_m$ as subgroups of $\mathbf{GL}(V)$.
			 The maximal compact subtorus $\bfZ_{\bfG,0}$ is generated by $(\Res_{\rmK/\bbQ} \bbG_m)^{\mathrm{Nm}_{\rmK/\rmF}=1}$ and  $(\Res_{\rmK_S/\bbQ} \bbG_m)^{\mathrm{Nm}_{\rmK_S/\rmF}=1}$.
			 We find that $$\bfZ_{\bfG^{\der}}=(\Res_{\rmK_S/\rmF}\mu_n)^{\mathrm{Nm}_{\rmK_S/\rmF}=1}\subset (\Res_{\rmK_S/\bbQ}\bbG_m)^{\mathrm{Nm}_{\rmK_S/\rmF}=1}$$ and hence  $$\bfZ_{\bfG}^c\cong(\Res_{\rmK/\bbQ} \bbG_m)^{\mathrm{Nm}_{\rmK/\rmF}=1}\times^{\Res_{\rmF/\bbQ}\mu_2}(\Res_{\rmK_S/\bbQ} \bbG_m)^{\mathrm{Nm}_{\rmK_S/\rmF}=1}.$$ 
			 Thus $Z_G^c$ is a quasi-tame torus since $p$ is odd.
	 
It remains to verify the last part  of (3). As in case of type $D_n^{\bbH}$, we first show that $\rho_i:G_i'\rightarrow \GL(V_i)$ extends to a very good Hodge embedding $(\calG'_i,\mu_i')\rightarrow(\GL(\calL),\mu_d)$ for $\calL$ a self-dual lattice chain. It suffices to consider those cases which are not covered by Theorem \ref{thm:main}. Thus	we may assume $G_i'^{\der}\cong \SL_{m_i}(D_i)$ as in \S\ref{ss:A_n}; we also assume $\mu_i'$ is non-trivial as otherwise the local model at that place is $0$-dimensional. Our choice of $S$ implies that there is an inclusion $G_1\subset G_i'$, where $G_1$ is the group considered in \ref{para: modification A-n}, and $V_i|_{G_1}$ is a sum of the representation denoted $\rho_1$ in \ref{cor:PELdiv}. Moreover, $\mu_i'$ factors through $G_1$, and $\Psi_i$ is of the form given in \ref{para: modification A-n} by \cite[2.1]{Satake}. The result then follows from Lemma \ref{lem: intersection vg} and  Corollary \ref{cor:PELdiv}. The rest follows as in case (1).
	\end{proof}

\end{para}
\begin{para}
	We now relax the assumption that $\bfG_2^{\ad}$ is  $\bbQ$-simple. 	The following is a generalization and refinement of \cite[Lem. 4.6.22]{KP}. 
	
	\begin{prop}\label{prop: HT lift general}
	Let $p>2$. Let $(\bfG_2,X_2)$ be a Shimura datum of abelian type and $\calG^\circ_2$ a parahoric of $G_2$.  
Then there exists a Shimura datum $(\bfG,X)$ of Hodge type together with a central isogeny $\bfG^{\der}\rightarrow \bfG^\der_{2}$ which induces an isomorphism $(\bfG^{\ad},X^{\ad})\cong (\bfG^{\ad}_{2},X^{\ad}_{2})$. Moreover, $(\bfG,X)$ may be chosen to satisfy the following conditions.

		\begin{altenumerate}

			\item $\pi_1(G^{\der})$ is a $2$-group and is trivial if $(\bfG_2^{\ad},X_2^{\ad})$ has no factors of type $D^{\bbH}$.  Moreover $(\bfG,X)$ satisfies assumption (D) of \S\ref{para: ab type reduction assumptions}.
			\item Any prime $v_2|p$ of $\bfE_2$  splits in the composite $\bfE':=\bfE.\bfE_2$.
			\item $Z_G$ is an $R$-smooth torus with $Z_G^c$ quasi-tame.
			\item $(\bfG,X)$ admits a Hodge embedding $$\iota:(\bfG,X)\rightarrow (\mathbf{GSp}(V),S^\pm)$$ which extends to a very good local Hodge embedding $(\calG,\mu)\rightarrow(\GL(\Lambda),\mu_d)$ for $\calG$ a stabilizer group scheme of $G$ lifting $\calG_2$ and  $\Lambda\subset V_{\bbQ_p}$ is a self-dual lattice.

		\end{altenumerate}
		In particular, the Shimura datum $(\bfG,X)$ satisfies Assumptions (A)--(E) of \S\ref{subsec: integral models}.

			If moreover, $G_2^{\ad}$ does not contain a simple factor involving division algebras with index divisible by $p$, then $(\bfG,X)$ may be chosen in addition to satisfy the property that $X_*(G^{\ab})_I$ is torsion free.

	\end{prop}
	\begin{proof}
		We write $(\bfG_2^{\ad},X^{\ad})=\prod_{i=1}^r(\bfG_2^{(i)},X_2^{(i)})$ where each $\bfG^{(i)}$ is $\bbQ$-simple.  For each $i=1,\dotsc,r$ we let $(\bfG^{(i)},X^{(i)})$ be the lifting constructed in Proposition \ref{prop: HT lift simple} if $(G_2^{(i)},\mu_{h_2}^{(i)})$ is (NE), and that constructed in Proposition \ref{prop: HT lift A/D} if $(G_2^{(i)},\mu_{h_2}^{(i)})$ contains factors of exceptional type $A$ or $D$. These are equipped with Hodge embeddings 
		$$
		(\bfG^{(i)},X^{(i)})\rightarrow (\bfGSp(V^{(i)}),S^{(i),\pm})
		$$
		 which extend to very good local Hodge embeddings $(\calG^{(i)},\mu_h^{(i)})\rightarrow (\GL(\Lambda^{(i)}),\mu^{(i)}_{d_i})$ where $\Lambda^{(i)}$ is a self dual lattice in $V^{(i)}_{\bbQ_p}$ and $\calG^{(i)}$ is a stabilizer scheme lifting the corresponding factor of the parahoric $\calG_2^{\ad}$ of $G_2^{\ad}$ corresponding to $G_2$. We let $c^{(i)}: \bfG^{(i)}\rightarrow\bbG_m$ denote the symplectic multiplier homomorphism. 
		
		We set 
		$$
		\bfG'=\prod_{i=1}^r\bfG^{(i)},\ \ \ \bfG:=(\prod_{i=1}^r\bfG^{(i)})\times_{\bbG_m^r}\bbG_m,
		$$ where $\prod_{i=1}^r\bfG^{(i)}\rightarrow \prod_{i=1}^r\bbG_m$ is given by the product of $c^{(i)}$, and $\bbG_m\rightarrow \prod_{i=1}^r\bbG_m$ is the diagonal embedding.   Then $\bfG$ is an extension of $\bbG_m$ by the  group $\prod_{i=1}^r\bfG^{(i),c}$ (cf. \S\ref{para: kernel multiplier}) and hence $\bfG$ is a connected reductive group.
	 If $h\in \prod_{i=1}^r X^{(i)}$, then $h$ factors through $\bfG$ and we let $X$ be the $\bfG_{\bbR}$ conjugacy class of $h$. We thus obtain a Shimura datum $(\bfG,X)$.
	 
	 Let $V=\oplus_{i=1}^rV^{(i)}$ equipped with the alternating form given by the direct sum of those on $V^{(i)}$.
	Then we obtain a Hodge embedding $\iota:(\bfG,X)\rightarrow (\bfGSp(V),S^\pm)$, which arises from a morphism $\rho':=\bfG'\rightarrow \mathbf{GL}(V)$.  This extends to a very good Hodge embedding $(\calG',\mu')\rightarrow (\GL(\Lambda),\mu_d)$, where $\Gg'=\prod_{i=1}^r\Gg^{(i)}$ and $\Lambda=\oplus_{i=1}^r\Lambda^{(i)}$ is a self dual lattice in $V_{\bbQ_p}$. We have closed immersions $\Gg^{(i)}\hook \GSp(\La^{(i)})$ and
		\[
		\Gg'\cap \GSp(\La)={\prod_{i=1}^r\Gg^{(i)}}\times_{ \bbG_m^r} \bbG_m
\]
where, in the fiber product, $\Gg'=\prod_{i=1}^r\Gg^{(i)}\to  \bbG_m^r $ is the product of the similitudes and $\bbG_m\to  \prod_{i=1}^r{\bbG_m^r}$ is the diagonal. 
We now see   that $\Gg'\cap \GSp(\La)$ is smooth, since by Lemma \ref{lem:sp smooth} the above fiber product is smooth.

		It now follows by Lemma \ref{lem: intersection vg}, that we obtain a very good Hodge embedding $(\calG,\mu)\rightarrow(\GL(\Lambda),\mu_d)$, and so we obtain (4). Property (1) follows  since   $\bfC=\ker(\bfG^{\mathrm{sc}}\rightarrow \bfG^{\der})$ is isomorphic to a product of groups of the form $\Res_{\rmF/\bbQ}\mu_2$ for $\rmF/\bbQ$ totally real, with non-trivial factors coming from simple factors of type $D^{\bbH}$. Property  (2) follows by the corresponding property for each $(\bfG^{(i)},X^{(i)})$.
		By assumption each $Z^c_{G^{(i)}}$ is a quasi-tame torus. Thus by Lemma \ref{lem: symp multiplier kernel}, $Z_{G}$ is an extension of  $\bbG_m$ by the quasi-tame torus $\prod_{i=1}^rZ^c_{G^{(i)}}$, and hence $Z_G$ is $R$-smooth giving property (3).
		
		Conditions (1)--(4) immediately implies Assumptions (A)--(E). (A) is satisfied by definition and (E) follows from (3). (B) follows from from (1), (3) and Lemma \ref{lem: Z R-sm implies T R-sm}. (C) follows from (4) and (D) is part of condition (1).

If in addition $G_2^{\ad}$ does not contain a simple factor involving division algebras with index divisible by $p$, then we have $$X_*(Z_G^c/Z_{G^{\der}})_I=\prod_{i=1}^rX_*(Z^c_{G^{(i)}}/Z_{G^{(i),\der}})_I	$$   is torsion free. Since $X_*(G^{\ab})_I$ is an extension of $\bbZ$ by $X_*(Z_G^c/Z_{G^{\der}})_I$, it is torsion free.
	\end{proof}
\end{para}

\begin{para}\label{para: main thm}
Combining \ref{prop: HT lift general} and Proposition \ref{prop: auxiliary SV construction} we obtain the main result on the existence of local model diagrams for Shimura varieties of abelian type.

\begin{thm}\label{thm: LMD abelian type}
Assume $p>2$.	Let $(\bfG_2,X_2)$ be a Shimura datum  of abelian type and $\rmK_{2,p}^\circ=\calG_2^\circ(\bbZ_p)$ a parahoric subgroup. There exists a pro-system of $\calO_{E_2}$-schemes   $\mathscr{S}_{\rmK^\circ_{2,p}\rmK_2^p}(\bfG_2,X_2)$ with generic fibers $\Sh_{\rmK_{2,p}^\circ\rmK_2^p}(\bfG_2,X_2)$ and with finite \'etale transition maps, for varying sufficiently small $\rmK_2^p\subset \G_2({\mathbb A}^p_f)$, such that the $\calO_{E_2}$-scheme
\[	\mathscr{S}_{\rmK^\circ_{2,p}}(\bfG_2,X_2)=\varprojlim_{\rmK_2^p}\SSh_{\rmK^\circ_{2,p}\rmK_2^p}(\bfG_2,X_2)
	\]
	with $\bfG_2(\bbA^p_f)$-action extends  $\Sh_{\rmK^\circ_{2,p}}(\bfG_2,X_2)=\varprojlim_{\rmK_2^p}\Sh_{\rmK^\circ_{2,p}\rmK_2^p}(\bfG_2,X_2)$ and satisfies
	\begin{altenumerate}
		\item For  $R$  a discrete valuation ring of mixed characteristic  $(0,p)$, the map 
		$$\mathscr{S}_{\rmK^\circ_{2,p}}(\bfG_2,X_2)(R)\rightarrow\Sh_{\rmK^\circ_{2,p}}(\bfG_2,X_2)(R[1/p])$$ 
		is a bijection.
		\item For $\rmK_2^p\subset \bfG_2(\bbA_f^p)$ a sufficiently small compact open subgroup, 
	$	\SSh_{\rmK^\circ_{2,p}\rmK_2^p}(\bfG_2,X_2)	$
	is \'etale locally isomorphic to the local model $\bbM^{\mathrm{loc}}_{\calG^\circ_2,\mu_{h_2}}$.
		\item There  exists a diagram 
		\begin{equation}\label{LMDthm7.2.16}\begin{aligned}
		\xymatrix{ &\widetilde{\mathscr{S}}^{\mathrm{ad}}_{\rmK^\circ_{2,p}}(\bfG_2,X_2)\ar[dr]^q\ar[dl]_\pi&\\
				\mathscr{S}_{\rmK^\circ_{2,p}}(\bfG_2,X_2) & &\bbM^{\mathrm{loc}}_{\Gg^\circ_2, \mu_{h_2}} }
				\end{aligned}\end{equation}
				where $\pi$ is a $\bfG_2(\bbA_f^p)$-equivariant ${\calG}^{\ad}$-torsor and $q$ is $\Gg^\ad$-equivariant, smooth of relative dimension				 $\dim \bfG^{\ad},$ and 
				$\bfG_2(\bbA_f^p)$-equivariant for the trivial $\bfG_2(\bbA_f^p)$-action on $\bbM^{\mathrm{loc}}_{\calG^\circ_2, \{\mu_{h_2}\}}$. If in addition $(G_2,\mu_{h_2})$ is (NE), then $\pi$ reduces to a ${\calG}^{\ad,\circ}$-torsor.
	\end{altenumerate}
\end{thm}
\begin{proof}
	Proposition \ref{prop: HT lift general} implies that we may choose $(\bfG,X)$ satisfying the assumptions of Proposition \ref{prop: auxiliary SV construction}, and so we obtain (1) and the first part of (3). If $(G_2,\mu_{h_2})$ is (NE), then we may choose $(\bfG,X)$ such that $X_*(G^{\ab})_I$ is torsion-free. The argument in the proof of \cite[Thm. 4.6.23]{KP} then shows that we may choose $\bx\in \calB(G,\bbQ_p)$ lifting $\bx^{\ad}$ such that $\calG=\calG^\circ$, and so the ``in addition" part follows. Part (2) follows formally from (3).
\end{proof}
\end{para}

\begin{para}
 Using recent work of Daniels--van Hoften--Kim--Zhang \cite{DvHKZ}, we can further relax the (NE) assumption in Theorem \ref{thm: LMD abelian type}. Since \emph{op. cit.} uses the theory of $p$-adic shtukas, we state this refinement as a separate corollary to make clear what can be done without resorting to this.
\begin{cor}\label{cor: torsor for connected gp} The $\calG^{\ad}$-torsor in Theorem \ref{thm: LMD abelian type} (3) can be refined to a $\calG^{\ad,\circ}$-torsor. This fits into a $\calG^{\ad,\circ}$-equivariant local model diagram refining (\ref{LMDthm7.2.16}). 
	
\end{cor}

\begin{proof}We choose $(\bfG,X)$ satisfying the assumptions of Proposition \ref{prop: auxiliary SV construction} as above. The local model diagram in Theorem \ref{thm: main SV Hodge} induces  a $\calG$-equivariant local model diagram for  $\SSh_{\rmK^\circ_{p}}(\bfG,X)$ by pullback.  By \cite[Proposition 4.3.3]{DvHKZ}, the corresponding $\calG$-torsor admits a reduction to a $\calG^\circ$-torsor, and hence we obtain a corresponding diagram	\begin{equation}\label{LMDthm Cor 7.2.24}\begin{aligned}
	\xymatrix{ &\widetilde{\mathscr{S}}^{\mathrm{ad}}_{\rmK^\circ_{p}}(\bfG,X)\ar[dr]^{q^\circ}\ar[dl]_{\pi^\circ}&\\
		\mathscr{S}_{\rmK^\circ_{p}}(\bfG,X) & &\bbM^{\mathrm{loc}}_{\Gg^\circ, \mu_{h}} } 
	\end{aligned}\end{equation}
	with $\pi^\circ$ a $\calG^{\circ}$-torsor. The construction in Proposition \ref{prop: auxiliary SV construction} then gives the desired refinement to a $\calG^{\ad,\circ}$-torsor for $\widetilde{\mathscr{S}}^{\mathrm{ad}}_{\rmK^\circ_{2,p}}(\bfG_2,X_2)\rightarrow \mathscr{S}^{\mathrm{ad}}_{\rmK^\circ_{2,p}}(\bfG_2,X_2)$.
\end{proof}
 \end{para}
% \addtocontents{toc}{\protect\setcounter{tocdepth}{1}}

 \subsection{Errata}\label{ss:Errata}
 
\begin{para} 1) Correction to the proof of \cite[Thm. 4.2.7]{KP}: The morphism $q^{\rm loc}$ is not a $\Gg$-torsor as stated there: Instead, it is isomorphic to the action morphism
$\Gg\times {\rm M}^{\loc}_{G,X}\to {\rm M}^{\loc}_{G, X}$. The action morphism is smooth since it is the composition of 
the isomorphism $\Gg\times {\rm M}^{\loc}_{G,X}\xrightarrow{\sim} \Gg\times {\rm M}^{\loc}_{G,X}$ given by $(g, m)\mapsto (g, g\cdot m)$ with the projection 
$\Gg\times {\rm M}^{\loc}_{G,X}\to {\rm M}^{\loc}_{G,X}$; the rest of the proof is the same.

2) Correction to the proof of \cite[Lem. 3.1.17]{KP}: The ring $\whW(A)[1/p]=\whW(A)\otimes_{\Z_p}\Q_p$  is not complete for the topology $\tau$ defined there and so proving $p^{-m}\phi^m(x)\to 0$ in $\tau$ is not enough to complete the proof (we thank M. Hoff for pointing this out). However, as we will show, $\whW(A)[1/p]$ is complete and separated for the $p$-adic topology and for $x\in \whW(\frakM_A)$, $p^{-m}\phi^m(x)\to 0$, in the $p$-adic topology. This is enough to complete the proof.  

Following  \cite[\S 2]{ZinkWindows} set $\N=\frakM_A$ which is a $p$-adic ring with no unit. Since $\frakM_A^N\subset pA$, for all $a\in \N/p\N$ we have $a^{N+1}=0$, and $\N$ is ``modulo $p$ bounded nilpotent'' in the terminology of \emph{loc. cit.}.  
We also have
\[
\whW(\calN)=\varprojlim_n \whW(\frakM_A/\frakM_A^n)=\varprojlim_n \whW(\calN/p^n\calN)\subset W(\calN).
\]
By \cite[Prop. 2.3, 2.4]{ZinkWindows}, $\whW(\N)$ is closed in $W(\N)$ and is $p$-adically complete and separated.
Since $\whW(A)=W(k)\oplus \whW(\N)
$ and $\whW(A)$ is $p$-torsion free, it follows that $\whW(A)[1/p]$ is $p$-adically complete and separated.

We now show that for $x\in \whW(\N)[1/p]$, $p^{-m}\phi^m(x)\to 0$, in the $p$-adic topology of $\whW(\N)[1/p]$. 
By \cite[Lem. 2.2]{ZinkWindows} the group $\whW(\N/p\N)$ is annihilated by a power of $p$. Hence, $p^a\cdot x\in \whW(p\N)$, for $a\gg 0$, and it is enough to assume $x\in  \whW(p\N)$. 
Since $p>2$ we can use Zink's logarithmic coordinates \cite[p. 35]{ZinkDisplay},   coming from the divided power structure on $p\N$:
There is a group homomorphism
\[
\log: \whW(p\N)\xrightarrow{\sim}     \widehat\bigoplus_{i\geq 0}   p\N \subset  \prod_{i\geq 0}  p\N,
\]
with  $\widehat\bigoplus $ signifying the subgroup of the product consisting of $z=[z_0,  \ldots , z_i,\ldots ]$, for which $z_i\to 0$, $p$-adically
(\cite{ZinkWindows}). By \cite[(49), p. 35]{ZinkDisplay}  the action of $p^{-m}\phi^m$  
on the target of $\log$ is given by 
\[
(p^{-m}\phi^m)([z_0, z_1,\ldots , z_i,\ldots ])=[z_m, z_{m+1},\ldots , z_{m+i},\ldots  ].
\]
Set $z=\log(x)$. Since $z_i\to 0$ in the $p$-adic topology of $\N$, this gives $p^{-m}\phi^m(z)\to 0$ in the $p$-adic topology of $\widehat\bigoplus_{i\geq 0} p\N\subset \prod_{i\geq 0} p\N$ and so $p^{-m}\phi^m(x)\to 0$  in the $p$-adic topology of $\whW(p\N)$.

3) Correction to \cite[Lemma 4.6.13]{KP} and \cite[Corollary 4.6.15]{KP}. The description of $\Sh_{\rmK_2^\circ}(G_2,X_2)$ in \cite[Lemma 4.6.13]{KP} is not correct; we thank Yu Luo and Peihang Wu for pointing this out. In the statement, the connected Shimura varieties appearing in the disjoint union should have different levels given by the conjugate of $\rmK_p^\circ$ by the element $j\in J$; here $J\subset G_2(\bbQ_p)$ is as in \cite[\S4.6.12]{KP}. More precisely, the correct description is as follows.

	 For each $j\in J$, let $\Sh_{\rmK_p^\circ}(G,X)^{j,+}$ denote the connected component of $\Sh_{\rmK_p^\circ}(G,X)$ containing $(j,X^+)$. Note that $\Sh_{\rmK_p^\circ}(G,X)^{j,+}$ is isomorphic to  $\Sh_{j\rmK_p^\circ j^{-1}}(G,X)^+$ and is independent of the choice of representative $j$.  Then taking the quotient of the isomorphism $$\Sh(G_2,X_2)\cong [\Sh(G,X)^+\times \mathscr{A}(G_2)]/\mathscr{A}(G)^\circ $$ by $\rmK_{2,p}^\circ$ gives  \begin{align*}
\Sh_{\rmK_{2,p}^\circ}(G_2,X_2)&\cong [\Sh(G,X)^+\times \mathscr{A}(G_2)/\rmK_{2,p}^\circ]/\mathscr{A}(G)^\circ) \\
&\cong \coprod_{j \in J}    [\Sh(G,X)^+\times j\mathscr{A}(G_{2,\bbZ_{(p)}})]/\tilde{\mathscr{A}}(G_{\bbZ_{(p)}})^\circ\\ 
&\cong\coprod_{j\in J}[\Sh_{\rmK_{p}^\circ}(G,X)^{j,+}\times \mathscr{A}(G_{2,\bbZ_{(p)}})]/\mathscr{A}(G_{\bbZ_{(p)}})^\circ. \end{align*}
Here the second isomorphism follows from the definition of $J$ and the last isomorphism follows from the fact that $$\mathrm{Stab}_{\tilde{\mathscr{A}}(G_{\bbZ_{(p)}})^\circ}(j\mathscr{A}(G_{2,\bbZ_{(p)}}))=j\rmK_p^\circ j^{-1}\cap \tilde{\mathscr{A}}(G_{\bbZ_{(p)}})^\circ$$ and  the fact that $$\Sh_{j\rmK_p^\circ j^{-1}}(G,X)^+\cong \Sh(G,X)^+/j\rmK_p^\circ j^{-1}\cap \tilde{\mathscr{A}}(G_{\bbZ_{(p)}})^\circ.$$

The corresponding construction of the integral model in \cite[Corollary 4.6.15]{KP} should then be $$\mathscr{S}_{\rmK_{2,p}^\circ}(G_2,X_2)=\coprod_{j\in J}[\mathscr{S}_{\rmK_{p}^\circ}(G,X)^{j,+}\times \mathscr{A}(G_{2,\bbZ_{(p)}})]/\mathscr{A}(G_{\bbZ_{(p)}})^\circ,$$ where $\mathscr{S}_{\rmK_{p}^\circ}(G,X)^{j,+}$ is the $\calO_{E^p}$-scheme given by the Zariski closure of $\Sh_{\rmK_p^\circ}(G,X)^{j,+}$ in 
$\mathscr{S}_{\rmK_p^\circ}(G,X)_{\calO_{E^p}}$ and $E^p={\rm E}^p_v$. The same proof as \cite[Corollary 4.6.15]{KP} now works, noting that this scheme has the correct generic fiber.
 \end{para}
 \begin{para}
 The assumption that $(\Gg,\mu)\hookrightarrow (\GL(\La),\mu_d)$ is very good as in Definition \ref{def:vgood},
 has to be added to the statements of the main results of \cite{PCan}.
More specifically, this condition has to be assumed for the constructions in \S 4.5, in Prop. 4.5.3, and for the results in \S 8 of \cite{PCan}. 
(\cite[Prop. 4.5.3]{PCan} asserts that the isomorphism $c$ respects the tensors, but the proof is based on the erroneous construction of $c$ in \cite[Lem. 3.1.9]{KP}; see the proof of Lemma \ref{319}.)
In particular, the independence of \cite[Thm. 8.1.6]{PCan} is for integral models constructed using different very good Hodge embeddings.
 \end{para}

 \end{document}